\pdfoutput=1

\documentclass[a4paper,10pt,oneside]{article}
\usepackage{latexsym,amsfonts,amssymb,amsmath}
\usepackage{amsthm}

\usepackage[hyperindex, breaklinks]{hyperref}

\usepackage{mathtools}
\usepackage{mathrsfs}
\usepackage{graphicx}
\usepackage[english]{babel}
\usepackage{diagrams}

\usepackage[hmargin=0.93in, vmargin=0.91in]{geometry}

\DeclareMathOperator{\rk}{rk}

\DeclareMathOperator{\codim}{codim}

\DeclareMathOperator{\Spec}{Spec}

\DeclareMathOperator{\Pic}{Pic}

\DeclareMathOperator{\sym}{sym}
\DeclareMathOperator{\alt}{alt}

\DeclareMathAlphabet{\mathpzc}{OT1}{pzc}{m}{it}
\DeclareMathOperator{\gr}{gr}
\DeclareMathOperator{\Grass}{Grass}
\DeclareMathOperator{\supp}{supp}
\DeclareMathOperator{\Res}{Res}

\newcommand{\mbb}{\mathbb}

\newcommand{\mc}{\mathcal}

\newcommand{\del}{\partial}

\newcommand{\FS}{\mathcal{O}}
\newcommand{\FSP}[1]{\mathcal{O}_{\mathbb{P}_{#1}}}
\newcommand{\comp}[1]{{#1}^{\bullet}}

\newcommand{\id}{{\rm id}}

\newcommand{\B}{\mathbf}

\newcommand{\Tor}{{\rm Tor}}

\newcommand{\perm}{\mathfrak{S}}

\newcommand{\trest}{{\big |} }

\newcommand{\bkrh}{\mathbf{\Phi}}

\newcommand{\transpose}[1]{{}^t\! {#1}}
\newcommand{\mf}{\mathfrak}

\newcommand{\Aut}{\mathrm{Aut}}

\newcommand{\tens}{\otimes}
\newcommand{\Tens}{\bigotimes}
\newcommand{\w}{\wedge}

\newcommand{\Stab}{\mathrm{Stab}}

\newcommand{\p}{\oplus}
\newcommand{\D}{\mathscr{D}}
\newcommand{\E}{\mathcal{E}}
\newcommand{\K}{\mathcal{K}}
\newcommand{\W}{\mathcal{W}}
\newcommand{\V}{\mathcal{V}}

\newcommand{\Q}{\mathcal{Q}}

\theoremstyle{plain}
\newtheorem{theorem}{Theorem}[section]
\newtheorem{lemma}[theorem]{Lemma}
\newtheorem{pps}[theorem]{Proposition}
\newtheorem{crl}[theorem]{Corollary}
\newtheorem*{expect}{Expectation}

\newtheorem{thma}{Theorem}[section]
\newtheorem{lemmaa}[thma]{Lemma}

\newtheorem{crla}[thma]{Corollary}

\theoremstyle{definition}
\newtheorem{definition}[theorem]{Definition}
\newtheorem{example}[theorem]{Example}
\newtheorem{remark}[theorem]{Remark}

\newtheorem{notat}[theorem]{Notation}
\newtheorem{conv}[theorem]{Convention}

\numberwithin{equation}{section}

\mathchardef\phi="0127
\mathchardef\varphi="011E

\mathchardef\alpha="710B
\mathchardef\beta="710C
\mathchardef\gamma="710D
\mathchardef\delta="710E
\mathchardef\epsilon="7122
\mathchardef\zeta="7110
\mathchardef\eta="7111
\mathchardef\theta="7112
\mathchardef\iota="7113
\mathchardef\kappa="7114
\mathchardef\lambda="7115
\mathchardef\mu="7116
\mathchardef\nu="7117
\mathchardef\xi="7118
\mathchardef\pi="7119
\mathchardef\rho="711A
\mathchardef\sigma="711B
\mathchardef\tau="711C
\mathchardef\upsilon="711D
\mathchardef\chi="711F
\mathchardef\psi="7120
\mathchardef\omega="7121
\mathchardef\varepsilon="710F
\mathchardef\vartheta="7123
\mathchardef\varpi="7124
\mathchardef\varrho="7125
\mathchardef\varsigma="7126

\theoremstyle{definition}

\newtheorem{remarka}[thma]{Remark}
\newtheorem{notata}[thma]{Notation}

\numberwithin{equation}{section}

\title{Higher symmetric powers of tautological bundles on Hilbert schemes of points on a surface}
\author{Luca Scala}
\date{}

\begin{document}
\maketitle


\begin{abstract}We study general symmetric powers $S^k L^{[n]}$ 
of a tautological bundle $L^{[n]}$ on the Hilbert scheme $X^{[n]}$ of $n$ points  over a smooth quasi-projective surface $X$, associated to a line bundle $L$ on $X$. 
Let $V_L$ be the $\perm_n$-vector bundle on $X^n$ defined as the exterior direct sum $L \boxplus \cdots \boxplus L$. 
We prove that the Bridgeland-King-Reid transform $\mathbf{\Phi}(S^k L^{[n]})$ of symmetric powers $S^k L^{[n]}$ is quasi isomorphic to the last term of a finite decreasing filtration  on the natural vector bundle $S^k V_L$, defined by kernels of operators $D^l_L$, which operate locally as higher order restrictions to pairwise diagonals. 
We use this description and the natural filtration on $(S^k V_L)^{\mathfrak{S}_n}$ induced by the decomposition in direct sum, to obtain, for $n =2$ or $k \leq 4$, a  finite decreasing filtration $\mathcal{W}^\bullet$ on the direct image $\mu_*(S^k L^{[n]})$ for the Hilbert-Chow morphism whose graded sheaves we control completely. As a consequence of this structural result, we obtain a chain of cohomological consequences, like 
a spectral sequence abutting to the cohomology of symmetric powers $S^k L^{[n]}$, an effective vanishing theorem for the cohomology of symmetric powers $S^k L^{[n]} \otimes \mathcal{D}_A$ twisted by the determinant, in presence of adequate positivity hypothesis on $L$ and $A$, as well as universal formulas for their Euler-Poincar\'e characteristic. 
\end{abstract}

\tableofcontents

\section*{Introduction}Let $X$ be a smooth complex quasi-projective algebraic surface and let $n \in \mbb{N}\setminus \{0 \}$. Denote with 
$X^{[n]}$ the Hilbert scheme of $n$ points over  $X$ and with $L^{[n]}$ the tautological bundle on $X^{[n]}$ associated to a line bundle $L$ over the surface $X$. The aim of this work is the study of general symmetric powers $S^k L^{[n]}$ of the tautological bundle $L^{[n]}$ with some applications to their cohomology. 

Tautological bundles have proven to be of fundamental importance for the geometry of Hilbert schemes of points over a surface $X$: indeed their Chern classes play a fundamental role in the structure of the cohomology ring 
of the Hilbert scheme $X^{[n]}$ and hence of its topology \cite{Lehn1999, EllingsrudGoettscheLehn2001}.  
Tautological bundles on $X^{[n]}$ are as well particularly relevant for the geometry of the  surface $X$ itself: as an example, we can mention
G\"ottsche conjecture on the number of 
$\delta$-nodal curves in given linear system $|L|$, whose statement, 
proofs and new refined versions directly involve the use of the tautological bundle $L^{[n]}$
\cite{Goettsche1998, KoolShendeThomas2011, GoettscheShende2014}.

The original motivation of this work is the Strange Duality conjecture over the projective plane, and, in particular, the understanding 
of the space of global sections of some particular determinant line bundle over  the moduli spaces of semistable sheaves of rank $2$. The program set up 
by Danila \cite{Danila2000} aimed  to prove the conjecture on $\mbb{P}_2$, for Grothendieck classes $u$ (of rank $2$, $c_1 =0$, $c_2 =n$) and $v$ (of rank $0$, $c_1 =d$, $\chi =0$), by interpolating the moduli space of semistable sheaves $M_u$ with a Hilbert scheme of points $\mbb{P}_2^{[m]}$ via moduli spaces of coherent systems (a technique already used in \cite{He1998} and \cite{LePotier1996}). The computations of global sections $H^0(M_u, \mc{D}_v)$ is then reconducted to the understanding of a second quadrant spectral sequence whose terms $E^{p,q}_1$ involve the cohomology groups
$$ H^q(\mbb{P}_2^{[m]}, S^{ld + p} L^{[m]} \tens \mc{D}^{\tens d}_{\FS(1)}) $$of symmetric powers of the tautological bundle $L^{[m]}$ associated to the line bundle $L= \FSP{2}(2l-3)$ over $\mbb{P}_2$, twisted by a natural line bundle 
$\mc{D}^{\tens d}_{\FSP{2}(1)}$ over the Hilbert scheme. Here $l$ is an auxiliary integer such that $m = n + l^2$ and that $l(l-1) \leq n < (l+1)(l+2)$.   To finally understand the original space of global sections $H^0(M_u, \mc{D}_v)$ two problems 
remain: the first,  \emph{the vanishing} of the higher cohomology of the twisted symmetric powers $S^{ld+p} L^{[m]} \tens \mc{D}^{\tens d}_{\FSP{2}(1)}$; the second,  \emph{the computation of the dimension of the spaces of global sections} $H^0(\mbb{P}_2^{[m]}, S^{ld + p} L^{[m]} \tens \mc{D}^{\tens d}_{\FSP{2}(1)})$. The first problem is solved by a general formula for the cohomology of tautological bundles \cite{Danila2001} and by an ad-hoc vanishing of the $H^1$ of the double symmetric powers \cite{Danila2000}; the second, via representation theory of $SL(3)$. Global sections for the dual moduli space $H^0(M_v, \mc{D}_u)$ have been computed 
with different techniques \cite{Danila2002}, not involving tautological bundles. Eventually, Danila proves the conjecture  for $d=1$, $n \leq 19$, $d=2,3$, $n \leq 5$ \cite{Danilathese, Danila2000, Danila2002}. However, in the case considered by Danila, that is, $X = \mbb{P}_2$ and Grothendieck classes $u$ and $v$ as above, the Strange Duality conjecture has been recently settled first by Abe \cite{Abe2010}, for $d =1,2$, using deformations of quasi-bundles, and later by Yao \cite{Yao2012, Yao2012A}, for $d \leq 4$, using results by 
G\"ottsche 
about Hilbert series of determinant line bundles over moduli spaces $M_u$ over the projective plane. This last approach uses heavily techniques of Nekrasov istanton counting like the ones appeared in \cite{GNY2009}. 
Both approaches extend  to rational surfaces. 

\sloppy
The work of Danila dealt with cohomology of tautological bundles \cite{Danila2000, Danila2001}, vanishing of $H^1(\mbb{P}_2^{[n]}, S^2 L^{[2]} \tens \mc{D}_A)$ for twisted symmetric powers for $L$ and $A$ positive, sections of $S^3 L^{[n]}$ over $\mbb{P}_2^{[n]}$ \cite{Danila2000, Danila2002},  symmetric powers $S^k L^{[n]}$ for $k \leq 2$ and $n \leq 3$ \cite{Danila2004},  and then with sections of general tensor powers \cite{Danila2007}. 
In our previous work \cite{ScalaPhD, Scala2006, Scala2009GD, Scala2009D}, among other things, we gave formulas for the cohomology of general double symmetric powers $S^2 L^{[n]}$ and for general exterior powers $\Lambda^k L^{[n]}$. 
Recently, Krug \cite{Krug2014} studied general tensor products of tautological bundles giving, among others, 
results 
for triple tensor products $E_1^{[n]} \tens E_2^{[n]} \tens E_3^{[n]}$ for general $n$. 

Here we address the case of general symmetric powers $S^k L^{[n]}$: we first obtain a general characterization of their Bridgeland-King-Reid transform  as kernels of some explicit  operators over $X^n$;  
then, for $n \leq 2$ or for $k \leq 4$ we prove a \emph{general structure theorem for their direct images $\mu_* S^k L^{[n]}$} for the Hilbert-Chow morphism over the symmetric variety $S^n X$. This structure result yields a series of cohomological consequences, among which, when $X$ is projective, an \emph{effective vanishing theorem} for twisted symmetric powers $S^k L^{[n]} \tens \mc{D}_A$, in presence of adequate positivity hypothesis for $L$ and $A$ and \emph{universal formulas for their Euler-Poincar\'e characteristic}. 

The starting point of our work is the description, given in \cite{ScalaPhD, Scala2009D}, of the Bridgeland-King-Reid transform 
of a $k$-fold tensor power $L^{[n]} \tens \cdots \tens L^{[n]}$ of the tautological bundle $L^{[n]}$. Recall that the Bridgeland-King-Reid transform \cite{BridgelandKingReid2001, Haiman2001, ScalaPhD, ScalaTrento2006} 
$$ \bkrh: \B{D}^b(X^{[n]}) \rTo \B{D}^b_{\perm_n}(X^n) $$is the equivalence of categories between the derived category of the Hilbert scheme of $n$ points over $X$ and the $\perm_n$-equivariant derived category of the product variety $X^n$ built (in this case) as the Fourier-Mukai transform through 
Haiman isospectral Hilbert scheme $B^n \subset X^{[n]} \times X^n$. 
The image 
$\bkrh(L^{[n]})$ of a tautological bundle is resolved in \cite{ScalaPhD, Scala2006, Scala2009D} by a complex 
$$\comp{\mc{C}}_L : \; \; 0 \rTo \oplus_{i=1}^n L_i \rTo^{\partial^0_L} \oplus_{|I|=2} L_I \rTo^{\partial^1_L} \oplus_{|J|=3} L_I \rTo^{\partial^2_L} \cdots \rTo^{\partial^{n-2}_L} L_{{1, \dots, n}} \rTo 0 $$where, if $I \subseteq \{1, \dots, n \}$ is a nonempty subset, $p_I: X^n \rTo X^I$ is the projection onto the factors indexed by $I$,  $i_I : X \rTo X^I$ is the diagonal immersion, the sheaf $L_I$ is defined as $L_I = p_I^* {i_I}_* L$ and the differentials $\partial^i_L$ are {\v C}ech-like maps (see subsection \ref{subsection: taut}). The image of the $k$-fold tensor power $\bkrh(L^{[n]} \tens \cdots \tens L^{[n]})$ is not the (derived) tensor power of the images. In \cite{Scala2009D} we proved that the image $\bkrh(L^{[n]} \tens \cdots \tens L^{[n]})$ is concentrated in degree $0$ and that the sheaf $p_* q^* (L^{[n]} \tens \cdots \tens L^{[n]})$ can be realized as 
the $E^{0,0}_\infty$ term of the hyperderived spectral sequence $$E^{p,q}_1 = \bigoplus_{i_1 + \dots + i_k=p} \Tor_{-q}(\mc{C}^{i_1}_L, \dots, \mc{C}^{i_k}_L)$$associated to the derived tensor product $\comp{\mc{C}}_L \tens^L \cdots \tens^L \comp{\mc{C}}_L$ and abutting to $\Tor_{-p-q}(\comp{\mc{C}}_L ,  \cdots,   \comp{\mc{C}}_L)$. 
Here the symmetric group $\perm_n$ acts geometrically on $X^n$ while a second symmetric  group $\perm_k$ acts permutating the factors in the tensor products $L^{[n]} \tens \cdots \tens L^{[n]}$. Hence, taking $\perm_k$-invariants, we get a characterization of the image  $\bkrh(S^k L^{[n]})$ of the symmetric power of tautological bundles as the term $(E^{0,0}_\infty)^{\perm_k}$ of the spectral sequence of $\perm_k$-invariants $(E^{p,q}_1)^{\perm_k}$. 
On the other hand, the bi-invariants $(E^{0,0}_\infty)^{\perm_n \times \perm_k}$ are quasi-isomorphic to the derived direct image 
$\B{R} \mu_* S^k L^{[n]} \simeq^{qis} \mu_* S^k L^{[n]} $ of symmetric powers $S^k L^{[n]}$ by the Hilbert-Chow morphism $\mu: X^{[n]} \rTo S^nX$. Both spectral sequences 
$(E^{p,q}_1)^{\perm_k}$ and $(E^{p,q}_1)^{\perm_n \times \perm_k}$ are still combinatorially difficult to treat. 
However, the sheaves $p_*q^* S^k L^{[n]}$ and $\mu_* S^k L^{[n]}$ satisfy the following important extension property. 
Let $W$ be the closed subscheme of points in $X^n$ lying in the intersection of two or more distinct pairwise diagonals $\Delta_{ij}$; 
let moreover $\pi: X^n \rTo S^nX$ be the quotient projection. The sheaves $p_*q^* S^k L^{[n]}$ and $\mu_* S^k L^{[n]}$ satisfy 
\begin{gather}
\label{eq: ext1} j_* j^* p_* q^* S^k L^{[n]} \simeq p_* q^* S^k L^{[n]} \\ 
\label{eq: ext2} j_* j^* \mu_* S^k L^{[n]} \simeq \mu_* S^k L^{[n]} 
\end{gather}where 
$j$ denotes the open immersion $X^n \setminus W \rInto X^n$ (in the first isomorphism) and 
the open immersion $S^nX \setminus \pi(W) \rTo S^n X$ (in the second). Isomorphisms (\ref{eq: ext1}) and (\ref{eq: ext2}) suggest
that one can extract all information about the sheaves $p_*q^* S^k L^{[n]}$ and $\mu_* S^k L^{[n]}$ from their restrictions 
to the open sets $X^n_{**} := X^n \setminus W$ and $S^nX_{**} = S^nX \setminus W$, respectively. 

We consequently look at the restrictions $j^*(E^{p,q}_1)^{\perm_k}$ and $j^* (E^{p,q}_1)^{\perm_n \times \perm_k}$ of the spectral sequences of $\perm_k$-invariants and of bi-invariants to the open sets $X^n_{**}$ and $S^nX_{**}$ respectively. In the analogous case of exterior powers, studied in \cite{Scala2009D}, the spectral sequence governing the problem was the spectral sequence of $\perm_n$-invariants and $\perm_k$-anti-invariants $(E^{p,q}_1 \tens \epsilon_k)^{\perm_n \times \perm_k}$ --- $\epsilon_k$ being here the 
alternating representation of the symmetric group $\perm_k$ --- and its restriction to $S^nX_{**}$ degenerated at level $E_2$.  In the present case of symmetric powers, the spectral sequences  $j^* (E^{p,q})_1^{\perm_n \times \perm_k}$ does not degenerate at level $E_2$, but only at level $E_k$; the spectral sequence $j^*(E^{p,q}_1)^{\perm_k}$ degenerates surprisingly at level $E_k$ in the same way. Despite the degeneration  occurs only later,  the degeneration pattern is clear: the only surviving terms at level $E_1$ are $j^*(E^{0,0}_1)^{\perm_k}$ and $j^*(E^{i,1-i}_1)^{\perm_k}$ for $0 \leq i \leq k-1$: therefore the kernels of the (restricted) higher differentials 
\begin{equation}\label{eq: filtr}  d^r : j^* (E^{0,0}_r)^{\perm_k} \rTo j^* (E^{r,1-r}_r)^{\perm_k} \simeq j^* (E^{r,1-r}_1)^{\perm_k}
\end{equation}define a filtration $$ \ker  d^{k-1} \subseteq  \ker  d^{k-2} \subseteq \cdots \subseteq  \ker d^{1} \subseteq j^* S^k \mc{C}^0_L $$of the vector bundle $j^* S^k \mc{C}^0_L$ over the open set $X^n_{**}$ with graded sheaves $j^*(E^{i,1-i}_1)^{\perm_k}$. The filtration (\ref{eq: filtr}) is nothing but the restriction to the open set $X^n_{**}$ of the natural filtration 
\begin{equation}\label{eq: filtr2} (E^{0,0}_{k})^{\perm_k} \subseteq ( E^{0,0}_{k-1})^{\perm_k} \subseteq \cdots \subseteq (E^{0,0}_1 )^{\perm_k} = S^k \mc{C}^0_L \;. \end{equation}The sheaves $(E^{0,0}_r)^{\perm_k} $ 
satisfy an extension property similar to (\ref{eq: ext1}); however, the functor $j_*$ is not right exact: consequently there is no way 
we can get information on the graded sheaves of the filtration (\ref{eq: filtr2}) from the graded sheaves $j^*(E^{i,1-i}_1)^{\perm_k}$ of filtration (\ref{eq: filtr}).  

In front of this difficulty, we performed a detailed analysis of the higher differentials $d^l$ of the restricted spectral sequence $j^*(E^{p,q}_1)^{\perm_k}$. As a result of this study, we found recursively defined operators 
\begin{gather*} D^0_L : S^k \mc{C}^0_L \rTo (j_* j^* E^{1, 0}_1)^{\perm_k} =: K^0(L) \\
D^l_L : \ker D^{l-1}_L \rTo (j_* j^* E^{l+1, -l}_1)^{\perm_k} =: K^l(L) 
\end{gather*}over the whole variety $X^n$ such that their restrictions to the open set $X^n_{**}$ are the higher differentials $d^l$: $$ j^* D^l_L = d^l \;.$$The operators $D^l_L$ are built globally over $X^n$ by glueing local operators $D^l_L$ defined via partial jet-projections: we provide explicit local formulas for them, inspired by finite difference calculus. Thanks to the operators $D^l_L$ we can state the first result of this work. Denote with $E^l(n,k)$ the subsheaf of $S^k \mc{C}^0_L$ defined by the kernel of $D^{l-1}_L$ over $X^n$. Then we have

\vspace{0.2cm}
\noindent
{\bf Theorem \ref{thm: mainimage}.} {\it The Bridgeland-King-Reid transform $\bkrh(S^k L^{[n]})$ of the $k$-symmetric power $S^k L^{[n]}$ of the tautological bundle $L^{[n]}$ associated to the line bundle $L$ over the surface $X$ is quasi-isomorphic to the term 
$E^{k-1}(n,k)$ of the filtration $E^\bullet(n,k)$ on the vector bundle $S^k \mc{C}^0_L$: }
$$ \bkrh(S^k L^{[n]}) \simeq E^{k-1}(n,k) \simeq \ker D^{k-2}_L \;.$$

In particular, for $X = \mbb{C}^2$ or an affine surface, we get a characterization of the image $\bkrh(\FS_{X}^{[n]})$ as the subsheaf of sections of the (trivial) sheaf $S^k \mc{C}^0_{\FS_X} = \oplus_{\lambda \in c_n(k)} \FS_{X^n}$ (where $c_n(k)$ is the set of compositions of $k$ supported in $\{1, \dots, n \}$) satisfying an \emph{explicit system of higher order restrictions}, that is, restrictions of higher derivatives (see corollary \ref{crl: diffaffine}) to pairwise diagonals. Similarly, we have a characterization of $\mu_* S^k L^{[n]}$ in terms of kernels of invariant operators $\mc{D}^l_L := (D^l_L)^{\perm_n}$. We obtain, moreover, explicit local formulas for the invariant operators $\mc{D}^l_L$ over $S^nX$: these formulas will be useful later, when proving the main structure theorem.  The appearence of higher order restrictions  of the kind of $D^l_L$ is natural when in presence of \emph{nontransverse intersections}: in our context the occurrence of the operators $D^l_L$ witness the fact that \emph{the irreducible components of Haiman polygraph $D(n,k) \subset X^n \times X^k$ are not transverse}.
The operators $D^l_L$ are essentially the $\perm_k$-invariant version of the operators $\phi_l$ in Krug's work \cite{Krug2014} and theorem \ref{thm: mainimage} can be compared to \cite[theorem 4.10]{Krug2014}.  However, the way these operators are found\footnote{The operators $D^l_L$, their local formulas and theorem \ref{thm: mainimage} are known to us since late 2009; some of their consequences --- for example theorem \ref{thm: main3} --- since soon later (mid-2010) \cite{Scala2010}. } is completely different (see remark \ref{rmk: Krug}). 

Since the filtration $E^\bullet(n,k)$ and its $\perm_n$-invariant version $\mc{E}^\bullet(n,k) := E^\bullet(n,k)^{\perm_n}$ can't be directly used to get informations about  the sheaves $p_* q^* S^k L^{[n]}$ and $\mu_* S^k L^{[n]}$, respectively, due to the difficulty in understanding their graded sheaves, we are lead to 
introduce a second filtration $\mc{W}^\bullet$ on the invariants $(S^k \mc{C}^0_L)^{\perm_n}$ as follows. 
The vector bundle $S^k \mc{C}^0_L$ is isomorphic to the direct sum 
$$ S^k \mc{C}^0_L = \oplus_{\lambda \in c_n(k)} L^\lambda $$--- over compositions $\lambda$ of $k$ supported in $\{1, \dots, n \}$ ---
 of line bundles $L^\lambda := \tens_{i=1}^n p_i^* L^{\tens \lambda_i}$; consequently, the $\perm_n$-invariants $(S^k \mc{C}^0_L)^{\perm_n}$ of the vector bundle $S^k \mc{C}^0_L$ split as a direct sum 
 $$ (S^k \mc{C}^0_L)^{\perm_n} = \oplus_{\lambda \in p_n(k)} \mc{L}^\lambda $$--- over partitions $\lambda$ of $k$ of length at most $n$ --- of sheaves $ \mc{L}^\lambda : = \pi_*(L^\lambda)^{\Stab_{\perm_n}(\lambda)} $ of $\Stab_{\perm_n}(\lambda)$-invariants of the sheaves $\pi_* L^\lambda$ over $S^n X$. Hence, to each total order $\preceq$ on the set  of partitions $p_n(k)$, we can associate a natural filtration $\mc{V}^\bullet$ on $(S^k \mc{C}^0_L)^{\perm_n}$ defined as 
 $$ \mc{V}^\mu = \oplus_{\substack{\lambda \in p_n(k) \\ \lambda \succeq \mu}} \mc{L}^\lambda \;.$$Consider for the moment 
 the reverse lexicographic order $\leq_{\rm rlex}$ on the set of partitions $p_n(k)$ and let $\mc{V}^\bullet$ the associated filtration. Then 
 we can define the filtration $\mc{W}^\bullet$ on $\mu_* S^k L^{[n]}$ as 
 $$ \mc{W}^\bullet :=  \mc{V}^{\bullet}  \cap \mc{E}^{k-1}(n, k) \;.$$The filtration $\mc{W}^\bullet$ has the great advantage that 
 \emph{we can try to guess what are the graded sheaves $\gr^{\mc{W}}_\mu$}, at least for $k \leq 4$. It turns out that 
 (at least for small $k$) \emph{the graded sheaves are subsheaves  of the sheaves $\mc{L}^\mu$ whose sections vanish with a prescribed order along some prescribed partial diagonals}. More precisely, define the sheaf $\mc{L}^\mu(-l \Delta)$ as the sheaf 
 $$ \mc{L}^\mu(-l \Delta) = \pi_*(L^\mu \tens \cap_{1 \leq i < j \leq l(\mu)} I_{\Delta_{ij}}^l)^{\Stab_{\perm_n}(\mu)}$$where here $l(\mu)$ is the length of the partition $\mu$ and $I_{\Delta_{ij}}$ is the ideal sheaf of the pairwise diagonal $\Delta_{ij}$ in $X^n$. The main result of this work is  the following structure theorem

\vspace{0.2cm}
\noindent
{\bf Theorem \ref{thm: main2}-\ref{thm: main3}-\ref{thm: main4}. } {\it 
Let $n \in \mbb{N}^*$, $k \in \mbb{N}$, such that $n \leq 2$ or $k \leq 4$. Then the graded sheaves of the filtration $\mc{W}^\bullet$ on the sheaf $\mu_* S^k L^{[n]}$, indexed by partitions $p_n(k)$, equipped with the reverse lexicographic order, are given by
 $$ \gr^{\mc{W}}_\mu \mu_* S^k L^{[n]} \simeq \mc{L}^\mu (- 2 m_\mu \Delta) \;,$$where $m_\mu :=0$ if $l(\mu)=1$, otherwise $m_\mu := \min_{2 \leq i \leq l(\mu)} \mu_i$. }
 
 \vspace{0.2cm}
The proof of theorems \ref{thm: main2}-\ref{thm: main3}-\ref{thm: main4} consists in the construction of natural 
 left exact sequences 
 $$ 0 \rTo \mc{W}^{\mu^\prime} \rTo \mc{W}^\mu \rTo \mc{L}^\mu(-2 m_\mu \Delta) $$which are proven to actually be right exact by an ad-hoc verification. Here $\mu^\prime$ is the partition following $\mu$ in the reverse lexicographic order. 
 
 \emph{Even if theorem 
 \ref{thm: main2}-\ref{thm: main3}-\ref{thm: main4} has been proven by ad-hoc methods, it seems that its content might be carried on to the general case}, possibly modifying the definitions of some of the objects involved. First, the reverse lexicographic order has to be replaced by another total order  on the set of partitions $p_n(k)$: the new order $\preceq$ has to take into account 
 the length of a partition in the first place and coincides with the reverse lexicographic order  between partitions of the same 
length. Secondly, it might be that we have to consider subsheaves of the sheaves $\mc{L}^\mu$ whose sections vanish
with different prescribed order along different prescribed pairwise diagonals. Our first guess are the sheaves $\mc{L}^\mu(-2 \mu \Delta)$, defined as 
$$ \mc{L}^\mu(-2 \mu \Delta) := \pi_* (L^\mu \tens \cap_{1 \leq i < j \leq l(\mu)} I_{\Delta_{ij}}^{2 \mu_j})^{\Stab_{\perm_n}(\lambda)} \;.$$
Several cases worked out for $k \geq 5$ confirm these facts. 
It is likely that a structure theorem of the kind of \ref{thm: main2}-\ref{thm: main3}-\ref{thm: main4}  hold for general tensor powers or even general tensor products of tautological bundles. Indeed, the analogue of filtration $E^\bullet(n,k)$ can be defined for general tensor products of tautological bundles, as kernels of operators $\phi_l$ in Krug's work \cite{Krug2014}. As for the filtration $\mc{V}^\bullet$, there should be a natural way of extending it in the case of tensor product of tautological bundles, and maybe even in the case of general tensor products.
We will discuss more in detail a possible extension to the general case in subsection \ref{toward}. 

Everything we said extends 
to the case in which symmetric powers $S^k L^{[n]}$ are twisted by the natural line bundle 
$\mc{D}_A = \mu^*(A \boxtimes \cdots \boxtimes A/\perm_n) $, defined as the pull back over $X^{[n]}$ of the descent over $S^n X$ of the line bundle $A \boxtimes \cdots \boxtimes A$ ($n$-factors) on $X^n$. For brevity's sake we will denote the descent $A \boxtimes \cdots \boxtimes A/\perm_n$ also with $\mc{D}_A$. 

Theorem \ref{thm: main2}-\ref{thm: main3}-\ref{thm: main4} implies a chain of cohomological consequences. First, we can compute --- for $n \leq 2$ or for $k \leq 4$ --- the cohomology of twisted symmetric powers $S^k L^{[n]} \tens \mc{D}_A$  via the spectral sequence of a filtered sheaf 
$$ E^{p,q}_1 = H^{p+q}(S^n X, \mc{L}^{\mu(p)}(-2 m_{\mu(p)} \Delta) \tens \mc{D}_A) \implies H^{p+q}(X^{[n]}, S^k L^{[n]} \tens \mc{D}_A) $$where $\mu(p)$ is the $(p+1)$-th partition 
in $p_n(k)$ according to reverse lexicographic order. Secondly, when $X$ is projective, we can work out an \emph{effective vanishing theorem} for the higher cohomology $H^p(X^{[n]}, S^k L^{[n]} \tens \mc{D}_A)$, in presence of adequate positivity hypothesis on $L$ and $A$. 
The key notion here is that of an $l$-very ample line bundle and the fact that if $L$ is $n$-very ample on $X$, 
then the determinant $\det L^{[n]}$ of the associated tautological bundle $L^{[n]}$ is very ample on $X^{[n]}$ \cite{BeltramettiSommese1991, CataneseGoettsche1990}. This fact is sufficient to make the higher cohomology 
of most graded sheaves $\gr^{\mc{W}}_\mu$ vanish via Kodaira theorem; still for one of them, the sheaf $\mc{L}^{2,1,1}(-2 \Delta) \tens \mc{D}_A$, we need to work over the isospectral Hilbert scheme $B_3$ of three points over $X$ 
and use that $B_3$ has canonical singularities \cite{Scala2015isospectral} before concluding with Kodaira  vanishing. The effective vanishing result reads precisely

\vspace{0.2cm}
\noindent
{\bf Theorem \ref{thm: vanishing}.} {\it Let $X$ be a smooth complex projective surface. Let $L$, $A$ be line bundles on $X$ such that $L$ is nef and $A \tens \omega_X^{-1}$ is big and nef. Let $n \in \mbb{N}^*$ and $k \in \mbb{N}$, with $n \leq 2$ or $k \leq 4$. Then 
$$ H^i(X^{[n]}, S^k L^{[n]} \tens \mc{D}_A) = 0 \qquad \mbox{for all $i>0$}$$if 
$L \tens A \tens \omega_X^{-1} = \tens_{j=1}^{k+1}B_j$ when $n=2$, if $L \tens A \tens \omega_X^{-1} = \tens_{j=1}^{5}B_j$ when $k =3$, if $L \tens A \tens \omega_X^{-1} = \tens_{j=1}^{7}B_j$ when $k=4$, 
where $B_j$ are very ample line bundles on $X$.} 

\vspace{0.2cm}
A variant of this result  for $n=2$ (theorem \ref{thm: vanishingvariant}) just requires that $L$ and $A \tens \omega_X^{-1}$ are both tensor product of two very ample line bundles. Finally, for $X$ projective and for $n=2$ or for $k \leq 4$ we give general \emph{universal formulas} for the Euler-Poincar\'e characteristic $\chi(X^{[n]}, S^k L^{[n]} \tens \mc{D}_A)$ of twisted symmetric powers in terms of Euler-Poincar\'e characteristics of $L, A, S^l \Omega^1_X, \omega_X$ 
and some tensor product of these over~$X$. 

\paragraph{Acknowledgements.} This work started in 2009 at University of Chicago, where I had the chance of talking with Madhav Nori. I'd like to thank him for the useful discussions and suggestions. In winter 2011 I visited University of Massachussets at Amherst: I'd like to thank  Eyal Markman for his interest in this work, for fruitful talks, and, finally, for providing example \ref{ex: markman}. I would also like to thank 
Lei Song for pointing out several missprints and imprecisions in the first version of this paper. 
This work was partially supported by CNPq, grant 307795/2012-8.

\section{The geometric setup}

Let $X$ be a smooth quasi-projective surface. We will indicate with $X^n$ the cartesian product of $n$ copies of $X$; let $G$ be the symmetric group $G:=\perm_n$; it acts on $X^n$ with quotient the symmetric variety $S^nX$: let $\pi: X^n \rTo S^nX$ be the quotient projection.  Let $X^{[n]}$ be the Hilbert scheme of $n$ points on $X$: the Hilbert-Chow morphism $\mu: X^{[n]} \rTo S^nX$, defined as $\mu(\xi)= \sum_{x \in X} ( {\rm length \:} \FS_{\xi, x}) x$, realizes $X^{[n]}$ as a semismall crepant resolution of the singularities of $S^nX$. 
Haiman \cite{Haiman2001} introduced the isospectral Hilbert scheme $B^n$, 
that is, the 
reduced fibered product $(X^{[n]} \times_{S^n X} X^n)_{\rm red}$:  
let $p$ and $q$ its projections onto $X^n$ and $X^{[n]}$, respectively. 
In the diagram \begin{diagram} B^n & \rTo^p & X^n \\
\dTo^q & & \dTo^\pi \\
X^{[n]} & \rTo^\mu & S^n X
\end{diagram}
$p$ and $\mu$ are birational, $q$ and $\pi$ are finite, $q$ is flat. The Bridgeland-King-Reid transform \cite{BridgelandKingReid2001, Haiman2001, ScalaTrento2006, Scala2009D}$$\bkrh: = \Phi^{\FS_{B^n}}_{X^{[n]} \rightarrow X^n}:= \B{R}p_* \circ q^* : \B{D}^b(X^{[n]}) \rTo \B{D}^b_{G}(X^n)$$is an equivalence of derived categories from the bounded derived category of coherent sheaves on $X^{[n]}$ to the  bounded derived category of $\perm_n$-equivariant  coherent sheaves over $X^n$.

\subsection{Tautological bundles and their images under the Bridgeland-King-Reid transform}\label{subsection: taut}
Consider now the universal subscheme $\Xi \subseteq X^{[n]} \times X$ of the Hilbert scheme $X^{[n]}$. 
If $L$ is a vector bundle on $X$, the \emph{tautological  bundle} $L^{[n]}$ over $X^{[n]}$ is defined as: $$L^{[n]} := {p_{X^{[n]}}}_* (\FS_{\Xi} \tens_{\FS_X} L) \;,$$where $p_{X^{[n]}}$ is the projection over the Hilbert scheme $X^{[n]}$. 
Since the subscheme $\Xi$ is flat and finite over $X^{[n]}$ of degree $n$, it turns out that $L^{[n]}$ is a vector bundle of rank  $n \rk L$ on $X^{[n]}$. 

\begin{notat}\label{notat: restrizioni}Let $I \subseteq \{1, \dots, n \}$, $I \neq \emptyset$. Denote with $p_I: X^n \rTo X^{I}$ the projection onto the factors in $I$, let $\imath_I : X \rTo X^I$ the diagonal immersion of $X$ into $X^I$. If $F$ is a coherent sheaf on $X$, we indicate with $F_I$ the sheaf
on $X^n$ defined as $F_I := p_I^* (i_I)_* L$. If $|I| \geq 2$, we indicate with $\Delta_I$ the inverse image by $p_I$ of the total (small) diagonal in $X^I$.  Note that in this case $F_I$ is supported on $\Delta_I$. If $n=2$ we will also denote the diagonal in $X^2$ and $S^2X$  just with $\Delta$. 
\end{notat}
For a vector bundle $L$ on $X$, define now the complex $(\comp{\mc{C}}_L, \partial^{\bullet}_{L})$ on $X^n$ in the following way:
$$ \mc{C}^p_L := \bigoplus_{|I|=p+1} L_I \;, \qquad \qquad \partial^p_L(x)_J = \sum_{i \in J} \epsilon_{i,J} x_{J\setminus \{i\}} \trest_{\Delta_J} \;,$$where $\epsilon_{i, J} = (-1)^{\sharp \{j \in J \; | \; j < i \}}$. 
\begin{remark}\label{farfalla}
  The complex $\comp{\mc{C}}_L$ is naturally $G$-equivariant, in the following way.  If $\sigma \in G$, denote with $\sigma_* $
the automorphism $\sigma_* \in \Aut(X^n)$ defined as: $\sigma_*(x_1, \dots, x_n) := (x_{\sigma^{-1}(1)}, \dots, x_{\sigma^{-1}(n)})$. We have $\sigma_*(L_I) = L_{\sigma(I)}$ for all 
$ \emptyset \neq I \subseteq \{1, \dots,n \}$. The $G$-linearization on $\mc{C}^p_L := 
\oplus_{|I|=p+1} L_I$ can then be defined setting $(\sigma.x)_J := \epsilon_{\sigma, J} \sigma_* x_{\sigma^{-1}(J)}$, where $(x_{I})_I$ is a local section of $\mc{C}^p_L$ and where
$\epsilon_{\sigma,J}$ is the signature of the only permutation 
$\tau$ such that $\sigma^{-1}\tau$ is strictly increasing. \end{remark}The main result about the Bridgeland-King-Reid transform of a tautological bundle is the following. 
\begin{theorem}[\cite{Scala2009D}] \label{B3}
Let $X$ be a smooth quasi-projective algebraic surface and $L$ be a vector bundle
on $X$. Let $L^{[n]}$ be the tautological bundle on the Hilbert
scheme $X^{[n]}$ associated to $L$.  Then the image of the tautological bundle 
$L^{[n]}$ for the Bridgeland-King-Reid equivalence $\bkrh$ is
  isomorphic in $\B{D}^b_{G}(X^n)$ to the complex
  $(\comp{\mc{C}}_L, \comp{\del}_L )$:
$$ \bkrh ( L^{[n]}) \simeq
\comp{\mc{C}}_L \; .$$\end{theorem}This result holds as well for tautological sheaves \cite{Scala2009GD}. As for the Bridgeland-King-Reid 
transform of a tensor product of tautological bundles $\bkrh(L_1^{[n]} \tens \cdots \tens L_k^{[n]})$, for vector bundles $L_1, \dots, L_k$ on $X$, of course one can't hope that it would be quasi-isomorphic to the (derived) tensor product $\comp{\mc{C}}_{L_1} \tens^L \cdots \tens^L \comp{\mc{C}}_{L_k}$ of the images. However, a precise comparison can be established in terms of a natural morphism $$ \alpha: \comp{\mc{C}}_{L_1} \tens^L \cdots \tens^L \comp{\mc{C}}_{L_k} \rTo \bkrh(L_1^{[n]} \tens \cdots \tens L_k^{[n]})$$in $\B{D}^b_{G}(X^n)$. It turns out that the mapping cone of $\alpha$ is acyclic in positive degree: this is equivalent to saying that the complex $\bkrh(L_1^{[n]} \tens \cdots \tens L_k^{[n]})$ is cohomologically concentrated in degree $0$, or, in other words, that the higher direct images $R^i p_* q^* (L_1^{[n]} \tens \cdots \tens L_k^{[n]})$ vanish for $i >0$. Moreover, in degree zero, the natural morphism $\alpha$ provides an epimorphism $ p_* q^* L_1^{[n]} \tens \cdots \tens p_* q^* L_k^{[n]} \rOnto p_* q^* (L_1^{[n]} \tens \cdots \tens L_k^{[n]})$ whose kernel is the torsion subsheaf. As a consequence of these facts one can extract the information of the sheaf $p_* q^* (L_1^{[n]} \tens \cdots \tens L_k^{[n]}) \simeq \mc{H}^0(\bkrh(L_1^{[n]} \tens \cdots \tens L_k^{[n]}))$  as the $E^{0, 0}_{\infty}$ term of the $G$-equivariant spectral sequence 
$$E^{p,q}_1 = \bigoplus_{i_1 + \dots + i_k = p} \Tor^{-q}(\mc{C}^{i_1}_{L_1}, \dots, \mc{C}^{i_k}_{L_k})$$abutting to 
$\mc{H}^{p+q}(\comp{\mc{C}}_{L_1} \tens^L \cdots \tens^L \comp{\mc{C}}_{L_k})$. If $L_i = L$ for all $i =1, \dots, k$, 
the $k$-fold tensor products $L \tens \dots \tens L$ and $\comp{\mc{C}}_{L} \tens^L \cdots \tens^L \comp{\mc{C}}_{L}$ both acquire an action of the symmetric group $\perm_k$ -- which operates permutating the factors -- in such a way that 
the morphism $\alpha$ becomes $\perm_k$-equivariant. From now on we will denote  with $H$ the symmetric group $\perm_k$. The corresponding spectral sequence $E^{p,q}_1$ is in this case naturally $H$-equivariant, and hence $G \times H$-equivariant. 
Details on how to equip the derived tensor product $\comp{\mc{C}}_{L} \tens^L \cdots \tens^L \comp{\mc{C}}_{L}$ and his associated spectral sequence $E^{p,q}_1$ with a $H$-action appeared in \cite{Scala2009D}, section 4.1, and will be briefly recalled later. As a consequence of the $G \times H$-equivariance discussed here, we have \cite[Corollary 4.1.3]{Scala2009D}: 
\begin{crl}\label{crl: spect} The Bridgeland-King-Reid transform $\bkrh(S^k L^{[n]})$ of the symmetric power $S^k L^{[n]}$ of 
a tautological bundle is quasi-isomorphic to the term $(E^{0,0}_\infty)^H$ of the spectral sequence $(E^{p,q}_1)^H$. Moreover, the direct image $\mu_* S^k L^{[n]}$ of the symmetric power $S^k L^{[n]}$ for the Hilbert-Chow morphism can be identified with the term $(E^{0,0}_\infty)^{G \times H}$ of the spectral sequence $(E^{p,q}_1 )^{G \times H}$. 
\end{crl}

\begin{conv}\label{conv: sym}Let $A$ a $\mbb{C}$-algebra and $M$ an $A$-module. For $n \in \mbb{N}\setminus \{0 \}$, consider the symmetric  power 
$S^n M$  of the module $M$. We consider $S^n M$ as the space of $\perm_n$-invariants of $M^{\tens n}$ for the action of $\perm_n$ permutating the factors in the tensor product. Throughout this article, we will use the following convention for the symmetric product $u_1. \cdots . u_n$ 
of elements $u_i \in M$: 
$$ u_1. \cdots . u_n := \sum_{\sigma \in \perm_n} u_{\sigma(1)} \tens \cdots \tens u_{\sigma(n)} \;,$$where the right hand side is seen in $M^{\tens n}$. In this way the explicit formulas for 
invariant operators $\D^l_L = (D^l_L)^{\perm_n}$, in subsection \ref{subsection: invariantoperators} and following, will be simpler 
and will not depend on $n$. We use an analogous convention for the exterior product: 
$ u_1 \wedge \cdots \wedge u_n := \sum_{\sigma \in \perm_n} (-1)^\sigma u_{\sigma(1)} \tens \cdots \tens u_{\sigma(n)} $, where $(-1)^\sigma$ is the signature of the of permutation $\sigma$ and where we see $\Lambda^n M$ as the space of anti-invariants for the action of $\perm_n$ over $M^{\tens n}$. 
\end{conv}

\subsection{Reduction to a big open set}\label{subsection: reductionopen}\label{subsec: bigopenset}
Consider the closed subscheme $$ W:=  \bigcup_{\substack{|I|=|I^{\prime}|=2 \\ I \neq I^{\prime}}} \Delta_I \cap \Delta_{I^{\prime}} $$of $X^n$: its irreducible components are smooth of codimension $4$.  Define now the open sets $X^n_{**}$, $S^nX_{**}$, $X^{[n]}_{**}$, $B^{n}_{**}$ of $X^n$, $S^nX$, $X^{[n]}$, $B^n$, respectively, as: 
\begin{equation}\label{eq: bigopensets} X^n_{**} := X^n \setminus W \:, \quad S^n X_{**} = S^nX \setminus \pi(W) \;, \quad X^{[n]}_{**}= \mu^{-1}(S^n_{**}X) \:, \quad B^n_{**} = p^{-1}(X^n_{**})
\end{equation}We will indicate just with $j$ the open immersion of each of these open sets into 
$X^n$, $S^nX$, $X^{[n]}$, $B^n$, respectively. 
The open sets $X^n_{**}$ and $S^nX_{**}$ are complementary of closed subschemes of codimension 4, while $X^{[n]}_{**}$ and $B^n_{**}$ are complementary of closed subschemes of codimension 2. Like any vector bundle on $X^{[n]}$, the symmetric power $S^k L^{[n]}$ satisfies the property \cite[Lemma 3.1.4]{Scala2009D}: 
$ \mu_* S^k L^{[n]} = j_* j^*  \mu_* S^k L^{[n]}$ and, with a similar proof, because of the vanishing $R^i p_* q^* S^k L^{[n]} = 0$ if $i >0$, it satisfies: 
\begin{equation}\label{eq: j} \bkrh(S^k L^{[n]}) \simeq^{qis}  j_* j^* p_* q^* S^k L^{[n]} \end{equation}as well. Consequently we can try to extract all the 
information about $\bkrh(S^k L^{[n]})$ and $\mu_* S^k L^{[n]}$ working over the big open sets (\ref{eq: bigopensets}). 
More precisely, we have $\bkrh(S^k L^{[n]}) \simeq^{qis} (j_* j^* E^{0.0}_\infty)^H$ and $\mu_* S^kL^{[n]} \simeq (j_* j^* E^{0,0}_\infty)^{G \times H} $. In other words, the image $p_* q^* S^k L^{[n]}$ can be seen as a recursive kernel of operators: 
$$D^{l-1}:= (j_* d^l)^H : \ker (j_* d^{l-1})^H \rTo (j_* j^*E^{l, 1-l}_l)^H \;,$$where $d^l : (j^* E^{0,0}_l)^H \rTo (j^* E^{l, 1-l}_l)^H$ are differentials of the spectral sequence $(j^* E^{p,q}_1)^H$, at level $l$. In the next sections we are going to \emph{prove that the spectral sequence $(j^*E^{p,q}_1)^H$ degenerates at level $k$} and we will \emph{explicitely compute the  operators $D^l$}. As a consequence we will get: $$\bkrh(S^k L^{[n]}) \simeq^{qis}  j_* j^* p_* q^* S^k L^{[n]}  \simeq \ker D^{k-2}$$and a similar formula for $\mu_* S^k L^{[n]}$. To end this section, we rewrite the terms $j_* j^* E^{p,q}_1$ in a simpler way. 
\begin{notat}\label{notaziona}
Let $\mc{P}:= \{ I \; | \; \emptyset \neq I \subseteq \{ 1, \dots, n \} \}$. 
For each $a \in \mc{P}^{ \{1, \dots,k \} }$ we will indicate with $A(a)$ the multi-index 
$A(a) := \cup_{i \, | \, |a(i)| \geq 2 } \, a(i)$, $k(a) := 
{\rm max} \{0, 2(|A(a)|-1) \}$, $l(a) = (\sum_{i=1}^k |a(i)|) -k$ and
$J(a) = \cup_{i \, | \, |a(i)|=1} \, a(i)$. We finally set 
$S_0(a):= a^{-1}(A(a))$ and 
$\lambda(a) :  J(a) \rTo \mbb{N}$ the function defined by :  $\lambda_j(a) = 
|a^{-1}(\{j \})|$ for all $j \in J(a)$. If there is no confusion, we will just write $A$, $J$, $S_0$, $\lambda$. 
\end{notat}With these notations $E^{p,q}_1$ can be rewritten as $$E^{p,q}_1 =  \bigoplus_{\substack{ a \in \mc{P}^{\{1, \dots, k\}} \\ l(a)=p}} \Tor_{-q}(L_{a(1)}, \dots, L_{a(k)}) \;.$$
Define now $\mc{I}^p$ as $\mc{I}^p  = \{ a  \in \mc{P}^{\{1, \dots,k\}}, \; l(a)=p \, , \; k(a)  \leq 2 \}$, and $(E^{p,q}_1)_0$ as the $G \times H$-equivariant subsheaf of $E^{p,q}_1$ given by (see also \cite[Remark 3.1.8]{Scala2009D}): 
\begin{equation}\label{eq: directsum} (E^{p,q}_1)_0 := \bigoplus_{a \in \mc{I}^p} \Tor_{-q}(L_{a(1)}, \dots, L_{a(k)}) \;. \end{equation}It turns out  that $(E^{p,q}_1)_0$ vanishes if $q<2(1-p)$ and if $p>k$. Moreover, since it is a vector bundle or a direct sum of restrictions of vector bundes over pairwise diagonals, and hence pure Cohen-Macauley coherent sheaf of codimension $\leq 2$, we have \cite[Remark 3.1.8, Lemma 3.1.10]{Scala2009D}:$$j_*j^* E^{p,q}_1 \simeq (E^{p,q}_1)_0 \;.$$
\begin{remark}\label{rmk: deg}As a consequence of what just said, in order to prove that the spectral sequence $(j^*E^{p,q}_1)^H$ degenerates at level $k$, it is sufficient to prove that $(E^{p,q}_1)_0^H \neq 0$ 
if and only if $(p,q)=(0,0)$ or if $q=1-p$ and $2-k \leq q \leq 0$. \end{remark}

\section{Degeneration of the spectral sequence}
In this section we will prove the degeneration of the spectral sequences $(j^*E^{p,q}_1)^H$ and $(j^*E^{p,q}_1)^{G \times H}$ by proving the vanishing of terms $(E^{p,q}_1)_0^H$ and $(E^{p,q}_1)_0^{G \times H}$ as explained in remark \ref{rmk: deg}. In the following two subsections we study the set $\mc{I}^p$, parametrizing the terms in the direct sum (\ref{eq: directsum}), as $G \times H$-set.
\subsection{Multi-indexes}
\begin{notat}\label{notat: comp}Let $h \in \mbb{N}$, $h \leq n$. We indicate with $c_n(h) := \{ \lambda: \{1, \dots, n\} {\rTo} \{0, \dots, h\} \; | \; 
 \sum_i \lambda_i = h \}$ the set of \emph{compositions} of $h$ of \emph{range} $n$. If $\lambda \in c_n(h)$, we denote with $\supp \lambda$ its support as a function, that is $\supp \lambda : = \{ i \in \{1, \dots, n\} \; | \;\lambda_i \neq 0 \}$. If $\lambda$ is a composition of some $h$ of range $n$ we also say that $h = \sum_{i=1}^n \lambda_i$ is its weight, and we denote it with $|\lambda|$.  
 
If $\lambda \in c_n(l)$, for brevity's sake, we will sometimes write $\lambda$ multiplicatively as $\lambda = \prod_{i=1}^n i^{\lambda_i}$; we will refer to this notation as the \emph{multiplicative notation}\footnote{This has nothing to do with the exponential notation of partitions. The exponential notation will be recalled in notation \ref{notat: exp} and used in subsection \ref{subsubsection: actglob}.} for compositions.  If $\lambda, \mu$ are two compositions, the function $\lambda + \mu$, as a composition, is written as the product $\lambda \mu$ in the multiplicative notation. 
\end{notat}
\begin{notat}Let $m \in \mbb{N}$. We denote with $p(m)$ the set of partitions of $m$. Set $p(0)= \{0 \}$.
We will write the length of a partition $\lambda$ as $l(\lambda)$. 
 Denote with $p_t(m)$ the set of partitions $\lambda$ of $m$ of lenght $l(\lambda) \leq t$. 
Any partition $p_n(m)$ can be seen naturally in $c_n(m)$ as a composition of $m$ of range $n$, by setting $\lambda_i = 0$ for $l(\lambda) < i \leq n$. The weight of a partition is the weight of the corresponding composition.\end{notat}
\begin{remark}The $G \times H$-action on the set of indexes $\mc{P}^{\{1, \dots, k\}}$ is given by 
$(\sigma, \tau). a := \sigma a \tau^{-1}$ and yields naturally a $G \times H$-action on the sets $\mc{I}^l$. 
\end{remark}
\begin{remark}\label{rmk: comp}
The group $G $ acts naturally on compositions $c_n(h)$ setting $\sigma.\lambda := \lambda \circ \sigma^{-1}$. Therefore, in the orbit of  a composition $ \mu$ in $c_n(h)$ there is a unique partition $\nu(\mu) \in p_n(h)$ of length  $|\supp \mu|$. We denote this partition with $\nu(\mu)$: it is an invariant of the $G$-orbit of $\mu$. 
\end{remark}
Denote now,  for $r \in \mbb{N}, r \leq n$, with $\mf{P}_n(r) = \coprod_{0 \leq m \leq r} p_n(m)$, that is, the set of partitions of natural numbers lower or equal than $r$ and of length lower or equal to $n$. Moreover, denote with $\mc{Q} = \{ I \subseteq \{1, \dots, n\} \; | \; |I | \leq 2 \}$. There is a natural $G$-action on the set $\mc{Q}$, induced by the $G$-action on $\{1, \dots, n\}$. 

\begin{notat}If $\lambda \in c_n(h)$ and $A \subseteq \{1, \dots, n\}$, we indicate with $\lambda_A$ the restriction $\lambda |_A : A {\rTo} \{0, \dots, h\}$. Sometimes (it will be clear from the context) $\lambda_A$ will also denote the extension of $\lambda |_A$ to the whole $\{1, \dots, n\}$ with zero. 
\end{notat}

\begin{remark}\label{rmk: multiindexinvariants}Define the 
map $\psi_l : \mc{I}^l \rTo  c_n(k-l) \times \mc{Q} $ as $\psi_l(a) = (\lambda(a), A(a))$ where $\lambda(a)$ is thought as a function: $\{1, \dots, n \} {\rTo} \{0, \dots , k-l\}$ by extending it with zero outside $J(a)$. Moreover define the map
$\eta_l: c_n(k-l) \times \mc{Q}  \rTo \mf{P}_n(k-l) \times \mf{P}_n(k-l)$ as $\eta_l(\mu, B) = (\nu(\mu |_B), \nu(\mu|_{\bar{B}}))$, where $\nu(\mu |_B)$ and $ \nu(\mu|_{\bar{B}})$ are the partitions in the $G$-orbits of  the compositions $\mu |_B$ and $\mu|_{\bar{B}}$, respectively. Let, finally, $\phi_l$ be the composition $\phi_l := \eta_l \circ \psi_l$. We have the commutative diagram: 
\begin{diagram} \mc{I}^l & \rTo^{\psi_l} & c_n(k-l) \times \mc{Q} \\ 
 &\rdTo^{\phi_l} & \dTo^{\eta_l} \\
 & &   \mf{P}_n(k-l) \times \mf{P}_n(k-l)
\end{diagram}The following facts are easy to establish and hence the proof is left to the reader. 
\begin{itemize}\item $\psi_l$ is $H$-invariant and $G$-equivariant, $\eta_l$ is $G$-invariant and hence $\phi_l$ is $G\times H$-invariant, where  
$c_n(k-l) \times \mc{Q} $ is naturally seen as $G$-sets, since products of $G$-sets; 
\item The image $B(k,l):= {\rm im} \: \psi_l$ is characterized by elements $(\lambda, A)$ such that
$|A| = \min \{ 2, 2l \}$; the quotient map: $\hat{\psi}_l : \mc{I}^l/H \rTo B(k,l)$ is a bijection; 
\item The image $A(k,l):= {\rm im} \: \phi_l$ is characterized by elements $(\lambda, \mu)$ such that 
\begin{enumerate}
\item $|\lambda| + |\mu| = k-l$; 
\item $l(\lambda) \leq \min\{ 2, 2l \}$. 
\end{enumerate}
Moreover the induced quotient map: $\hat{\phi}_l: \mc{I}^l/G \times H \rTo A(k, l)$ is a bijection. 
\end{itemize}\end{remark}

\begin{notat}If $a \in \mc{I}^l$, it will be useful to set $t(a):=|A(a) \cap J(a)|$; we will just write $t$ when there is no risk of confusion. It is an invariant of the $G \times H$-orbit of $a$. Moreover, denote as $\nu_A(a)$ and $\nu_{\bar{A}}(a)$ --- or with $\nu_A$, $\nu_{\bar{A}}$
 when there is no risk of confusion --- the partitions $\nu(\lambda(a) \trest_{A(a)})$ and $\nu(\lambda(a) \trest_{\overline{A(a)}})$, respectively,  defined by $\phi_l(a)$. \end{notat}

\subsection{Stabilizers}
\begin{remark}\label{rmk: equiv}Let $A$ and $B$ be two non empty  finite sets, and 
let $f : A \rTo B$ be a map. We will indicate with $\sim_f$ the equivalence 
relation defined by $x \sim_f y \iff f(x)=f(y)$ and with $\{ S_i(f) \}_i$ the 
associated partition. It defines a partition $\nu(f)$ of $|A|$ of length $|f(A)|$. \label{notaancora}
\end{remark}
\begin{notat}If $A \subseteq \{1, \dots,n\}$ (resp. $A \subseteq \{1, \dots,k\}$) we will indicate with $G(A)$ (resp. $H(A)$) the subgroup of $G$ (resp. of $H$) consisting of permutations fixing the complementary of $A$ in $\{1, \dots,n\}$ (resp. in $\{1, \dots,k\}$). The groups $G(A)$ (resp. $H(A)$) will also be identified with the symmetric group of $A$. 
\end{notat}

\begin{notat}Let $a \in \mc{I}^l$, $l \leq k$, let $A(a)$ and $J(a)$ its associated multi-indexes and let $\lambda(a)$ its associated composition. 
Let now $R \subseteq J(a)$. 
As explained in the notation above, $\lambda_R(a)$ as the restriction of $\lambda(a)$ to $R$. 
Let $S_i(\lambda_R(a))$, $i=1, \dots,s$ be  
the associated partition of $R$, defined by the equivalence relation 
$\sim_{\lambda_R}$, defined in remark \ref{rmk: equiv}. 
Define now the subgroup $M^R(a)$ of $\Aut(J(a))$ as being $M^R(a) := 
\prod_{i=1}^s \Aut(S_i(\lambda_R(a)))$.
Define, moreover, the subgroup $H^R(a)$ of $H$ as: $H^R(a) := \prod_{r \in R} H(a^{-1}(\{ r \}))$. \end{notat}
\begin{notat}We naturally identify the group $\Aut(J(a))$ with the subgroup  $G(J(a))$ of $G$. 
Let now $\sigma \in M^{J(a)}(a)$: this means that, for all $j \in J(a)$, $|a^{-1}(\{ j \})| = |a^{-1}(\{ \sigma(j) \})|$. 
There is now a unique way of lifting $\sigma$ to a permutation $\tilde{\sigma} \in \Aut \{1, \dots, k\}$ by imposing that $\tilde{\sigma}$ sends $a^{-1}( \{ j \})$ to 
$a^{-1}(\{ \sigma(j) \})$ in such a way that
	$ \tilde{\sigma}\trest_{a^{-1}( \{ j \})} : a^{-1}( \{ j \}) \rTo 
a^{-1}(\{ \sigma(j) \})$ is strictly increasing. In this way we get a monomorphism of groups $M^{J(a)}(a) \rInto \Aut \{1, \dots,k\}=H$
	sending $\sigma$ to $\tilde{\sigma}$. If, moreover, $R \subseteq J(a)$, we can lift $ \sigma \in M^R(a)$ to a $\tilde{\sigma} \in
H$ via the composition $M^R(a) \rInto M^{J(a)}(a) \rInto H$. 
Consider now the monomorphism:
\begin{diagram}[height=0.5cm,LaTeXeqno]\label{diago}
M^R(a) & \rInto & G \times H \\
\sigma & \rMapsto & (\sigma, \tilde{\sigma})	
\end{diagram}
where $\tilde{\sigma}$ is the lifting of $\sigma$ to $H$. 
The image of (\ref{diago}) is denoted with $\Delta^R(a)$. 
\end{notat}
\begin{lemma}
	Consider the groups $H^R(a)$ and $\Delta^R(a)$. 
Then, if $\hat{\tau} =(\tau, \tilde{\tau}) \in \Delta^R(a)$ and 
$\sigma = \prod_{r \in R}\sigma_r \in H^R(a)$, we have in $G \times H$: 
	\begin{equation} \label{eq:azsemidir}\hat{\tau} \cdot \sigma = \prod_{r \in R} (\tilde{\tau} \sigma_{r} \tilde{ \tau} ^{-1}) \cdot \hat{\tau} \in H^R(a) \cdot \hat{\tau}
		\end{equation}As a consequence, $H^R(a)$ is normal in 
$\langle \Delta^R(a), H^R(a) \rangle$ and hence
$ \langle \Delta^R(a), H^R(a) \rangle 
\simeq \Delta^R(a) \ltimes  H^R(a)$. 
\end{lemma}
\begin{proof}We have, just formally:
$\hat{\tau} \sigma =  (\tau, \tilde{\tau}) \sigma = (\tau, \tilde{\tau}
\sigma) = (\tau, \tilde{\tau} \sigma \tilde{\tau}^{-1} \tilde{\tau})
=  (1, \tilde{\tau} \sigma \tilde{\tau}^{-1}) (\tau, \tilde{\tau}) 
=  \tilde{\tau} \sigma \tilde{\tau}^{-1} \hat{\tau}$. 
Moreover:
\begin{align*}
\tilde{\tau} \sigma \tilde{\tau}^{-1} = 
\tilde{\tau} \sigma_1 \cdots \sigma_r \tilde{\tau}^{-1} 
=  \tilde{\tau} \sigma_1 \tilde{\tau}^{-1} \tilde{\tau} \cdots
\tilde{\tau}^{-1} \tilde{\tau} \sigma_r \tilde{\tau}^{-1} = \prod_{r
  \in R} (\tilde{\tau} \sigma_r \tilde{\tau}^{-1}) \;.
\end{align*}Now, for each $r \in R$, $\tilde{\tau} \sigma_r
\tilde{\tau}^{-1}$ is in $H(a^{-1}(\tau(r))) \subset H^R(a)$, by definition of $\tilde{\tau}$ and $H^R(a)$.
\end{proof}
\begin{remark}Let $a \in \mc{I}^l$ and let $R \subseteq J(a)$. We will indicate with $D^a(R)$ the semidirect
product $ 
D^a(R) := \Delta^{R}(a) \ltimes H^R(a)$.
	\label{rmk:equal}
Remark that if $\lambda_i(a) \neq \lambda_j(a)$ for all $i \neq j \in
R$, then $D^a(R) \simeq H^R(a)$, since any $S_i(\lambda_R(a))$
consists of just one element and therefore $M^R(a) \simeq \{1\}$.
\end{remark}
In what follows we  write for brevity's sake $S_0$, $A$ and $J$, instead of $S_0(a)$, $A(a)$ and~$J(a)$. 
\begin{pps}\label{stabiliziamo}Let $a \in \mc{I}^l$. The stabilizers of $a$ in $G \times H$ and in $H$ are
\begin{gather*} \Stab_{G \times H}(a)  \simeq  G(\overline{A \cup J}) 
\times G(A \setminus J) \times H(S_0) \times
	 D^a(A \cap J) \times D^a(J \setminus A)\;, \\
\Stab_{H}(a)  \simeq  H(S_0) \times
	 H^{J}(a) \simeq H(S_0) \times H^{A \cap J}(a) \times H^{J \setminus A}(a) \;,\end{gather*}respectively. If $l=0$  formulas make sense setting $S_0 = \emptyset$ and 
$\perm(\emptyset) =\{1 \}$. 
\end{pps}\begin{proof}
An element $(\sigma, \tau)$ is in $\Stab_{G \times H}(a)$ if and only if 
$\sigma a = a \tau$. Since the sets $A, J$ and, consequently, 
$A \cap J$, $A \setminus J$, $J \setminus A$, $\overline{A \cup J}$ are fixed by $\sigma$, it follows that $$\sigma = \sigma_1 \sigma_2 \sigma_3 \sigma_4 \in G(A \setminus J) \times G(A \cap J) \times 
G( J \setminus A) \times G(\overline{A \cup J})$$Since $\sigma_1$, $\sigma_4$ are already in $\Stab_{G \times H}(a)$, then we have: 
$ \sigma^{\prime} a = a \tau $, where $\sigma^{\prime} = \sigma_2 \sigma_3 \in G(A \cap J) \times 
G( J \setminus A) \subset \Aut(J)$. Suppose that $\sigma (i) =j$, with $i,j \in J$ such that $\{ i,j \} \subseteq A \cap J$ or $\{ i,j \} \subset J \setminus A$. 
Then one has to have $\tau(a ^{-1}(i)) = a^{-1}(j)$ which 
implies $\lambda_i(a) 
= \lambda_j(a)$. Since we can argue like this for all $i,j$ moved by $\tau$, we can always write $\tau$ as: 
$ \tau = \tilde{\sigma_2} \tau_2 \tilde{\sigma_3} \tau_3 \tau^{\prime}$ where $(\tau^{\prime}, 
\tau_2, \tau_3) \in H(S_0) \times H^{A \setminus J}(a) \times H^{J \setminus A}(a)$. 
This proves the first statement. The second follows immediately from $\Stab_H(a) = \Stab_{G \times H}(a) \cap \{1 \} \times H$. 
\end{proof}
\subsection{Invariants}\label{subsection: Invariants}
In this subsection we will compute invariants $(E^{p,q}_1)_0^H$ and $(E^{p,q}_1)_0^{G \times H}$. As a consequence of remark \ref{rmk: deg}, this will prove that the spectral sequence $(j^* E^{p,q}_1)^H$ -- and hence $(j^* E^{p,q}_1)^{G \times H}$ --
degenerates at level $k$. \emph{From now on $L$ will always denote a line bundle over the surface $X$}. To understand the comprehensive $G \times H$ action on the terms
$$ (E^{p,q}_1)_0 := \bigoplus_{a \in \mc{I}^p} \Tor_{-q}(L_{a(1)}, \dots, L_{a(k)}) \:, $$induced by the $G \times H$-action on the spectral sequence $E^{p,q}_1$, we have to be careful \emph{to take into account that the sheaves $L_{a(i)}$ are 
direct summands of the terms $\mc{C}^{|a(i)|-1}_L$ of the complex $\mc{C}^\bullet_L$}: now, not only \emph{the complex $\mc{C}^\bullet_L$ is $G$-equivariant}, but \emph{there is a $H$-action on the derived tensor product $\mc{C}^\bullet_L \tens^L \cdots \tens^L \mc{C}^\bullet_L$}. These facts induce additional 
signs which have been thouroughly studied in \cite[sections 4.1, 4.2, Appendix B]{Scala2009D} and which we will explicit in the sequel. 

\begin{remark}\label{rmk: identification}Let $A \subseteq \{ 1, \dots, n \}$, $|A|=2$. 
In what follows we  will frequently make the identification $X^n \simeq X^A \times X^{\bar{A}}$. More precisely, let $\sigma_A$ be the unique permutation of $\{1, \dots, n\}$ carrying $A$ onto $\{1, 2\}$ and such that its restrictions to $A$ and $\bar{A}$ preserve the order. The identification $X^n \simeq X^A \times X^{\bar{A}}$ is nothing but the automorphism
$(\sigma_A)_* : X^n \rTo X^n$ induced by the permutation $\sigma_A$. In the sequel the phrase ``in the identification $X^n \simeq X^A \times X^{\bar{A}}$" will mean ``modulo the automorphism $(\sigma_A)_*$ " . Moreover, when using the identification, functions defined over $A$ and $\bar{A}$ will, in the identification, be thought as functions defined over $\{1, 2\}$ and $\{3, \dots, n \}$, respectively; permutations in $\perm(A)$ and $\perm(\bar{A})$ wil be thought as permutations in $\perm_2$, $\perm(\{3, \dots, n \})$, respectively. 
\end{remark}

Denote first of all with $F^L_q(a)$ the multitor $\Tor_{-q}(L_{a(1)}, \dots, L_{a(k)})$. Let now $(\lambda, A) \rMapsto a(\lambda, A)$ and $(\nu, \mu) \rMapsto a(\mu, \nu)$ fixed sections of the quotient surjections $\mc{I}^l \rTo B(k, l)$ and $\mc{I}^l \rMapsto 
A(k, l)$, respectively. Danila's lemma \cite[Lemma 2.2]{Danila2001}  for the $H$- and $G \times H$-actions, respectively, on $(E^{p,q}_1)_0$ can be rephrased by saying: 
\begin{gather}\label{eq: homH}
(E^{p,q}_1)_0^H \simeq \bigoplus_{(\lambda, A) \in B(k,l)} F^L_q(a(\lambda, A))^{\Stab_H(a(\lambda, A))} \\
\label{eq: homGH} (E^{p,q}_1)_0^{G \times H} \simeq \bigoplus_{(\mu, \nu) \in A(k,l)}  \left ( \pi_* F^L_q(a(\mu, \nu)) \right) ^{\Stab_{G \times H}(a(\mu, \nu))} 
\end{gather}It will be then sufficient to understand the action of the stabilizers $\Stab_H(a)$ and $\Stab_{G \times H}(a)$ on the terms $F^L_q(a)$ and $\pi_* F^L_q(a)$, respectively. More explicitely 
$$F^{L}_q(a) := \Tor_{-q}(L_A, \cdots, L_A) \tens \Tens_{j \in A \cap J} L_j^{\lambda_j(a)} \tens \Tens_{j \in J \setminus A} L_j^{\lambda_j(a)}$$Denote now with $F_{1, q}^L(a)$ and $F_{2,q}^L(a)$  the sheaves on $X^{A}$, $X^{\bar{A}}$, 
respectively by 
\begin{gather*}
F_{1, q}^L (a) :=  \Tor_{-q}(L_A, \cdots, L_A) \tens \Tens_{j \in A \cap J} L_j^{\lambda_j (a)} \simeq \Lambda^{-q}(N_{A}^* \tens \mbb{C}^{p-1}) \tens L_A^{ p+|\nu_A | }\\
F_{2, q}^L(a) := \Tens_{j \in J \setminus A} L_j^{\lambda_j (a)}  \:,
\end{gather*}
where we used \cite[Lemma B.3]{Scala2009D} and 
where we indicate with $N_A^*$ the conormal bundle of the diagonal in $X^A$, where
the $L_A$'s are naturally seen as sheaves over $X^A$ and the $L_j$'s are seen as sheaves on $X^A$ or $X^{\bar{A}}$ depending if $j \in A$ or $j \in \bar{A}$. Of course we have, in the identification $X^n \simeq X^A \times X^{\bar{A}}$: $$F^L_q(a) := F_{1, q}^L(a) \boxtimes F^L_{2, q}(a)  \:.$$Write now 
$\Stab_{G \times H}(a)$ as 
$\Stab_{G \times H}(a) = G_1(a) \times G_2(a)$ where 
\begin{gather*}
G_1(a) = H(S_0) \times G(A \setminus J) \times D^a(A \cap J) \:, \quad
G_2(a) = D^a(J \setminus A) \times G(\overline{A \cup J})\:.
\end{gather*}Analoguosly, setting $H_i(a) = G_i(a) \cap (\{1\} \times H)$, that is: 
$$ H_1(a)= H(S_0) \times H^{A \cap J}(a) \:, \quad H_2(a) = H^{J \setminus A}(a) \:,$$
we have that $\Stab_H(a) \simeq H_1(a) \times H_2(a)$. 
Denote with $v$  the projection
$$v: S^{|A|}X \times S^{|\bar{A}|}X \times  \rTo S^n X\:.$$For any $\emptyset \neq B \subseteq \{1, \dots, n\}$ denote with 
$\pi_B$ the quotient projection: $\pi_B: X^B \rTo S^{|B|} X$. 
We have: 
\begin{gather} \label{eq: xH}F^L_q(a)^{\Stab_{H}(a)} \simeq F^L_{1,q} (a)^{H_1(a)} \boxtimes F_{2,q} ^L(a)^{H_2(a)}    \\ \mbox{} 
\label{eq: xGH}
( \pi_* F^L_q(a)) ^{\Stab_{G \times H}(a)} \simeq v_* \left(  {\pi_A}_* F^L_{1,q} (a)^{G_1(a)} \boxtimes {\pi_{\bar{A}}}_* F_{2,q} ^L(a)^{G_2(a)} \right)\:.
\end{gather}Hence, to understand precisely the $G \times H$-action on $(E^{p,q}_1)_0$ we just need to understand the action of each group $G_i(a)$ on the sheaf $F_{i,q}^L(a)$. 
\begin{remark}\label{rmk: globact}The $G_2(a)$-action on $F_{2,q}^L(a)$ reduces easily to the $\Delta^{J \setminus A}(a)$-action, which is clear. 
As for the $G_{1}(a)$-linearization on $F_{1,q}^L(a)$ we remark the following. Recall that $t=t(a)=|A(a) \cap J(a)|$. 
\begin{itemize}
\item The group $G(A \setminus J) \simeq \perm_{2-t}$ acts on $F_{1,q}^L(a)$, fiberwise on $\Delta \subseteq X^A$, with the representation $\Lambda^{-q}( \id_{\mbb{C}^2 } \tens \epsilon_{2-t}  \tens \mbb{C}^{p-1})$. We have still to take into account that the sheaves $L_A$ are direct summands of the term $\mc{C}^1_L$ of the complex $\mc{C}^\bullet_L$: hence there is an additional sign induced by the $G$-linearization on $\mc{C}^\bullet_L$. Since $G(A \setminus J)$ acts with a sign on each $L_A$, this accounts for an additional representation $\epsilon_{2-t}^{\tens p}$. Comprehensively, $G(A \setminus J)$ acts fiberwise on $F_{1, q}^{L}(a)$ with the representation
$$ \Lambda^{-q}( \id_{\mbb{C}^2 }  \tens \mbb{C}^{p-1}) \tens \epsilon_{2-t}^{\tens p-q} \:.$$
\item The factors in $H^{A \cap J}(a)$ act trivially on $F_{1, q}^L(a)$, hence the $D^a(A \cap J)$-action reduces to the action of $\Delta^{A \cap J}(a)$. The group $\Delta^{A \cap J}(a)$ is trivial  if $t(a)=0$ 
or if $|A \cap J|=t(a)=1$. The action of $\Delta^{A \cap J}(a)$ is trivial on the factor 
$\tens_{j \in J \setminus A} L_{j}^{\lambda_j(a)}$. On the other hand, nontrivial factors in $\Delta^{A \cap J}(a) \subseteq D^a(A \cap J)$ act nontrivially on $N_A^*$ and hence on $\Tor_{-q}(L_A, \dots, L_A)$. This can happen if and only if $t(a) = |A \cap J|=2$, that is, if $A \subseteq J$ \emph{and} $\lambda_i(a) = \lambda_j(a)$ if $A = \{i, j\}$, with $i \neq j$, that is, if $\nu_A(a)$ is a partition of the form $(h,h)$. In any  case the subgroup $\Delta^{A \cap J}(a)$ acts naturally on the fibers of $F_{1,q}^L(a)$ over the diagonal 
with the representations $\Lambda^{-q}(\id_{\mbb{C}^2 \tens \mbb{C}^{p-1}} \tens \epsilon_t) $. Taking into account the $G$-linearization of $\mc{C}^\bullet_L$, $\Delta^{A \cap J}(a)$ acts with a sign on each $L_A$ and hence
there is an additional representation $\epsilon_t^{\tens p}$ to consider.
The comprehensive fiberwise $D^a(A \cap J)$-action on $F_{1,q}^L(a)$ is then reduced to the $\Delta^{A \cap J}(a) \simeq \perm_t$-action, given by: 
$$ \Lambda^{-q}(\id_{\mbb{C}^2 \tens \mbb{C}^{p-1}} ) \tens \epsilon_t^{\tens p -q} \;.$$  
\item As a $H(S_0) \simeq \perm_p$-sheaf, $F_{1,q}^L(a)$ is naturally isomorphic to $\Lambda^{-q}(N_A^* \tens \rho_p) \tens L_A^{p+|\nu_A|}$, where $\rho_p$ is the standard representation of $\perm_p$. 
However, since the $L_A$'s in the multitors $\Tor_{-q}(L_A, \dots, L_A)$ are not just sheaves, but direct summands of the term $\mc{C}^1_L$  of the complex $\mc{C}^\bullet_L$, the sheaf $F_{1,q}^L(a)$ inherits an additional sign $\epsilon_p$, induced by the $H$-action on the spectral sequence $E^{p,q}_1$ (see \cite[sections 4.1, 4.2, Appendix B]{Scala2009D}). Hence, the comprehensive $H(S_0)$-action on $F_{1,q}^L(a)$ is given by the representation 
$$ \Lambda^{-q}(N_A^* \tens \rho_{p}) \tens L_A^{p+ |\nu_A|} \tens \epsilon_p $$
\end{itemize}
Putting all the pieces together, the $G_1(a)$-linearization on the sheaf $F_{1,q}^L(a)$ reduces to the $\perm_{2-t} \times \perm_{t} \times \perm_p$-linearization given by the representation: 
\begin{equation}\label{eq: GHrepr} \Lambda^{-q}(N_A^* \tens \rho_p) \tens \epsilon_t^{p-q} \tens \epsilon_{2-t}^{p-q} \tens \epsilon_p \;.
\end{equation}
\end{remark}
\begin{notat}\label{notat: Lmu}Let $\lambda \in c_{m}(l)$ a composition of $l$ of range $m$. Denote with $L^\lambda$ the 
line bundle on $X^m$ defined by: $$L^\lambda := \Tens_{i =1}^m L_i^{\lambda_i} \;.$$
Denote with $\mc{L}^{\lambda}_{m}$, or simply with $\mc{L}^{\lambda}$ when there is no risk of confusion, 
the sheaf over $S^m X$ defined by: 
$$ \mc{L}^{\lambda}_{m} := \pi_* (L^\lambda)^{\Stab_{\perm_m}(\lambda)} \: ,$$where $\pi: X^m \rTo S^mX$ is the quotient projection. 
\end{notat}
\begin{remark}It is clear that if $\lambda$ and $\mu$ are compositions of range $m$ in the same $\perm_m$-orbit, 
that is $\lambda = \mu \circ \sigma$, for $\sigma \in \perm_m$, then, $\mc{L}^\lambda$ and $\mathcal{L}^\mu$ are isomorphic. 
Indeed, 
denoting with $\sigma$ the automorphism of $X^m$ such that $\sigma(x_1, \dots, x_m) = (x_{\sigma^{-1}(1)}, \dots, x_{\sigma^{-1}(m)})$, we have that $\sigma_* L^\lambda \simeq L^\mu$ and that 
$$ \mc{L}^\mu = \pi_* (L^\mu)^{\Stab_{\perm_m}(\mu)} = \pi_*(\sigma_* L^{\lambda})^{\Stab_{\perm_m}(\mu)} \simeq 
\pi_*(L^{\lambda})^{\Stab_{\perm_m}(\lambda)} = \mc{L}^\lambda$$where the middle isomorphism is induced by $\pi_*(\sigma_*)$. Hence 
the isomorphism class of $\mc{L}^{\lambda}$ does not  depend on the composition $\lambda$ but only on its associated partition $\nu(\lambda)$. 
\end{remark}
\begin{notat}\label{notat: brevdelta}If $F$ is a coherent sheaf on $X$, for brevity's sake, we will indicate with $F_\Delta$ the sheaf on $X^2$ (resp. $S^2X$) defined by 
${i_{\Delta}}_* F$, where $i_\Delta : X \rInto X^2$ (resp. $i_\Delta: X \rInto S^2X$) is the diagonal immersion. 
\end{notat}
\begin{pps}Let $a \in \mc{I}^p$, $p>0$. The invariants $F_{q}^L(a)^{\Stab_H(a)}$ are zero if and only if $q \neq 1-p$ or $p > k$, otherwise: 
$$  F^L_{1-p}(a)^{\Stab_H(a)} = S^{p-1}N^*_{A} \tens L_{A}^{p+|\lambda_A(a)|} \tens L^{\lambda_{\bar{A}}(a)} \;.$$
The invariants $\pi_* F_{q}^L(a)^{\Stab_{G \times H}(a)}$, over the symmetric variety $S^{n}X$, are zero if $q \neq 1-p$ or $p> k$, otherwise they are given by: 
$$ \pi_* F_{1-p}^L(a)^{\Stab_{G \times H}(a)} = \left \{ \begin{array}{cl}
0 & \quad \mbox{if $\nu_A =(h,h)$, for $h \in \mbb{N}$} \\
v_* \big( ( S^{p-1} \Omega_X^1 \tens L^{p + |\nu_A(a)|} )_{\Delta} \boxtimes \mc{L}^{\nu_{\bar{A}}(a)} \big) & \quad \mbox{if $\nu_A  \neq (h,h)$\:.}
 \end{array} \right. $$
 \end{pps}
\begin{proof}As for the $\Stab_H(a)$-invariants, remark that in (\ref{eq: xH}) the group $H_2(a)$ acts trivially on $F_{2,q}^L(a) = L^{\lambda_{\bar{A}}}$. Moreover the $H_1(a)$ action reduces to the $H(S_0) \simeq \perm_p$-action. Since $F_{1,q}^L(a) \simeq \Lambda^{-q}(N_\Delta^* \tens \rho_p) \tens L_A^{p + |\lambda_A(a)|}$ as a $\perm_p$-sheaf, 
the first statement follows from corollary \ref{crl: torantiinv}. 

Let's consider now the $\Stab_{G \times H}(a)$-invariants. After remark \ref{rmk: globact}, lemma \ref{lmm: antiinv} and corollary \ref{crl: torantiinv}, the 
$G_1(a)$-invariants  ${\pi_A}_* F_{1,q}^{L}(a)^{G_1(a)}$ are zero if $q \neq 1-p$ and coincide with 
the $\perm_t \times \perm_{2-t}$-invariants of the sheaf: 
$$ {\pi_A}_* ( S^{p-1}N_\Delta^* \tens L^{p + |\nu_A(a)|} )  \tens \epsilon_t^{2p-1} \tens \epsilon_{2-t}^{2p-1} $$if 
$q = 1-p$ and $\nu_A(a) = (h,h)$ for a certain $h \in \mbb{N}$ and with the $\perm_t \times \perm_{2-t}$-invariants of the sheaf: $$
{\pi_A}_* ( S^{p-1}N_\Delta^* \tens L^{p + |\nu_A(a)|} )  \tens \epsilon_{2-t}^{2p-1}$$if 
$q = 1-p$ and $\nu_A(a)$ is not of the form $(h,h)$. For the latter, we necessarily have $t \geq 1$ and the invariants are ${\pi_A}_* ( S^{q-1}N_\Delta^* \tens L^{p + |\nu_A(a)|} )$. For the former, that is, when the partition $\nu_A(a)$ is of the form $(h,h)$, necessarily $t=2$ or $t=0$ and the $\perm_t \times \perm_{2-t}$ representation above reduces in any case to the 
$\perm_2$-representation
$ {\pi_A}_* ( S^{p-1}N_\Delta^* \tens L^{p + |\nu_A(a)|} )  \tens \epsilon_2$, which has no invariants. Hence there are nonzero invariants only if $q=1-p$ and in this case
$$ {\pi_A}_* F_{1,1-p}^{L}(a)^{G_1(a)} \simeq \left \{ \begin{array}{lcl} 
{\pi_A}_* ( S^{p-1}N_\Delta^* \tens L^{p + |\nu_A(a)|} ) & \quad &\mbox{if $\nu_A  \neq (h,h)$}  \\
0 & \quad &\mbox{if $\nu_A =(h,h)$, for $h \in \mbb{N}$} \: .
\end{array}
\right. $$

Let's now consider  ${\pi_{\bar{A}}}_* F_{2,q}^L(a)^{G_2(a)}$. We have: 
$${\pi_{\bar{A}}}_* F_{2,q}^L(a)^{G_2(a)} = {\pi_{\bar{A}}}_* (L^{\lambda_{\bar{A}}})^{G_2(a)} \;.$$It is now sufficient to remark that $G_2(a) \simeq \Stab_{G(\bar{A})}(\lambda_{\bar{A}})$: as a consequence:$${\pi_{\bar{A}}}_* F_{2,q}^L(a)^{G_2(a)} = {\pi_{\bar{A}}}_* (L^{\lambda_{\bar{A}}})^{\Stab_{G(\bar{A})}(\lambda_{\bar{A}})}
\simeq \mc{L}^{\nu_{\bar{A}}(a)} $$over $S^{|\bar{A}|} X$. Note that, over $S^{|A|}X$, we can rewrite
${\pi_A}_* ( S^{p-1}N_\Delta^* \tens L^{p + |\nu_A(a)|} ) \simeq ( S^{p-1} \Omega_X^1 \tens L^{p + |\nu_A(a)|} )_{\Delta}$. 
We now conclude by (\ref{eq: xGH}). 
\end{proof}
\begin{notat}
Denote with $A_0(k,l)$ the set of all couples $(\lambda, \mu) \in
A(k,l)$ such that $\lambda \neq 0$, $\lambda$ not of the form $(h,h)$.
\end{notat}As an immediate consequence of (\ref{eq: homH}), (\ref{eq: homGH}), of the previous proposition
and of remark \ref{rmk: deg} we have
\begin{crl}\label{crl: EGH}The term $(E^{p,q}_1)_0^H$ is nonzero if and only if $p=q=0$ or $q = 1-p$, $1 \leq p \leq k-1$. In these cases the term $(E^{p,q}_1)_0^H$ is
\begin{gather*} 
(E^{0,0}_1)_0^H \simeq \bigoplus_{\lambda \in c_n(k)} L^{\lambda} \\
(E^{p,1-p}_1)_0^H \simeq \bigoplus_{(\lambda, A) \in B(k,p)} S^{p-1}N^*_{A} \tens L_{A}^{p+|\lambda_A|} \tens L^{\lambda_{\bar{A}}} 
\end{gather*} 
The term $(E^{p,q}_1)_0^{G \times H}$ is nonzero if and only if $p=q=0$ or $q = 1-p$, $1 \leq p \leq k-1$. In these cases the term $(E^{p,q}_1)_0^H$ is
\begin{gather*} 
(E^{0,0}_1)_0^{G \times H} \simeq \bigoplus_{\lambda \in p_n( k)} \mathcal{L}^{\lambda} \\
(E^{p,1-p}_1)_0^{G \times H} \simeq \bigoplus_{(\mu, \nu) \in A_0(k, p)} v_* \big( ( S^{p-1} \Omega_X^1 \tens L^{p + |\mu|} )_{\Delta} \boxtimes \mc{L}^{\nu} \big) \end{gather*} 
As a consequence the spectral sequences $j^*(E^{p,q})^H$ and $j^* (E^{p,q}_1)^{G \times H}$ degenerate at level $k$. 
\end{crl}
\begin{remark}\label{rmk: dlsurj}As a consequence of corollary \ref{crl: EGH} and since the 
derived tensor power $\mc{C}^\bullet_L \tens^L \cdots \tens^L \mc{C}^\bullet_L$ is acyclic in degree $>0$, 
the higher differentials $d^r: j^* (E^{0,0}_r)^H \rTo j^* (E^{r, 1-r}_r)^H$ are surjective; analogously, the $G$-invariant higher differentials $(d^r)^G : j^* (E^{0,0}_r)^{G \times H} \rTo j^* (E^{r, 1-r}_r)^{G \times H}$ are surjective. In particular, for $n =2$,  the differentials $d^r$ and $(d^r)^G$ are surjective over the whole $X^2$ and $S^2 X$, respectively. 
\end{remark}

\section{Higher differentials and operators \texorpdfstring{$D^l_L$}{DL}.}

\label{subsection: Koszul}\sloppy
In this section we study higher differentials $d^l$ and $(d^l)^G$ of the spectral sequences $j^* (E^{p,q}_1)^{H}$ and $j^* (E^{p,q}_1)^{G \times H}$, respectively. 
We will prove that the higher differentials $d^l$ are, up to signs, restrictions of globally defined operators $D^l_L$, built recursively as higher order restrictions of sections to pairwise diagonals. As a consequence, we get a  description of the Bridgeland-King-Reid transform $\bkrh(S^k L^{[n]})$ of symmetric powers of tautological bundles, as well as their  direct image $\mu_* (S^k L^{[n]})$, in terms 
of kernels of operators $D^l_L$ and their $G$-invariants, respectively. 

In order to explicitly compute higher differentials in the $G $-equivariant spectral sequence $(j^* E^{p,q}_1)^H$, 
we will locally solve the complex $j^*\comp{\mc{C}}_L \tens^L \cdots \tens^L j^*\comp{\mc{C}}_L $ with a bicomplex of locally free sheaves. In order to do so, we need to solve term by term the complex $\comp{\mc{C}}_L$. First of all we introduce some auxiliary open sets covering $X^n$ that will be used throughout this section.

\subsection{Auxiliary open covers}
\begin{lemma}\label{lmm: auxU}Let $X$ be a smooth quasi-projective surface and $L$ be a line bundle on $X$. Let $n \in \mbb{N}$, $n \geq 1$. 
Then $X^n$ and $S^n X$ are covered by affine open subsets of the form $U^n$ and $S^n U$, respectively, with $U$ an affine open subset of $X$ such that $L$ is trivial over $U$. 
\end{lemma}
\begin{proof}We just need to prove that, if $x_1, \dots, x_n$ are $n$ points over $X$  (not necessarily distinct), there is an affine open subset $U$ of $X$ containing all the $x_i$, $i = 1, \dots, n$ and such that $L$ is trivial over $U$. 
A proof of the first property was reported in \cite[Lemma 1.27]{ScalaPhD}. Therefore, consider an affine open set $U$ of $X$ containing $x_i$, $i = 1, \dots, n$. We just have to show that we can shrink $U$ sufficiently to achieve the triviality of $L$. 
Over $U$ consider the epimorphism $ L \rOnto L \tens \FS_Z $, where $Z$ is the scheme theoretic union $Z = \cup_{i=1}^n \{ x_i \} $ (and hence reduced). Taking global sections, since 
$U$ is affine, we get an epimorphism $ \Gamma(U, L) \rOnto \Gamma(U, L_Z) = \oplus_{j} L_{x_{i_j}}$, where $\{x_{i_1}, \dots, x_{i_l} \} = \{ x_1, \dots, x_n \}$ and $x_{i_j}$ are all distinct. Hence there exists a section $s \in \Gamma(U, L)$ such that 
$s(x_i) \neq 0$ for all $i =1, \dots, n$. The open set $U \setminus Z(s)$ is the affine open subset we want. 
\end{proof}

\begin{remark}\label{rmk: auxiliaryopen}Let $U$ a smooth affine open set of $X$ over which $L$ is trivial, found in lemma \ref{lmm: auxU}. Consider a partial diagonal $\Delta_I \subseteq U^n$, $|I| \geq 2$: since smooth, it is locally complete intersection inside $U^n$. On the other hand, even shrinking $U$, we can't guarantee that $\Delta_I$ is complete intersection inside $U^n$. However, 
 it easy, possibly shrinking $U$, to build a smooth 
  complete intersection of dimension $2$ inside $U^n$ having $\Delta_I$ as an irreducible component. Hence, by shrinking $U$ and further removing a 
 closed subset of $U^n$, we can succeed building
 an affine open subset $\tilde{V} \subseteq U^n$ containing $\Delta_I$ and such that $\Delta_I$ is a complete intersection inside $\tilde{V}$. 

Analogously, if $p \in U^n $ it is possible to find an affine open subset $V_{U^n, p} \subseteq U^n$, containing $p$ and such that any partial diagonal $\Delta_J$, $J \subseteq \{1, \dots, n \}$, $|J| \geq 2$ is 
 a complete intersection inside $V_{U^n,p}$, or empty. The open sets $V_{U^n, p}$, with $U$ affine open of $X$ as in lemma \ref{lmm: auxU} and $p \in U^n$ cover $X^n$. In what follows we will drop the point $p$ and we will just denote the affine open sets $V_{U^n, p}$ with $V_{U^n}$. 

We will also denote with $U^n_{**}$ the intersection $U^n \cap X^n_{**}$ and $V_{U^n, **}$ the intersection $V_{U^n} \cap X^n_{**}$. 
\end{remark}
 
\subsection{A term-by-term Koszul resolution of $\comp{\mc{C}}_L \tens^L \cdots \tens^L  \comp{\mc{C}}_L$.  }
Consider now an affine open set of the form $U^n$, built in lemma \ref{lmm: auxU}. 
Since on the affine open set $U$ the line bundle $L$ is trivial, over $U^n$ the complex $\comp{\mc{C}}_L$ can be written as: 
$$ \mc{C}^p_L = \bigoplus_{|J|=p+1} L_J= \bigoplus_{|J|=p+1} \FS_J \;.$$Now, the sheaves $\FS_J$, for $|J| \geq 2$,  are nothing but the structural sheaves of partial diagonals $\Delta_J$: hence they can be solved, restricting to smaller open sets $V_{U^n}$ if necessary, by a Koszul complex $K^\bullet (F_J, s_J)$, where $s_J$ is a section of the trivial vector bundle $F_J$ of rank $\rk F_J = \codim \Delta_J = 2p$, transverse to the zero section. Set $F_i = \FS_i$ and $K^p(F_i, s_i) = F_i$ if $p =0$ and $K^p(F_i, s_i) = 0$ if $p \neq 0$.  As a consequence the bicomplex $R^{\bullet, \bullet}$, defined, for $p \geq 0$ and $q \leq 0$, as: 
$$ R^{p, q} := \bigoplus_{|J|=p+1} K^{q}(F_{J}, s_J)$$with natural differentials, provides a term-by-term locally free resolution $R^{\bullet, \bullet} \rTo \comp{\mc{C}}_L$ of the complex $\comp{\mc{C}}_L$ over the open set $V_{U^n}$. 
\begin{remark}In the direct sum here above, and whenever we write Koszul complexes through subsection \ref{subsection: higher}, it is implicit that 
we just take multi-indexes $J$ such that $\Delta_J \cap V_{U^n} \neq \emptyset$. \end{remark}
The vertical differential $\delta$ of the bicomplex $R^{\bullet, \bullet}$ is induced by differentials of the
Koszul complexes, of which $R^{p,
  \bullet}$ is direct sum. 
The horizontal differential $\partial: R^{p, q} \rTo R^{p+1, q}$ is
given by$$(\partial^p (x))_J = \sum_{i \in J} \epsilon_{i,J} i_{i,J}^{-q}(x_{J
  \setminus \{i \} }) $$where $i_{i,J}^{-q}$ is the injection
$i_{i,J}^{-q} : \Lambda^{-q}(F_{J \setminus \{ i \}}) \rTo \Lambda^{-q}(F_J)$ and where the sign $\epsilon_{i, J}$ has been explained in the definition of the complex $\comp{\mc{C}}_L$ (subsection \ref{subsection: taut}). 

The derived $k$-fold tensor product $\comp{\mc{C}}_L \tens^L \cdots \tens^L \comp{\mc{C}}_L$ is then solved, on $V_{U^n}$, by the bicomplex $$\mbb{R}^{\bullet, \bullet}:= R^{\bullet, \bullet} \tens \dots \tens
R^{\bullet, \bullet} $$(for the sum of the first indexes and the sum of
the second indexes). We can rewrite the bicomplex $R^{\bullet, \bullet}$ as: 
\begin{equation*}\mbb{R}^{p,q} =\bigoplus_{\substack{p_1 + \dots + p_k=p \\ q_1 + \dots + q_k=q}}
R^{p_1, q_1} \tens \dots \tens
R^{p_k, q_k} \\ 
= 
\bigoplus_{|J_1|+ \dots
  +|J_k|=p+k}  \Lambda^{-q}( F_{J_1} \oplus \dots \oplus F_{J_k}) \; .
 \end{equation*}
With our notations (see notation \ref{notaziona}), this can be also written:
$$ \mbb{R}^{p,q} := \bigoplus_{a \in \mc{P}^k , \: l(a) =p }   
\Lambda^{-q}(F_{a(1)} \oplus \dots \oplus F_{a(k)}) \;.
$$
\begin{remark}[Differentials] 
Since $\mbb{R}^{\bullet, \bullet}$ is a tensor product of bicomplexes,
the vertical differential  
$ \delta : \mbb{R}^{p,q} \rTo \mbb{R}^{p,q+1} $ is given, on the component
$R^{p_1, q_1} \tens \dots \tens R^{p_i, q_i } \tens \cdots \tens
R^{p_k, q_k} \rTo R^{p_1, q_1} \tens \dots \tens R^{p_i, q_i+1 } \tens \cdots \tens
R^{p_k, q_k}$ by $$ \delta = 
(-1)^{q_1 + \dots +
  q_{i-1}} \delta_i \;,$$where $\delta_i$ is the vertical differential of the $i$-th factor $R^{\bullet, \bullet}$, and hence coincides
  with the direct sum of 
differentials of the
Koszul complexes $$ R^{\bullet}_{J_1, \dots, J_k} := 
K^{\bullet}(F_{J_1} \p \cdots \p F_{J_k},
s_{J_1} \p \cdots \p s_{J_k})\;.$$The horizontal differential \sloppy
$ \partial : \mbb{R}^{p,q} \rTo \mbb{R}^{p+1,q}$, on the component
$R^{p_1, q_1} \tens \dots \tens R^{p_i, q_i } \tens \cdots \tens~R^{p_k, q_k} 
\rTo R^{p_1, q_1}~\tens~\dots~\tens~R^{p_i +1, q_i }~\tens~\cdots~\tens
R^{p_k, q_k}$ is given by\begin{equation}\label{eq: horizdiff} \partial =
(-1)^{p_1 + \dots +
  p_{i-1}} \partial_i \end{equation}where $\partial_i$ is the horizontal
differential of the $i$-th bicomplex $R^{\bullet, \bullet}$. \end{remark}
\begin{remark}The horizontal differential $\partial$ depends just on $p$ and not
  on the $q$ involved; consequently it just depends on the multi-indexes 
$J_i$ involved. Therefore writing $\mbb{R}^{p,q}$ as a direct sum of Koszul complexes
$$ \mbb{R}^{p,\bullet}=\bigoplus_{|J_1|+ \dots
  +|J_k|=p+k}  \Lambda^ {-\bullet}( F_{J_1} \oplus \dots \oplus F_{J_k}) \simeq \bigoplus_{|J_1|+ \dots
  +|J_k|=p+k}  R^{\bullet}_{J_1, \dots, J_k} $$the
horizontal differential on 
$ \Lambda^{-q}( F_{J_1} \oplus \dots  \oplus F_{J_i \setminus \{h \}} \oplus \dots \oplus 
F_{J_k}) \rTo \Lambda^{-q}( F_{J_1} \oplus \dots \oplus F_{J_i }\oplus \cdots \oplus
F_{J_k}) $
is given by
\begin{equation}\label{eq: horizdiff2} (-1)^{|J_1| + \dots + |J_{i-1}| +i-1}\epsilon_{h, J_i} i_{h,J_i}^{-q} \end{equation}where
we still indicate with $i_{h,J_i}^{-q} : \Lambda^{-q}( F_{J_1} \oplus \dots  \oplus F_{J_i \setminus \{h \}} \dots \oplus
F_{J_k}) \rTo \Lambda^{-q}( F_{J_1} \oplus \dots \oplus F_{J_i }\oplus
\dots \oplus
F_{J_k})$ the injection induced by $F_{J_i \setminus \{h \}} \rInto F_{J_i }$. 
\end{remark}
We now introduce another bicomplex $\mbb{K}^{\bullet, \bullet}$, easier to deal with, also related to the spectral sequence $j^*E^{p,q}_1$. 
\begin{remark}\sloppy 
Consider, on the open set $V_{U^n}$, the bicomplex $\mbb{K}^{\bullet, \bullet}$
defined as$$\mbb{K}^{p,q} : = \bigoplus_{a \in \mc{I}^p}  R^q_{a(1), \dots,
  a(k)} = \bigoplus_{a \in \mc{I}^p}  \Lambda^{-q}(F_{a(1)} \p \cdots \p
F_{a(k)}) \;,$$with differentials defined exactly as in the previous remarks and where the direct sum is taken over the $a\in \mc{I}^p $ such that $\Delta_{a(i)} \cap V_{U^n} \neq \emptyset$ for all $i$.  It is a quotient bicomplex of $\mbb{R}^{\bullet, \bullet}$. Consider the open immersion $j_U: V_{U^n, **} \rTo V_{U^n}$. Since the Koszul 
complexes 
$j_U^*R^\bullet_{a(1), \dots, a(k)}$ are exact if $a \in \mc{P}^p \setminus \mc{I}^p$, the complexes $j_U^*\mbb{R}^{\bullet, \bullet}$ and $j_U^* \mbb{K}^{\bullet, \bullet}$ are quasi-isomorphic: 
$$ j_U^*\mbb{R}^{\bullet, \bullet} \simeq^{qis} j_U^* \mbb{K}^{\bullet, \bullet} \;.$$
\end{remark}
Consider now  the spectral sequence $K^{p,q}_1 :=H^q_{\delta} (\mbb{K}^{p, \bullet})$ associated to the bicomplex $\mbb{K}^{\bullet, \bullet}$ over the open set $V_{U^n}$. 
As a consequence of what we said, 
\begin{pps}\label{pps: spectlocal}The restriction of the spectral sequence $j^* E^{p,q}_1$ to the open set $V_{U^n,**}$ is isomorphic to the spectral sequence 
$ j^*_U K^{p,q}_1 :=H^q_{\delta} (j_U^*  \mbb{K}^{p, \bullet})$; analogously, the $H$-invariant spectral sequence 
$(j^* E^{p,q}_1)^H$ is isomorphic, over $V_{U^n,**}$, to the spectral sequence $(j^*_U K^{p,q}_1) ^H$. 
\end{pps}
\begin{remark}\label{rmk: horizK}For $a \in \mc{I}^p$, denote more briefly with $R^q_a$ the sheaf $R^q_{a(1), \dots a(k)}$. Consider the horizontal differentials
$$ \oplus_{a,b} \partial_{a}^b : \mbb{K}^{p,q} = \bigoplus_{a \in \mc{I}^p} R^q_a \rTo 
\bigoplus_{b \in \mc{I}^{p+1}} R^q_b = \mbb{K}^{p+1,q} \;.$$
The map $\partial_{a}^b : R^q_a \rTo R^q_b$ is nonzero  if and only if 
$A(a) = A(b)$, $S_0(a) \subset S_0(b)$ and $a_h \subseteq A(b)$
if $\{h \} = S_0(b) \setminus S_0(a)$. This means that, for a fixed $b
\in \mc{I}^{p+1}$ all possible nonzero maps to $R^q_b$ are of the
form
$\partial_{a}^b: R^q_a \rTo R^q_b$, where $a$ is of the form
$a_j=b_j$ for all $j \in \{1, \dots, k\} \setminus \{ l \}$, for some $l
\in S_0(b)$. For this $l$, $a_l \subseteq A(b)$, and can hence take
only $2$ possible values. Hence there are exactly $2p+2$ 
nonzero maps to
$R^q_b$. 
\end{remark}
\subsection{The $H$-invariant bicomplex $(\mbb{K}^{\bullet, \bullet})^H$ }
\paragraph{Homogeneous components.}
Denote, as in subsection \ref{subsection: Invariants}, with $(\lambda, A) \rMapsto a(\lambda, A)$ a fixed section of the quotient surjection $ \psi_l: \mc{I}^l \rTo B(k,l)$. For any $(\lambda, A) \in B(k,l)$ denote with $W^{\bullet}_{\lambda, A}$ the homogeneous component of $\mbb{K}^{\bullet, \bullet}$ indexed by $(\lambda, A)$, defined as: 
$$W^{q}_{\lambda, A} := \bigoplus_{a \in \psi_l^{-1}(\lambda, A)} R^{q}_a \;.$$Danila's lemma for the 
action of $H$ on the complex $\mbb{K}^{\bullet, \bullet}$ can be rephrased writing: 
$$ (\mbb{K}^{p, \bullet}) ^H \simeq \bigoplus_{(\lambda, A) \in B(k, p)} ( W^{\bullet}_{\lambda, A}) ^H \simeq 
 \bigoplus_{(\lambda, A) \in B(k, p)} ( R^\bullet_{a(\lambda, A)})^{\Stab_H(a(\lambda, A))} \;.$$Since, over the open set $U$ the line bundle $L$ is trivial, there is no harm in tensorizing each $R^q_{a(\lambda, A)}$ with the (trivial) line bundle $L^\lambda$: in this way we can keep track of the label $\lambda$. As a result we can write: 
 \begin{equation}\label{eq: W}  (\mbb{K}^{p, \bullet}) ^H \simeq \bigoplus_{(\lambda, A) \in B(k, p)} ( W^{\bullet}_{\lambda, A}) ^H \simeq 
 \bigoplus_{(\lambda, A) \in B(k, p)} ( R^\bullet_{a(\lambda, A)} \tens L^\lambda)^{\Stab_H(a(\lambda, A))} \end{equation}
\paragraph{Vertical and horizontal differentials.}
In what follows  we will use notations and facts from appendix \ref{section: appendixKoszul}. 
Let $(\lambda, A) \in B(k, p)$. Note that the $\Stab_H(a(\lambda, A))$-action on the complex 
$R^\bullet_{a(\lambda, A)}$ reduces to the $H(S_0(a)) \simeq \perm_p$ action. 
The complex $R_{a(\lambda, A)}^\bullet $ is isomorphic to the tensor product of $p$ Koszul complexes
$K^{\bullet}(F_A, s_A) \tens \cdots \tens K^{\bullet}(F_A, s_A) $. However, 
the $H(S_0(a))$-action on $R_{a(\lambda, A)}^\bullet$ has to be compatible with the $H(S_0(a))$-action on $\comp{\mc{C}}_L \tens^L \cdots \tens^L \comp{\mc{C}}_L$, which introduces a sign when permutating two consecutive $\mc{C}^1_L$ factors (see \cite[section 4.1]{Scala2009D}): hence, we have to take into account a further $\epsilon_p$-representation. Consequently, in 
the identification $H(S_0(a)) \simeq \perm_p$, the $H(S_0(a))$-equivariant complex $R_{a(\lambda, A)}^\bullet $ is isomorphic to the $\perm_p$-equivariant complex given by tensor product of $p$ Koszul complexes twisted by $\epsilon_p$: 
\begin{equation} \label{eq: tensorKoszul} R_{a(\lambda, A)}^\bullet \simeq K^{\bullet}(F_A, s_A) \tens \cdots \tens K^{\bullet}(F_A, s_A) \tens \epsilon_p \simeq K^\bullet(F_A \tens R_p, s_A \tens \sigma_p) \tens \epsilon_p\;. \end{equation}or with the Koszul complex $K^\bullet(F_A \tens R_p, s_A \tens \sigma_p)$, twisted by $\epsilon_p$ (see remark \ref{rmka: tensorKoszul}), where $R_p$ is the natural representation of the symmetric group $\perm_p$ and $\sigma_p$ is the invariant element in $R_p$. 
Now, 
from (\ref{eq: W}) and from corollary 
\ref{crla: invkoszul}, we deduce
\begin{lemma}[Vertical differentials] \label{lemma: verticaldiff}Let $p \geq 1$. Then the vertical complexes $(\mbb{K}^{p, \bullet})^H$ are isomorphic to 
the direct sum of twisted and shifted Koszul complexes: 
$$ (\mbb{K}^{p, \bullet})^H \simeq \bigoplus_{(\lambda, A) \in B(k, p)} K^{\bullet}(F_A, s_A) \tens S^{p-1}F_A^* \tens L^\lambda [p-1]  $$where $K^\bullet(F_A, s_A)$ is the Koszul complex resolving the partial diagonal $\Delta_A$. Hence the complexes $(\mbb{K}^{p, \bullet})^H $ are quasi-isomorphic to the sheaves $\bigoplus_{(\lambda, A) \in B(k,p)} S^{p-1}N_A^* \tens L^\lambda$, placed in degree $1-p$. \end{lemma}
\begin{crl}Consider the spectral sequence $(K^{p,q}_1)^H$, associated to the $H$-invariant bicomplex $(\mbb{K}^{\bullet, \bullet})^H$. At level $(K_1)^H$, the only nonzero terms are $(K^{0,0}_1)^H$ and $(K^{l, 1-l}_1)^H$, for $1 \leq l \leq k-1$. 
Moreover, for $1 \leq l \leq k-1$ we have $(K^{l, 1-l}_1)^H \simeq (K^{l, 1-l}_l)^H$. Consequently, $(K^{p,q}_1)^H$ degenerates at level $k$. \end{crl}

\begin{notat}\label{notat: orderpartitions}Let $\lambda \in c_n(l_1)$, $\mu \in c_n(l_2)$. We write that $\lambda \leq \mu$ if $\lambda_i \leq \mu_i$ for all $i \in \{1, \dots, n\}$. 
\end{notat}

\begin{lemma}[Relevant horizontal differentials]\label{lmm: possiblehoriz}Let $(\lambda, A) \in B(k, p-1)$, $(\lambda^\prime, A^\prime) \in B(k,p)$. If $p \geq 2$, the horizontal differential  between two homogeneous
  components
$$\partial_{\lambda, A}^{\lambda^{\prime}, A^{\prime}}: (W^{1-p}_{\lambda, A})^H \rTo (W^{1-p}_{\lambda^{\prime},
  A^{\prime}})^H$$is nonzero if and only if
a) $A=A^{\prime}$; 
b) $\lambda^\prime \leq \lambda$
, $\lambda^{\prime}_{\bar{A}} = \lambda_{\bar{A}}$. 
In this case, the map
\begin{equation*}\label{eqn: maphomcomp} \partial_{\lambda, A}^{\lambda^{\prime}, A}: (W^{1-p}_{\lambda, A})^H \simeq S^{p-2}F_A^* \tens F_A^* \tens L^\lambda \rTo S^{p-1}F_A^* \tens L^{ \lambda^\prime} \simeq (W^{1-p}_{\lambda^{\prime},
  A})^H 
\end{equation*}
can be identified\footnote{the identification is not canonical and, with different choices, it could differ by a positive constant, as explained in remark \ref{rmka: inclusion}} with $(p-1) \sym$ -- where $\sym$ is the symmetrization map\footnote{The symmetrization map, as well as the alternating map, will be recalled in remark \ref{rmka: symalt}} $\sym:  S^{p-2}F_A^*
\tens F_A^* \rTo S^{p-1} F^*_A$ -- twisted by the sign $(-1)^{p-2}\epsilon_{r, A}$, where $r$ is the unique element of $A$ such that 
$\lambda_r = \lambda^\prime_r + 1$.
\end{lemma}
\begin{proof}Consider $b_0 \in \psi_p^{-1}(\lambda^\prime, A^\prime)$; hence 
for any fixed $a_0 \in
\mc{I}^{p-1}$, such that $\lambda(a_0)=\lambda$ and $A(a_0)=A$, by the lemma \cite[Lemma A.1]{Scala2009D}, the map between homogeneous components 
$(W^{1-p}_{\lambda, A})^H \rTo (W^{1-p}_{\lambda^{\prime},
  A})^H$ identifies to the map: 
\begin{equation}\label{eqn: horizinvH} \big( R_{a_0}^{1-p} \big) ^{\Stab_H(a_0)} \rTo
  \big( R_{b_0}^{1-p} \big)^{\Stab_H(b_0)} \end{equation}given by: 
$x \rMapsto \sum_{[g] \in H/\Stab_H(a_0) } f_{g a_0, b_0 } (g x)$. Now
when $[g] \in H/\Stab_H(a_0)$, $ga_0$ spans all possible $a$ in the orbit $\psi_{p-1}^{-1}(\lambda, A)$ inside $\mc{I}^{p-1}$. By remark \ref{rmk: horizK}, $f_{a,b_0} \neq 0$ if and only
if $a$ is of the form $a_i = (b_0)_i$ for all $i \in \{1, \dots,
k \}\setminus \{ l \}$, $S_0(a) = S_0(b_0) \setminus \{l\}$ and $a_l \subseteq A(b_0)$. All other $a$ such that $f_{a,b_0} \neq 0$
are of the form $a= g a_0$ for $g \in H(S_0(b_0)) \simeq \perm_p$. 

The necessity of the listed conditions now follows  immediately. 
To prove sufficiency, take $(\lambda,
A)$ and $(\lambda^{\prime}, A)$ satisfying the
conditions a), b) in the statement. 
We necessarily have that $\lambda_{A}
\leq \lambda^{\prime}_{A}$, with 
$| \lambda^{\prime}_{A}|=| \lambda_{A} |-1$. Therefore there exists a unique
$h \in A$ such that
$\lambda^{\prime}_{h}=\lambda_{h}-1$; 
for any other $j$, 
$\lambda^{\prime}_{j}=\lambda_{j}$. 
Take $b_0 \in \psi_p^{-1}(\lambda^\prime, A)$, pick $l \in S(b_0)$ and 
define now $a_0$ as: $(a_0)_l=\{ h \}$
and $(a_0)_i=
(b_0)_i$ for any other $i$. 
Now $\Stab_H(a_0)$ and $\Stab_H(b_0)$ satisfy the hypothesis of lemma
\cite[Lemma A.2]{Scala2009D}, hence, recalling (\ref{eq: tensorKoszul}) or remark \ref{rmka: tensorKoszul},  the map (\ref{eqn: horizinvH}) identifies, up to signs, to the map of
$\perm_p$-invariants: 
$$ \left [\bigoplus_{i=1}^p \Lambda^{p-1}(F^*_A \tens R_{p-1}(i)) \tens
  \epsilon_{p-1}(i) \right]^{\perm_p} \tens L^\lambda {\rTo} [\Lambda^{p-1} (F^*_A \tens R_p) \tens
\epsilon_p]^{\perm_p} 
\tens L^{\lambda^{\prime}} \; .$$By lemma \ref{lemma: mappainvariants}, this map identifies to 
the map 
$$ (p-1) \sym \tens \id : S^{p-2}F_A^* \tens F_A^* \tens
L^\lambda \rTo S^{p-1}F_A^* \tens
L^{\lambda^\prime}$$where
${\rm sym}: S^{p-2}F_A^* \tens F_A^* \rTo S^{p-1}F_A^*$ is the symmetrization map. As for the sign, from (\ref{eq: horizdiff}) and (\ref{eq: horizdiff2}) we deduce immediately that we have to twist $(p-1) \sym \tens \id $ by $(-1)^{p-1}\epsilon_{s, A}$, if $A = \{r, s \}$ and the differential is induced by $i^{p-1}_{s, A}$. But $\epsilon_{s, A} = - \epsilon_{r, A}$ and $r$ is the unique 
element in $A$ such that $\lambda_r = \lambda^\prime_r +1$. 
\end{proof}
\begin{remark}
After the previous lemma and after remark \ref{rmk: horizK}, for
$(\lambda^{\prime}, A^{\prime}) \in B(k,p)$, there are
exactly two different $(\lambda, A) \in B(k,p-1)$ such that
$\partial^{\lambda^{\prime}, A^{\prime}}_{\lambda, A}$ is
non zero. They are exactly the $(\lambda, A)$ such that $A=
A^{\prime}$, 
and that 
for all 
$h \in \{ 1, \dots, n \}$, $\lambda_h= \lambda_{h}^{\prime}$, apart for one $h_0 \in A$, for which $\lambda_{h_0}=\lambda_{h_0}^{\prime}+1$. It is
clear that the two choices depend just on the two choices of $h_0 \in
A$. Using the multiplicative notation of compositions (see notation \ref{notat: comp}), if $A=\{ a_0, a_1 \}$, the two elements we
are considering are $(\lambda_0, A)$ and $(\lambda_1, A)$ where
$\lambda_0 = a_0 \lambda^\prime$, $\lambda_1 = a_1
\lambda^\prime$; hence $\supp \lambda_0 = \supp \lambda^\prime \cup \{ a_0 \}$,  $\supp \lambda_1 = \supp \lambda^\prime \cup \{a_1 \}$. 
\end{remark}The analogous of lemma \ref{lmm: possiblehoriz} for $p=1$ is the
\begin{pps}\label{pps:d1}
The horizontal differential $\partial: (\mbb{K}^{0,0})^H \rTo (\mbb{K}^{1,0})^H = \oplus_{(\mu, A) \in B(k, 1)} \FS_X $  
is given by $$( \partial (x_\lambda)_\lambda)_{\mu, A} = x_{a_1 \mu} - x_{a_0 \mu} \;.$$
\end{pps}
\begin{proof}
The horizontal differential for $p=0, q=0$: $$\partial :  (\mbb{K}^{0,0})^H \simeq \bigoplus_{\lambda \in c_n(k)} \FS_{X^n} \rTo \bigoplus_{(\mu, A) \in B(k, 1)} \FS_{X^n} \simeq  (\mbb{K}^{1,0})^H$$is defined by (\ref{eq: horizdiff2}). In this case
the maps $i^0_{h, J_i}$ coincide with the identity map. We have that $[\partial(x_\lambda)_\lambda]_{\mu, A} = 0$ unless $\lambda_{A} \neq 0$ and $\lambda_{\bar{A}} = \mu_{\bar{A}}$. The only possibility is, again, that, in multiplicative notation, if $A = \{ a_0, a_1 \}$, with $a_0 < a_1$, $\lambda = a_0 \mu$ or $\lambda = a_1 \mu$. Taking into account the sign, the $(\mu, A)$-component of the differential $\partial$ is given by $[\partial(x_\lambda)_\lambda]_{\mu, A}$ is given by $- \epsilon_{a_0, A} x_{a_0 \mu} - \epsilon_{a_1 , A} x_{a_1 \mu} = x_{a_1 \mu} - x_{a_0 \mu}$. \end{proof}
\begin{crl}The induced map $d^1_K:  (K^{0,0}_1)^H = H^0_\delta ( \mbb{K}^{0, \bullet})^H \rTo H^0_\delta(\mbb{K}^{1, 0})^H= K^{1,0}_1$ is given by 
\begin{equation}\label{eq: diffspec}[d^1_K(x_\lambda)_\lambda ]_{\mu, A} =  x_{a_1 \mu} \trest_{\Delta_A} - x_{a_0 \mu}  \trest_{\Delta_A} \end{equation}\end{crl}
\begin{remark}\label{rmk: d0}The map (\ref{eq: diffspec}) is the local expression on an open set of the form $V_{U^n}$ for the global map: 
\begin{equation}\label{eq: diffspec2}  S^k \mc{C}^0_L  \simeq \bigoplus_{\lambda \in c_n(k)} L^\lambda \rTo \bigoplus_{(\mu, A) \in B(k,1)} L_A \tens L^\mu \simeq \mc{C}^1_L \tens S^{k-1} \mc{C}^0_L\end{equation}
given by the the differential $d^0$ of the complex $S^k \comp{\mc{C}}_L = (\comp{\mc{C}}_L \tens \cdots \tens \comp{\mc{C}}_L)^H$. Moreover the expression in (\ref{eq: diffspec}) holds globally. 
Consequently the differential
$d^1 :  (j^* E^{0,0}_1)^H \rTo ( j^*E^{1,0}_1)^H$ of the spectral sequence $(j^* E^{p,q}_1)^H$ is given globally 
by the restriction of the map (\ref{eq: diffspec2}) to $X^n_{**}$. 
\end{remark}

\subsection{Higher order restrictions}

The aim of this subsection is to build explicitely on the whole variety $X^n$ recursively defined $G$-equivariant operators
$$ D^l_L : \ker D^{l-1}_L \rTo K^l(L)$$where the sheaves $K^l(L)$ coincide with the sheaves $(j_*j^* E^{l+1, -l}_1)^H$, 
and where $D^0_L$ coincides with the map~(\ref{eq: diffspec2}). In the next subsection we will prove  that over the open set $X^n_{**}$ the  operators $D^l_L$ coincide up to signs  with the higher differentials $d^{l+1
}$ of the spectral sequence $( j^* E^{p,q}_1)^H $. The operators $D^l_L$ and their $G$-invariants $(D^l_L)^G$ will be extremely useful to deduce global results about 
the sheaves $p_* q^* (S^k L^{[n]})$ and $\mu_* S^k L^{[n]}$ over $X^n$ and $S^nX$, respectively.

For brevity's sake, denote with $V_L $ the vector bundle on $X^n$ defined as $L \boxplus \cdots \boxplus L := \oplus _{i=1}^n L_i$. 
The bundle $S^k V_L$ on $X^n$ splits as a direct sum $S^k V_L \simeq 
\bigoplus_{\lambda \in c_n (k) } L^ \lambda $. 
\subsubsection{Definition of higher order restrictions $D^l_L$ when $L$ is trivial}
Suppose for the moment that $L$ is trivial. 
Then $S^k V_{\FS_X}$ can be written as the direct sum $S^k V_{\FS_X} \simeq \oplus_{\lambda \in c_n(k)} \FS_{X}^\lambda$; since $\FS_X^\lambda \simeq \FS_{X^n}$, to avoid confusion, we will write $\FS_{X^n}^\lambda$ instead of $\FS_X^\lambda$, using $\lambda$ just as a mere label. A section of $S^k V_{\FS_X}$ is then given by a collection of sections $(x_\lambda)_\lambda$ of $\FS_{X^n}^\lambda$ for all $\lambda \in c_n( k)$. Inspired by finite difference calculus, and using the multiplicative notation for compositions (see notation \ref{notat: comp}), we give the following
\begin{definition}[Higher  differences]
Let $l \in \mbb{N}$, $0 \leq l \leq k$, let $\mu \in c_n(k-l)$ and let $A$ be a subset  
of $\{1, \dots, n\}$ of order $|A| =2$. Define the morphism of $\FS_{X^n}$-modules $$ \Delta^l_{\mu, A}: S^k V_{\FS_X}\rTo \FS_{X^n} $$as 
$$ \Delta^l_{\mu, A} (x_{\lambda})_{\lambda} := \sum_{\substack{\beta \in c_n(l) \\ \beta_{\bar{A}}= 0 }}
(-1)^\beta \binom{l}{\beta} x_{\beta \mu} $$where $A = \{a_0, a_1 \}$, with $a_0 < a_1$, where $(-1)^\beta := (-1)^{\beta(a_0)}$, and where $\displaystyle \binom{l}{\beta}:= \binom{l}{\beta(a_0)}= \binom{l}{\beta(a_1)}$.  
Note that in the limit case $l=0$, $\Delta^0_{\mu, A}(x_\lambda)_\lambda = x_\mu$. 
\end{definition}

\begin{lemma}\label{lmm:diff}Let $l \in \mbb{N}$, $1  \leq l \leq  k$. 
We have the identity: 
$$ \Delta^l_{\mu, A} = -\Delta^{l-1}_{a_0 \mu, A} + \Delta^{l-1}_{a_1 \mu, A} \;.$$
\end{lemma}
\begin{proof}Take local sections $x_\lambda$ of $\FS_{X^n}^\lambda$. 
The identity to prove is the following:
\begin{equation}\label{eq: binomial} \sum_{\substack{\gamma \in c_n(l) \\ \gamma_{\bar{A}}=0}}
(-1)^{\gamma}
  \binom{l}{\gamma} x_{\gamma \mu }  = -\sum_{\substack{\beta \in c_n(l-1) \\ \: \beta_{\bar{A}}=0}} 
(-1)^{\beta}
  \binom{l-1}{\beta} x_{\beta a_0 \mu }  + \sum_{\substack{\beta \in c_n(l-1)  \\ \beta_{\bar{A}}=0}} 
(-1)^{\beta}
  \binom{l-1}{\beta} x_{\beta a_1 \mu }\end{equation}and can be straightforwardly proved, writing multiplicatively a general composition $\beta$ in (\ref{eq: binomial}) as $a_0^{l-j-1}a_1^j$, for $0 \leq j \leq l-1$. 
\end{proof}
\begin{notat}
Denote now with $I_A$ the ideal sheaf of the partial diagonal $\Delta_A$ in $\FS_{X^n}$. Consider the quotient sheaf $\FS_{X^n} / I_A^{l+1}$. The class in $\FS_{X^n} / I_A^{l+1}$ of a local section $f \in \FS_{X^n}$ is denoted with $[f]_{I_A^{l+1}}$. 
\end{notat}

\begin{definition}[Higher order restrictions]\label{def: Hor}Let $l$ be an integer, $0 \leq l \leq k-1$. Let $(\mu, A) \in B(k, l+1)$, that is, $\mu \in c_n(k-l-1)$, $A \subseteq \{1, \dots, n\}$, $|A|=2$. 
Define the operator: 
$$ D^l_{\mu, A}: S^k V_{\FS_X} \rTo  \FS_{X^n}/I_A^{l+1} 
$$
as $$ D^l_{\mu, A} ( (x_\lambda)_\lambda) : = [\Delta^{l+1}_{\mu, A}((x_{\lambda})_\lambda)]_{I_A^{l+1}} \;.$$ 
\end{definition}
As a consequence of the definition of the operators $D^l_{\mu, A}$ and of lemma \ref{lmm:diff} we get immediately:
\begin{lemma}\label{lmm:diffnabla}
Let $A = \{ a_0, a_1 \}$, with $a_0 < a_1$ and $l \in \mbb{N}^*$. The composition
$$ S^k V_{\FS_X} \rTo^{D^l_{\mu, A}}   \FS_{X^n}/I_A^{l+1}  \rOnto^{p_{l-1}}  \FS_{X^n}/I_A^{l}  $$is given by
$$ p_{l-1} \circ D^{l}_{\mu, A} = -D^{l-1}_{a_0 \mu, A}
+ D_{a_1 \mu, A}^{l-1} \;.$$
\end{lemma}

\begin{notat}
Let $M$ be a line bundle on $X$,  $l \in \mbb{N}$, $l \leq k-1$ let $(\mu, A) \in B(k, l+1)$. 
Denote with $P^l_{\mu, A}(M)$ and with $K^l_{\mu, A}(M)$, respectively, the following sheaves, supported on the partial diagonal $\Delta_A$: \begin{gather*}
P^l_{\mu, A}(M) :=   \FS_{X^n}/I_A^{l+1}  \tens M^{l+1}_A \tens M^{\mu}  \\
K^l_{\mu, A}(M):=  I_A^l/I_A^{l+1} \tens M^{l+1}_A \tens M^{\mu} \simeq (S^l \Omega^1_X \tens M^{l+1})_A  
\tens M^{\mu}
\;.
\end{gather*}Finally, denote: 
\begin{gather*} P^l(M) := \bigoplus_{(\mu, A) \in B(k, l+1)} P^l_{\mu, A}(M) \\
K^l(M) := \bigoplus_{(\mu, A) \in B(k, l+1)} K^l_{\mu, A}(M) \:.
\end{gather*}If $M$ is trivial, we will denote $P^l(M)$ simply with $P^l$, $P^l_{\mu, A}(M)$ simply with $P^l_{\mu, A}$, $K^l(M)$ with $K^l$ and $K^l_{\mu, A}(M)$ with 
$K^l_{\mu, A}$; in this case $P^l_{\mu, A} \simeq  \FS_{X^n}/I_A^{l+1}$ and $K^l_{\mu, A} \simeq I_A^l/I_A^{l+1} \simeq (S^l \Omega^1_X)_A$. As we see, if $M$ is trivial, $\mu$ plays  no role for the moment and should be considered as a mere label. 
\end{notat}
\begin{remark}Let $(\mu, A) \in B(k, l+1)$. 
Note that, since $P^l_{\mu, A}(M) \simeq   \FS_{X^n}/I_A^{l+1} \tens M_A^{ l+ 1+ |\mu_A| } \tens M^{ \mu_{\bar{A}} }$, 
for any composition $\mu^\prime \in c_n(k-l)$, such that $\mu_{\bar{A}} \simeq \mu^\prime_{\bar{A}}$, there is a canonical 
epimorphism $P^l_{\mu, A}(M) \rOnto P^{l-1}_{\mu^\prime, A}(M)$, whose kernel is isomorphic to $K^l_{\mu, A}(M)$. 
\end{remark}
\begin{definition}\label{def: D}Let $l \in \mbb{N}$, $l \leq k-1$. Let $(\mu, A) \in B(k, l+1)$. Define the 
operator $D^l: S^k V_{\FS_X} \rTo P^l$ as: 
$$ D^l ((x_\lambda)_\lambda)_{\mu, A} : = D^l _{\mu, A}(x_{\lambda})_\lambda \;.$$
\end{definition}
\begin{remark}
It follows immediately from lemma \ref{lmm:diffnabla}
that the operator $D^l_{\mu, A}$, restricted to $\ker D^{l-1}$, 
takes values in $K^l_{\mu, A}$. 
Hence the  operator 
$D^l$, restricted to $\ker D^{l-1}$, takes values in $K^l$: 
$$ D^l : \ker D^{l-1} \rTo K^l \;.$$
\end{remark}

\subsubsection{Definition of higher order restrictions $D^l_L$ for general $L$}

Let's consider now the case in which $L$ is non trivial. In this case we will not be able to build operators $D^l_L : S^k V_L \rTo P^l(L)$, but we do can recursively build  operators 
$D^l_L : \ker D^{l-1}_L \rTo K^l (L)$. 
 Indeed, let $l$ be an integer, $0 \leq l \leq k-1$, $\mu \in c_n(k-l-1)$, $A \subseteq \{1, \dots, n \}$, with $|A|=2$. 
Consider an affine open subset of the form $U^n$, with $L$ trivial over $U$. 
Let $s \in \Gamma(U, L)$ be a nowhere zero section of $L$ over the open set $U$, trivializing $L$. Consider now sections $s_j$, $j=1, \dots, n$, of $L_j$ over $U^n$ defined as 
$s_j = p_j^*(s)$. Then $s^\lambda := \tens_{j=1}^{n} s_j^{\lambda_j}$ trivializes $L^\lambda$ over $U^n$.  
We can now write an element of $S^kV_L $ on $U^n$ as $(x_\lambda s^\lambda)_\lambda$. 
Define $$ D^l_{\mu, A}: S^k V_L \rTo P^l_{\mu, A}(L)$$over $U^n$ as 
\begin{equation}\label{eq: defdiff}D^l_{\mu, A} ((x_\lambda s^\lambda)_\lambda) = (D^l_{\mu, A}(x_\lambda)_\lambda) 
s^{l+1}\tens s^{\mu} \;.\end{equation}Here $(D^l_{\mu, A}(x_\lambda)_\lambda)$ is the image of the section  $(x_\lambda)_\lambda$ of $S^k V_{\FS_X}$ for the operator $D^l_{\mu, A}: S^k V_{\FS_X} \rTo P^l_{\mu, A}$,  introduced in definition \ref{def: D}. Hence $(D^l_{\mu, A}(x_\lambda)_\lambda)$ is 
a local section of $\FS_{X^n}/I_A^{l+1}$
and $s^{l+1} \otimes s^{\mu}$ is a local section of $L_A^{l+1} \tens L^{\mu}$. 
Hence the product in (\ref{eq: defdiff}) defines correctly a local section of $P^l_{\mu, A}(L)$. 
Over $U^n$ we define furthermore operators 
$$D^l_L : S^kV_L \rTo P^l(L)$$ as $$D^l_L := \bigoplus_{(\mu, A) \in B(k, l+1)} D^l_{\mu, A} \;,$$and where 
each $D^l_{\mu, A} : S^kV_L \rTo P^l_{\mu, A}(L)$ is defined as in (\ref{eq: defdiff})

\begin{remark}\label{rmk: D0}Note that, if $(u^\lambda)_\lambda$ is a local section of $S^k V_L$ over an open set $U^{n}$  and if $A = \{ a_0, a_1 \}$, with $a_0 < a_1$, then 
$D^0_{\mu, A}(u^\lambda)_\lambda = u_{a_1 \mu} \trest_{\Delta_{A}} - u_{a_0 \mu} \trest_{\Delta_A} \in L_A \tens L^\mu$. Hence $D^0_{\mu, A}$ and $D^0_L$ are globally defined operators on $S^k V_L$. Moreover the operator $D^0_L$ coincides with the map (\ref{eq: diffspec2}). 
\end{remark}

\begin{pps}The operators $D^l_{\mu, A}$ and $D^l_L$
are recursively well defined globally on $X^n$ as operators 
\begin{gather*} D^l_{\mu, A} : \ker  D^{l-1}_L \rTo K^l_{\mu, A}(L) \;, \\ D^l_L  : \ker D^{l-1}_L \rTo K^l (L)\;. \end{gather*}
 \end{pps}
\begin{proof}
We prove the goodness of the definition 
by induction on $l$. For $l=0$ this fact follows from remark \ref{rmk: D0}. Suppose now that, for any composition $\nu \in c_n(k-l-1)$ and any $B \subseteq \{1, \dots, n \}$, $|B| =2$, the operators $D^l_{\nu, B}: \ker D^{l-1}_L \rTo K^l_{\nu, B}(L)$
 are globally well defined. Take now $a = (x_\lambda s^\lambda)_\lambda$ a local section of $S^k V_L$ over an open set $U^n$, where $U$ is an affine open set of $X$ such that $L$  trivializes on $U$ via the nowhere zero section $s$. Suppose that $a \in \ker D^{l}_L$: this means that for any 
composition $\nu$ and for any subset $B$ as above, we have $D^l_{\nu, B} (x_\lambda s^\lambda)_\lambda =0$, that is, 
$D^l_{\nu, B}(x_\lambda)_\lambda = 0$ in $(S^l \Omega^1_X)_B \simeq I_B^l/I_B^{l+1} \subseteq \FS_{X^n}/I_B^{l+1}$. By definition of $D^l_{\nu, B}$, this means that the functions $$\Delta^{l+1}_{\nu, B}(x_\lambda)_\lambda \in I_B^{l+1} \quad \mbox{for all $(\nu, B) \in B(k, l+1)$} \;.$$
Let now $(\mu, A) \in B(k, l+2)$. To prove that $D^{l+1}_{\mu, A}$ is well defined, write $a$ in term  of a nowhere zero local section $t$ of $L$, over another affine open set $V$ over which $L$ trivializes, and such that $U \cap V \neq \emptyset$. On $U \cap V$ we have $s = \gamma t$, for $\gamma \in \FS_{U \cap V}^*$; hence $a = (x_\lambda s^\lambda)_\lambda = (x_\lambda \gamma^\lambda t^\lambda)_\lambda$. We have to prove that 
$$D^{l+1}_{\mu, A} (x_\lambda s^\lambda)_\lambda = D^{l+1}_{\mu, A} (x_\lambda \gamma^\lambda t^\lambda)_\lambda \;.$$Now \begin{align*}
D^{l+1}_{\mu, A} (x_\lambda s^\lambda)_\lambda - D^{l+1}_{\mu, A} (x_\lambda \gamma^\lambda t^\lambda)_\lambda = & 
D^{l+1}_{\mu, A}(x_\lambda)_\lambda s^{l+2} \tens s^\mu - D^{l+1}_{\mu, A}(x_\lambda \gamma^\lambda)_\lambda t^{l+2} \tens t^\mu \\
= & D^{l+1}_{\mu, A}(x_\lambda)_\lambda( \gamma^{l+2} t^{l+2} \tens \gamma^\mu t^\mu) -  D^{l+1}_{\mu, A}(x_\lambda \gamma^\lambda)_\lambda t^{l+2} \tens t^\mu \\
= & D^{l+1}_{\mu, A}(x_\lambda)_\lambda \gamma^{l+2}_{a_0} \gamma^\mu (  t^{l+2} \tens t^\mu) -  D^{l+1}_{\mu, A}(x_\lambda \gamma^\lambda)_\lambda( t^{l+2} \tens t^\mu) \\
= & \left[  \Delta^{l+2}_{\mu, A}(x_\lambda)_\lambda \gamma^{l+2}_{a_0} \gamma^\mu  -  \Delta^{l+2}_{\mu, A}(x_\lambda \gamma^\lambda)_\lambda \right]_{I_A^{l+2}} (  t^{l+2} \tens t^\mu) \;.
\end{align*}Hence it is sufficient to prove that the function on $(U \cap V)^n$ 
$$\Delta^{l+2}_{\mu, A}(x_\lambda)_\lambda \gamma^{l+2}_{a_0} \gamma^\mu  -  \Delta^{l+2}_{\mu, A}(x_\lambda \gamma^\lambda)_\lambda  = \gamma^\mu \sum_{|\beta|=l+2, \beta_{\bar{A}=0}}(-1)^\beta
\binom{l+2}{\beta} x_{\beta \mu} (\gamma_{a_0}^{l+2}-\gamma^{\beta} )$$is in the ideal $I_A^{l+2}$. 
Now, writing $\beta=a_0^{l+2-j}a_1^j$, for $0 \leq j \leq l+2$. 
we have: 
\begin{align*}
\sum_{|\beta|=l+2, \beta_{\bar{A}}=0}(-1)^\beta
\binom{l}{\beta} x_{\beta \mu} (\gamma_{a_0}^{l+2}-\gamma^{\beta} )
= &(-1)^l \sum_{j=0}^{l+2} (-1)^j
\binom{l+2}{j}x_{a_0^{l+2-j}a_1^j\mu} (\gamma_{a_0}^{l+2}-\gamma_{a_0}^{l+2-j}
\gamma_{a_1}^{j} ) \\
=& (-1)^l\sum_{j=1}^{l+2} (-1)^j
\binom{l+2}{j}x_{a_0^{l+2-j}a_1^j\mu} (\gamma_{a_0}^{l+2}-\gamma_{a_0}^{l+2-j}
\gamma_{a_1}^{j} ) \\
= & (-1)^l \sum_{j=1}^{l+2} (-1)^j
\binom{l+2}{j}x_{a_0^{l+2-j}a_1^j \mu} \gamma_{a_0}^{l+2
-j}(\gamma_{a_0}^j-
\gamma_{a_1}^j) \\
\end{align*}where, in the second equality we used that,  for $j=0$, $\gamma_{a_0}^{l+2}-\gamma_{a_0}^{l+2-j}
\gamma_{a_1}^{j} =0$. 
Now $\displaystyle
(\gamma_{a_0}^j-
\gamma_{a_1}^j) = -\sum_{i=1}^j \binom{j}{i} 
\gamma_{a_0}^{j-i}(\gamma_{a_1}-\gamma_{a_0})^i$, hence: \small
\begin{align*}
\sum_{\substack{|\beta|=l+2 \\ \beta_{\bar{A}}=0}}(-1)^\beta
\binom{l}{\beta} x_{\beta \mu} (\gamma_{a_0}^{l+2}-\gamma^{\beta} )
= &- (-1)^l\sum_{j=1}^{l+2} \sum_{i=1}^j (-1)^j
\binom{l+2}{j}\binom{j}{i}  x_{a_0^{l+2-j}a_1^j \mu}  
\gamma_{a_0}^{l+2-i}(\gamma_{a_1}-\gamma_{a_0})^i \\ 
= &
- (-1)^l\sum_{i=1}^{l+2} \sum_{j=i}^{l+2} (-1)^{j}
\binom{l+2}{j}\binom{j}{i}  x_{a_0^{l+2-j}a_1^j \mu}
\gamma_{a_0}^{l+2-i}(\gamma_{a_1}-\gamma_{a_0})^i \\
=& - (-1)^l\sum_{i=1}^{l+2} \sum_{r=0}^{l+2-i} (-1)^{r+i}
\binom{l+2}{r+i}\binom{r+i}{i}  x_{a_0^{l+2-r-i}a_1^{r+i} \mu}
\gamma_{a_0}^{l+2-i}(\gamma_{a_1}-\gamma_{a_0})^i \\
= & (-1)^l \sum_{i=1}^{l+2} (-1)^{i+1}\gamma_{a_0}^{l+2-i}(\gamma_{a_1}-\gamma_{a_0})^i \left(\sum_{r=0}^{l+2-i} (-1)^{r}
\binom{l+2}{r+i}\binom{r+i}{i}  x_{a_0^{l+2-i-r}a_1^{r+i}\mu} \right)  \\
= &  \sum_{i=1}^{l+2} (-1)^{i+1}\binom{l+2}{i} \gamma_{a_0}^{l+2-i}(\gamma_{a_1}-\gamma_{a_0})^i \left(\sum_{r=0}^{l+2-i} (-1)^{l+2-r}
\binom{l+2-i}{r}  x_{a_0^{l+2-i-r}a_1^{r} a_1^i \mu}
\right)  \\
=& \sum_{i=1}^{l+2} (-1)^{i+1}\binom{l+2}{i}
\gamma_{a_0}^{l+2-i}(\gamma_{a_1}-\gamma_{a_0})^i
\left(\sum_{\substack{|\beta|=l+2-i \\ \beta_{\bar{A}=0}}} (-1)^{\beta}
\binom{l+2-i}{\beta}  x_{\beta a_1^i\mu}
\right)  
\end{align*}\normalsize since $\displaystyle \binom{l+2}{r+i}\binom{r+i}{i} =
\binom{l+2}{i}\binom{l+2-i}{r}$. As a consequence
\begin{equation}\label{eq: difference} \Delta^{l+2}_{\mu, A}(x_\lambda)_\lambda \gamma^{l+2}_{a_0} \gamma^\mu  -  \Delta^{l+2}_{\mu, A}(x_\lambda \gamma^\lambda)_\lambda  = \gamma^\mu \sum_{i=1}^{l+2} (-1)^{i+1}\binom{l+2}{i}
\gamma_{a_0}^{l+2-i}(\gamma_{a_1}-\gamma_{a_0})^i \Delta^{l+2-i}_{a_1^i \mu, A} (x_\lambda)_\lambda \;.\end{equation}Note that, since $(x_{\lambda})_{\lambda} \in \ker D^l_L$,  one
has that, for $1 \leq i \leq l+1$, $D^{l+1-i}_{a_1^i\mu, A} (x_{\lambda})_{\lambda}
=0$ in $\FS_{X^n }/I_A^{l+2-i}$, and hence, for all $i = 1, \dots, l+2$ the function $\Delta^{l+2-i}_{a_1^i\mu,
  A} (x_{\lambda})_{\lambda} \in I_A^{l+2-i}$. Since 
$(\gamma_{a_1}-\gamma_{a_0})^i \in I_A^i$, we conclude that the difference in (\ref{eq: difference}) is in $I_A^{l+2}$, which is what we wanted to prove. 
\end{proof}

\subsection{Higher differentials}\label{subsection: higher}
Let's go back to the higher differential of the spectral sequence $(j^* E^{p,q}_l)^H$ on $X^n_{**}$. Note that 
with the definitions of last sections, for $l$ integer $0 \leq l \leq k-2$, we have $(E^{0,0}_1)^H \simeq S^k V_L$ and $$K^l(L) \simeq (E^{l+1, -l}_1)_0^H \;,$$where the latter has been defined in subsection \ref{subsection: reductionopen}, and hence 
$$ j^*K^l(L) \simeq (j^* E^{l+1, -l}_1)^H \;.$$\sloppy Hence it is natural to compare the restrictions $j^* D^l_L$ of the operators $D^l_L$, found in the last subsection, with 
the higher differentials $d^{l+1}: ( j^*E^{0,0}_l) ^H \rTo(  j^* E^{l+1,-l}_l) ^H$.  We will compare the morphisms by induction on $l$. 
Since the morphisms $j^* D^l_L$ and $d^{l+1}$ are globally well defined on $X^n_{**}$  
it will be sufficient to compare them locally on an open set of the form $V_{U^n,**}$. Since the restriction of $(j^* E^{p,q}_1)^H$ to such an open set $V_{U^n,**}$ coincides with the restriction to $V_{U^n,**}$ of the spectral sequence $(K^{p,q}_1)^H$, associated to the $H$-invariant bicomplex $(\mbb{K}^{\bullet, \bullet})^H$ by proposition \ref{pps: spectlocal}, it will be sufficient to compare
the operators $D^l_L$ and the higher differentials $d_K^{l+1}$ of the spectral sequence $(K^{p,q}_1)^H$ over the open sets $V_{U^n}$. 

\begin{remark}\label{rmk: openset}To be precise, consider an affine open set of the form 
$V_{U^n}$ of $X^n$, as defined in remark \ref{rmk: auxiliaryopen}. We recall that $L$ is trivial on every open set $U$ of $X$ and that 
that all pairwise diagonals $\Delta_{A}$ with $A \subseteq \{1, \dots,
n\}$, $|A| \geq 2$,  are complete intersection inside $V_{U^n}$, or empty.  
For $A \subseteq \{1, \dots, n \}$, $|A|=2$,  such that $\Delta_A \cap V_{U^n} \neq \emptyset$, 
let $f_{A, i}$, $i=1,2$ be the
generators of $I_A$. Let $F_A$ be the trivial vector bundle $\oplus_{i=1}^2 \FS_{V_{U^n}} e_{A,i}$ of rank $2$, 
where $e_{A, i}$, $i=1,2$ is the standard basis of $\mbb{C}^2$, seen as a local frame for $F_A$. Let $e_{A, i}^*$ the dual basis. 
The zeros of the section $s_A : \FS_{X^n} \trest_{V_{U^n}} \rTo F_A$, 
given by $s_A = \sum_{i=1}^2f_{A,i} e_{A,i}$, define scheme-theoretically the pairwise diagonal
$\Delta_A$. The Koszul complex $\comp{K}_A(F_A, s_A)$ resolves the structural sheaf $\FS_A$ of $\Delta_A$. 
We identify $S^l N^*_A$ with $I^l_A / I^{l+1}_A$. Remember that, after lemma \ref{lemma: verticaldiff}, we have: 
$$ (\mbb{K}^{l, \bullet})^H \simeq \bigoplus_{(\mu, A) \in B(k,l)} (W^{\bullet}_{\mu, A})^H $$where 
$$ (W^{ \bullet}_{\mu, A})^H \simeq S^{l-1} F^*_A \tens K^\bullet(F_A, s_A) [l-1] \tens L^\mu \:, $$and its $1-l$-cohomology is isomorphic to 
$$H^{1-l} ( (W^{\bullet}_{\mu, A})^H )  \simeq S^{l-1} N^*_A \tens  L^{\mu} \simeq I_A^{l-1}/I_A^{l} \tens  L^{\mu} \subseteq (K^{l,1-l}_1)^H \simeq (K^{l,1-l}_l) ^H \:.$$\end{remark}
\begin{lemma}\label{lmm: difference}
Let $(x_\lambda)_{\lambda} \in (K^{0,0}_l)^H$, and suppose that, for $(\nu, B) \in B(k,l)$, we have $$d_K^l[(x_{\lambda})_\lambda]_{\nu,
B}=[h_{\nu, B}] \in I_B^{l-1}/I_B^{l} \tens L^\nu \subseteq  (K^{l,1-l}_l) ^H \;,$$where $h_{\nu, B}$
are elements of $I_{B}^{l-1}$. Suppose moreover that $d^l_K [(x_{\lambda})_\lambda]_{\nu,
B}=0$ for all $(\nu, B) \in B(k, l)$. Let $(\mu, A) \in B(k, l+1)$, $A = \{a_0, a_1 \}$, $a_0 < a_1$. 
 Then $(x_{\lambda})_{\lambda} \in
E^{0,0}_{l+1}$ and 
$$d^{l+1}_K[(x_{\lambda})_{\lambda}]_{\mu, A}=(-1)^{l-1} l[-h_{a_0\mu, A }+h_{a_1\mu, A}]$$
\end{lemma}
\begin{proof}\sloppy For brevity's sake, and since it is trivial, let's neglect the term $L^{\lambda}$. An element $[a]$ in the
$1-l$ -cohomology of $(W^{\bullet}_{\nu, B})^H$, $[a]=
[\sum_{|\alpha|=l-1,
  l(\alpha)=2} a_{\alpha} f_B^\alpha] \in I^{l-1}_B/ I^{l}_B$ is
represented by an element $$\frac{1}{(l-1)!}\Big[\sum_{|\alpha|=l-1,
  l(\alpha)=2} a_{\alpha} (e^*_B) ^\alpha \Big]$$in  $S^{l-1}F_B^*/I_B S^{l-1} F^*_B \simeq I_B^{l-1} /I_B^{l}$. The
element is zero if and only if it comes, in the complex $(W^{\bullet}_{\nu, B})^H $, from some element in $
(W^{-l}_{\nu, B})^H \simeq
S^{l-1}F_B^*
\tens F_B^*$, that is if 
$$ \sum_{|\alpha|=l-1,
  l(\alpha)=2} a_{\alpha} (e^*_B) ^\alpha = \delta \left(\sum_{i=1}^2\sum_{|\alpha|=l-1,
  l(\alpha)=2} b_{\alpha, i} (e^*_B)^\alpha \tens e^*_{B, i} \right)= \sum_{i=1}^2\sum_{|\alpha|=l-1,
  l(\alpha)=2} b_{\alpha, i} f_{B, i} (e^*_B)^\alpha $$that is, if $a_{\alpha}=\sum_{i=1}^2
b_{\alpha, i} f_{B,i}$, or, equivalently, if $a_{\alpha} \in I_B$,
which means that the original function $a$ is in $I_B^{l}$. Take now $(\mu, A) \in B(k, l+1)$. 
The only two elements $(\nu, B) \in B(k, l)$ such that the horizontal differential 
$\partial^{\mu, A}_{\nu, B} : (W^{-l}_{\nu, B})^H \rTo (W^{-l}_{\mu, A})^H$ is nonzero are $(a_0 \mu, A)$ and $(a_1 \mu, A)$. For the first we have, by lemma \ref{lmm: possiblehoriz}, taking into account that $l \sym( (e^*_A)^\alpha \tens e^*_{A, i} )
=  b_{\alpha, i} (e^*_A)^\alpha e^*_{A, i}$, 
$$\partial^{\mu, A}_{a_0 \mu, A}  \frac{1}{(l-1)!}\Big(\sum_{i=1}^2\sum_{|\alpha|=l-1,
  l(\alpha)=2} b_{\alpha, i} (e^*_A)^\alpha \tens e^*_{A, i} \Big ) =   \frac{(-1)^{l-1}}{(l-1)!} \epsilon_{a_0, A}  \sum_{i=1}^2\sum_{|\alpha|=l-1,
  l(\alpha)=2} b_{\alpha, i} (e^*_A)^\alpha e^*_{A, i}  \in S^{l} F_A^*$$whose class, in the $-l$-cohomology of 
  $(W^{\bullet}_{\mu, A})^H$, identified with $I_A^l / I_A^{l+1}$, is
  \begin{align*} \Big[    (-1)^{l-1} l    \epsilon_{a_0, A}   \sum_{i=1}^2\sum_{|\alpha|=l,
  l(\alpha)=2} b_{\alpha, i} (f_A)^\alpha f_{A, i} \Big] = & \:  \Big[ 
  (-1)^{l-1} l   \epsilon_{a_0, A}   \sum_{|\alpha|=l,
  l(\alpha)=2} a_\alpha f^\alpha_{A} 
  \Big] \\ = & \: (-1)^{l-1} l \epsilon_{a_0, A} [ a] \in I^{l+1}_A/ I^{l+2}_A \;. \end{align*}
  We have an analogous expression for $a_1 \mu$. We conclude by lemma \ref{lemma: highdiffspec} and corollary \ref{crla: highdiffspec}. 
 \end{proof}

\begin{definition}The sheaves $ E^l(n, k) := S^k V_L \cap \ker D^0_L \cap \dots \cap \ker D^{l-1}_L $ define a finite decreasing filtration
 $$ E^{k-1}(n,k) \subseteq E^{k-2}(n,k) \subseteq \dots \subseteq E^{1}(n,k) \subseteq E^0(n,k) = S^kV_L $$of $S^k V_L$. 
\end{definition}

\begin{pps}Let  $0 \leq l \leq k-1$. Over an open set $V_{U^n}$, the two decreasing filtrations
 $E^{\bullet}(n,k)$ and $(K^{0,0}_{\bullet+1})^H$ of $S^k V_L$ coincide. Moreover,  
the differential 
$d^{l+1}_K: (K^{0,0}_{l+1})^H \rTo (K^{l+1, -l}_{l+1})^H$ coincides
 with the the operator $(-1)^{\frac{l(l-1)}{2}}l! D^l_L : E^l(n,k) \rTo K^l(L)$. 
\end{pps}
\begin{proof}By induction on $l$. For $l=0$, it follows directly from the fact that $S^k V_L = E^0(n,k) = (K^{0,0}_1)^H$ 
and from remarks \ref{rmk: D0} and \ref{rmk: d0}.  Suppose by induction that for a certain $l \geq 1$ we have, over $V_{U^n}$, 
that $E^{l-1}(n,k) = (K^{0,0}_{l}) ^H$ and that  $$d^l_K =(-1)^{\frac{(l-1)(l-2)}{2}} (l-1)! D^{l-1}_L \;,$$as operators from 
$E^{l-1}(n,k) = (K^{0,0}_{l}) ^H \rTo (K^{l, 1-l}_l)^H$. Hence $(K^{0,0}_{l+1})^H = E^l(n,k)$. 
To prove the formula for $l+1$ over $V_{U^n}$, 
we dispose of  the explicit definition of operators $D^l_L$ over such open sets and  of lemma \ref{lmm: difference}. Take $(x_\lambda)_\lambda \in E^l(n,k)$. Then, by definition, 
$D^{l-1}_L (x_\lambda)_\lambda = 0$. Hence, for any $(\nu, B) \in B(k,l)$, we have that $D^{l-1}_{\nu, B}(x_\lambda)_\lambda = [\Delta^l_{\nu, B}(x_\lambda)_\lambda] =0 \in S^{l-1} N_B^* \simeq I_B^{l-1} /I_B^{l}$. Hence
$$[ d^l_K (x_\lambda)_\lambda]_{\nu, B} = (-1)^{\frac{(l-1)(l-2)}{2}}(l-1)! [\Delta^l_{\nu, B}(x_\lambda)_\lambda] =0$$ in 
$ S^{l-1} N_B^* \simeq I_B^{l-1} /I_B^{l}$. By lemma \ref{lmm: difference} we have that, for $(\mu, A) \in B(k, l+1)$ 
\begin{align*} [ d^{l+1}_K (x_\lambda)_\lambda ]_{(\mu, A)} = & (-1)^{l-1+ \frac{(l-2)(l-1)}{2}} l (l-1)!  [- \Delta^{l}_{a_0 \mu, A} (x_\lambda)_\lambda + \Delta^l_{a_1 \mu, A} (x_\lambda)_\lambda ]_{I_A^{l+1}}  \\ = & (-1)^{\frac{l(l-1)}{2}} l!   [\Delta^{l+1}_{\mu, A}]_{I_A^{l+1}} =  (-1)^{\frac{l(l-1)}{2}} l!  D^l_{\mu, A}(x_\lambda)_\lambda \\ = &(-1)^{\frac{l(l-1)}{2}} l! \left( D^l_L(x_\lambda)_\lambda \right)_{(\mu, A)} \;.\end{align*}
\end{proof}As an immediate consequence of what just proved and of proposition \ref{pps: spectlocal} we have: 
\begin{crl}\label{crl: differentialscomparison}Consider the differential $d^l : ( j^*E^{0,0}_l )^H \rTo ( j^*E^{l, 1-l}_l)^H$ of the spectral sequence $( j^* E^{p,q}_1)^H$. 
 We have $( j^* E^{0,0}_l )^H\simeq j^* E^{l-1}(n,k)$
   for all $1 \leq l \leq k$ and $$  d^l = (-1)^{\frac{l(l-1)}{2}} l! j^* D^{l-1}_L \;.$$
\end{crl}

\begin{remark}\label{rmk: DLsurj}The previous corollary implies that, for $n=2$, the operators $D^{l}_L$ coincide, up to constants, with the higher differentials $d^{l+1}$ of the spectral sequence $(E^{p,q}_1)^H$ over the whole $X^2$; consequently, after remark \ref{rmk: dlsurj} they are surjective. 

\end{remark}
We have finally come to the main theorem of this section: the characterization of the image $\bkrh(S^k L^{[n]})$
of symmetric powers of tautological bundles in terms of kernels of operators $D^{k-2}_L$. 
\begin{theorem}\label{thm: mainimage}On the whole variety $X^n$ the Bridgeland-King-Reid transform $\bkrh(S^k L^{[n]})$ 
of the $k$-symmetric power of a tautological vector bundle $L^{[n]}$, associated to a line bundle $L$ over the surface $X$, is quasi-isomorphic to the term $E^{k-1}(n,k)$ 
of the filtration $E^\bullet(n,k)$ on the vector bundle $S^k V_L$: 
$$ \bkrh (S^k L^{[n]}) \simeq^{qis} E^{k-1}(n,k) \simeq \ker D^{k-2}_L \;.$$
\end{theorem}
\begin{proof}Recall that the big open set $X^n_{**}$ is defined as $X^n \setminus W$, where the closed subscheme $W$ has been defined in subsection \ref{subsection: reductionopen} and has codimension $4$ in $X^n$. 
Let $E^{p,q}_1$ be the spectral sequence introduced in subsection \ref{subsection: taut}. By the vanishing $R^i p_* q^* S^k L^{[n]} = 0$ for $i>0$ and after corollary \ref{crl: spect} we have: 
$\bkrh(S^k L^{[n]}) \simeq p_* q^* S^k L^{[n]} \simeq (E^{0,0}_\infty) ^H $ and hence $j^* p_* q^* S^k L^{[n]} \simeq j^* ( E^{0,0}_\infty) ^H \simeq  j^* ( E^{0,0}_k) ^H$ since $(j^* E^{p,q}_1)^H$ degenerates at level $k$, after corollary \ref{crl: EGH}.  We also know by (\ref{eq: j}) that $p_* q^* S^k L^{[n]} \simeq j_* j^* p_* q^* S^k L^{[n]} \simeq j_* j^* ( E^{0,0}_\infty) ^H$ and hence, by the previous corollary \begin{equation}\label{eq: jj}
\bkrh(S^k L^{[n]}) \simeq ^{qis} (j_* j^* E^{0,0}_k)^H \simeq j_* j^* E^{k-1}(n,k) \;.
\end{equation}But the sheaves $E^l(n,k)$, for $0 \leq l \leq k-2$ fit in the diagrams
\begin{diagram}[height=0.7cm]
0 & \rTo & E^{l+1}(n,k) & \rTo & E^l(n,k) & \rTo^{D^l_L} & K^l(L)   \\
 & & \dTo & & \dTo& & \dTo \\
0 & \rTo & j_* j^* E^{l+1}(n,k) & \rTo & j_* j^* E^l(n,k) & \rTo^{j_* j^* D^l_L} & j_* j^* K^l(L)  
\end{diagram}But now the vertical map $K^l(L) \rTo j_* j^* K^l(L)$ is an isomorphism after \cite[ Lemma 3.1.9]{Scala2009D}, since $K^l(L)$ is a direct sum of restrictions of vector bundles to pairwise diagonals $\Delta_A$ of codimension $2$ and every irreducible component of the closed subscheme $W$ has codimension $4$. 
Since $j_* j^* E^0(n,k) \simeq E^0(n,k)$ because $E^0(n,k) \simeq S^k V_L$ is a vector bundle, we conclude by induction on $l$ that $$j_* j^* E^l(n,k) \simeq E^l(n,k)$$and 
that $j_* j^* D^l_L = D^l_L $ for all $0 \leq l \leq k-1$. Together with (\ref{eq: jj}), this gives the result. 
\end{proof}

\begin{remark}In the sequel we will never use the spectral sequence $E^{p,q}_1$, or its invariants $(E^{p,q}_1)^H$ or $(E^{p,q}_1)^{G \times H}$ any more; we will just need the operators $D^l_L$ explained above. Therefore, we will no more use the notation $d^l $ for the higher differential of a spectral sequence; on the other hand we will use this notation for 
another differential, as the following remark explains. 
\end{remark}

 \begin{remark}\label{rmk: affine}\sloppy Suppose  that $X$ is a smooth affine surface, $X = \Spec(R)$, with $R$ a finitely generated $\mbb{C}$-algebra. Consider the sheaves $J^l \FS_X$ of $l$-jets over $X$ (see, for example \cite[16.3.2]{EGAIV.4}). 
 The sequences 
 \begin{equation}\label{eq: splits} 0 \rTo S^l \Omega^1_X \rTo J^l \FS_X \rTo J^{l-1} \FS_X \rTo 0 \end{equation}split because $X$ is affine: let $\tau : J^l \FS_X \rTo S^l \Omega_X^1$ the morphism of $\FS_X$-modules defining the splitting. 
 As a consequence one can define a higher differential 
 $d^l : \FS_X \rTo S^l \Omega^1_X$ as the composition
 $$ d^l : \FS_X \rTo^{j_l} J^l \FS_X \rTo^\tau S^l \Omega^1_X $$where the first morphism $j_l$ is the jet projection. Equivalently, we can define $d^l$ via the formula $d^lf = j_l f - j_{l-1} f$, where $J^{l-1} \FS_X$ is seen inside $J^{l} \FS_X$ via the splitting. 
 On the other hand we can define another differential $d^l_\Delta$ in the following way. Note first that the sheaves $\FS_{X^2}/I_\Delta^{l+1}$ and $(J^l \FS_X)_\Delta$ are isomorphic; however, they are just isomorphic as sheaves, and not as $\FS_{X^2}$-modules. In any case the splitting of (\ref{eq: splits}), together with the isomorphism of sheaves 
 $\FS_{X^2}/I_\Delta^{l+1} \simeq (J^l \FS_X)_\Delta$ induce a morphism of sheaves 
 \begin{equation}\label{eq: splits2} \FS_{X^2}/I_\Delta^{l+1} \rTo (S^l \Omega^1_X)_\Delta \end{equation}
The differential $d^l_\Delta$ is then defined as the composition: 
 $$ d^l_\Delta: \FS_{X^2} \rTo \FS_{X^2}/I_\Delta^{l+1} \rTo (S^l \Omega^1_X)_\Delta $$of the jet projection $j^l_\Delta: \FS_{X^2} \rTo \FS_{X^2}/I_\Delta^{l+1}$ followed by the morphism (\ref{eq: splits2}). 
 
Analogously, for any $A \subseteq \{ 1, \dots, n \}$, $|A|=2$, we can define a higher differential $d^l_A: \FS_{X^n} \rTo (S^l \Omega_X^1)_A$.  Again, $d^l_A$ is just defined as a morphism of sheaves and not of $\FS_{X^n}$-modules.
We will denote with $j^l_A:  \FS_{X^n} \rTo \FS_{X^n}/I_A^{l+1}$ the partial $l$-jet projection.  If $A = \{ a_0, a_1 \}$, $a_0 < a_1$ we have that, on a global section $f_1 \tens \cdots \tens f_n \in R^{\tens n}$ of $\FS_{X^n}$, where $f_i \in R$, 
 the partial differential $d^l_A$ acts as 
 $$ d^l_A (f_1 \tens \cdots \tens f_n) = d^l_\Delta(f_A) \boxtimes f_{\bar{A}} = ( f_{a_0}d^l f_{a_1}) \boxtimes f_{\bar{A}} \in S^l \Omega^1_R \tens R^{\tens n-2}$$where
 for a multiindex $B \subseteq \{1, \dots, n \}$, we indicated with $f_B$ the element $f_B := \tens_{i \in B} f_i$, and where we identified $(S^l \Omega^1_X)_A$ with $(S^l \Omega^1_X )_\Delta \boxtimes \FS_{X^{\bar{A}}}$ via the automorphism $(\sigma_A)_*$. 
 
 With these notations 
 the operators $D^l_{\mu, A}: E^l(n,k) \rTo K^l$ can be written as: 
$$ D^l_{\mu, A} ( x_\lambda)_\lambda = \sum_{\substack{\beta \in c_n(l+1) \\ \beta_{\bar{A}}= 0 }}
(-1)^\beta \binom{l+1}{\beta} d^l_A(x_{\beta \mu})  \;,$$and can hence be extended to operators $D^l_{\mu, A} : S^k V_{\FS_X} \rTo K^l$. 
 \end{remark}

 As a consequence of theorem \ref{thm: mainimage} we have the following explicit characterization for the image of the Bridgeland-King-Reid transform $p_* q^* S^k \FS_X^{[n]}$ as a subsheaf of $S^k V_{\FS_X}$. The proof follows immediately form remark \ref{rmk: affine}. 
 \begin{crl}\label{crl: diffaffine}Let $X$ be a smooth affine surface. The sheaf $p_*q^*S^k \FS_X^{[n]}$ over $X^n$ is the subsheaf 
of $S^k V_{\FS_X}$ whose  sections $x =(x_{\lambda})_{\lambda} \in S^k V_{\FS_X} =\oplus_{\lambda \in c_n(k)} \FS_X$ satisfy the  system of higher order restrictions $D^0 x = 0 = D^1 x = \dots = D^{k-2} x$, or, equivalently 
$$ 
 \sum_{\substack{\beta \in c_n(l+1) \\ \beta_{\bar{A}}= 0 }}
(-1)^\beta \binom{l+1}{\beta} d^l_A(x_{\beta \mu}) =0 \qquad \mbox{    $\forall (\mu, A) \in B(k, l+1)$,  $0 \leq l \leq k-2$} $$  
\end{crl}
\begin{remark}\label{rmk: Krug}The operators $D^l_L$ are essentially the $\perm_k$-invariant version of the operators $\phi_l$ in Krug's work \cite{Krug2014}. However, the way these operators are found is completely different: Krug analyzes 
closely the natural morphism $p_* q^* L^{[n]} \tens \cdots p_* q^* L^{[n]} \rTo p_* q^* ( L^{[n]} \tens \cdots \tens L^{[n]} )\subseteq \mc{C}^0_L \tens \cdots \tens \mc{C}^0_L$, while we work out explicitely the higher differential in the spectral sequence $( j^* E^{p,q}_1)^{\perm_k}$. 
\end{remark}
\begin{remark}As mentioned in the introduction, higher order restrictions of the kind of $D^l_L$ appear naturally in presence of nontransverse intersections. The easiest nontrivial example is that of three distinct lines $\ell_1, \ell_2, \ell_3$ in the plane $\mbb{C}^2$, intersecting in a point $P$. Let $Z $ be the scheme-theoretic union of the three lines. Consider its structural sheaf $\FS_Z$. The natural sequence 
$$ 0 \rTo \FS_Z \rTo \oplus_i \FS_{\ell_i} \rTo^{d^0} \oplus_{i<j} \FS_{\ell_i \cap \ell_j}  \rTo 0 $$is not exact at the center (nor at the right). Indeed, to determine the sheaf $\FS_Z$ it is not sufficient to consider triples $(f_1, f_2, f_3)$ of regular functions $f_i$ over lines 
$\ell_i$ pairwise coinciding at the intersection point $P$ (that is, such that $d^0(f_1, f_2, f_3)=0$). We have to impose a second condition that the derivatives of $f_i$ in $P$ have to be dependent. The condition on derivatives may be given in terms of a higher order restriction $d^1: \ker d^0 \rTo \FS_P$ (see \cite[Remark 2.1]{ScalaPhD}). Hence, setting $F^i := \ker d^{i-1}$, the sheaf $\FS_Z$ can be see as the last term of a decreasing filtration $\FS_Z = F^2 \subset F^1 \subset F^0 = \oplus_i \FS_{\ell_i}$ of $\oplus_i \FS_{\ell_i}$. 

In our situation, to be non-transverse are the irreducible components $Z_i$ of the Haiman polygraph $D(n,k) \subseteq X^n \times X^k$ or better of its $\perm_k$-quotient $Z = D(n,k)/\perm_k \subseteq X^n \times S^k X$. Using the hyperderived spectral sequence associated to the derived tensor product $K_{Z_1} \tens^L \cdots \tens^L K_{Z_r}$ of the complexes 
$$ K_{Z_j}: 0 \rTo \FS_{X^n \times S^n X} \rTo \FS_{Z_j} \rTo 0 $$in the spirit of \cite[section 3.2]{Scala2009GD}, one could characterize the structural sheaf of the polygraph in terms of higher order restrictions $\tilde{D}^l$. Pushing everything forward to $X^n$, we would get the higher order restrictions $D^l_L$ for trivial $L$. For nontrivial $L$, one has just to straightforwardly tensorize the complexes $K_{Z_j}$ with adequate pull backs of the line bundles $L$ as in \cite{Scala2009GD}. 
\end{remark}

We finish this subsection proving a useful formula for the higher order restrictions $D^l_{\mu, A}$ in terms 
of analogous  operators  for the case of two points. 

\begin{notat}
If $n=2$, then the indication of the subset $A=\{1, 2\}$ will be
omitted and we will just denote with $D^l_\theta$ the operator
$D^l_{\theta, \{1, 2\}}$, with $\theta$ a composition of $k-l-1$
supported in $\{1, 2 \}$. 
\end{notat}

\begin{notat}Let $A \subseteq \{1, \dots , n\}$, $|A|=2$, and let $\mu$ be a composition supported in $\bar{A}$. 
Denote with 
$\hat{q}^\mu_A$ the projection 
$$\hat{q}^\mu_A: S^k V_L  \simeq \bigoplus_{\lambda \in c_n(k) } L^\lambda \rTo \bigoplus_{ \substack{\beta \in c_n(k-|\mu| ) \\ \beta_{\bar{A}}=0 }}L^{\beta \mu}  \;,$$Under 
the automorphism $(\sigma_A)_*$, defined in remark \ref{rmk: identification},  the term on the right corresponds to 
$S^{k-|\mu|} (L_{a_0} \oplus L_{a_1}) \boxtimes L^{\mu}$. Denote with $q^\mu_A$ the composition of morphisms
$$ q^\mu_A :=( \sigma_A)_* \circ \hat{q}^\mu_A: S^k V_L \rTo S^{k-|\mu|} (L_{a_0} \oplus L_{a_1}) \boxtimes L^{\mu} \;.$$
\end{notat}

\begin{remark}\label{rmk: localpartial}
Let $(\mu, A) \in B(k, l+1)$. Consider an affine open subset of the form $U^n$, such that $U$ is an affine open subset of $X$ over which $L$ is trivial. 
In the identification $X^n \simeq X^A \times X^{\bar{A}}$, we have, over $U^n \simeq U^A \times U^{\bar{A}}$: $$ D^l_{\mu, A} = (D^l_{\mu_A} \boxtimes \id) \circ q^{\mu_{\bar{A}}}_A \;.$$
\end{remark}\begin{proof}Remark first of all that, necessarily, $l+1 \leq k - |\mu_{\bar{A}}|$. 
In the identification $X^n \simeq X^A \times X^{\bar{A}}$ the sheaf $\FS_{X^n}/I_A^{l+1}$ corresponds
to $\FS_{X^2}/I_\Delta^{l+1} \boxtimes \FS_{X^{\bar{A}}}$. 
Now the automorphism $(\sigma_A)_*
: \FS_{X^n} \rTo \FS_{X^A} \boxtimes \FS_{X^{\bar{A}}}$, introduced in remark \ref{rmk: identification} satisfies the commutation relation
$ (j^l \boxtimes \id ) \circ (\sigma_A)_* = (\sigma_A)_* \circ j_A^l $. Hence, in the identification $U^n \simeq U^A \times U^{\bar{A}}$, 
\begin{align*} D^l_{\mu, A}((x_\lambda)_\lambda) = & \sum_{\substack{\beta \in c_n(l+1) \\ \beta_{\bar{A}}=0 } } (-1)^\beta \binom{l+1}{\beta} (\sigma_A)_* j_A^l (x_{\beta \mu}) \\
= & \sum_{\substack{\beta \in c_n(l+1) \\ \beta_{\bar{A}}=0 } } (-1)^\beta \binom{l+1}{\beta} (j^l \boxtimes \id)( \sigma_A)_* (x_{\beta \mu}) \\ 
= & \sum_{\substack{\beta \in c_n(l+1) \\ \beta_{\bar{A}}=0 } } (-1)^\beta \binom{l+1}{\beta} (j^l \boxtimes \id)( \sigma_A)_* (x_{\beta \mu_{A} \mu_{\bar{A}}}) \\
 = & \: (D^l_{\mu_A} \boxtimes \id) \circ q^{\mu_{\bar{A}}}_A ((x_\lambda)_\lambda) \end{align*}
\end{proof}

\begin{pps}\label{pps: tens}Let $r$ be an integer, $0 \leq r \leq k-1$. If $x \in E^r(n,k) \subseteq S^k V_L$,
then, for all composition $\nu$ supported in $\bar{A}$, with $|\nu| \leq k-r-1$, we have
$q^\nu_A(x) \in E^r(2, k-|\nu|)\boxtimes L^\nu$. Moreover, if $x \in E^r(n,k)$, $r \leq k-2$, for all $0 \leq l \leq r$ and for all $(\mu, A) \in B(k, l+1)$, 
then $|\mu_{\bar{A}} | \leq k-l-1$ and 
\begin{equation}\label{eq: partialformula}D^l_{\mu, A} x = (D^l_{\mu_A} \boxtimes \id) \circ q^{\mu_{\bar{A}}}_A (x) \;. \end{equation}
\end{pps}
\begin{proof}By complete induction on $r$. For $r=0$ we just have $E^0(n,k) = S^k V_L$ and $E^0(2, k-|\nu|) = S^{k-|\nu|}(L_{a_0} \oplus L_{a_1})$: hence the first statement is trivially verified. For $(\mu, A) \in B(k, 1)$, we just have to compare
$D^0_{\mu, A}$ with $(D^0_{\mu_A} \boxtimes \id) \circ q^{\mu_{\bar{A}}}_A$ under the identification given by $\sigma_A$: 
this follows trivially from the definitions, since the operators in question are just difference of restrictions. Suppose the proposition holds for $r$. Let $x \in E^{r+1}(n,k) $ and $\nu$ a composition supported in $\bar{A}$, $|\nu| \leq k-r -2$. 
Hence $|\nu| \leq k-r-1$, and by inductive hypothesis, $q^\nu_A (x) \in E^r(n,k-|\nu|) \boxtimes L^\nu$. We need to prove that 
$q^\nu_A (x) \in E^{r+1}(n,k-|\nu|) \boxtimes L^\nu$; to do this, we just need to prove that for all $\theta \in c_2(k-|\nu|-r-1)$
$(D^r_\theta \boxtimes \id) \circ q^\nu_A(x) = 0$. Via the identification $\sigma_A$ we can think of $\theta$ as a composition $\theta \in c_n(k-|\nu|-r-1)$ supported in $A$, and hence $\theta \nu$ as a composition of $k-r-1$. By the inductive hypothesis we have $$D^r_{\theta \nu} x = (D^r_\theta \boxtimes \id) \circ q^\nu_A (x) =0$$since 
$x \in E^{r+1}(n,k)$. Hence $q^\nu_A(x) \in E^{r+1}(n,k-|\nu|) \boxtimes L^\nu$. 

In order to prove the second statement we just have to prove the case $l = r+1$. Therefore, 
suppose that $(\mu, A) \in B(k, r+2)$. Then $k-r-2 = |\mu| = |\mu_A| + |\mu_{\bar{A}}|$ which proves that $|\mu_{\bar{A}}| \leq k-r-2 = k-l-1$. Then, by what we just proved,  $q^{\mu_{\bar{A}}}_A(x) \in E^{r+1}(2, k-|\mu_{\bar{A}}|) \boxtimes L^{\mu_{\bar{A}}}$. 
But then we can apply the operator $D^{r+1}_{\mu_A} \boxtimes \id$ to $q^{\mu_{\bar{A}}}_A(x)$ and hence compare, via the identification provided by $(\sigma_A)_*$,  the operators $D_{\mu, A}^{r+1}$ and $(D^{r+1}_{\mu_A} \boxtimes \id) \circ q^{\mu_{\bar{A}}}_{A}$. This can be done locally on an affine open set, and hence follows from remark  \ref{rmk: localpartial}. 
\end{proof}

\subsection{Invariant operators}\label{subsection: invariantoperators}
In this last subsection we apply theorem \ref{thm: mainimage} to obtain a description of the direct image $\mu_* S^k L^{[n]}$ 
over the symmetric variety $S^n X$ in terms of kernels of invariant operstors $(D^l_L)^G$. We will also give  explicit local formulas for the operators $(D^l_L)^G$. 
\subsubsection{$G$-invariant operators}We start by defining $G$-invariants of the sheaves appeared in the previous sections. 
Remember that the map $\eta_l: B(k,l) \rTo A(k,l)$ sending a couple $(\theta, A) \in B(k, l)$ to $(\nu(\theta_A), \nu(\theta_{\bar{A}}))$ --- where $\nu(\theta_A)$ and $\nu(\theta_{\bar{A}})$ are the partitions associated to the compositions $\theta_A$ and $\theta_{\bar{A}}$, respectively --- is the quotient map for the group $G$ acting on $B(k, l)$. 

 \begin{notat}Let $l$ an integer, $0 \leq l \leq k-2$. Write $\mc{K}^l(L) := K^l(L)^G$. 
 For $(\mu, \nu) \in A(k,l+1)$, denote with $\K^l_{\mu, \nu}$ the sheaf: 
 $$ \K^l_{\mu, \nu} := \Bigg( \bigoplus_{\substack{ (\theta, A) \in B(k,l+1)\\ 
\eta_l(\theta, A) =(\mu, \nu) }  } \pi_*  K^l_{\theta, A} (L) \Bigg) ^{G} \;.$$Remark that the sheaf $\K^l_{\mu, \nu}$ is isomorphic, by Danila's lemma, to 
 $ ( \pi_* K^l_{(\theta_0, A_0)} )^{\Stab_G(\theta_0, A_0)}$, where $(\theta_0, A_0)$ is a chosen point 
 in $\eta_l^{-1}(\mu, \nu)$. 
  \end{notat}
 
 \begin{remark}\label{rmk: same}It is evident, after what we said in the beginning of subsection \ref{subsection: higher} and even more after corollary \ref{crl: differentialscomparison} and theorem \ref{thm: mainimage}, that the $G$-sheaves  $K^l(L)$ are isomorphic to the terms  $(E^{l+1, -l}_1)^H_0$
 of subsection \ref{subsection: Invariants}. In particular, if $\eta_l (\theta, A) = (\mu, \nu)$ with $\mu$ of the form $(h,h)$, 
 $\Stab_G(\theta, A)$ contains the subgroup $G(A) \simeq \perm_2$, which acts on $\pi_* K^{l}_{(\theta, A)}(L)$ via the alternanting representation $\epsilon_2$  and makes the invariants $(\pi_* K^{l}_{(\theta, A)}(L)) ^{\Stab_G(\theta, A)}$ vanish. 
 \end{remark}
 
 By remarks \ref{rmk: same} and  \ref{rmk: multiindexinvariants} we have, over $S^n X$: 
 $$ (S^k V_L)^G \simeq \bigoplus_{\lambda \in p_n( k)} \mathcal{L}^\lambda, \qquad (K^l(L))^G \simeq 
 \bigoplus_{(\mu, \nu) \in A(k,l+1)} \K^l_{\mu, \nu} = \bigoplus_{(\mu, \nu) \in A_0(k,l+1)} \K^l_{\mu, \nu} \;.$$

 \begin{notat}Denote with $\E^l(n,k)$ the sheaf of invariants $\E^l(n,k) := E^l(n,k)^G$ over the symmetric variety $S^n X$ and 
 with $\D^l_L$ the $G$-invariant 
 operators $$\D^l_L := (D^l_L)^G : \E^l(n,k) \rTo \K^l(L) \;.$$For $(\mu, \nu) \in A(k,l+1)$, denote moreover
 with $\D^l_{\mu, \nu}$ the $(\mu, \nu)$ component of $\D^l_L$. Note that, if $(\mu, \nu) \not \in A_0(k, l+1)$
 the component $\D^l_{\mu, \nu}$ is identically zero.  \end{notat}
 
 Since in our context taking $G$-invariants is an exact functor, we have $\ker \D^l_L = \E^{l+1}(n,k) $. As an immediate consequence of theorem \ref{thm: mainimage} we have the following 
 \begin{theorem}\label{thm: directimage}Over the symmetric variety $S^n X$ the sheaves $\E^l(n,k)$ define a finite decreasing filtration of $ (S^k V_L)^G$ such that 
 $$ \B{R} \mu_* S^k L^{[n]} \simeq^{qis} \mu_* S^k L^{[n]} \simeq \E^{k-1}(n,k) = \ker \D^{k-2}_L \;.$$
 \end{theorem}
 We now proceed to find explicit local formulas for the operators $\D^l_L$: these formulas will be essential in section \ref{section: filtration}. 
 Before we begin, we give the following notation. 
 \begin{notat}Let $r$ a nonnegative integer.  For $\lambda \in c_n(r_1)$,  $\mu \in c_n(r_2)$, with
 $\lambda \geq \mu$ (according to notation \ref{notat: orderpartitions}), set as $\beta = \lambda/\mu$ the composition defined by $\beta_i = \lambda_i - \mu_i$.  \end{notat}

 \subsubsection{The case $n=2$.} For $n=2$ we have that $B(k,l)$ identifies to the set $c_2(k-l)$ of compositions of $k-l$ of range 2  and  $A(k,l)$  identifies to the set $p_2(k-l)$ of partitions of $k-l$ of lenght $ \leq 2$. Hence we will simply write $K^l_\mu(L)$, with $\mu \in c_2(k-l-1)$ for the components of $K^l(L)$ and  $D^l_\mu$ for the components of $D^l_L$; analogously if $\mu \in p_2(k-l-1)$ we will just write $\K^l_\mu $ for the components of $\K^l(L)$ and $\D^l_\mu$ for the components of $\D^l_L$. 
If $\mu \in p_2(k-l-1)$, 
  the component $\D^l_\mu$ of the differential $\D^l_L$ is given by the composition
  $$ \D^l_\mu : \E^l(2, k) = \pi_*( E^l(2,k))^G \rInto \pi_* ( E^l(2,k)) \rTo^{\pi_* D^l_{\mu}} \pi_* K^l_\mu (L) \;.$$
 Since $\E^l(2, k) \subseteq (S^k V_L)^G \simeq \bigoplus_{\lambda \in p_2(k)} \mathcal{L}^\lambda$, we write a local section
 $x \in \E^l (2, k)$ as $(x_\lambda)_\lambda$.

 Recall that if $I$ is a finite set equipped with a transitive action of a finite group $G$ and $M= \oplus_{i \in I} M_i$ is a $G$-module, where the $G$ action on $M$ is compatible with the $G$ action on $I$ (see \cite[section 2.4]{Scala2009D}), 
 the isomorphism of Danila's lemma $M_{i_0}^{\Stab_G(i_0)} \simeq M^G$, with $i_0 \in I$, is given by
 $ u \rMapsto \sum_{[g ] \in G/\Stab_G(i_0)} gu $.  Here $G/\Stab_G(i_0)$ has to be seen just as a collection of cosets. Hence the map 
 $$\bigoplus_{\lambda \in p_2(k)} \mathcal{L}^\lambda \simeq \pi_* (S^k V_L)^G \rInto 
 \pi_* S^k V_L = \bigoplus_{\lambda^\prime \in c_2(k)} \pi_*L^{\lambda^\prime}$$is given by 
 $(x_\lambda)_\lambda \rMapsto (x_{\lambda^\prime})_{\lambda^\prime}$ where $\lambda^\prime$ runs among compositions of the form $\lambda^\prime = \sigma \lambda$ with $\sigma \in G$ and $\lambda \in p_2(k)$ and where $x_{\sigma \lambda}
 = \sigma_* x_\lambda$.  
 Now take $(x_\lambda)_\lambda$ a local section  of $\E^l(2,k)$ over an affine open set of the form $S^2 U$, $U$ being an affine open set of $X$ such that $L$ is trivial over $U$. By what we  said, if $\mu \in p_2(k-l-1)$, 
 \begin{align*} \D^l_\mu (x_\lambda)_\lambda =  (\pi_* D^l_\mu)(x_{\lambda^\prime})_{\lambda^\prime} = &
  \sum_{\substack{\beta \in c_2(l+1)}} (-1)^{\beta} \binom{l+1}{\beta}[x_{\beta \mu}]_{I_\Delta^{l+1}}
\\
= &   \Bigg [  \sum_{\substack{\lambda^\prime \in c_2(k) , \:\lambda^\prime \geq \mu 
\\    \lambda^\prime = \sigma \lambda , \: \sigma \in G, \lambda \in p_2(k)  }} (-1)^{\lambda^\prime / \mu} \binom{l+1}{\lambda^\prime / \mu} \sigma_* x_{\lambda} \Bigg]_{I_{\Delta}^{l+1}}\end{align*}where the last term is seen as a local section of $\K^l_\mu = \pi_* (S^l \Omega^1_X \tens L^k)_\Delta$, since the sum inside the square bracket belongs to $I_\Delta^l \tens L^k$. 
\begin{example}Take $k=3$. Then $(S^3 V_L)^G \simeq \bigoplus_{\lambda \in p_2(3)} \mathcal{L}^\lambda = \mathcal{L}^{3} \oplus \mathcal{L}^{2,1}$. A local section $x$ of 
$(S^3 V_L)^G$ over an affine open set $S^2 U$, where $U$ is an affine open set of $X$ such that $L$ is trivial over $U$, can be written as $(x_{3,0}, x_{2,1})$. Suppose that $x \in \E^1(2, 3)$. 
Then \begin{align*} \D^1_{(1,0)}x = & \Bigg[ (-1)^{(3,0)/(1,0)}\binom{2}{(3,0)/(1,0)} x_{(3,0)} + (-1)^{(2,1)/(1,0)} \binom{2}{(2,1)/(1,0)} x_{(2,1)} + \\ & \qquad \qquad+ (-1)^{(1,2)/(1,0)}\binom{2}{(1,2)/(1,0)} \sigma_* x_{2,1} \Bigg]_{I_\Delta^2} \\ 
= &  \Bigg[ (-1)^{2}\binom{2}{2} x_{(3,0)} + (-1)^{1} \binom{2}{1} x_{(2,1)} + (-1)^{0}\binom{2}{0} \sigma_* x_{2,1} \Bigg]_{I_\Delta^2} =  [ x_{3,0} -2 x_{2,1}+ \sigma_* x_{2,1}]_{I_\Delta^2} \;.
\end{align*}If $x_{3,0} = f \tens a \in H^0(X, L^3) \tens H^0(X, \FS_X)$, $x_{2,1} = g_1 \tens g_2 \in H^0(X, L^2) \tens H^0(X, L)$ are global sections of $\mathcal{L}^3$ and $\mathcal{L}^{2,1}$, respectively, 
we can write, if $x \in \E^1(2,3)(S^2 X)$, 
$$ \D^1_{(1,0)}(f \tens a, g_1 \tens g_2) =
fda  -[g_1, g_2] \in H^0(X, \Omega^1_X \tens L^3)$$where 
$[ \cdot, \cdot]$ is the braket $H^0(X, L^2) \tens H^0(X, L) \rTo H^0(X, \Omega^1_X \tens L^3)$ defined locally by
$[g_1, g_2] := 2g_1 dg_2 - g_2 dg_1$. This bracket appears, for example, in \cite{Green1994}. 

\end{example}
\subsubsection{The general case.}  Let's discuss now the case with general
$n$. We will use the following diagram: 
\begin{diagram}
X^n &  & \\
\dTo_{\hat{\pi}} & \rdTo^{\pi} & \\
S^2 X \times S^{n-2}X & \rTo_{v} & S^n X  
\end{diagram}where $\hat{\pi}$ is the $\perm_2 \times \perm_{n-2}$ quotient map and $v$ is the induced map from the partial to the total quotient. Take $(\theta, \nu) \in A_0(k, l+1)$. The component $\D^l_{\theta, \nu}$ of the $G$-invariant 
differential $\D^l_L$ is, by definition, the composition: 
$$ \D^{l}_{\theta, \nu} :  
\E^l(n,k) = (\pi_*E^l(n,k))^G \rInto \pi_* E^l(n,k)
\rTo ^{\pi_* D^l_{\theta \nu, \{1 2 \} } } 
\pi_* K^l_{\theta\nu, \{1, 2 \}}(L)\;.$$Note that starting with an element 
in $\E^l(n,k) = (\pi_*E^l(n,k))^G$, that is, invariant by $G$, we will correctly finish with an element $\D^l_{\theta, \nu} x \in 
\pi_* (K^l_{\theta \nu, \{1, 2\}}(L))^{\Stab_{\perm_{n-2}}(\nu)} \simeq \K^{l}_{\theta, \nu} \simeq
v_*( (S^l \Omega^1_X \tens L^{l +1+ |\theta|})_{\Delta} \boxtimes \mathcal{L}^{\nu} ) $. We want to find a simpler way to write it, along with  an explicit local expression. 
\begin{notat}Let $\mu$, $\nu$ two partitions of $i$ and $m-i$, respectively, for $i$, $m$ nonnegative integers with $i \leq m$. 
We denote with $\mu \coprod \nu$ the unique partition of lenght $l(\mu)+ l(\nu)$ 
associated to the composition of $m = |\mu|+|\nu|$ with value $\mu_j$ if $1 \leq j \leq l(\mu)$ and $\nu_{j-l(\mu)}$ if $j > l(\mu)$. 
\end{notat}With this notation, it is easy to see that over $S^2X \times S^{n-2}X$, 
\begin{equation}\label{eq: s2xsn-2} (S^k V_L)^{\perm_2 \times \perm_{n-2}} \simeq \bigoplus_{\substack{
    \mu , \nu \\ l(\mu) \leq 2, l(\nu) \leq n-2 \\ \mu \coprod \nu \in p_n(k)}}
\mathcal{L}^{\mu} \boxtimes \mathcal{L}^{\nu} \;. \end{equation}

\begin{notat}For a partition $\nu_0$ of $l$, $0 \leq l \leq k$, $l(\nu_0) \leq n-2$, denote with 
$q_{\nu_0}$ the composition of maps
$$q_{\nu_0} : (\pi_* S^k V_L)^{G} \rInto (\pi_* S^k V_L)^{\perm_2 \times \perm_{n-2}} \simeq   \bigoplus_{\substack{
    \mu, \nu \\ l(\mu) \leq 2, l(\nu) \leq n-2 \\ \mu \coprod \nu \in p_n(k)}}
v_*( \mathcal{L}^{\mu} \boxtimes \mathcal{L}^{\nu})  \rTo \bigoplus_{\substack{\mu \in p_2(k-|\nu_0|)  
}} v_*(\mathcal{L}^\mu \boxtimes \mathcal{L}^{\nu_0}) \;.$$where the isomorphism in the middle is induced by \ref{eq: s2xsn-2}. 
\end{notat}\begin{lemma}\label{lmm:comm2}
The diagram
\begin{diagram}( \pi_* S^k V_L)^G& \rTo^{q_{\nu_0}} &  \bigoplus_{\substack{
    \mu \in c_2(k-|\nu_0|) \\  }} \pi_*( L^\mu \boxtimes L^{\nu_0})^{\perm_2 \times \Stab_{\perm_{n-2}}(\nu_0)}  \simeq \bigoplus_{\substack{
    \mu \in p_2( k-|\nu_0|)}}
v_* (\mathcal{L}^{\mu} \boxtimes \mathcal{L}^{\nu_0}) \\
\dTo^{ \imath_G} & & \dTo^{\imath_{\perm_2 \times \Stab_{\perm_{n-2}} (\nu_0)}} \\
\pi_* S^k V_L & \rTo^{\pi_*(q^{\nu_0}_{\{1,2\} })} & \bigoplus_{\substack{
    \mu \in c_2(k-|\nu_0|) }} 
\pi_* (L^{\mu} \boxtimes L^{\nu_0}) 
\end{diagram}where the vertical maps are natural injections, \sloppy
is commutative. In the second line of the diagram $\nu_0$ is thought as a composition of supported in $\{3, \dots, n\}$ via the unique order-preserving bijection $\{1, \dots, {n-2} \} \simeq \{3, \dots, n \}$. 

\end{lemma}
\begin{proof}It is sufficient to check the commutativity locally on an affine open set, where we can identify sheaves with the modules of sections. Let  $M$ be the module of sections of $\bigoplus_{\lambda \in c_n(k)} \pi_*( L^\lambda)$, $G= \perm_n$, $I=c_n(k)$, $J= c_2(k-|\nu_0|) \times \{ \nu_0 \}$, 
$H= \perm_2 \times \Stab_{\perm_{n-2}} (\nu_0)$. Then the commutativity of the diagram follows directly from the following lemma, whose proof is straightforward, noting that $q_{\nu_0} = p_J^H \circ i_{G, H}$. 

\end{proof}
\begin{lemma}Let $R$ a commutative $\mbb{C}$-algebra. 
Let $G$ a finite group and $H \subseteq G$ a subgroup. Suppose that $G$ acts on a finite set $I$
and let $M=\oplus_{i \in I} M_i$ be a $R[G]$-module such that the $G$-action on $M$ is compatible with the $G$-action on $I$. Let moreover $J \subseteq I$ be a subset of $I$ such that $J$ is $H$-invariant. Set $M_J = \oplus_{j \in J}M_j$ and let $p_J$ the projection $M \rTo M_J$ on the summands in $J$. 
Then the following diagram is commutative:
\begin{diagram} M^G & \rInto^{i_{G,H}} & M^H & \rOnto^{p_J^H} & M_J^H \\
\dInto^{i_G} & \ldInto^{i_H} & &\ldInto^{i_{J, H}} & \\
M & \rOnto^{p_J} & M_J & & \\
\end{diagram}where $i_G$, $i_{G, H}$, $i_{J, H}$ and $i_H$ are natural injections. 
\end{lemma}
We now come to the $G$-invariant version of formula \ref{eq: partialformula}, expressing an operator $\D^l_{\theta, \nu}$ in terms of the operator $\D^l_\theta$ over $S^2X$. 
\begin{pps}\label{pps: formula2}Let $r$ be a nonnegative integer, $r \leq k-1$, let $x$ be a local section of $ \E^r(n,k)$ and let $\nu_0$ be a partition of weight $|\nu_0| \leq k-r-1$ of length $l(\nu_0) \leq n-2$. Then $q_{\nu_0} (x) \in v_*(\E^r(2,k-|\nu_0|) \boxtimes \mathcal{L}^{\nu_0})$. Moreover for all integers $l$, $0 \leq l \leq r$, and for all $\theta \in p_2(k-|\nu_0|-l-1)$, $\theta \neq (h,h)$,  we have
\begin{equation}\label{eq: formula2} \D^l_{\theta, \nu_0} (x) = v_*(\D^l_\theta \boxtimes \id) \circ q_{\nu_0} (x) \;.\end{equation}
\end{pps}\begin{proof}If $x \in \E^{r}(n,k)$ then $x \in \pi_* E^{r}(n,k)^G$. 
Now by proposition \ref{pps: tens} the image of $E^r(n,k)$ for $q_{\{1,2\}}^{ \nu_0}$ is in $E^r(2,k-|\nu_0|) \boxtimes L^{\nu_0}$. 
Hence by  lemma 
\ref{lmm:comm2}  we have
$$q_{\nu_0}(x) = \pi_* ( q_{\{ 1, 2 \}}^{\nu_0}) (x) \in \pi_* (E^r(2,k-|\nu_0|) \boxtimes L^{\nu_0})^{\perm_2 \times \Stab_{\perm_{n-2}} (\nu_0)} \simeq v_*( \E^r(2, k-|\nu_0|) \boxtimes \mathcal{L}^{\nu_0} )\;.$$This proves the first statement. 

As for the second, the invariant operator $\D^l_{\theta, \nu_0}$ is computed by Danila's lemma as $\D^l_{\theta, \nu_0} = 
\pi_*(D^l_{\theta\nu_0, \{ 1, 2 \}}) \circ i_G$. Hence, by lemma \ref{lmm:comm2} and by proposition \ref{pps: tens}
\begin{align*}
\D^l_{\theta, \nu_0}(x) = & \: \pi_*(D^l_{\theta \nu_0, \{1, 2 \}} ) \circ \imath_G(x) 
= \pi_* [(D^l_\theta \boxtimes \id) q_{\{ 1, 2\}}^{\nu_0} ] \imath_G(x)\\
= & \: \pi_* (D^l_\theta \boxtimes \id) \pi_*(q_{\{ 1, 2\}}^{\nu_0} ) \imath_G(x)
=  \pi_*(D^l_\theta \boxtimes \id) \imath_{\perm_2 \times \Stab_{\perm_{n-2}}(\nu_0)} q_{\nu_0}(x) \\
= &\:  \pi_*(D^l_\theta \boxtimes \id) v_*( \imath_{\perm_2} \boxtimes \imath_{\Stab_{\perm_{n-2}}(\nu_0)} ) q_{\nu_0}(x)  \;.
\end{align*}Write now $\hat{\pi}: X^ n \rTo S^2X \times S^{n-2}X$ as $\hat{\pi}_1 \times \hat{\pi}_2$, where $\hat{\pi}_1$ and $  \hat{\pi}_2$ are the two
projections from $X^n \rTo S^2X$ and $X^ {n-2} \rTo S^{n-2}X$, respectively. Then: 
\begin{align*}\D^l_{\theta, \nu_0}(x) = & v_* \hat{\pi}_* ( D^l_\theta \boxtimes \id) v_*( \imath_{\perm_2} \boxtimes \imath_{\Stab_{\perm_{n-2}}(\nu_0)} ) q_{\nu_0}(x) \\
= & v_*  \Big[  ( ( \hat{\pi_1}_*  D^l_\theta)   \circ \imath_{\perm_2}) \boxtimes \imath_{\Stab_{\perm_{n-2}}(\nu_0)} \Big] q_{\nu_0}(x) \\ = & v_* (\D^l_\theta \boxtimes \imath_{\Stab_{\perm_{n-2}}  (\nu_0)} ) q_{\nu_0}(x) 
=   v_* (\D^l_\theta \boxtimes \id ) q_{\nu_0}(x) \;.
\end{align*}The last equality follows from the fact that the term $v_* (\D^l_\theta \boxtimes \imath_{\Stab_{\perm_{n-2}}  (\nu_0)} ) q_{\nu_0}(x) $ is naturally $\Stab_{\perm_{n-2}}(\nu_0)$-invariant and hence seen in 
$\pi_* (K^l_{\theta, \nu_0}(L))^{\perm_2 \times \Stab_{\perm_{n-2}} (\nu_0)}$ 
and not just in $\pi_* (K^l_{\theta, \nu_0}(L))^{\perm_2} $.  \end{proof}

\subsubsection{Action on global sections.} \label{subsubsection: actglob}We finish the section by expliciting how the invariant operators $\D^l_{\theta, \nu}$ act on global sections of the sheaves $\E^l(n, k)$. 
\begin{notat}\label{notat: exp}Let $\lambda$ a partition of $k$, $\lambda = (\lambda_1, \dots, \lambda_r)$. The \emph{exponential notation} for the partition $\lambda$ is the writing $(1^{\alpha_1(\lambda)}, 2^{\alpha_2(\lambda)}, \dots, m_{\lambda}^{\alpha_{m_\lambda}(\lambda)})$, where 
$\alpha_{\lambda_i}(\lambda)$ is the number of times the integer $\lambda_i$ appears in the list $(\lambda_1, \dots, \lambda_r)$ and where $\lambda_1 = m_\lambda$. 
We will also write $\lambda = \prod_{i=1}^{m_\lambda} i^{\alpha_i(\lambda)}$. 
\end{notat}
Thanks to the exponential notation of partitions we can express the global sections of sheaves $(S^k V_L)^G$ and 
$\K^l_{\theta, \nu}$. For brevity's sake we just write $H^i(F)$ for the $i$-th cohomology of a coherent sheaf $F$ over the surface $X$;  
if $\lambda \in p_m(k)$, recalling notation \ref{notat: Lmu}, we will also write $H^i(\mc{L}^\lambda_m)$ instead of the longer $H^i(S^mX, \mc{L}^\lambda_m)$. First of all, if $\mu$ is a partition of $h$ of length $l(\mu) \leq m$,  we have: 
$$ H^0(S^m X, \mathcal{L}^\mu) \simeq \Tens_{i=1}^{m_{\mu}}S^{\alpha_i(\mu)}H^0(L^i) \tens S^{m-l(\mu)}H^0(\FS_X) \;.$$
Then: 
\begin{align*} H^0(S^n X, (S^k V_L)^G) \simeq &\: \bigoplus_{\lambda \in p_n(k)} H^0(S^n X, \mathcal{L}^\lambda) \simeq
 \bigoplus_{\lambda \in p_n(k)} \Tens_{i=1}^{m_\lambda} S^{\alpha_i(\lambda)}H^0(L^i) \tens S^{n-l(\lambda)}H^0(\FS_X) \\ 
 \qquad H^0(S^n X, \K^{l}_{\theta, \nu}) = & \: H^0(S^{l} \Omega^1_X \tens L^{|\theta|+l+1}) \tens H^0(S^{n-2} X, \mathcal{L}^\nu) \\ 
 = &  \: H^0(S^{l} \Omega^1_X \tens L^{|\theta|+l+1}) \tens \Tens_{i=1}^{m_\nu} S^{\alpha_i(\nu)}H^0(L^i) \tens S^{n-l(\nu)-2} H^0(\FS_X)
\end{align*}
Now, by proposition \ref{pps: formula2}, in order to explicit the action of the operator $\D^l_{\theta, \nu}$ on global sections, it is just sufficient to understand how the morphism $q_{\nu_0}$ acts at the level of global sections. 
Let $\lambda$ be a partition of $k$ and let
$\lambda =\mu \coprod \nu$:  
in the exponential 
notation we can write $ \lambda = \prod_{i=1}^n i^{\alpha_i}$, $\mu = \prod_{i=1}^n i^{\beta_i}$, $\nu = \prod_{i=1}^n i^{\gamma_i}$, with $\alpha_i = \beta_i + \gamma_i$. We define the operators $$\imath_{\mu, \nu}^\lambda : H^0(\mathcal{L}^\lambda_{n}) \rTo H^0(\mathcal{L}^\mu_{2}) \tens H^0(\mathcal{L}^\nu_{n-2})$$ as follows: 
\begin{itemize}
\item if $l(\mu) = 2$, then $\imath^\lambda_{\mu, \nu}$ is defined as the natural map 
$$ \imath^\lambda_{\mu, \nu}: H^0(\mathcal{L}^\lambda_{n}) \simeq \Tens_{i} S^{\alpha_i}H^0(L^i) \Tens S^{n-l(\lambda)}H^0(\FS_X) \rInto \Tens_i S^{\beta_i}H^0(L^i) \tens S^{\gamma_i}H^0(L^i) \Tens  S^{n-l(\lambda)}H^0(\FS_X) $$induced by the inclusions
$ S^{\alpha_i}H^0(L^i) \rInto S^{\beta_i}H^0(L^i) \tens S^{\gamma_i}H^0(L^i)$; 
\item If $l(\mu)=1$, then $\mu$, in exponential notation is $\mu = i_0^1$ for some $i_0$: 
for this $i_0$ we have $\alpha_{i_0} = 1 + \gamma_{i_0}$. Hence 
the map $\imath^\lambda_{\mu, \nu}$ is then the compositon 
\begin{multline*}  \imath^\lambda_{\mu, \nu}: H^0(\mathcal{L}^\lambda_{n}) \simeq \Tens_i S^{\alpha_i} H^0(L^i)  \Tens S^{n-l(\lambda)}H^0(\FS_X)\rInto 
H^0(L^{i_0}) \tens  \Tens_{i}  S^{\gamma_i}H^0(L^i) \tens S^{n-l(\lambda)}H^0(\FS_X) \\ \rInto  H^0(L^{i_0}) \tens H^0(\FS_X) \tens S^{\gamma_i}H^0(L^i) \tens S^{n-l(\lambda)-1}H^0(\FS_X)  \simeq 
H^0(\mathcal{L}^\mu_{2}) \tens H^0(\mathcal{L}^\nu_{n-2}) \end{multline*}where the first map is induced by the natural injection 
$S^{\alpha_{i_0}} H^0(L^{i_0}) \rInto 
H^0(L^{i_0}) \tens    S^{\gamma_{i_0}}H^0(L^{i_0})$ and the second by the injection 
$S^{n-l(\lambda)}H^0(\FS_X) \rInto H^0(\FS_X) \tens S^{n-l(\lambda)-1}H^0(\FS_X)$. 
\end{itemize}For a partition $\nu_0$ with $l(\nu_0) \leq n-2$, the morphism $q_{\nu_0}$, at the level of global sections is the composition: 
$$ Q_{\nu_0}: \bigoplus_{\lambda \in p_n(k)} H^0(\mathcal{L}^\lambda_n) \rTo \bigoplus_{\substack{\mu, \nu \in p_n(k) \\ l(\mu) \leq 2, l(\nu) \leq n-1 \\ \mu \coprod \nu \in p_n(k)}} H^0(\mathcal{L}^\mu_{2}) \tens H^0(\mathcal{L}^\nu_{n-2}) \rTo \bigoplus_{\substack{\mu \in p_n(k) \\ l(\mu) \leq 2 \\ \mu \coprod \nu_0 \in p_n(k)}} H^0(\mathcal{L}^\mu_{2} ) \tens H^0(\mathcal{L}^{\nu_0}_{n-2})$$where the first map is induced by the morphisms $i^\lambda_{\mu, \nu}$, while the second one is the projection onto some summands. Then, at the level of global sections we have: 
\begin{crl}Let $r$ be a nonnegative integer, $r \leq k-1$ and $\nu_0$ a partition of weight $|\nu_0| \leq k-r-1$ and of length $l(\nu_0) \leq n-2$.  If $x \in H^0(S^n X, \E^r(n,k))$, then 
$Q_{\nu_0}(x) \in H^0(S^2 X, \E^r(2, k-|\nu_0|) ) \tens H^0(S^{n-2}X, \mathcal{L}^{\nu_0}_{n-2})$. Then, for all integers $l$, $0 \leq l \leq r$, and for all 
$\theta \in p_2(k-|\nu_0|-l-1)$, $\theta \neq (h,h)$,  we have
$$ \D^l_{\theta, \nu_0} = (\D^l_\theta \tens \id) \circ Q_{\nu_0} $$where $\D^l_\theta$ is the operator for the $n=2$ case, acting on the factor $H^0(S^2 X, \E^r(2, k-|\nu_0|))$. \end{crl}
As an application of this corollary we could now give another proof of the following theorem by Danila \cite{Danila2007}. 
We will indicate more briefly with $H^0(\E^r(m,k))$ the global sections $H^0(S^m X, \E^r(m,k))$. 
\begin{theorem} Let $X$ a smooth algebraic surface with $H^0(\FS_X) \simeq \mbb{C}$ (for example a smooth projective surface) and let $L$ be a line bundle on $X$. Let $k, n \in \mbb{N}$, with $n \neq 0$ and $k \leq n$. We have 
$$ H^0(X^{[n]}, S^k L^{[n]}) \simeq S^k H^0(X, L) \;.$$
\end{theorem}\begin{proof}[Sketch of the proof]
Just to give a sketch of the idea, $H^0(X^{[n]}, S^k L^{[n]})$ can be identified as the kernel of the recursively defined operator 
$\D^{k-2}_L: H^0(\E^{k-2}(n,k)) \rTo H^0(\K^{k-2}(L))$. 
Now, if $H^0(\FS_X) \simeq \mbb{C}$  it is not difficult to see that $\ker \D^0_L \simeq H^0(\E^1(n,k))$ is already isomorphic to $S^k H^0(X, L)$. 
Indeed, there is a natural map $H^0(\E^1(n,k)) \rTo S^kH^0(L)$, given by the composition \begin{equation}\label{eq: lambdacontraction} H^0(\E^1(n,k)) \rInto H^0((S^k V_L)^G) \simeq \oplus_{\lambda \in p_n(k)} H^0(\mathcal{L}^\lambda_n) \rTo H^0(\mathcal{L}^{1^k}_n) \simeq S^k H^0(L) \;. \end{equation}Analyzing more in detail the map $\D^0_L: H^0((S^k V_L)^G) \rTo H^0(\K^0(L))$ one could prove that the natural map above is injective. We won't give more details here. 

On the other hand
since $H^0(\FS_X) \simeq \mbb{C}$, for any $\lambda \in p_n(k)$, written as $\lambda = \prod_{i=1}^{m_\lambda} i^{\alpha_i}$ in exponential notation, there is a natural contraction $c_\lambda: S^k H^0(L) \rTo H^0(\mathcal{L}^\lambda_n) \simeq \Tens_{i=1}^{m_\lambda} S^{\alpha_i}H^0(L^i)$. It is not difficult to see that, if $x \in S^k H^0(L)$, the element $(c_{\lambda} x)_{\lambda \in p_n(k)}$ is in $H^0(\E^1(n,k))$; this proves that the natural map (\ref{eq: lambdacontraction}) is an isomorphism.  

Moreover, one can see with a little more work, thanks to the characterization of $H^0(\E^{1}(n,k))$ given by the above isomorphism, that  all higher operators $\D^l_L$ are zero restricted to $ H^0(\E^{1}(n,k))$; hence $S^k H^0(X, L)$ can be identified to $\ker \D^{k-2}_L \simeq H^0(\E^{k-2}(n,k)) \simeq H^0(X^{[n]}, S^k L^{[n]})$. \end{proof}

\vspace{0.2cm}
We now make a couple of examples of computations of operators $\D^l_{\theta, \nu}$which will be useful in the sequel. 

\begin{example}(Affine case for $k=3$). \label{ex: local3} Let $X= \Spec(R)$ a smooth affine surface, where $R$ is a finitely generated $\mbb{C}$-algebra, and let $L$ be the trivial line bundle over $X$. We will identify coherent sheaves over $X$ with 
their modules of global sections.  Hence a global section $x$ of $\mathcal{L}^ 3 \oplus \mathcal{L}^{2,1} \oplus \mathcal{L}^{1,1,1}$ can be seen as an element of $( R \tens S^ {n-1}R)  \oplus (R \tens R \tens S^ {n-2}R) \oplus (S^3 R \tens S^ {n-3}R)$. Then $x$ is a sum of elements of the form $y=(f \tens a, g_1 \tens g_2 \tens b, h_1h_2 h_3 \tens c)$. 
where $f, g_1, g_2, h_1, h_2, h_3 \in R$, $a=a_1 \dots a_{n-1} \in S^{n-1}R$, $b= b_1 \dots b_{n-2}\in S^{n-2}R$, $c \in S^{n-3}R$.
. Suppose that $x$ belongs to $\E^1(n,3)$. The operator
$\D^1_{(1)(0)}$ acts on each of the summand of $x$ of the form $y$ as: 
$$ \D^1_{(1)(0)} (y) = 
\sum_i f da_i \tens \widehat{a_i} - [2 g_1 dg_2 - g_2 dg_1] \tens b \in \Omega^1_R \tens S^{n-2}R =\K^1_{(1)(0)}\;,$$
where we indicate with $\hat{a}_i = a_1 \cdots a_{i-1}a_{i+1} \cdots a_{n-1} 
\in S^{n-2}R$. 
\end{example}
\begin{example}(Affine case for $k=4$). \label{ex: local4}With the same hypothesis of the previous example, a global section $x$ of $\mathcal{L}^{4} \oplus \mathcal{L}^{3,1} \oplus \mathcal{L}^{2,2} \oplus \mathcal{L}^{2,1,1} \oplus \mathcal{L}^{1,1,1,1} $ can be seen as an element in $
( R \tens S^{n-1}R ) \oplus ( R \tens R \tens S^{n-2}R) \oplus (S^{2}R \tens S^{n-2}R) \oplus
(R \tens S^{2}R \tens S^{n-3}R) \oplus (S^{4}R \tens S^{n-4}R)$. Then $x$ is a sum of elements of the form 
$y=(f \tens a, g_1 \tens g_2 \tens b, h_1 h_2 \tens c, k \tens k_2 k_3 \tens d,  m_1 \dots m_4 \tens e)$, where $f, g_i, h_j, k_l, m_s \in R$, $a=a_1 \dots a_{n-1} \in S^{n-1}R$, $b= b_1 \dots b_{n-2}, c=c_1 \dots c_{n-2} \in S^{n-2}R$, $d \in S^{n-3}R$, $e \in S^{n-4}R$. Suppose that $x$ is in $\E^1(n,4)$. Then the operator $\D^1$ acts on each of the summand of $x$ of the form $y$ as: 
\begin{gather*}
\begin{split}  \D^1_{(2)(0)}(y) = & \sum_i f da_i \tens \widehat{a_i} - 2 g_1 dg_2 \tens b  + [ h_1 dh_2 + h_2 dh_1]  \tens c \hfil \in \Omega^1_R \tens S^{n-2}R =\K^1_{(2)(0)}\\
\D^1_{(1)(1)}(y) =  &\sum_i g_1 db_i \tens g_2 \tens \widehat{b_i} -2k [dk_2 \tens k_3 + dk_3\tens k_2 ] \tens d + 
\\ & \qquad \qquad +  [k_2 dk \tens k_3 + k_3 dk \tens k_2] \tens d \hfil \in \Omega^1_R \tens R \tens S^{n-3}R =\K^1_{(1)(1)}\end{split}\end{gather*}
Suppose now $x \in \E^{2}(n,4)$. Then the operator $\D^2$ acts on
each summand of $x$ of the form $y$ as:
\begin{equation*} \begin{split} \D^2_{(1)(0)}(y) =   ( -f d^2 a_i \tens \widehat{a_i} +3 g_1 d^2 g_2 \tens b - & 3 [h_1 d^2 h_2 +  h_2 d^2 h_1] \tens c \\ & +
g_2 d^2 g_1 \tens b \in S^2 \Omega^1_R \tens S^{n-2}R = \K^2_{(1)(0)} \end{split}\end{equation*}
\end{example}

\section{The filtration \texorpdfstring{$\W^\bullet$}{W} on the direct image \texorpdfstring{$\mu_* S^k L^{[n]}$}{muSkL[n]}.} \label{section: filtration}
The filtration $\E^\bullet (n,k)$ of $(S^k V_L)^G$ found in the previous section, together with theorems \ref{thm: mainimage} and \ref{thm: directimage}, cannot yet be used to get precise information on the sheaf $\mu_* S^k L^{[n]}$ for general $n$, 
since, in the case $n > 2$, the operators $\D^l_L : \E^l(n,k) \rTo \K^l(L)$ are not surjective, as it will be clear in the sequel; at the same time, it seems difficult to understand directly the graded sheaves $\gr^{\E}_l$ of the filtration $\E^\bullet := \E^\bullet (n,k)$. In this section we will introduce another filtration $\V^\bullet$  of $(S^kV_L)^G$. The two filtrations $\V^\bullet$ and $\E^\bullet$ will induce a bifiltration $\V^\bullet \cap \E^\bullet$ of $(S^k V_L)^G$ and hence a filtration 
$\W^\bullet :=\V^\bullet \cap \E^{k-1}(n,k)$ of $\E^{k-1}(n,k) \simeq \mu_*S^k L^{[n]}$.  For the filtration $\W^\bullet$ things 
behave in an a better way, so that it is easier to guess the graded sheaves $\gr^{\W}_l$; as a consequence we will be able to deduce results on the cohomology of $S^k L^{[n]}$ for $n =2$ and general $k$ and for general $n$ and low $k$. We think, however, that similar results may be obtained in complete generality (see subsection \ref{toward}). 
We begin the section defining the filtrations $\V^\bullet$ and $\W^\bullet$ and with some general notations and preliminary lemmas. 
\begin{definition}\label{def: filtrationVW} Let $n \in \mbb{N}$, $n \geq 1$. 
Denote with $\leq_{\rm rlex}$ the reverse lexicographic order on the set of partitions. For $\mu \in p_n(k)$ set: 
\begin{equation}\label{eq: filtrationV} \V^\mu := \bigoplus_{\substack{\lambda \in p_n(k) \\ \lambda \geq_{\rm rlex} \mu }} \mathcal{L}^\lambda \; .\end{equation}The sheaves $\V^\mu$ define a finite decreasing filtration $\V^\bullet$ on $(S^k V_L)^G$. Denote with 
$ \W^\bullet$ the finite decreasing filtration of $\E^{k-1}(n,k) \simeq \mu_* S^k L^{[n]}$, indexed on $p_n(k)$,  induced by $\V^\bullet$, that is
$$ \W^\bullet : = \V^\bullet \cap \E^{k-1}(n,k) \;.$$
\end{definition}
\begin{remark}For any total order $\preccurlyeq$ 
on the set of partitions $p_n(k)$, the analogue of definition \ref{eq: filtrationV}, that is
$  \V^\mu := \bigoplus_{\lambda \in p_n(k), \:   \lambda \succcurlyeq \mu } \mathcal{L}^\lambda$,  defines a decreasing filtration on $(S^k V_L)^G$ and hence induces a decreasing filtration $\W^\bullet =  \V^\bullet \cap \E^{k-1}(n,k)$ on $\E^{k-1}(n,k) \simeq \mu_* S^k L^{[n]}$. In this section we will use the reverse lexicographic order as total order on $p_n(k)$; 
on the other hand in subsection \ref{toward} we will explain how, if we hope to express a nice general result on the graded sheaves $\gr^{\W}_l$, it's better to use another total order $\preccurlyeq$,  
introduced in definition \ref{def: toward}. The two orders concide on partitions $p_n(k)$ with $k \leq 5$. 
\end{remark}
\begin{notat}\label{notat: diagm}Let $m, n \in \mbb{N}$, $2 \leq m \leq n$. Denote with $\Delta_m$  the subscheme of $X^n$ whose ideal sheaf $I_{\Delta_m}$ is 
 the intersection of ideal sheaves $I_{\Delta_I}$ of pairwise partial diagonals $\Delta_I$, $I  \subseteq \{1, \dots, m\}, |I|=2$. 
 It is the inverse image of the so-called \emph{big diagonal} in $X^m$ via the projection $p_m : X^n \rTo X^m$ onto the first $m$ factors. 
\end{notat}
\begin{remark}\label{rmk: bigdiagonal}An important fact about powers of the ideal sheaf of the scheme $\Delta_m$ in $X^n$ is the following equality. Over $X^n$, for all $s \in \mbb{N}$ one has: 
\begin{equation} \label{eq: bigdiagonal} 
 \bigcap_{\substack{I \subseteq \{1, \dots, m \} \\ | I | =2}} I_{\Delta_I}^s = 
 \Bigg ( \bigcap_{\substack{I \subseteq \{1, \dots, m \} \\ | I | =2}} I_{\Delta_I} \Bigg) ^s = I_{\Delta_m}^s 
\end{equation}This fact has been proven by Haiman in \cite[Corollary 3.8.3]{Haiman2001} for $n=m$ and follows trivially for $m \leq n$ by taking the pull back
$p^{*}_m$ of (\ref{eq: bigdiagonal}) via the projection $p_m : X^n \rTo X^m$. 
We will use this fact in what follows. 
\end{remark}
\begin{notat}\label{notat: lDelta}Let $\mu \in c_n(k)$, $l \in \mbb{N}$, $m \in \mbb{N}$, with $ 2 \leq m \leq n$. We denote with 
$L^\mu(-l \Delta_m)$ the sheaf over $X^n$ defined by:
$$ L^\mu(-l \Delta_m) := L^ \mu \tens \bigcap_{\substack{I \subseteq \{1, \dots, m \} \\ | I | = 2}} I_{\Delta_I}^l \;.$$Note that, by equality (\ref{eq: bigdiagonal}), the sheaf $L^\mu(-l \Delta_m)$ can also be defined by 
$L^\mu(-l \Delta_m) := L^\mu \tens I_{\Delta_m}^l$, which justifies the notation. Now, 
if $\mu \in p_n(k)$ is a partition of $k$ of lenght $l(\mu)$ denote with $\mathcal{L}^\mu(-l \Delta)$ the sheaf  over $S^nX$ defined by
$$ \mathcal{L}^\mu(-l \Delta) := \pi_* \left( L^\mu(-l \Delta_{l(\mu)}) \right)^{\Stab_G(\mu)} \;.$$In case we need to emphasize that the sheaf $\mathcal{L}^\mu(-l \Delta)$ lives on $S^nX$ we will write $\mathcal{L}^\mu_n (-l \Delta)$ instead of just $\mathcal{L}^\mu(-l \Delta)$. 
\end{notat}
\begin{remark}
\label{rmk: differential}Let $I \subseteq \{ 1, \dots , n\}$, $|I|=2$.  
The restriction $ j^l_I: I^l_{\Delta_I} \rTo I^l_{\Delta_I}/I^{l+1}_{\Delta_I} \simeq (S^l \Omega^1_X )_I $ of the partial jet projection $j_I^l$ (see remark \ref{rmk: affine}) to the ideal $I^l_{\Delta_I} \subseteq \FS_{X^n}$
induces a global 
map, that  will be indicated again with $d^l_I$:  $$d^l_I: L^\mu \tens I^l_{\Delta_I} \rTo L^\mu \tens I^l_{\Delta_I}/I^{l+1}_{\Delta_I} \simeq (S^l \Omega^1_X \tens L^{|\mu_I|} )_I \tens L^{\mu_{\bar{I}}} \;.$$Notation $d_I^l$ has first been introduced in remark \ref{rmk: affine} if $L = \FS_X$; the notation we are giving here is compatible with the previous one. With the help 
of the differentials $d^l_I$
we can define a 
differential 
$$ d^l_{\Delta_m}: L^\mu(-l \Delta_m) \rTo \bigoplus_{\substack{I \subseteq \{ 1, \dots, m\} \\ |I|=2} } L^\mu \tens I^l_{\Delta_I} \rTo^{\oplus_I d^l_I} \bigoplus_{\substack{I \subseteq \{ 1, \dots, m \} \\ |I|=2} } L^\mu  \tens I^l_{\Delta_I}/I^{l+1}_{\Delta_I}  \simeq 
 \bigoplus_{\substack{I \subseteq \{ 1, \dots, m\} \\ |I|=2} } (S^l \Omega^1_X \tens L^{|\mu_I|})_I \tens L^{\mu_{\bar{I}}} \;.$$
 whose kernel is easily seen, after (\ref{eq: bigdiagonal}),  to be $L^\mu(-(l+1) \Delta_m)$. Remark that if $\supp \mu = \{1, \dots, m \}$ the differential 
 $d^l_{\Delta_m}$ is $\Stab_G(\mu)$-equivariant. 
Therefore,  if $\mu \in p_n(k)$ and $m= l(\mu)$, applying the functor $\pi_*^{\Stab_G(\mu)}$,  the previous map descends to a differential, noted with $d^l_{\Delta}$
\begin{equation}\label{eq: partialdifferentialsnx} d^l_\Delta : \mathcal{L}^{\mu} (-l \Delta) \rTo  \Bigg[ \bigoplus_{\substack{I \subseteq \{ 1, \dots, l(\mu) \} \\ |I|=2} } \pi_* \left( (S^l \Omega^1_X \tens L^{|\mu_I|})_I \tens L^{\mu_{\bar{I}}} \right) \Bigg]^{\Stab_G(\mu)} \;.\end{equation}Its kernel is $\ker d^l_{\Delta} = \mathcal{L}^\mu(-(l+1) \Delta)$. 
\end{remark}
\begin{example} For any $\mu \in c_2(k)$, the differential 
$d^l_{\Delta_2}$ is a map $d^l_{\Delta_2}: L^\mu (-l \Delta_2) \rTo (S^l \Omega_X^1 \tens L^k)_{\Delta_2} $ and its  kernel is $L^\mu (-(l+1) \Delta_2)$. For any $\mu \in p_2(k)$ of lenght exactly $2$, $d^l_\Delta$ is a map  
$$d^l_\Delta : {\mathcal{L}}^\mu (-l \Delta) {\rTo} [ v_* (  (S^l \Omega_X^1 \tens L^k)_\Delta  \boxtimes 
\FS_{ S^{n-2}X } ) ]^{ \Stab_{\perm_2 } (\mu) } $$ of kernel $\mathcal{L}^\mu (-(l+1) \Delta)$. 
Again, if $\mu \neq (h,h)$, the stabilizer is trivial.
\end{example}
\begin{example}\label{ex: 1k}Let $\mu = 1^k$ in exponential notation. Then 
$d^l_\Delta$ is defined over $S^kX$ as: 
$$ d^l_\Delta: \mathcal{L}^{1^k}(-l\Delta) \rTo \Big[ v_*  ( ( S^l \Omega^1_X \tens L^2)_\Delta \boxtimes \mathcal{L}^{1^{k-2}})  \Big]^{\perm_2  }\;,$$where $\perm_2$ is naturally acting on the factor $(S^l \Omega^1_X \tens L^2)_\Delta$ with the sign $(-1)^l$.
Hence, if $l$ is odd, $\mathcal{L}^{1^k}(-l \Delta) \simeq \mathcal{L}^{1^k}(-(l+1) \Delta)$. 
 \end{example}
\begin{remark}\label{rmk: restriction1k}The restriction of the operator $\D^0_{L}$ to $\mathcal{L}^{1^k}$: 
$$ \D^0_L \trest_{\mathcal{L}^{1^k}}: \mathcal{L}^{1^k} \rTo \K^0_{(1)(1^{k-2})} \simeq v_*( L^2_\Delta \boxtimes \mathcal{L}^{1^{k-2}}) $$coincides with 
$d^0_\Delta$. Indeed it follows by definitions that the component
$\D^0_{(1)(1^{k-2})}$ is defined over $\mathcal{L}^{1^k}$ as: 
$$ \mathcal{L}^{1^k} \rTo^{q_{(1^{k-2})}} v_*(  \mathcal{L}^{(1,1)} \boxtimes \mathcal{L}^{1^{k-2}}) \rTo^{v_*(\D^0_{(1)}  \boxtimes \id)} v_*(L^2_\Delta \boxtimes \mathcal{L}^{1^{k-2}}) \;,$$and $\D^0_{(1)}$ identifies to the restriction to $\Delta$, as seen with the definition. 
On the other $d^0_\Delta$ is defined as: 
$$ \pi_*(L^{1^k})^{\Stab_G(\mu)} \rTo \bigoplus_{\substack{I \subseteq \{ 1, \dots, k \} \\ |I | =2}} \pi_* (L^{1^k} \tens \FS_{X^n}/I_{\Delta_I})^{\Stab_G(1^k)} \simeq \pi_*(L^{1^k} \tens \FS_{X^n}/I_{\Delta_{12}})^{\perm_2 \times \perm_{n-2}} $$where the last identification is due to Danila's lemma. It's clear that $d^0_\Delta$ factors via 
$$\pi_*(L^{1^k})^{\Stab_G(\mu)} \rInto \pi_*(L^{1^k} )^{\perm_2 \times \perm_{n-2}} \rTo \pi_*(L^{1^k} \tens \FS_{X^n}/I_{\Delta_{12}})^{\perm_2 \times \perm_{n-2}} \;.$$Now 
$\pi_*(L^{1^k} )^{\perm_2 \times \perm_{n-2}} \simeq v_*(\mathcal{L}^{(1,1)} \boxtimes \mathcal{L}^{(1^{k-2})})$, 
$\pi_*(L^{1^k} \tens \FS_{X^n}/I_{\Delta_{12}}) ^{\perm_2 \times \perm_{n-2}} \simeq v_*(L^2_{\Delta} \boxtimes \mathcal{L}^{1^{k-2}} )$: 
the first map here above identifies to $q_{(1^{k-2})}$, and the second one to the restriction to $\Delta$ on the first factor, and hence to
$v_*(\D^0_{(1)} \boxtimes \id)$. Hence we have $d^0_\Delta = {\D}^0_{(1)(1^{k-2})}$. 
\end{remark}
\paragraph{The natural line bundle $\mc{D}_A$.} Before proving the main result in the next subsections, in order to state it in the most general form, 
we introduce the \emph{natural line bundle}  on the Hilbert scheme of points $X^{[n]}$, associated to a line bundle on $X$. More precisely, if $A$ is a line bundle on $X$, the line bundle $A \boxtimes \cdots \boxtimes A$ ($n$-factors) on $X^n$ \emph{descends} to a line bundle $\mc{D}_A$ on $S^n X$, in the sense that $\pi^* \mc{D}_A = A \boxtimes \cdots \boxtimes A$ \cite[Thm 2.3]{drezetNarasimhan1989}. As a consequence, the line bundle 
$\mc{D}_A$ on $S^n X$ coincides with the sheaf of $G$-invariants, on $S^n X$, of the line bundle $A \boxtimes \cdots \boxtimes A$. Pulling-back the line bundle $\mc{D}_A$ on $S^n X$ via the Hilbert-Chow morphism $\mu: X^{[n]} \rTo S^n X$ we get a line bundle $\mu^* \mc{D}_A$ on the Hilbert scheme, called the \emph{natural line bundle} on $X^{[n]}$ associated to the line bundle $A$ on $X$. For brevity's sake, we will denote it again with $\mc{D}_A$.  
\begin{remark}We have natural filtrations 
$\mc{V}^\bullet \tens \mc{D}_A$ and $\E^\bullet (n,k) \tens \mc{D}_A$ of $(S^k V_L)^G \tens \mc{D}_A$ over the symmetric variety $S^n X$, and hence a natural filtration $\W^\bullet \tens \mc{D}_A = ( \V^\bullet \cap \E^{k-1} (n,k) ) \tens \mc{D}_A$; but, because of projection formula 
$ \mu_* (S^k L^{[n]} \tens \mc{D}_A) := \mu_*(S^k L^{[n]} \tens \mu^* \mc{D}_A) \simeq \mu_* S^k L^{[n]} \tens \mc{D}_A $,  the last filtration yields naturally a filtration of $\mu_*(S^k L^{[n]} \tens \mc{D}_A)$.  
\end{remark}Since the line bundle $\mc{D}_A$ on the symmetric variety $S^n X$ is the descent of the line bundle $A \boxtimes \cdots \boxtimes A$ on $X^n$, using projection formula  and taking $G$-invariants we have the following
\begin{lemma}\label{lemma: invariantsFDA}Let $F$ be a $G$-equivariant  coherent sheaf over $X^n$. Then, over $S^nX$ we have
$$ (\pi^G_*F) \tens \mc{D}_A \simeq \pi^G_*( F \tens A^{\boxtimes n}) \;.$$
\end{lemma}

\begin{remark}\label{rmk: k2}For $k=2$ and any $n \geq 2$ it is immediate to prove the exact sequence: 
$$ 0 \rTo \mc{L}^{1,1}(-2 \Delta) \tens \mc{D}_A \rTo \mu_* S^2 L^{[n]} \tens  \mc{D}_A \rTo \mathcal{L}^2  \tens \mc{D}_A \rTo 0 \;.$$Indeed, 
it is sufficient to prove it for $A$--and hence $\mc{D}_A$ -- trivial. The sequence above is obtained 
taking kernels in the diagram
\begin{diagram} 0 & \rTo & \mc{L}^{1,1} & \rTo & \V^0 & \rTo & \mathcal{L}^2 & \rTo & 0 \\
& &           \dTo^{\D^0_L \trest_{\mathcal{L}^{1,1}}} & &  \dTo^{\D^0_L} & & \dTo^0 & & \\
0 &\rTo & \K^0(L) & \rTo^{\simeq} & \K^0(L) & \rTo & 0 & \rTo & 0 
\end{diagram}Note that $\K^0(L) \simeq v_*( L^2_\Delta  \boxtimes \FS_{S^{n-2}X})$ and that the restriction $\D^0_L \trest_{\mathcal{L}^{1,1}}$ coincides with $d^0_\Delta~:~\mathcal{L}^{1,1} \rTo v_*( L^2_\Delta \boxtimes \FS_{S^{n-2}X})$, that is, with the restriction of sections of $\mathcal{L}^{1,1}$ to the big diagonal. The exactness on the right of the sequence of kernels comes from the fact that 
the operator $\D^0_L \trest_{\mathcal{L}^{1,1}}$ is surjective. 
\end{remark}

\subsection{The case $n=2$}
\subsubsection{Toeplitz Matrices} 
\begin{definition}
A $m \times n$ ($m$ rows
and $n$ columns) complex matrix $A=(a_{ij}) \in M_{m \times
  n}(\mbb{C})$ is Toeplitz if there exist numbers $t_{-m+1}, \dots, t_0,
\dots, t_{n-1}$ such that $a_{ij} = t_{j-i}$, that is, if the diagonals of $A$ are constant. 
\end{definition}
\begin{remark}To a squared Toeplitz $m \times m$ matrix $A$ we can associate a function $f_A: \mbb{R} \rTo \mbb{C}$, by the formula 
$$ f_A(y) = \sum_{k=1-m}^{m-1} t_k e^{iky} \;.$$It is clear that $ \overline{f_{A}} = f_{\transpose{\bar{A}}} $ and that $f_{A^h} = {\rm Re} f_A$, where $A^h = (1/2)(A+ \transpose{\bar{A}})$ is the hermitian part of $A$. Hence the function $f_A$ is real if and only if $A$ is hermitian and if $A$ is real, then $f_A$ is real if and only if $A$ is symmetric. Given the function $f_A$, we can reconstruct the matrix $A$ via the anti-Fourier transform: 
$$ t_l(A) = \frac{1}{2\pi } \int^{2 \pi}_0 f_A(y) e^{-ily} dy \;.$$
If $A$ is a complex $m \times m$ Toeplitz matrix, then for any 
$x = (x_1, \dots, x_m) \in \mbb{C}^m$ we have
\cite[page 42]{Gray2006}:  
\begin{equation}\label{eq: txAx} {}^{t}\bar{x} A x = 
 \frac{1}{2 \pi}\int_0^{2 \pi}
\Big | \sum_{l=1}^m x_l e^{-il y} \Big |^2 f_A(y) dy \;.\end{equation}\end{remark}

\begin{crl}\label{crl: hermitiantoeplitz}
Let $A$ be an hermitian Toeplitz matrix such that $f_{A}(y)$, 
the function associated to $A$, is positive (negative),
apart from at most a countable number of points. Then $A$  is positive
(negative) definite and hence nondegenerate. 
\end{crl}
\begin{proof}
If $f_{A}(y)$ is positive (outside a countable number of points) 
the integral in (\ref{eq: txAx}) above is strictly positive, and the matrix
$A$ is positive definite. 
\end{proof}

We will now introduce a few Toeplitz matrices which will be useful in the sequel. 

\begin{notat}Given a vector $a = (a_0, \dots, a_{n-1}) \in \mbb{C}^{n}$,    we will denote with $T_{m \times n}(a)$ the $m \times n$ 
Toeplitz matrix with
$t_i = a_i$ if $i \geq 0$ and  $t_i=0 $ if $i<0$.
\end{notat}
\begin{notat}Let $m, n \in \mbb{N}$. 
Denote with $T_{2n,m}$ and $T_{2n+1,m}$ the $m \times m$ Toeplitz matrices defined by: 
\begin{align*} \qquad t_{k} =  & (T_{2n,m})_{i+k,i}  =  (-1)^{k}\binom{2n}{n+k} \; \; & \mbox{if $-n \leq k \leq n$} \quad \qquad \qquad & t_k= 0 \mbox{ otherwise } \qquad  \\ 
\qquad s_{k} =&  (T_{2n+1, m})_{i+k,i}  =  (-1)^{k}\binom{2n+1}{n+k+1} \; \; & \mbox{if $-n-1 \leq k \leq n$} \quad  \qquad \qquad & s_k= 0 \mbox{ otherwise.} \qquad \end{align*}\end{notat}
\begin{lemma}\label{lemma: Tnondeg} For all $m,n \in \mbb{N}$, $m \neq 0$, the matrices $T_{2n, m}$, $T_{2n+1,m}$ are nondegenerate. 
\end{lemma}
\begin{proof}It suffices to prove that $T_{2n,m}$ is nondegenerate, since $T_{2n+2,m}= T_{m \times m}(e_1 - e_2) T_{2n+1,m}$ and $T_{m \times m}(e_1 - e_2)$ is trivially nondegenerate, where we indicated with $e_i$ the vectors of the canonical basis of $\mbb{C}^n$. 

We prove first that $T_{2n,m}$ is nondegenerate for $m \geq 2n$. In this case  $T_{2n, m}$ is a real symmetric matrix
whose associated function is $f_{T_{2n,m}}(y)= (-1)^n (e^{iy/2}-e^{-iy/2})^{2n}= (2 \sin y/2)^{2n}$ which is positive apart from the countable set of zeros of $\sin y/2$. Hence, by corollary \ref{crl: hermitiantoeplitz} it is positive definite, and hence nondegenerate, since symmetric. 

For $m < 2n$, take $l \geq 2n$. Then, for all $k \leq 2n$, $\det T_{2n,k}$ is a leading principal minor of $T_{2n,l}$, which is positive definite. 
Hence, by Jacobi criterion, $\det T_{2n,k}$ is positive for all $k \leq 2n$. Hence $T_{2n,m}$ is nondegenerate (and positive definite). 
\end{proof}
\begin{notat}
For $l, k, j \in \mbb{N}$, $ l \leq k$, $2j \leq k$,  denote with $R_{l,k}(j)$ the $(k-l+1) \times (k-2j+1)$ the Toeplitz matrix such that $$ t_i = (-1)^i \binom{l}{j+i} \qquad \mbox{for $-j \leq i \leq l-j$}, \hspace{3cm} t_i =0 \mbox{ otherwise} \:.$$ 
\end{notat}
\begin{remark}\label{rmk: Rmatrix}
 If $l \leq 2j < k+1$ then $R_{l,k}(j)$ is the matrix of an injective morphism and 
$R_{2j,k}(j) = 
T_{2j, k+1-2j}$ is the matrix of an isomorphism. 
Indeed, if $l \leq 2j <k+1$, deleting the first $j-[(l+1)/2]$ and the last $j-[l/2]$ rows of $R_{l, k}(j)$, we get the matrix $T_{l, k+1-2j}$,  which is nondegenerate by lemma \ref{lemma: Tnondeg}. 
\end{remark}

\subsubsection{The filtration.}
Recall that we denoted  the diagonal in $X^2$ and $S^2 X$ just with $\Delta$. 
Denote for brevity's sake the filtration $E^l(2, k)$ just with $E^l$. 
Since differentials $D^l_L : E^l \rTo K^l(L)$ are surjective for $n=2$ by remark \ref{rmk: DLsurj}, the graded sheaves are exactly $\gr^E_l = K^l(L)= \bigoplus_{\mu \in c_2(k-l-1)}K^l_{\mu}(L)$, where, recalling notation \ref{notat: brevdelta},  $K^l_\mu(L) = (S^l \Omega_X \tens L^k)_\Delta$ and we can see $\mu$ just as a label. 

\begin{definition}
For $j = 0, \dots, k$, denote with $V^j = \bigoplus_{i=j}^{k-j} L^{(k-i, i)} \subseteq E^0 =
S^k V_L$. \end{definition}
\begin{remark}
The subsheaves $V^j$ of $S^k V_L$ define a $G$-equivariant 
decreasing filtration of $S^k V_L$: 
$$ 0=V^{[k/2]+1} \subsetneq V^{[k/2]} \subseteq \cdots \subseteq V^1
\subsetneq V^0 = S^k V_L \;.$$Taking $G$-invariants, we get exactly the filtration $\V^\bullet$ of $(S^k V_L)^G$, introduced in definition \ref{def: filtrationVW}. 
\end{remark}Set now $S_{j,l}:= \oplus_{i=j}^{k-j} L^{(k-i, i)}(-l \Delta)$. Note that, if $l \leq 2j-1$, $S_{j,l} \subseteq V^j \cap E^l$: indeed, 
$S_{j,l}$ consists of elements $(x_i)_i$ $j \leq i \leq k-j$, with $x_i \in L^{k-i, i} \tens I^l_\Delta$ and hence all operators $D^h_L$, for $h <l$,  vanish on them. 
Hence we can apply the operator $D^l_L$ to elements in $S_{j,l}$. 

\begin{lemma}\label{lmm: factornoninv}For $0 \leq l \leq 2j-1$
the $G$-equivariant restriction 
$ D^l_L \trest_{S_{j,l}} :  S_{j,l} \rTo \gr_l^E $ of the operator $D^{l}_L$  to $S_{j,l}$ factorizes as: 
$$
S_{j,l} = \oplus_{i=j}^{k-j} L^{(k-i, i)}(-l \Delta) 
\rTo^{\oplus_{i=j}^{k-j} d_{\Delta}^l}  \oplus_{i=j}^{k-j}
 (S^l \Omega^1_X \tens L^k)_\Delta  \\
 \rTo^{A} \gr_l^E \simeq \oplus_{r=0}^{k-l-1}
 (S^l \Omega^1_X \tens L^k)_\Delta 
$$where $d^l_{\Delta} : L^{k-s, s}(-l \Delta) \rTo (S^l \Omega_X \tens L^k)_\Delta$ is the differential introduced in remark \ref{rmk: differential} and 
where $A$ is a $G$-equivariant \emph{injective}
morphism of vector bundles over $\Delta$ having as matrix the
$ (k-l) \times (k+1-2j)$ Toeplitz matrix $R_{l+1, k}(j)$. 
\end{lemma}
\begin{proof}Let $x =(x_i)_i \in S_{j,l}$, $j  \leq i \leq k-j$. 
The component $D^l_{(r, k-l-1-r)}$ of the operator $D^l_L$  can be explicitely written as 
\begin{align*} D^l_{(r, k-l-1-r)} x =& \sum_{s=0}^{l+1}(-1)^s \binom{l+1}{s} j^l_\Delta(x_{(s, l+1-s)(r, k-l-1-r)}) \\
& =  \sum_{s=0}^{l+1}(-1)^s \binom{l+1}{s} j^l_\Delta(x_{(r+s, k-r-s)}) \; ,
\end{align*}where $j^l_\Delta$ denotes the $l$-jet projection. 
It is immediate that $D^l_L$ is a composition of $\oplus_{i=j}^{k-j} d^l_\Delta: S_{j,l} \rTo \oplus_{i=j}^{k-j}
(S^l \Omega^1_X \tens L^k)_\Delta$ and a linear map $A: \bigoplus_{i=j}^{k-j} (S^l \Omega^1_X \tens L^k)_\Delta \rTo K^l(L)$ 
whose $(k-l) \times (k-2j+1)$ matrix has entries  $(-1)^s \binom{l+1}{s}$. Ordering the terms, that is, 
setting $i = j+a$, $0 \leq a \leq k-2j$, we have $r+s = j+a$ and $s = j+a-r$; we then find that the matrix of the  map $A$ is given 
by $$A_{r,a} = (-1)^{j+a-r} \binom{l+1}{j+a-r}$$which is Toeplitz  and coincides easily with $R_{l+1, k}(j)$. Hence $A$ is injective if $l \leq 2j-1$ and an isomorphism if $l = 2j-1$  by remark \ref{rmk: Rmatrix}. 
\end{proof}
\begin{remark}Note that the term $\oplus_{i = j}^{k-j}(S^l \Omega^1_X \tens L^k)_\Delta$ in the preceding lemma is isomorphic to $$\bigoplus_{i = j}^{k-j}(S^l \Omega^1_X \tens L^k)_\Delta \simeq 
\oplus_{i = j}^{k-j} L^{ k-i, i}(I_\Delta^l/I_{\Delta}^{l+1})$$and hence naturally $G$-equivariant. The $G$ actions 
exchanges terms $ L^{i, k-i}(I_\Delta^l/I_{\Delta}^{l+1})$ and $ L^{k-i, i}(I_\Delta^l/I_{\Delta}^{l+1})$ $i \neq k/2$, 
that is, it exchanges factors $(S^l \Omega^1_X \tens L^k)_\Delta$ indexed by $i \neq k/2$, 
and, if $k$ is even, it acts on the term  $ L^{k/2, k/2}(I_\Delta^l/I_{\Delta}^{l+1})$, that is, on the term $(S^l \Omega^1_X \tens L^k)_\Delta$ indexed by $k/2$, with the sign $(-1)^l$. \end{remark}
Let's now descend on $S^2X$. For $l \leq 2j-1$ we naturally have that $S_{j,l}^G \subseteq \V^j \cap \E^l$: in particular $S^G_{j, l}$ is in the domain of definition of $\D^l_L$. We indicate with $\oplus_{i \geq k/2}^{h}$ a direct sum indexed by natural numbers $i$ such that $k/2 \leq i \leq h$; analogously for $\oplus_{i \gneq  k/2}^{h}$ or $\oplus_{i > k/2}^{h}$. 
Taking $G$-invariants in lemma \ref{lmm: factornoninv}, we get the following. 
\begin{pps}\label{pps: SjlG}Suppose that $0 \leq l \leq 2j-1$. 
Then
the restriction $\D^l_L \trest_{S_{j, l}^G}:  S_{j,l}^G 
 \rTo  \gr_l^{\E}$ of the operator $\D^l_L$ to $S_{j,l}^G$ factorizes
as: 
$$ S_{j,l}^G = \bigoplus_{i \geq k/2}^{k-j} \mathcal{L}^{i, k-i}(-l\Delta)
\rTo^{\oplus_{i \geq k/2}^{k-j} d^l_{\Delta}} \left[ \oplus_{i \geq j}^{k-j} (S^l \Omega_X^1 \tens L^k)_{\Delta} \right]^G \rTo^{A^G} \gr_l^{\E}
$$where $A^G$ is an isomorphism if $l=2j-1, 2j-2$ and injective if $l \leq 2j-3$.
 Hence, if $l \leq 2j-1$,  the kernel of the
restriction is $$\ker \D^l_L \trest_{S_{j,l}^G} \simeq  
\bigoplus_{i \geq k/2}^{k-j} \mathcal{L}^{i,
  k-i}(-(l+1)\Delta) \simeq S^G_{j, l+1}\;.$$
  \end{pps}\begin{proof}
Let's prove first the statement about $A^G$. 
It follows from lemma \ref{lmm: factornoninv} that $A^G$ is injective if $l \leq 2j-1$. 
The graded sheaf $\gr_l^{\E}$ is given by $$ \gr_l^{\E} \simeq \bigoplus_{\substack{ \mu \in p_2(k-l-1) \\ \mu \neq (h,h)}}  \K^l_{\mu} $$and hence consists of $\displaystyle \big[  \frac{k-l}{2} \big]$ terms, all isomorphic to $(S^l \Omega_X^1 \tens L^k)_{\Delta}$. On the other hand, the sheaf $\big[ \oplus_{i \geq j}^{k-j} (S^l \Omega_X^1 \tens L^k)_{\Delta}
\big]^G$ changes if $k$ is odd or even. Let's see the two cases.
\begin{itemize}
\item $k$ is odd. Then 
$$ \left[ \oplus_{i \geq j}^{k-j} (S^l \Omega_X^1 \tens L^k)_{\Delta} \right]^G
\simeq \oplus_{i \gneq k/2}^{k-j} (S^l \Omega_X^1 \tens L^k)_{\Delta} $$ and hence it consists of 
a direct sum of $\displaystyle \frac{k+1}{2} -j $ terms. If $l =2j-1$ or $l = 2j-2$, we have, since $k$ is odd, that
$$ \frac{k+1}{2} - j = \frac{k -2j +1}{ 2} = \big[ \frac{k - l}{ 2} \big] \;.$$
Hence in the two cases $l = 2j-1, 2j-2$ the two sides are direct sums with the same number of copies of $(S^l \Omega^1_X \tens L^k)_\Delta$. Since $A^G$ is a constant injective endomorphism of a vector  bundle (over the diagonal $\Delta$), it is an isomorphism in these cases. 
\item $k $ is even. The 
sheaf of invariants  $\left[ \oplus_{i \geq j}^{k-j} (S^l \Omega_X^1 \tens L^k)_{\Delta} \right]^G$ is given by
$$\left[ \oplus_{i \geq j}^{k-j} (S^l \Omega_X^1 \tens L^k)_{\Delta} \right]^G \simeq  \left \{ \begin{array}{cc}
 \oplus_{i > k/2}^{k-j} (S^l \Omega_X^1 \tens L^k)_{\Delta} & \mbox{if $l$ odd} \\ \mbox{} & \\
\oplus_{i \geq k/2}^{k-j} (S^l \Omega_X^1 \tens L^k)_{\Delta} & \mbox{if $l$ even} 
\end{array}
\right. $$
since, for the $G$-action, the term 
$(S^l\Omega_{X}\tens L^k)_{\Delta}$, indexed by $i=k/2$, is pure of parity $(-1)^l$. Hence the sheaf of invariants is a direct sum of 
$ k/2 -j $ copies of $(S^l \Omega_X^1 \tens L^k)_\Delta$
 if $l$ is odd and of 
$ k/2 -j + 1$ copies if $l$ is even. 
 We have the following cases.
\begin{enumerate} 
\item $l=2j-1$, hence odd. We have: $\displaystyle \big[\frac{k- l}{2} \big] = \frac{k}{2} - j
$;
\item $l = 2j-2$, even. We have $\displaystyle   \big[\frac{k-l}{2} \big] =  \big[\frac{k - 2j + 2}{2} \big] =\frac{k}{2} -j + 1$.
\end{enumerate}
Hence for $l = 2j-1, 2j-2$ $A^G$ is an isomorphism and  injective if $l < 2j-2$. 
\end{itemize}
In order to compute the kernel of the restriction, for $l \leq 2j-1$, we just have to compute the kernel of the map $$S_{j,l}^G = \bigoplus_{i \geq k/2}^{k-j} \mathcal{L}^{k-i, i}(-l\Delta)
 \rTo^{ \oplus_{i \geq k/2}^{k-j} d^l_{\Delta}}
\Big [ \bigoplus_{i  \geq j}^{k-j} (S^l \Omega_X^1 \tens L^k)_{\Delta} \Big] ^G \;.$$
In the case $k$ is odd, it is the kernel of the map: 
$$ \bigoplus_{i > k/2}^{k-j} \mathcal{L}^{k-i, i}(-l\Delta) \rTo ^{\oplus_{i \geq k/2}^{k-j} d^l_{\Delta}} \bigoplus_{i  > k/2}^{k-j} (S^l \Omega_X^1 \tens L^k)_{\Delta}$$and hence $S^G_{j,l+1}$. In the case $k$ is even, 
$$\Big[ \bigoplus_{i  \geq j}^{k-j} (S^l \Omega_X^1 \tens L^k)_{\Delta} \Big] ^G \simeq \left \{ \begin{array}{ll}  \displaystyle
\bigoplus_{i  > k/2}^{k-j} (S^l \Omega_X^1 \tens L^k)_{\Delta}  & \qquad \mbox{if $l$ is odd} \\
\displaystyle
\bigoplus_{i  \geq k/2}^{k-j} (S^l \Omega_X^1 \tens L^k)_{\Delta} & \qquad \mbox{if $l$ is even} 
\end{array}
\right.
$$
In the case $l$ even, the kernel   is easily 
$ \oplus_{i \geq k/2}^{k-j} \ker d^l_{\Delta} \simeq S^G_{j,l+1}$. 
Now in the case $l$ odd, $\mathcal{L}^{k/2,k/2}(-l \Delta)$ is isomorphic to 
$\mathcal{L}^{k/2,k/2}(-(l+1) \Delta)$, which is in the kernel of 
$[\oplus_{i \geq j}^{k-j} d^l_{\Delta} ]^G = \oplus_{i \geq k/2}^{k-j} d_{\Delta}^l$ and hence the kernel is again
$\oplus_{i > k/2}^{k-j} \ker d_{\Delta}^l \oplus \mathcal{L}^{k/2, k/2}(-(l+1)\Delta)  \simeq \oplus_{i \geq k/2}^{k-j} \mathcal{L}^{k-i, i}(-(l+1)\Delta) \simeq S^{G}_{j,l+1}$. 
\end{proof}As an immediate consequence of proposition \ref{pps: SjlG} we have
\begin{crl}\label{crl: bifiltrationVE}The natural map $\V^j \cap \E^l \rTo \gr^{\E}_l$ is surjective for $l \geq 2j-2$. 
\end{crl}
\begin{proof}We just have to notice that, if $l = 2j-2$ or $l= 2j-1$ the map $\oplus_{i \geq k/2}^{k-j} d^l_{\Delta}$ is surjective and $A^G$ is an isomorphism. The case $l \geq 2j-2$ follows because $\V^j$ is a decreasing filtration. 
\end{proof}
\begin{lemma}\label{lemma: bifiltration}Let $F$ be a coherent sheaf over an algebraic variety $Y$. Let $M^h \subseteq M^{h-1} \subseteq \cdots  \cdots \subseteq M^1 \subseteq M^0 = F$ and $N^k \subseteq N^{k-1} \subseteq \cdots  \subseteq N^1 \subseteq N^0 = F$, $h, k \in \mbb{N}$, two finite decreasing filtrations of $F$. The filtration $N^\bullet$ induces a finite decreasing filtration $N^\bullet \cap M^l$ on each coherent sheaf $M^l$, for $0 \leq l \leq h$.  Suppose that, for a certain $j$ and $l$, the natural map $$ N^{j+1} \cap M^l \rTo M^l / M^{l+1} \simeq \gr^{M}_l F $$is surjective; then, for all $i \leq j$,  
$$ \gr^{N}_i M^{l+1} \simeq \gr^{N}_i M^l \;.$$
\end{lemma}\begin{proof}
It is sufficient to prove the lemma for $i=j$,  since it is evident
that, if $N^{j+1} \cap M^l \rTo
\gr_l^M F$ is surjective, then for all $i \leq j$, then also $N^{i+1}
\cap M^l \rTo \gr_l^M F$ is surjective. For $i=j$, the lemma is a
consequence of the following commutative diagram, whose rows and columns are all
exact. 
\begin{diagram}
& & 0   &    & 0    &  & 0  &  & \\
& & \dTo&    & \dTo &  &\dTo&  & \\
0 &\rTo  &N^{j+1} \cap M^{l+1} & \rTo & N^j \cap M^{l+1} & \rTo &
\gr_j^N M^{l+1} & \rTo & 0 \\
& & \dTo&    & \dTo &  &\dTo&  & \\
0 &\rTo  &N^{j+1} \cap M^{l} & \rTo & N^j \cap M^{l} & \rTo &
\gr_j^N M^{l} & \rTo & 0 \\
& & \dTo&    & \dTo &  &\dTo&  & \\
0 &\rTo  &\gr_l^M F & \rTo & \gr_l^M F& \rTo &
0 & \rTo & 0 \\
& & \dTo&    & \dTo &  &\dTo&  & \\
& & 0   &    & 0    &  & 0  &  & \\
\end{diagram}
\end{proof}
\begin{lemma}\label{lemma: gradedE}Consider the bifiltration $\V^j \cap \E^l$ on
  $(S^k V_L)^G$. For all $l \geq 2i$ we have
 $$\gr_i^{\V} \E^l \simeq \gr_i^{\V} \E^{2i}  \;.$$
\end{lemma}
\begin{proof}
By corollary \ref{crl: bifiltrationVE} the map $\V^j \cap \E^l \rTo \gr^{\E}_l$ is surjective for $j \leq [ l/2] + 1$. Hence, by lemma 
\ref{lemma: bifiltration}, for all $i \leq [l/2]$ we have that $\gr^{\V}_i \E^{l+1} \simeq \gr^{\V}_i \E^l$. But this means that, for all $l \geq 2i$, $ \gr^{\V}_i \E^{l+1} \simeq \gr^{\V}_i \E^l $, which implies the statement. 
\end{proof}
\begin{pps}\label{pps: bifiltrationdue}Let $h, l \in \mbb{N}$ such that $ \displaystyle h \geq \big[ \frac{l +1}{2 } \big]$. Then 
$$ \V^h \cap \E^l \simeq \bigoplus_{i \geq k/2}^{ k- h} \mathcal{L}^{(i,k-i)}(-l \Delta)  \simeq S^G_{h, l}\;.$$
\end{pps}
\begin{proof}It is sufficient to prove that, if $l = 2j-1$ or $l = 2j$ and if $h \geq j$, we have
\begin{equation}\label{eq: VE} \V^h \cap \E^l \simeq \bigoplus_{i \geq k/2}^{ k-  h } \mathcal{L}^{(i,k-i)}(-l \Delta) \;.\end{equation}It's
 evident that it is sufficient to prove (\ref{eq: VE}) for $h = j$, since for higher $h$ it follows immediately from the definition of $\V^h$. We prove it by induction on $j$. For $j=0$ there is nothing to prove, since $l=0$ and (\ref{eq: VE}) is the definition of 
 $\V^0$. Suppose the result is true for $j \geq 0$. Then we have, in the notations of proposition \ref{pps: SjlG}
 \begin{gather*}
\V^j \cap \E^{2j-1} \simeq \bigoplus_{i \geq k/2}^{k-j} \mathcal{L}^{(i,
  k-i)}(-(2j-1) \Delta) = S^G_{j, 2j-1}\\
\V^j \cap \E^{2j} \simeq \bigoplus_{i \geq k/2}^{k-j} \mathcal{L}^{(i,
 k- i)}(-2j \Delta) = S^G_{j, 2j} \;.
\end{gather*}Then, by definition of $\V^{j+1}$ we have that 
$\V^{j+1} \cap \E^{2j-1} \simeq \bigoplus_{i \geq k/2}^{k-j-1} \mathcal{L}^{(i,
 k- i)}(-(2j-1) \Delta)=S^G_{j+1, 2j-1}$, and $
\V^{j+1} \cap \E^{2j} \simeq \bigoplus_{i \geq k/2}^{k-j-1} \mathcal{L}^{(i,
 k- i)}(-2j \Delta) =S^G_{j+1, 2j}$. Consider now the sequence
$$ 0 \rTo \V^{j+1} \cap \E^{2j+1} \rTo \V^{j+1} \cap \E^{2j} \rTo^{\D^l_L}
\gr_{2j}^\E (S^k V_L)^G \rTo 0 \;.$$It is left exact by definition, but, since $2j= 2(j+1)-2$, by corollary \ref{crl: bifiltrationVE}, 
it is exact also on the right. Hence, by proposition \ref{pps: SjlG} $$ \V^{j+1} \cap \E^{2j+1}  \simeq S^G_{j+1, 2j+1} \simeq 
\bigoplus_{i \geq k/2}^{k-j-1}\mathcal{L}^{(i,k-i)}(-(2j+1) \Delta) \;.$$Now, for $l = 2j+2$, by corollary \ref{crl: bifiltrationVE} the sequence 
$$ 0 \rTo \V^{j+1} \cap \E^{2j+2} \rTo \V^{j+1} \cap \E^{2j+1} \rTo
\gr_{2j+1}^\E (S^k V_L)^G \rTo 0 $$is exact, since $2j+1 = 2(j+1) -1$. Hence, by proposition \ref{pps: SjlG}, $$ \V^{j+1} \cap \E^{2j+2} = 
\ker \D^l_L \trest_{S^G_{j+1, 2j+1}} = S^G_{j+1, 2j+2} \simeq \bigoplus_{i \geq k/2}^{k-j-1}\mathcal{L}^{(i,k-i)}(-(2j+2) \Delta) \;.$$
\end{proof}
\begin{theorem}\label{thm: main2}
The finite decreasing filtration $\W^{[ \frac{k}{2}] } \tens \mc{D}_A \subseteq \cdots \W^1 \tens \mc{D}_A \subseteq \W^0  \tens \mc{D}_A \simeq \mu_* S^k L^{[2]} \tens \mc{D}_A$ of length $[k/2]$ on the  on the sheaf $\mu_* ( S^k L^{[2]} \tens \mc{D}_A)$  over the symmetric variety $S^2 X$
has graded sheaves $$\gr^{\W \tens \mc{D}_A}_j ( \mu_* S^k L^{[2]} \tens \mc{D}_A ) \simeq \mathcal{L}^{(k-j, j)}(-2j \Delta) \tens \mc{D}_A \qquad \quad \mbox{$j = 0, \dots, \displaystyle \big[ \frac{k}{2} \big]$ }\;.$$
\end{theorem}\begin{proof}It is sufficient to prove it for $A$ and $\mc{D}_A$ trivial.  By  proposition \ref{pps: bifiltrationdue} and by lemma \ref{lemma: gradedE} we immediately have  that for $l \geq 2j$, $$ \gr^\V_j \E^l \simeq \gr^\V_j \E^{2j} \simeq \mathcal{L}^{(k-j, j)}(-2j \Delta) \;.$$
Since $\E^{k-1} \simeq \mu_* S^k L^{[2]}$, this implies the statement of the theorem for $j<[k/2]$. For $j=[k/2]$ proposition \ref{pps: bifiltrationdue} implies directly that 
$$ \gr^{\V}_{[k/2]} \E^{k-1} \simeq \mathcal{L}^{(k-[k/2], [k/2])}(-2 [k/2] \Delta) \;.$$
\end{proof}

\subsection{The case $k=3$.}
Fix $k=3$. Over $S^nX$ we have: 
\begin{align}
\V^{3} = & \;\mathcal{L}^{3} \oplus \mathcal{L}^{2,1} \oplus \mathcal{L}^{1,1,1}  \label{eq: directsumdecomposition1} \\
\K^0(L) = & \; \K^0_{(2)} \oplus \K^0_{(1)(1)} \simeq v_* (L_\Delta^3 \boxtimes \FS_{S^{n-2}X})  \oplus v_*(L_\Delta^2 \boxtimes 
\mathcal{L}) \label{eq: directsumdecomposition2}\\
\K^1(L) = &\; \K^1_{(1)}  \simeq v_*((\Omega_X^1 \tens L^3)_\Delta \boxtimes \FS_{S^{n-2}X}) \notag
\end{align}The operator $\D^0_L: \V^3 \rTo \K^0(L)$, according to the direct sum decompositions (\ref{eq: directsumdecomposition1}) and (\ref{eq: directsumdecomposition2}) above, 
can be written as: 
\begin{equation}\label{eq: matrixD} \D^0_L =  \left( \begin{array}{ccc} \D^3_{(2)} & \D_{(2)}^{2,1} & 0 \\ 0
    & \D^{2,1}_{(1), (1)} & \D^{(1,1,1)}_{(1),
      (1)} \end{array}\right ) \;.\end{equation}
Remark that for $\mu= (2,1)$ and $n \geq 2$ the differential $d^1_\Delta$ in (\ref{eq: partialdifferentialsnx}) is a map
$d^1_\Delta: \mathcal{L}^{2,1}(-\Delta) \rTo \K^1_{(1)}$. 
\begin{lemma}\label{lemma: D1factors}The restriction of the operator $\D^1_{(1)} $ to $\V^{2,1} \cap \E^1$ factors through: 
 $$ \V^{2,1} \cap \E^1 \rTo \mathcal{L}^{2,1}(-\Delta) \rTo^{-3d^1_\Delta} \K^1_{(1)} \;,$$where the first map is induced by the  projection
 $\V^{2,1} \rTo \mathcal{L}^{2,1}$ onto a summand of the direct sum. 
\end{lemma}\begin{proof}
First of all, if $(x,y) \in \mathcal{L}^{2,1} \oplus \mathcal{L}^{1,1,1}= \V^{2,1}$ is in $\E^1$ this means 
that $\D^0_L(x,y)=0$ and this implies, by the form of the matrix (\ref{eq: matrixD}) that $\D^{2,1}_{(2)} x = 0$ and hence, by definition of $\D^{2,1}_{(2)}$, that $d^0_\Delta x =0$: therefore $x \in \mathcal{L}^{2,1}(-\Delta)$. As a consequence, it is sufficient to compare the two morphism in the statement locally; by lemma \ref{lmm: auxU} it is sufficient to compare them on an affine open set of the form $S^n U$, with $U=\Spec(R)$ an affine open set of $X$ over which $L$ is trivial. 

Over $S^n U = \Spec(S^n R)$, identifying sheaves with their module of sections, we have $\V^{2,1} \simeq (R \tens R \tens S^{n-2}R) \oplus 
(S^3 R \tens S^{n-3} R)$. Let $(x,y) \in \V^{2,1} \cap \E^1$. It is clear that $\D^1_L (x,y) $ just depends on $x$ and not on $y$. 
Let then $x$ be of the form $\sum_i a_i \tens b_i$, with $b_i \in S^{n-2}R$, and with $a_i = \sum_{ij} u_{ij} \tens v_{ij} \in I_\Delta \subseteq R \tens R$; we have: $\sum_j u_{ij}v_{ij} =0$ for all $i$. Hence $\sum_{j} ( v_{ij}d u_{ij} + u_{ij}d v_{ij}) =0$ in $\Omega^1_R$.  
Then, by the local formula in example \ref{ex: local3}, 
\begin{align*} \D^1_L (x,y) =  & \sum_i \Big[ -2 \sum_{j} u_{ij}dv_{ij} + \sum_{j} v_{ij} du_{ij} \Big] \tens b_i \\
 = & \: -3  \sum_{ij}u_{ij}dv_{ij} \tens b_i 
= -3 d_\Delta^1 x \;.
\end{align*}This proves the lemma.
 \end{proof}

 \begin{pps}\label{pps: shortsequences3}We have the following short left exact sequences over $S^n X$: 
\begin{gather}
0 \rTo \W^{2,1} \rTo \W^3  \rTo \mathcal{L}^3  \label{eq: 3grad}\\
0 \rTo \W^{1,1,1} \rTo \W^{2,1}  \rTo \mathcal{L}^{2,1}(-2\Delta) \label{eq: 2,1grad} \;.
\end{gather}Moreover the term $\W^{1,1,1}$ is isomorphic to 
$\W^{1,1,1} \simeq \mathcal{L}^{1,1,1}(-\Delta) \simeq \mathcal{L}^{1,1,1}(-2\Delta)$. \end{pps}
\begin{proof}
We start with the split short exact sequence: 
$ 0 \rTo \V^{2,1} \rTo \V^3 \rTo \mathcal{L}^3 \rTo 0 $ and we fit it into the diagram: 
\begin{diagram}
0 & \rTo & \V^{2,1}         &\rTo& \V^3    & \rTo &\mathcal{L}^3 & \rTo & 0   \\
 &           &  \dTo^{\D^0_L}&       & \dTo^{\D^0_L} &  & \dTo & & \\
0 & \rTo & \K^0(L) & \rTo^\sim & \K^0(L) & \rTo & 0  & \rTo & 0 
\end{diagram}Taking kernels yields the left exact sequence 
$ 0 \rTo \V^{2,1} \cap \E^1 \rTo \V^3 \cap \E^1 \rTo \mathcal{L}^3 $. We fit the latter into the diagram: 
\begin{diagram} 
0 & \rTo & \V^{2,1} \cap \E^1 & \rTo & \V^3 \cap \E^1 & \rTo & \mathcal{L}^3 \\
 &           &  \dTo^{\D^1_L} &       & \dTo^{\D^1_L} &  & \dTo^0 \\
 0 & \rTo & \K^1(L) & \rTo^\sim & \K^1(L) & \rTo & 0 
\end{diagram}and the kernels of the vertical maps give the left exact sequence
$ 0  \rTo  \V^{2,1} \cap \E^2  \rTo  \V^{3} \cap \E^{2}  \rTo  \mathcal{L}^3 $, which is precisely the first the statement. 

To get the second, consider the term $\V^{2,1}$. By definition we have the splitting short exact sequence; 
$ 0 \rTo \mathcal{L}^{1,1,1} \rTo \V^{2,1} \rTo \mathcal{L}^{2,1} \rTo 0 $. Moreover, the term $\K^0(L)$ splits as
$\K^0(L) \simeq  \K^0_{(2)} \oplus \K^0_{(1)(1)}$. Since the matrix of the restriction of 
$\D^0_L$ to $\V^{2,1}$ is lower triangular --- see (\ref{eq: matrixD})--- with respect to the direct sum decompositions  (\ref{eq: directsumdecomposition1}), (\ref{eq: directsumdecomposition2}), the following diagram is commutative: 
\begin{diagram}
0 & \rTo & \mathcal{L}^{1,1,1} & \rTo & \V^{2,1} & \rTo & \mathcal{L}^{2,1} & \rTo & 0 \\
&           &  \dTo^{\D^0_{(1)(1)}} &       & \dTo^{\D^0_L} &  & \dTo^{\D^0_{(2)}} \\
0 & \rTo & \K^0_{(1)(1)} & \rTo & \K^0(L) & \rTo & \K^0_{(2)} & \rTo &  0  \;.
\end{diagram}Now, by remark \ref{rmk: restriction1k}, the operator $\D^0_{(1)(1)}$ coincides with $d^0_\Delta$; 
moreover, the operator $\D^0_{(2)}$ coincides  as well with $d^0_\Delta$, as seen directly with the definitions. Hence
taking kernels for vertical maps in the above diagrams  and recalling example \ref{ex: 1k} and remark \ref{rmk: restriction1k}, we get the sequence
$ 0 \rTo \mathcal{L}^{1,1,1}(-2\Delta) \rTo \V^{2,1} \cap \E^1 \rTo \mathcal{L}^{2,1}(-\Delta) $. Now by lemma \ref{lemma: D1factors} the operator $\D^1_L$, restricted to $\V^{2,1} \cap \E^1$ factors through the differential $-3 d^1_\Delta : \mathcal{L}^{2,1}(-\Delta) {\rTo} \K^1_{(1)}$. Moreover the restriction of $\D^1_L$ to $\mathcal{L}^{1,1,1}(- 2 \Delta)$ is zero. Hence we get the commutative diagram: 
\begin{diagram}
0 & \rTo & \mathcal{L}^{1,1,1}(-2\Delta) & \rTo & \V^{2,1} \cap \E^1 & \rTo & \mathcal{L}^{2,1}(-\Delta) & &  \\
&           &  \dTo^{0} &       & \dTo^{\D^1_L} &  & \dTo^{-3 d^1_\Delta} & & \\
0 & \rTo &0 & \rTo & \K^1(L) & \rTo^\simeq & \K^1(L) & \rTo &  0 
\end{diagram}from which, taking kernels of vertical maps, by remark \ref{rmk: differential}, we get exactly the second left exact sequence in the statement. 
\end{proof}
\begin{notat}Let $R$ a commutative $\mbb{C}$-algebra. Consider elements $a_1, \dots, a_m$ of $R$. 
For all $I \subseteq \{1, \dots, m \}$, $I \neq \emptyset$, we denote with $\widehat{a}_I $ the element of $S^{m-|I|}R$ given by the symmetric product of $a_i$, with $i \not \in I$. 
\end{notat}
\begin{notat}\label{notat: mmu}Let $n , k \in \mbb{N}$, $n \geq 1$.  Consider $\mu \in p_n(k)$, that is, a partition of $k$ of length $l(\mu) \leq n$. 
Define the integer $m_\mu$ as $0$ if $l(\mu)=1$ and 
$$ m_\mu := \min_{2 \leq i \leq l(\mu)} \mu_i $$
if $l(\mu) \geq 2$. 
In other words $m_\mu = \mu_{l(\mu)}$ if $l(\mu) \geq 2$ and  $m_\mu = 0$ if $l(\mu) = 1$. 
\end{notat}
\begin{theorem}\label{thm: main3}Let $n \in \mbb{N}$, $n \geq 1$. The graded sheaves for the filtration 
$\mc{W}^\mu \tens \mc{D}_A$ on $\mu_* ( S^3 L^{[n]} \tens \mc{D}_A )$, indexed by partitions $\mu \in p_n(3)$, equipped with the reverse lexicographic order, are given by 
$$ \gr^{\W \tens \mc{D}_A }_{\mu} \simeq \mathcal{L}^\mu(-2 m_\mu \Delta) \tens \mc{D}_A \;.$$
\end{theorem}
\begin{proof}{\it Step 0.} We will prove the theorem for $n \geq 3$, since for $1 \leq n \leq 2$ the proof is analogous and easier; the result for $n =2$ is actually already provided by theorem \ref{thm: main2}. 
It is sufficient to prove the theorem for $A$, and hence $\mc{D}_A$, trivial. Moreover, it is sufficient to prove that the that the sequences in the statement proposition \ref{pps: shortsequences3} are right exact, that is, that the maps $\W^3 \rTo \mathcal{L}^3$ and $\W^{2,1} \rTo \mathcal{L}^{2,1}(-2\Delta)$ are surjective. 
Analogously to what done in the proof of lemma \ref{lemma: D1factors}, it is sufficient to prove the surjectivity of the same maps over an affine open set of the form $S^nU$, with $U=\Spec(R)$ an affine open set of $X$ over which $L$ is trivial.

{\it Step 1.} Over $S^n U= \Spec(S^n R)$ we can identify coherent sheaves with the modules of their sections; hence, as in example \ref{ex: local3}, $\mathcal{L}^3$ will be identified with $R \tens S^{n-1} R$, $\mathcal{L}^{2,1}$ with 
$R \tens R \tens S^{n-2} R$, and $\mathcal{L}^{1,1,1}$ with $S^{3} R \tens S^{n-3}R$. Moreover $\mathcal{L}^{2,1}(-2 \Delta)$ will be identified with $I_\Delta^2 \tens S^{n-2} R$, where $I_\Delta$ is the ideal of the diagonal in $R \tens R$. 
To prove that the map 
$\W^{2,1} \rTo \mathcal{L}^{2,1}(-2 \Delta)$ is surjective, it is sufficient to show that we can lift any element of the form $y =fg \tens a$
to $\W^{2,1}$, where $f, g \in I_\Delta$ and $a = a_1 \cdots a_{n-2} \in S^{n-2}R$, $a_i \in R$. Write $f = \sum_i s_i \tens t_i$ and $g = \sum_j u_i \tens v_i$, $s_i, t_i , u_j, v_j \in R$, such that $\sum_i s_i t_i =0 \in R$, $\sum_j u_j v_j = 0 \in R$. Then 
$fg= \sum_{ij}s_i u_j \tens t_i v_j$. The element $$x = \Big( 2 \sum_{ij} s_i u_j \tens t_i v_j \tens a, \sum_{ij h} s_i . t_i v_j . u_ja_h \tens \widehat{a_h} \Big)$$obviously projects to $2 y$ via the map $\V^{2,1} \rTo \mathcal{L}^{2,1} $. 
Let's prove that it is in $\W^{2,1}$. Indeed 
$$ \D^0_{(2)} x = 0 $$because $fg \trest_\Delta = 0$. Moreover, 
$$ \D^0_{(1)(1)} x = 2 \sum_{ijh} s_i u_j a_h \tens t_i v_j \tens \widehat{a_h} - 2 \sum_{ijh} s_i u_j a_h \tens t_i v_j \tens \widehat{a_h}$$since any other contraction of $\sum_{ij h} s_i . t_i v_j . u_ja_h$ is zero. Finally 
$$ \D^1_{(1)} x = -3 d^1_\Delta y = -3 d^1_\Delta (fg) \tens a = 0 $$since $fg \in I_\Delta^2$. 

{\it Step 2.} Let's now prove that the map $\W^3 \rTo \mathcal{L}^3$ is surjective. It is sufficient to prove that we can lift any element of the form $y \tens a \in R \tens S^{n-1}R$, where $y \in R$, $a \in S^{n-1}R$, to $\W^3$. Consider the element
$$ x= \Big(3 y \tens a, 2 \sum_i y \tens a_i \tens \widehat{a_i} + \sum_i a_i \tens y \tens \widehat{a_i}, \sum_{ij} [ y \tens a_i \tens a_j  + a_i \tens y \tens a_j + a_i \tens a_j \tens y ] \tens \widehat{a_{ij}} \Big ) \;.$$
The element $x$ projects to 
the element $3 y \tens a$ via the map $\V^3 \rTo \mathcal{L}^3$. We just need to show that $x \in \V^3 \cap \E^2$. 
We have: $$ \D^0_{(2)} x = [  \sum_i -3y a_i +2 y a_i +ya_i] \tens \widehat{a_i} =0 \;.$$
$$\D^0_{(1)(1)} x= -2 \sum_{ij} y a_j \tens a_i \tens \widehat{a_{ij}} - \sum_{ij} a_i a_j \tens y \tens \widehat{a_{ij}} +2 \sum_{ij} y a_j \tens a_i \tens \widehat{a_{ij}} +  \sum_{ij} a_i a_j \tens y \tens \widehat{a_{ij}} =0 \;.$$
Finally
$$\D^1_{(1)} x = 3 \sum_i y d a_i \tens \widehat{a_i} - 2 \sum_{i} [ 2 y da_i +a_i dy] \tens \widehat{a_i} + \sum_i [2 a_i dy + y da_i ] \tens \widehat{a_i} =0 \;.$$

\end{proof}

\subsection{The case $k=4$.}
Fix now $k=4$. Over $S^nX$ we have
\begin{align}  \V^4 = &\; \mathcal{L}^4 \oplus \mathcal{L}^{3,1} \oplus \mathcal{L}^{2,2} \oplus \mathcal{L}^{2,1,1} \oplus \mathcal{L}^{1,1,1,1}  \label{eq: V4}\\
\K^0(L)= & \;\K^0_{(3)} \oplus \K^0_{(2,1)} \oplus \K^0_{(2)(1)} \oplus \K^0_{(1)(2)} \oplus \K^0_{(1)(11)} \label{eq: K14} \\
\K^1(L)= & \;  \K^1_{(2)} \oplus \K^1_{(1)(1)} \notag \\
\K^2(L) = & \; \K^2_{(1)} \notag \;.
\end{align}The operator $\D^0_L$, according the direct sum decompositions (\ref{eq: V4}) and (\ref{eq: K14}) can be written as: 
\begin{equation}\label{eq: matrixD4} \left( \begin{array}{ccccc}
\D^4_{(3)} & \D^{3,1}_{(3)}     &        0                      &    0     & 0\\
0                 & \D^{3,1}_{(2,1)}  & \D^{2,2}_{(2,1)} & 0 & 0 \\
0                 & \D^{3,1}_{(2)(1)} & 0                          & \D^{2,1,1}_{(2)(1)} & 0 \\
0                 & 0                            & \D^{2,2}_{(1)(2)}&\D^{2,1,1}_{(1)(2)} & 0 \\
0                 &     0                        &                  0           &  \D^{2,1,1}_{(1)(11)}& \D^{1,1,1,1}_{(1)(11)}
\end{array}
\right)\end{equation}
We begin by some technical lemmas about the restriction of operators $\D^l_L$ to terms of the bifiltration of the kind $\V^\mu \cap \E^l$. 
\begin{lemma}\label{lemma: restrictionD14}The restriction of the operator $\D^1_{(2)}$ to $\V^{3,1} \cap \E^1$ factors through
$$ \V^{3,1} \cap \E^1 \rTo^r \mathcal{L}^{3,1}(-\Delta) \rTo^{-2 d^1_\Delta} \K^1_{(2)} \;,$$where the map $r$ is induced by the projection $\V^{3,1} \rTo \mathcal{L}^{3,1}$ onto a summand of the direct sum. 
Moreover, the restriction of the same operator 
to $\V^{2,2} \cap \E^1$ is zero. 
\end{lemma}\begin{proof}The proof is very similar to the proof of lemma \ref{lemma: D1factors}, so we will just sketch it. 
Write $\V^{3,1} = \mathcal{L}^{3,1} \oplus \mathcal{L}^{2,2} \oplus \V^{2,1,1}$. Take an element $(x,y, z) \in ( \mathcal{L}^{3,1} \oplus \mathcal{L}^{2,2} \oplus \V^{2,1,1} ) \cap \E^1$. 
Hence we have that $\D^0_L(x,y,z) = 0$ which implies, because of the form of the matrix \ref{eq: matrixD4}, $\D^{3,1}_{(3)} x =0$, which, as seen by definitions, implies in turn $d^0_\Delta x =0$ and hence that $x \in \mathcal{L}^{3,1}(-\Delta)$. 
Notice that $\K^0_{(3)} \simeq \K^0_{(2,1)} $. 
Moreover we  have
$\D^{3,1}_{(2,1)} x + \D^{2,2}_{(2,1)} y = 0$, which becomes $-d^0_\Delta x + d^0_\Delta y =0$, and hence implies $d^0_\Delta y=0$. Therefore $y \in \mathcal{L}^{2,2}(-\Delta) = \mathcal{L}^{2,2}(-2 \Delta)$. 
Now we just have to check that $\D^1_{(2)} \trest_{\V^{3,1} \cap \E^1}$ coincide locally with $-2 d^1_\Delta x  \circ r$; by lemma \ref{lmm: auxU} 
it is sufficient to do this over an open set of the form $S^nU$, with 
$U$ an affine open set of $X$ over which $L$ trivial. In this case, 
the local formula 
in example \ref{ex: local4} gives $$\D^1_{(2)}(x,y,z) = -2d^1_\Delta x + d^1_\Delta y = -2 d^1_\Delta x \;.$$
The statement for $\V^{2,2} \cap \E^1$ now follows taking $x=0$. 
\end{proof}
\begin{lemma}\label{lemma: restrictionD214}The restriction of the operator $\D^0_{(2)(1)} \oplus \D^0_{(1)(2)} $ to $\mathcal{L}^{2,1,1}$:
$$ \D^0_{(2)(1)} \oplus \D^0_{(1)(2)}: \mathcal{L}^{2,1,1} \rTo \K^{0}_{(2)(1)} \oplus \K^0_{(1)(2)} $$
 coincides with 
$d^0_\Delta$. Hence its kernel is $\mathcal{L}^{2,1,1}(-\Delta) \simeq \pi_* (L^{2,1,1} \tens I_{\Delta_{12}} \cap I_{\Delta_{13}}  \cap I_{\Delta_{23}}^2 )^{\Stab_G(2,1,1)}$. 
\end{lemma}\begin{proof}
The differential $d^0_\Delta$ is defined as 
$$d^0_\Delta : \pi_*(L^{2,1,1})^{\Stab_G(2,1,1)} \rTo
\pi_*(L^{2,1,1} \tens \FS_{X^n}/I_{\Delta_{12}})^{\perm(\{ 4,\dots, n \})} \oplus  \pi_*(L^{2,1,1} \tens \FS_{X^n}/I_{\Delta_{23}})^{\perm(\{ 2,3\}) \times 
 \perm(\{4,\dots, n\})} \;.$$The component 
 $\pi_*(L^{2,1,1})^{\Stab_G(2,1,1)} \rTo \pi_*(L^{2,1,1} \tens \FS_{X^n}/I_{\Delta_{12}})^{\perm(\{4,\dots, n\})} $ is given by 
 the composition $\pi_*(L^{2,1,1})^{\Stab_G(2,1,1)}  \rTo \pi_*(L^{2,1,1}) \rTo \pi_*(L^{2,1,1} \tens \FS_{X^n}/I_{\Delta_{12}})$. The second map is $\pi_*(d^0_{\Delta_2} \boxtimes \id) = \pi_*(D^0_{(2)} \boxtimes \id)$: hence the composition coincides with 
 $\D^0_{(2)(1)}$. Analogously, the component
 \small
 $$\pi_*(L^{2,1,1})^{\Stab_G(2,1,1)} \rTo  \pi_*(L^{2,1,1} \tens \FS_{X^n}/I_{\Delta_{23}})^{\perm(\{2,3\}) \times 
 \perm(\{4,\dots, n\})} \simeq \pi_*(L^{1,1,2} \tens \FS_{X^n}/I_{\Delta_{12}} )^{\perm(\{1,2\}) \times \perm(\{4, \dots, n\})}$$
 \normalsize
 coincides with $\D^0_{(1)(2)}$. 
 \end{proof}
\begin{lemma}\label{lemma: restrictionD114}The restriction of the operator $\D^1_{L}$ to
$\V^{2,1,1} \cap \E^1$ factors through 
$$ \V^{2,1,1} \cap \E^1 \rTo \mathcal{L}^{2,1,1}(-\Delta) \rTo \K^{1}_{(1)(1)} $$
where the first map is induced by the projection $\V^{2,1,1} \rTo \mathcal{L}^{2,1,1}$ onto a summand of the direct sum and 
where the second map is 
the only nonzero component \footnote{the other component $\mathcal{L}^{2,1,1}(-\Delta) \rTo \pi_*(L^{2,1,1} 
\tens I_{\Delta_{23}}/I_{\Delta_{23}}^2)^{\perm(\{2,3\}) \times \perm(\{4, \dots, n\})}$ is zero since 
$\perm(\{2, 3\})$ acts with a sign on $I_{\Delta_{23}}/I_{\Delta_{23}}^2$. } of $-3 d^1_\Delta$. 
\end{lemma}\begin{proof}By lemma \ref{lmm: auxU}, 
it is sufficient to prove the statement over  an open set of the form $S^nU$, where $U= \Spec(R)$ is an affine open set of $X$ over which $L$ is trivial. Identifying coherent sheaves with the modules of their sections, we have that, if $
x = a \tens b_1 b_2 \tens c \in R \tens S^2 R \tens S^{n-3} R$ is a global section of $\mathcal{L}^{2,1,1}(-\Delta)$, we have, by the local formula in example \ref{ex: local4}: 
\begin{align}  \D^1_{(1)(1)} x = & \; -2 [ a db_1 \tens b_2 + a db_2 \tens b_1 ] \tens c + [ b_1 da \tens b_2 + b_2 da \tens b_1 ] \tens c  \notag \\
= & \; [-2  a db_1 \tens b_2   -2 ab_2 \tens b_1] \tens c  
+[ b_1 da \tens b_2  + b_2 da \tens b_1  ] \tens c   \notag\\
= & \; -3 adb_1 \tens b_2 \tens c - 3 a db_2 \tens b_1 \tens c \label{eq: expression}
\end{align}since $ab_1 \tens b_2 + ab_2 \tens b_1 = 0$ and hence 
$$ 0 = d(ab_1) \tens b_2 + d(ab_2) \tens b_1 = (a db_1 \tens b_2 + a db_2 \tens b_1) + (b_1 da \tens b_2 + b_2 da \tens b_1 )
\;.$$
 Now it is immediate to see that 
(\ref{eq: expression}) is the expression of the only nonzero component of  $-3 d^1_\Delta$. 
\end{proof}
\begin{lemma}\label{lemma: restrictionD24}The restriction of the operator $\D^2_L$ to $\V^{2,2} \cap \E^2$ factors through
$$ \V^{2,2} \cap \E^2 \rTo \mathcal{L}^{2,2}(-2 \Delta) \rTo^{-3 d^2_\Delta} \K^2_{(1)}  \;,$$where the first map is induced by the projection $\V^{2,2} \rTo \mathcal{L}^{2,2}$ onto a summand of the direct sum. 
\end{lemma}
\begin{proof}By lemma \ref{lemma: restrictionD14} it is clear that if $u=(x, y, z)$ is a local section of $\mathcal{L}^{2,2} \oplus \mathcal{L}^{2,1,1} \oplus \mathcal{L}^{1,1,1,1} = \V^{2,2}$ and $u$ is in $\V^{2,2} \cap \E^2$, then $u \in \V^{2,2} \cap \E^1$ and 
$\D^0_{(2,1)} x =0$ which implies that $x$ is in $\mathcal{L}^{2,2}(-2 \Delta)$. Then we can compare the two morphisms of the statement; the comparison is done via the local formula for $\D^2_L$ for $X$ a complex affine surface and $L$  trivial --- see example \ref{ex: local4} --- and analogously to the proofs of lemmas \ref{lemma: D1factors}, \ref{lemma: restrictionD14} and \ref{lemma: restrictionD114}. 
\end{proof}

\begin{pps}\label{pps: shortsequences4}We have the following short left exact sequences on $S^n X$: 
\begin{gather}
0 \rTo \W^{3,1}  \rTo \W^4  \rTo \mathcal{L}^4 \label{eq: sequence4}\\
0 \rTo \W^{2,2} \rTo \W^{3,1}  \rTo \mathcal{L}^{3,1}(-2\Delta) \label{eq: sequence31}\\
0 \rTo \W^{2,1,1}  \rTo \W^{2,2}  \rTo \mathcal{L}^{2,2}(-4 \Delta) \label{eq: sequence22}\\
0 \rTo \W^{1,1,1,1}  \rTo \W^{2,1,1}  \rTo \mathcal{L}^{2,1,1}(-2\Delta) \label{eq: sequence211}\;.
\end{gather}Moreover the term $\W^{1,1,1,1}$ is isomorphic to $\mathcal{L}^{1,1,1,1}(- 2 \Delta)$. 
\end{pps}\begin{proof}To get the first left exact sequence, one proceeds analogously as in the case $k=3$, that is, 
taking recursively kernels of vertical maps in the diagrams
\begin{diagram} 0 & \rTo & \V^{3,1} \cap \E^l & \rTo & \V^4 \cap \E^l & \rTo & \mathcal{L}^4 \\
 &           &  \dTo^{ \D^l_L \trest_{\V^{3,1} \cap \E^l} } &       & \dTo^{\D^l_L} &  & \dTo^{0} & \\
0 & \rTo & \K^l(L) & \rTo^{\simeq} & \K^l(L) & \rTo & 0 &  
\end{diagram}for $l=0, 1, 2$. 

As for the second, consider the term $\V^{3,1}$: for convenience, we split the term 
$\K^0(L)$ as $\K^0_{\leq (2,1)}  \oplus \K^0_{(3)}  $, where $\K^0_{\leq (2,1)}  = \K^0_{(2,1)}  \oplus \K^0_{(2)(1)}  \oplus \K^0_{(1)(2)}  \oplus \K^0_{(1)(11)} $. The matrix of $\D^0_L$ with respect to the direct sum decompositions 
$\V^{3,1} = \mathcal{L}^{3,1} \oplus \V^{2,2}$ and $\K^0(L) = \K^0_{(3)}  \oplus \K^0_{\leq (2,1)} $ is lower triangular, as seen in
(\ref{eq: matrixD4}). Hence we can form the commutative diagram \sloppy
\begin{diagram}
0 & \rTo & \V^{2,2}  & \rTo & \V^{3,1} & \rTo & \mathcal{L}^{3,1} & \rTo & 0  \\
&           &  \dTo^{\D^0_L \trest_{\V^{2,2}}} &       & \dTo^{\D^0_L} &  & \dTo^{\D^{3,1}_{(3)}} &  \\
0 & \rTo & \K^0_{\leq (2,1)}  & \rTo & K^0(L) & \rTo & \K^0_{(3)} & \rTo & 0
\end{diagram}which yields, taking kernels of vertical maps and noting that $\D^{3,1}_{(3)}$ coincides with $d^0_\Delta$, the left exact sequence $0 \rTo \V^{2,2} \cap \E^1 \rTo \V^{3,1} \cap \E^1 \rTo \mathcal{L}^{3,1}(-\Delta) $. Lemma \ref{lemma: restrictionD14} immediately gives the diagram
\begin{diagram}
0 & \rTo & \V^{2,2} \cap \E^1 & \rTo & \V^{3,1} \cap \E^1 & \rTo & \mathcal{L}^{3,1}(-\Delta)  \\
 &           &  \dTo^{\D^1_{(1)(1)} }&       & \dTo^{\D^1_L} &  & \dTo^{-2 d^1_\Delta} &  \\
0 & \rTo & \K^1_{(1)(1)}  & \rTo & \K^1(L) & \rTo & \K^1_{(2)} & \rTo & 0
\end{diagram}from which, taking kernels of vertical maps, we get the left exact sequence 
$ 0 \rTo \V^{2,2} \cap \E^2 \rTo \V^{3,1} \cap \E^2 \rTo \mathcal{L}^{3,1}(-2 \Delta) $. Finally, the commutative diagram 
\begin{diagram}
0 & \rTo & \V^{2,2} \cap \E^2 & \rTo & \V^{3,1} \cap \E^2 & \rTo & \mathcal{L}^{3,1}(-2\Delta)  \\
 &           &  \dTo^{\D^2_L} &       & \dTo^{\D^2_L} &  & \dTo^{0} &  \\
0 & \rTo & \K^2(L)  & \rTo^\simeq & \K^2(L) & \rTo & 0& 
\end{diagram}yields, taking kernels of vertical maps, the second left exact sequence of the statement. 

Let's now work at the third sequence. We split the term $\K^0_{\leq (2,1)} $ as $
\K^0_{(2,1)}  \oplus \K^0_{\leq (2)(1)} $ where $\K^0_{\leq (2)(1)} = \K^0_{(2)(1)} \oplus \K^0_{(1)(2)}  \oplus \K^0_{(1)(11)} $. Since the matrix of the restriction of $\D^0_L$ to $\V^{2,2}$, with respect to the 
direct sum decompositions $\V^{2,2} = \mathcal{L}^{2,2} \oplus \V^{2,1,1}$ and $\K^0_{\leq (2,1)}  = \K^0_{(2,1)} \oplus \K^0_{\leq (2)(1)}  $ is lower triangular --- see (\ref{eq: matrixD4})~--- we can consider the commutative diagram
\begin{diagram}0 & \rTo & \V^{2,1,1}  & \rTo & \V^{2,2} & \rTo & \mathcal{L}^{2,2} & \rTo & 0  \\
 &           &  \dTo^{\D^0_L \trest_{\V^{2,1,1}} }&       & \dTo^{\D^0_L } &  & \dTo^{\D^0_{(2,1)}} &  & \\
0 & \rTo & \K^0_{\leq (2)(1)}  & \rTo & \K^0_{\leq (2,1)} & \rTo & \K^0_{(2,1)} & \rTo & 0 \;.
\end{diagram}Taking kernels of vertical maps and noting that the restriction of $\D^0_{(2,1)}$ to $\mathcal{L}^{2,2}$ coincides with $d^0_\Delta$, we get the left exact sequence $0 \rTo \V^{2,1,1} \cap \E^1  \rTo \V^{2,2} \cap \E^1 \rTo \mathcal{L}^{2,2}(-\Delta)$. Note that $\mathcal{L}^{2,2}(-\Delta) \simeq \mathcal{L}^{2,2}(-2 \Delta)$. Lemma \ref{lemma: restrictionD14} allows to build the commutative diagram
\begin{diagram}0 & \rTo & \V^{2,1,1} \cap \E^1 & \rTo & \V^{2,2} \cap \E^1& \rTo & \mathcal{L}^{2,2}(-2\Delta)  \\
 &           &  \dTo^{\D^1_{L} \trest_{\V^{2,1,1} \cap \E^1} }&       & \dTo^{\D^1_L \trest_{\V^{2,2} \cap \E^1}} &  & \dTo^{0}   \\
0 & \rTo & \K^1_{(1)(1)}  & \rTo^{\simeq} & \K^1_{(1)(1)} & \rTo & 0 
\end{diagram}from which we get, taking kernels, the exact sequence $0 \rTo \V^{2,1,1} \cap \E^2 \rTo \V^{2,2} \cap \E^2 \rTo \mathcal{L}^{2,2}(-2 \Delta)$. Finally, since the restriction of $\D^2_L$ to $\V^{2,1,1} \cap \E^2$ is zero, 
and thanks to lemma \ref{lemma: restrictionD24}, we get the commutative diagram
\begin{diagram}0 & \rTo & \V^{2,1,1} \cap \E^2 & \rTo & \V^{2,2} \cap \E^2& \rTo & \mathcal{L}^{2,2}(-2\Delta)  & & \\
 &           &  \dTo^0&       & \dTo^{\D^2_L \trest_{\V^{2,2} \cap \E^2}} &  & \dTo^{-3 d^2_{\Delta}}  & &  \\
0 & \rTo & 0 & \rTo & \K^2(L) & \rTo^\simeq & \K^2(L) &  \rTo & 0 \;.
\end{diagram}Taking kernels of vertical maps we get the third left exact sequence in the statement. 

For the fourth sequence in the statement, we can consider the commutative diagram
\begin{diagram}0 & \rTo & \mathcal{L}^{1,1,1, 1} & \rTo & \V^{2,1,1} & \rTo & \mathcal{L}^{2,1,1}  & \rTo & 0 \\
 &           &  \dTo^{\D^{1,1,1,1}_{(1)(11)}}&       & \dTo^{\D^0_L \trest_{\V^{2,1,1}}} &  & \dTo_{\D^{2,1,1}_{(2)(1)} \oplus \D^{2,1,1}_{(1)(2)}}  & &  \\
0 & \rTo & \K^0_{(1)(11)}  & \rTo & \K^0_{\leq (2)(1) } & \rTo & \K^0_{(2)(1)}  \oplus \K^0_{(1)(2)}  &  \rTo & 0 
\end{diagram}since the matrix of the restriction of $\D^0_L$ to $\V^{2,1,1}$, with respect to the direct sum decompositions $\V^{2,1,1} = \mathcal{L}^{2,1,1} \oplus \mathcal{L}^{1,1,1,1}$ and $\K^0_{\leq (2)(1)}  =  ( \K^0_{(2)(1)}  \oplus \K^0_{(1)(2)} ) \oplus \K^0_{(1)(11)} $ is lower triangular --- see (\ref{eq: matrixD4}). Taking kernels of vertical maps, we get, using lemma \ref{lemma: restrictionD214} and remark \ref{rmk: restriction1k}, the left exact sequence $0 \rTo \mathcal{L}^{1,1,1,1}(-\Delta) \rTo \V^{2,1,1} \cap \E^1 \rTo \mathcal{L}^{2,1,1}(-\Delta) $. Note now that the restriction of $\D^1_L$ to $\mathcal{L}^{1,1,1,1}(-\Delta)$ is zero. Hence, by lemma \ref{lemma: restrictionD114} we can form the commutative diagram
 \begin{diagram}0 & \rTo & \mathcal{L}^{1,1,1, 1}(-\Delta)  & \rTo & \V^{2,1,1} \cap \E^1 & \rTo & \mathcal{L}^{2,1,1}(-\Delta)  \\
 &           &  \dTo^{0}&       & \dTo^{\D^1_L \trest_{ \V^{2,1,1} \cap \E^1}} &  & \dTo^{-3 d^1_\Delta}  & &  \\
0 & \rTo & 0 & \rTo & \K^1_{(1)(1) }   & \rTo^\simeq &  \K^1_{(1)(1) }  &  \rTo & 0 \;.
\end{diagram}Taking kernels of vertical maps in the last diagram yields the fourth left exact sequence in the statement, since the restriction of $\D^2_{L}$ to $\V^{2,1,1} \cap \E^2$ is zero. 
\end{proof}
The following fact will be needed in the proof of the main theorem \ref{thm: main4}. 
\begin{lemma}\label{lemma: ideal4inva}Let $R$ be a commutative $\mbb{C}$-algebra. Let $I_\Delta$ the ideal of the diagonal in $R \tens R$.
Then in $S^2 R$ we have: $$( I_{\Delta}^4 )^{\perm_{2}} = (I_{\Delta}^2)^{\perm_2}(I_{\Delta}^2)^{\perm_2} \;, $$where the group $\perm_2$ operates on $R \tens R$ exchanging the factors. 
\end{lemma}
\begin{proof}An element in $I_{\Delta}^4$ can be written as a sum of products $xyzw$, were $x, y, z, w \in I_{\Delta}$. 
Let $x= \sum_i a_i \tens \alpha_i $, $y = \sum_j b_j \tens \beta_j$, $z = \sum_h c_h \tens \gamma_h $, $w = \sum_l d_l \tens \delta_l$. Then $xyzw = \sum_{ijhl} a_i b_j c_h d_l \tens \alpha_i \beta_j \gamma_h \delta_l $. We have that 
$$ xyzw = \sum_{ijhl}( a_i b_j c_h d_l \tens 1)(1\tens \alpha_i - \alpha_i \tens 1)(1 \tens \beta_j - \beta_j \tens 1)(1 \tens \gamma_h - \gamma_h \tens 1)(1 \tens \delta_l - \delta_l \tens 1) \;.$$
If $xyzw \in (I_{\Delta}^4)^{\perm_2}$ then $xyzw$ is equal to its invariant part, and hence
$$ xyzw = \frac{1}{2}   \sum_{ijhl}( a_i b_j c_h d_l . 1)(1\tens \alpha_i - \alpha_i \tens 1)(1 \tens \beta_j - \beta_j \tens 1)(1 \tens \gamma_h - \gamma_h \tens 1)(1 \tens \delta_l - \delta_l \tens 1)$$since 
any term of the form $1 \tens f - f \tens 1$ is anti-invariant and the term $( a_i b_j c_h d_l  . 1)$ is invariant. 
\end{proof}
\begin{remark}\label{rmk: ideal2antiinva}Similarly, one can easily prove that the anti-invariant 
part $(I^2_\Delta)^{-}$ of $I^2_\Delta$, that is, $(I_\Delta^2)^{-} := \{ f \in I_\Delta^2 \: | \: \sigma(f) = -f \}$ --- where $\sigma$ is the transposition $(12)$ --- is contained in $I^3_\Delta$. 
\end{remark}
Recalling notation \ref{notat: mmu}, we can now prove the following
\begin{theorem}\label{thm: main4}Let $n \in \mbb{N}$, $n \geq 1$. The graded sheaves for the filtration 
$\W^\mu\tens \mc{D}_A$ on $\mu_* ( S^4 L^{[n]} \tens \mc{D}_A )$,  indexed by partitions $\mu \in p_n(4)$,  equipped with the reverse lexicographic order, are given by 
$$ \gr^{\W \tens \mc{D}_A }_{\mu} \simeq \mathcal{L}^\mu(-2 m_\mu \Delta) \tens \mc{D}_A\;.$$ 
\end{theorem}
\begin{proof}{\it Step 0}. We will prove the theorem for $n \geq 4$, since for $1 \leq n \leq 3$ the proof is analogous and easier. 
It is sufficient to consider the case in which $A$ is trivial. Then, exactly as we did in step 0 of the proof of theorem \ref{thm: main3}, by lemma \ref{lmm: auxU}, we reduce the proof to showing 
the surjectivity of the maps 
on the right in (\ref{eq: sequence4}), (\ref{eq: sequence31}), (\ref{eq: sequence22}), (\ref{eq: sequence211}) 
over an open set of the form $S^n U = \Spec(S^n R)$, where 
$U = \Spec(R) $ is an affine open set of $X$ over which $L$ is trivial. Over $S^n U$ we can identify coherent sheaves with the modules of their sections; hence as in example \ref{ex: local4}, $\mathcal{L}^4$ will be identified with $R \tens S^{n-1} R$, $\mathcal{L}^{3,1}$ with 
$R \tens R \tens S^{n-2} R$, $\mathcal{L}^{2,2}$ with $S^2R \tens S^{n-2}R$, $\mathcal{L}^{2,1,1}$ with $R \tens S^2 R \tens S^{n-3}R$ and finally $\mathcal{L}^{1,1,1, 1}$ with $S^{4} R \tens S^{n-4}R$. Denoting with $I_{\Delta}$ the ideal of the diagonal in $R \tens R$, the sheaf $\mathcal{L}^{3,1}(-2 \Delta)$ will be identified with $I_\Delta^2 \tens S^{n-2} R$ and $\mathcal{L}^{2,2}(-4\Delta)$ with $(I_{\Delta}^4)^{\perm_2} \tens S^{n-2}R$. 

{\it Step 1: surjectivity of the map $\W^4 \rTo \mathcal{L}^4$}. It is sufficient to prove that we can lift any element of the form $y \tens a \in R \tens S^{n-1}R$ to $\W^4$, where $y \in R$  and  $a =a_1 \cdots a_{n-1} \in S^{n-1}R$, $a_i \in R$. Consider the element
\begin{multline*}x =\Big( 4 y \tens a,  \sum_i  3 y \tens 
a_i \tens \widehat{a_i} + a_i \tens y \tens \widehat{a_i}, \sum_i 2 y \tens a_i \tens \widehat{a_i} + 2 a_i \tens y \tens \widehat{a_i}, \\ \sum_{ij} 2 y \tens a_i \tens a_j \tens \widehat{a_{ij}} + a_i \tens y \tens a_j \tens \widehat{a_{ij}} + a_i \tens a_j \tens y \tens \widehat{a_{ij}}, \\ \sum_{ijk}y \tens a_i \tens a_j \tens a_k \tens \widehat{a_{ijk} } + a_i \tens y \tens a_j \tens a_k \tens \widehat{a_{ijk}} + a_i \tens a_j \tens y \tens a_k \tens \widehat{a_{ijk}} + a_i \tens a_j \tens a_k \tens y \tens \widehat{a_{ijk}} \Big) \end{multline*}Let's prove that this element is in $\W^4 = \V^4 \cap \E^3$. It is in $\E^1$ since
 \begin{align*} \D^0_{(3)} x =& \: \sum_i (-4 y a_i + 3 y a_i + y a_i) \tens \widehat{a_i} =0 \;. \\
  \D^0_{(2,1)} x = & \: \sum_i ( -3 y a_i - a_i y + 2 y a_i + 2 ya_i) \tens \widehat{a_i} =0 \\
  \D^0_{(2)(1)}x = & \: \sum_{ij} \big( -3 y a_j \tens a_i - a_i a_j \tens y + 2 y a_j \tens a_i + a_i a_j \tens y + a_i y \tens a_j) \tens \widehat{a_{ij}} = 0 \\
  \D^0_{(1)(2)} x = & \: \sum_{ij} ( -2 y a_j \tens a_i -  2 a_i a_j \tens y + 2 a_i a_j \tens y + y a_j \tens a_i + y a_j \tens a_i) \tens \widehat{a_{ij}} =0 \\
  \D^0_{(1)(11)} x= & \: \sum_{ijk} \big( -2 y a_k \tens a_i \tens a_j - a_i a_k \tens y \tens a_j - a_i a_k \tens a_j \tens y \\ & \hspace{1.5cm}  +2 y a_i \tens a_j \tens a_k +  a_i a_j \tens y \tens a_k + a_i a_j \tens a_k \tens y \big) \tens \widehat{a_{ijk}}  =0
\end{align*}  
To show that $x \in \E^2$, we compute: 
\begin{align*} \D^1_{(2)} x = & \: \sum_i \Big[ 4 y da_i - 2 (3 y da_i + a_i dy) + 2 y d a_i + 2 a_i d y \Big] \tens \widehat{a_i} = 0 \\ 
\D^1_{(1)(1)} x = & \: \sum_{ij} \Big[ 3 y da_j \tens a_i + a_i da_j \tens y - 2 (2 y d a_i \tens a_j + a_i d y \tens a_j+a_i d a_j \tens y ) + \\ & \hspace{1cm} 2 a_i d y \tens a_j + y d a_i \tens a_j + a_j d a_i \tens y \Big] \tens \widehat{a_{ij}} = 0 
\end{align*}Finally in $\K^2_{(1)}  $ we have: 
$$ \D^2_{(1)} x = \sum_{i} \Big[ -4 y d^2 a_i + 3 (3 y d^2 a_i + a_i d^2 y ) - 3(2 y d^2 a_i + 2 a_i d^2 y) +3 a_i d^2 y + y d^2 a_i \Big] \tens \widehat{a_i} = 0 \;.$$

{\it Step 2: surjectivity of the map $\W^{3,1} \rTo \mathcal{L}^{3,1}(-2\Delta)$.} Over the open set $S^nU = \Spec(S^nR)$ the sheaf $\mathcal{L}^{3,1}(-2 \Delta)$ can be identified with the module $(I_\Delta^2) \tens S^{n-2} R$. It is then sufficient to show that it is possible to lift to $\W^{3,1}$ elements of the form $y \tens a$, where $y \in I_\Delta^2 \subseteq R \tens R$ and $a = a_1 \dots a_{n-2} \in S^{n-2}R$. We can split $y$ in its invariant and anti-invariant part, for the action of $\perm_2$ exchanging the two factors in $R \tens R$: 
$ y = y^+ + y^-$, where $y^+ = (1/2)(y + \sigma(y))$, $y^- = (1/2)(y - \sigma(y))$, and where $\sigma$ is the transposition $(12)$. Hence it is sufficient to show that it is possible to lift to $\W^{3,1}$ each of the summands $y^+ \tens a , y^- \tens a$.

Let's begin with $y^+ \tens a$. The element $y^+ $ belongs to $(I^2_\Delta)^{\perm_2}$ and hence can be written as the invariant part of an element $2 \sum_i s_i \tens t_i \in I^2_\Delta$. Hence $y^+ = \sum_i s_i . t_j$, with 
$\sum_i s_i t_i =0$ and \begin{equation}\label{eq: sidti=0}\sum_i s_i dt_i = \sum_i t_i d s_i =0\end{equation}and hence 
$d_\Delta^1 y^+ = \sum_i s_i d t_i + t_i ds_i =0$, since $\sum_i s_i \tens t_i \in  I^2_\Delta$. Then the element 
$$ x = (6 y^+ \tens a, 8 y^+ \tens a, S, T) \in \V^{3,1}$$where 
\begin{align*}
S = & \; 4 \sum_{ij} s_i \tens 
t_i . a_j \tens \widehat{a_j} + 4 \sum_{ij} t_i \tens s_i . a_j \tens \widehat{a_j} + 2 \sum_{ij} a_j \tens s_i .t_i \tens \widehat{a_j}
\in R \tens S^2 R \tens S^{n-3}R\\
T = & \; \sum_{ijl} s_i . t_i . a_j . a_l \tens \widehat{a_{jl}} \in S^{4}R \tens S^{n-4}R
\end{align*}is in  $\W^{3,1}$ and lifts the element $6 y^+ \tens a$. Indeed $\D^0_{(3)} x =0 = \D^0_{(2,1)} x =0$ since $y^+ \in I_\Delta^2$. 
We have to check
\begin{align*}
\D^0_{(2)(1)} x = & \; - 6 \sum_{ij} ( s_i a_j \tens t_i  +  t_i a_j \tens s_i ) \tens \widehat{a_j}  \\&  \hspace{1.4cm}  +  \sum_{ij} ( 4 s_i a_j \tens t_i + 4 t_i a_j \tens s_j  + 2s_i a_j \tens t_j +2 t_i a_j \tens s_i )  \tens \widehat{a_j} =0 \;,
\end{align*}and 
\begin{align*}\D^0_{(1)(2)} x = & - 8 \sum_{ij} ( s_i a_j \tens t_i + t_i a_j \tens s_i ) \tens \widehat{a_j}   + 8 \sum_{ij} (  t_i a_j \tens s_i + s_i a_j \tens t_i ) \tens \widehat{a_j} = 0 \;.
\end{align*}As for the operator $\D^1_{(2)}$, we have immediately $\D^1_{(2)} x =0$, since $x \in \V^{3,1} \cap \E^1$
and since, by lemma  \ref{lemma: restrictionD14}, $$ \D^1_{(2)} x = -2 d^1_\Delta ( 6y^+ \tens a) = 0 \;,$$since $y^+ \tens a \in \mathcal{L}^{3,1}(-2 \Delta)$. We still have to check that: 
\begin{align*} \D^1_{(1)(1)} x = & \; 6 \sum_{ij} ( s_i da_j \tens t_i  + 
t_i d a_j  \tens s_i ) \tens \widehat{a_j}  \\ & \; - 2 
 \sum_{ij} ( 4 s_i d t_i \tens a_j  + 4 s_i da_j \tens t_i  + 4 t_i 
d s_i \tens a_j  \\ & \hspace{2cm}+ 4 t_i d a_j \tens s_i  +  2 a_j d s_i \tens t_i  + 2 a_j d t_i \tens s_i ) 
\tens \widehat{a_j} 
  \\ & \qquad + \sum_{ij} 
 \left ( 4 t_i d s_i \tens a_j + 4 a_j d s_i \tens t_i + 4 
s_i d t_i \tens a_j  \right. \\ & \hspace{2.2cm} \left. + 4 a_j d t_i \tens s_i + 
2 s_i d a_j \tens t_i + 2 t_i d a_j \tens s_i \right ) \tens \widehat{a_j} =0 \;,
\end{align*}because of (\ref{eq: sidti=0}). Finally the operator $\D^2_{(1)}$ vanishes on $x$, since 
\begin{align*} \D^2_{(1)}x =3 \cdot 6 \sum_{ij} (s_i d^2 t_i + t_i d^2 s_i ) \tens a 
 - 3 \cdot 8 \sum_{ij} (s_i d^2 t_i + t_i d^2 s_i ) \tens a 
+ 6 \sum_{ij} (s_i d^2 t_i + t_i d^2 s_i ) \tens a  =0 \;.
\end{align*}

Let's now lift the element $y^- \tens a$. By remark \ref{rmk: ideal2antiinva}, we can write $y^-$ as a product $\alpha \beta \gamma$, 
$\alpha, \beta, \gamma \in I_\Delta$. Writing $\alpha= \sum_i s_i \tens t_i$, $\beta = \sum_j u_j \tens v_j$, $\gamma = \sum_h z_h \tens w_h$, we have $y^- = \sum_{ijh} s_i u_j z_h \tens t_i v_j w_h$. Consider the element 
$$ x = (6 y^- \tens a, 0, S, T) \in \V^{3,1}$$where 
\begin{align*}
S = & \;  2 \sum_{ijhl} \big( s_i u_j  \tens t_i v_j w_h . z_h a_l + s_i z_h \tens t_i v_j w_h . u_j a_l + u_j z_h a_l \tens t_i v_j w_h . s_i  \big) \tens \widehat{a_l} \in R \tens S^2 R \tens S^{n-2} R\\
T =  & \; \sum_{ijhlm} s_i . t_i v_j  w_h . u_j a_l . z_h a_m \tens \widehat{a_{lm}}  \in S^4 R \tens S^{n-4}R \;.
\end{align*}Then $x \in \W^{3,1}$ and $x$ lifts the element $6 y^- \tens a$. Indeed we have that 
$\D^0_{(3)} x  = 0  = \D^0_{(2,1)} x  $ since $y^- \in I_\Delta^3$. Moreover 
\begin{align*} \D^0_{(2)(1)} x =  & \; - 6 \sum_{ijhl} s_i u_j z_h a_l \tens t_i v_j w_h \tens \widehat{a_l}  
+ 2 \sum_{ijhl} s_i u_j z_h a_l \tens t_i v_j w_h \tens \widehat{a_l} \\
& \; \hspace{1.4cm}+ 2 \sum_{ijhl} s_i u_j z_h a_l \tens t_i v_j w_h \tens \widehat{a_l} 
+ 2 \sum_{ijhl} s_i u_j z_h a_l \tens t_i v_j w_h \tens \widehat{a_l} 
  = 0 \;. \\  \D^0_{(1)(2)} x =  & \: 4   \sum_{ijl} \sum_h (z_h w_h) t_i v_j a_l \tens s_i u_j  \\ & \hspace{1cm} \;  + 4 \sum_{ihl} \sum_j (u_j v_j) t_i w_h a_l \tens s_i z_h 
   +   4 \sum_{jlh} \sum_i (s_i t_i) v_j w_h \tens u_j z_h  a_l  \tens \widehat{a_l} = 0 \\
 \D^0_{(1)(11)} x = & \, -2 \sum_{ijhlm} \Big ( 
s_i u_j a_m \tens t_i v_j w_h . z_h a_l + s_i z_h a_m \tens t_i v_j w_h . u_j a_l + u_j z_h a_l a_m \tens t_i v_j w_h . s_i  \Big) \tens \widehat{a_{lm}} \\ &\hspace{0.6cm}  +2 \sum_{ijlhm} \Big (  s_i u_j a_l \tens t_i v_j w_h . z_h a_m + s_i z_h a_m \tens t_i v_j w_h . u_j a_l +  u_j z_h a_l a_m \tens t_i v_j w_h . s_i 
\Big) \tens \widehat{a_{lm}} \\ = & \; 0  \;.
\end{align*}
As for the operator $\D^1_L$, the component $\D^1_{(2)} x$ is immediately zero by lemma \ref{lemma: restrictionD14}. 
For the component $\D^1_{(1)(1)} x $ we have: 
\begin{align*}\D^1_{(1)(1)} x = & \: 6 \sum_{ijhl} s_i u_j z_h d a_l \tens t_i v_j w_h \tens \widehat{a_l} - 4 \sum_{ijhl} \big ( s_i u_j d (t_i v_j w_h ) \tens z_h a_l + 
s_i u_j d(z_h a_l) \tens t_i v_j w_h \big ) \tens  \widehat{a_l} + \\
& \: \hspace{0.5cm}- 4 \sum_{ijhl} \big( s_i z_h d (t_i v_j w_h) \tens u_j a_l  + s_i z_h d(u_j a_l) \tens t_i v_j w_h \big) \tens \widehat{a_l} \\
& \: \hspace{0.5cm} - 4 \sum_{ijhl} \big(  u_j z_h a_l d (t_i v_j w_h) \tens s_i +  u_j z_h a_l d s_i \tens t_i v_j w_h \big ) \tens \widehat{a_l}  \\
& \: \hspace{0.5cm} + 2\sum_{ijhl} \big( 
t_i v_j w_h d(s_i u_j) \tens z_h a_l + 
z_h a_l d (s_i u_j) \tens t_i v_j w_h \tens \widehat{a_l }  \\ 
& \: \hspace{0.5cm} + 2\sum_{ijhl} \big( t_i v_j w_h d(s_i z_h) \tens u_j a_l + u_j a_l d (s_i z_h) \tens t_i v_j w_h ) \tens \widehat{a_l} \\
& \: \hspace{0.5cm} + 2\sum_{ijhl} \big( t_i v_j w_h d(u_j z_h a_l ) \tens s_i + 
s_i d (u_j z_h a_l) \tens t_i v_j w_h 
\big) \tens \widehat{a_l}  \\
 = & \: \sum_{ijhl} \Big( 
6 s_i u_j z_h d a_l - 4s_i u_j d(z_h a_l) -4s_i z_h d(u_j a_l) -4   u_j z_h a_l d s_i \\ & \hspace{1cm} +  2z_h a_l d (s_i u_j) + 2u_j a_l d (s_i z_h) + 2s_i d (u_j z_h a_l) \Big)
\tens t_i v_j w_h \tens \widehat{a_l}  =0 \;.
\end{align*}As for the operator $\D^2_{(1)}$, we have immediately that $\D^2_{(1)} x = 0$ since 
$$ \D^2_{(2)} x = 3 \cdot 6 \sum_{ijh} s_i u_j z_h d^2 (t_i v_j w_h) \tens a = 18 d^2_\Delta (y^- \tens a) = 0 $$since $y^-$ was chosen in
$I^3_\Delta$.

{\it Step 3: surjectivity of the map $\W^{2,2} \rTo \mathcal{L}^{2,2}(-4 \Delta)$.} Over the open set $S^nU = \Spec(S^nR)$ the sheaf $\mathcal{L}^{2,2}(-4 \Delta)$ can be identified with the module $(I_\Delta^4)^{\perm_2} \tens S^{n-2} R$;  by lemma \ref{lemma: ideal4inva}
$(I_\Delta^4)^{\perm_2} = (I_\Delta^2)^{\perm_2} (I_\Delta^2)^{\perm_2}$. 
It is hence sufficient to lift 
any element of the form $y_1 y_2 \tens a$, with $y_i \in (I_\Delta^2)^{\perm_2}$, $a =a_1 \dots a_{n-2} \in S^{n-2}R$, $a_i \in R$. Since they are $\perm_2$-invariants, we can write $y_1 = \sum_i s_i.t_i$, $y_2 = \sum_j u_j .v_j$. The product  $y_1 y_2$ will then be $y_1 y_2 = \sum_{ij} (s_i u_j . t_i v_j + s_i v_j . t_i u_j)$. Consider the element 
$$ x = ( 4y_1 y_2 \tens a, S, T) \;,$$where 
\begin{align*} S = & \: 2 \sum_{ijl} [s_i u_j \tens (a_l t_i). v_j + 
t_i u_j \tens (a_l s_i).v_j + s_i v_j \tens (a_l t_i). u_j + 
t_i v_j \tens (a_l s_i). u_j] \tens \widehat{a_l} \in R \tens S^2 R \tens S^{n-3}R \\
T = & \: \sum_{ijlh}  (a_h s_i). (a_l t_i). u_j.v_j \tens \widehat{a_{hl}} \in S^{4} R \tens S^{n-4}R \;.
\end{align*}Note that $S$ is in  $  [ I_{\Delta_{12}}^2 \cap I^2_{\Delta_{13}} ] ^{\perm(\{2,3\})} \tens S^{n-3}R$, since the starting elements $y_i$ have been chosen in $I_\Delta^2$. Then $x$ is in $\W^{2,2}$ and it lifts $4 y_1 y_2 \tens a$. Indeed $\D^0_{(3)} x =0 $  because $x \in \V^{2,2}$ and 
$\D^0_{(2,1)} x = -4 d^0_\Delta ( y_1y_2 \tens a )=0 $ since $y_1 y_2  \in I_\Delta^4 \subset I_\Delta^2$. Moreover
$\D^0_{(2)(1)} x = 0$ because it just depends on $S$ and $S$ is evidently in $ (I_{\Delta_{12}} \cap I_{\Delta_{13}})^{\perm(\{2,3\})} \tens S^{n-3} R$. Then
\begin{multline*}  \D^0_{(1)(2)} x = - 4 \sum_{ijl} [ s_i u_j a_l \tens t_i v_j
+ t_i v_j a_l \tens s_i u_j  +s_i v_j a_l
 \tens t_i u_j + t_i u_j a_l \tens s_i v_j ] \tens \widehat{a_l}   \\
+ 4 \sum_{ijl} [ a_l  t_i v_j \tens s_i u_j   + 
 a_l s_i v_j \tens t_i u_j   +  a_l t_i u_j \tens  s_i v_j  + 
a_l s_i u_j  \tens t_i v_j ] \tens \widehat{a_l}   = 0\end{multline*}and 
\begin{align*} \D^0_{(1)(11)} x = & - 2 \sum_{ijlh} \Big [ s_i u_j a_h \tens (a_l t_i). v_j + 
t_i u_j a_h \tens (a_l s_i).v_j  + s_i v_j a_h \tens (a_l t_i). u_j + 
t_i v_j a_h \tens (a_l s_i). u_j \Big] \tens \widehat{a_{hl}}  \\
& \hspace{0.3cm} +2 \sum_{ijlh} \Big [ s_i u_j a_h \tens (a_l t_i). v_j + s_i v_j a_h 
\tens  (a_l t_i) . u_j   + t_i u_j a_l \tens (a_h s_i).v_j + t_i v_j a_l \tens ( a_h s_i) . u_j \Big] \tens \widehat{a_{hl}} \\ =&\; 0 \;. 
\end{align*}We have $ \D^1_{(2)} x = 0$ by lemma \ref{lemma: restrictionD14}. Moreover
$$\D^1_{(1)(1)} x = -3 d^1_\Delta S = 0 \;$$by the same argument in the proof of lemma $\ref{lemma: restrictionD114}$, since $S$ is in $ [ I_{\Delta_{12}}^2 \cap I^2_{\Delta_{13}} ] ^{\perm(\{2,3\})} \tens S^{n-3}R$. Finally 
$$ \D^2_{(1)} x = -12 d^2_\Delta (y_1 y_2 \tens a) = 0 \;,$$by lemma \ref{lemma: restrictionD24},  since $y_1 y_2 \in I_\Delta^4$.

{\it Step 4: surjectivity of the map $\W^{2,1,1} \rTo \mathcal{L}^{2,1,1}(-2 \Delta)$.} Denote first of all with $c_{hl}$, for $1 \leq h < l \leq 3$ the contractions $c_{hl}: 
R \tens R \tens R \rTo R \tens R $ defined as $c_{hl} (b_1 \tens b_2 \tens b_3) = (b_h b_l) \tens \widehat{b_{hl}}$. 
The sheaf $\mathcal{L}^{2,1,1}(-2 \Delta)$ can be 
identified with the module $( I^2_{\Delta_{12}} \cap I^2_{\Delta_{13}} \cap I^2_{\Delta_{23}})^{\perm_{(\{2,3\})}} \tens S^{n-3} R$. The ideal $( I^2_{\Delta_{12}} \cap I^2_{\Delta_{13}} \cap I^2_{\Delta_{23}})^{\perm(\{2,3\})}$ is naturally seen in $R \tens S^2 R = (R \tens R \tens R)^{\perm(\{2,3\})}$. To prove the surjectivity of the map $\W^{2,1,1} \rTo \mathcal{L}^{2,1,1}(-2 \Delta)$ it is sufficient to show that it possible to lift to $\W^{2,1,1}$ elements of the form $y \tens a$, $y \in ( I^2_{\Delta_{12}} \cap I^2_{\Delta_{13}} \cap I^2_{\Delta_{23}})^{\perm(\{2,3\})}$, $a \in S^{n-3}R$. 
By remark \ref{rmk: bigdiagonal} and equation (\ref{eq: bigdiagonal}), for $s =2$, $m=3$,
we can write $y$ as the invariant sum $fg + \sigma(fg)$, 
of a product $fg$, where $f, g \in I_{\Delta_{3}} = I_{\Delta_{12}} \cap I_{\Delta_{13}} \cap I_{\Delta_{23}}$, and where $\sigma$ is the permutation exchanging the second and third factor. 
Hence, writing $f = \sum_i r_i \tens s_i \tens t_i$, $g = \sum_j u_i \tens v_i \tens z_j$, with 
\begin{equation}\label{eq: contrazioni} c_{hl} f = c_{hl} g = 0 
\end{equation} for all $1 \leq h < l \leq 3$, the element $y$ can be finally written as
$ y = \sum_{ij} r_i u_j \tens s_i v_j . t_i z_j $. We claim that the element 
$$ x = (2 y \tens a, \sum_{ijk} u_j. s_i v_j. t_i z_j.r_i a_k \tens \widehat{a_k} ) $$is in $\W^{2,1,1}$ and lifts 
the element $2 y \tens a$. Indeed the operators $\D^0_{(3)}$, $\D^0_{(2,1)}$ are trivially zero on $\V^{2,1,1}$; the operators $\D^0_{(2)(1)}$, $\D^0_{(1)(2)}$ and $\D^1_{(1)(1)}$ just involve $2 y \tens a$ 
and hence are zero by definition of $\mathcal{L}^{2,1,1}(- 2 \Delta)$. We just have to verify the vanishing of $\D^0_{(1)(11)} x$: 
$$ \D^0_{(1)(11)} x = -2 \sum_{ijk} r_i u_j a_k \tens s_i v_j . t_i z_j \tens \widehat{a_k} +2 \sum_{ijk} r_i u_j a_k \tens  s_i v_j . t_i z_j \tens \widehat{a_k} = 0$$which holds, since all the other contractions of $\sum_{ijk} u_j. s_i v_j. t_i z_j.r_i a_k$ are zero, because of (\ref{eq: contrazioni}). For example the contraction 
\begin{align*} 2 \sum_{ijk} u_j s_i v_j  \tens t_i z_j . r_i a_k = & \; 2 \sum_i (s_i \tens t_i \tens r_i) \sum_k (1 \tens 1 \tens a_k)\sum_j (u_j v_j \tens z_j \tens 1)  \\
& \quad \quad + 2 \sum_i (s_i \tens r_i \tens t_i) \sum_k (1 \tens a_k \tens 1) \sum_j( u_j v_j \tens 1 \tens z_j) =0 \end{align*}
which vanishes because of the vanishing of $c_{12}g = \sum_j u_j v_j \tens z_j $ in (\ref{eq: contrazioni}). 
The other terms are analougous. 
\end{proof}

\subsection{Toward the general case.}\label{toward}
The aim of this subsection is to explain what we expect about a possible 
filtration $\W^\mu \mu_*S^k L^{[n]}$ of the direct image $\mu_* S^k L^{[n]}$, for general $n$ and $k$, and its graded sheaves $\gr^\W_\mu$. 
We start with the direct sum decomposition 
$$ (S^k V_L)^G \simeq \bigoplus_{\lambda \in p_n(k)} \mc{L}^\lambda \;.$$In the previous subsections we used the reverse lexicographic order $\leq_{\rm rlex}$ to define a filtration $\V^\mu$ on the sheaf $(S^k V_L)^G$ and, consequently, a filtration $\W^\mu$ on $\E^{k-2}(n,k) \simeq \mu_* S^k L^{[n]}$.

The first remark is that, to hope to have a nice description in the general case, the reverse lexicographic order has to be replaced by another total order $\preceq$ on $p_n(k)$, which can be defined as follows. 
\begin{definition}\label{def: toward}For two partitions $\mu_1$, $\mu_2$ in $p_n(k)$ we write 
$ \mu_1 \preceq \mu_2$ if $l(\mu_1) \leq l(\mu_2)$ and, in case of equality, $\mu_1 \leq_{\rm rlex} \mu_2$. 
\end{definition}
\begin{remark}For $k\leq 5$ the total order $\preceq$ on $p_n(k)$ coincides with the reverse lexicographic order; as soon as $k \geq 6$, the two order are different. Indeed, we have
$$ (6) \prec (5,1) \prec (4,2) \prec (3,3) \prec (4,1,1) \prec (3,2,1) \prec (2,2,2) \prec (2,2,1,1) \prec (2,1,1,1,1) \prec (1^6) $$which is not an increasing sequence in the reverse lexicographic order, since $(3,3) \not \leq_{\rm rlex} (4,1,1)$.  
\end{remark}Thanks to the  total order $\preceq$ just defined, consider now a filtration $\V^\mu$ on $(S^k V_L)^G$ by setting 
$$  \V^\mu := \bigoplus_{\substack{\lambda \in p_n(k) \\ \lambda \succcurlyeq \mu }} \mathcal{L}^\lambda \;.$$As a consequence we have an induced finite decreasing filtration on any sheaf $\E^{l}(n,k)$ and hence an induced finite decreasing filtration 
$$ \W^\mu := \V^\mu \cap \E^{k-1}(n,k) $$of $\mu_* S^k L^{[n]}$. 

Secondly, in order to hope to give a general description of the graded sheaves for the filtration $\W^\bullet \mu_* S^k L^{[n]}$ for any $n$ and $k$, we propose the following definition. 
\begin{definition}\label{def: gradedgeneral}Let $\mu \in p_n(k)$. For any $m \in \mbb{N}$ define the sheaf $\mathcal{L}^\mu( - m \mu \Delta)$ over the symmetric variety $S^n X$ as 
$$ \mathcal{L}^\mu(-m \mu \Delta) := \pi_* \Big( L^\mu \tens \bigcap_{1 \leq i <j \leq l(\mu)} I_{\Delta_{ij}}^{m \mu_j} \Big)^{\Stab_G(\mu)} \;.$$
\end{definition}With this definition we can write down our hope for a general structure result about $\mu_* S^k L^{[n]}$: 
\begin{expect} For any $n, k \in \mbb{N}$, $n \geq 1$, the graded sheaves of the finite decreasing filtration 
$\W^\bullet$ of $\mu_* S^k L^{[n]}$ are given by: 
$$ \gr^\W_\mu \simeq \mathcal{L}^{\mu}(-2 \mu \Delta) \;.$$
\end{expect}
\begin{remark}Note that the graded sheaves $\gr^\W_\mu$ of the filtration $\W^\bullet$, considered in theorems  \ref{thm: main2}, \ref{thm: main3}, \ref{thm: main4} are all of the form $\mathcal{L}^\mu(-2 \mu \Delta)$: hence the above expectation is true for $k \leq 4$ or for $n \leq 2$. 
\end{remark}
\begin{remark}We worked out by hand several cases for $k \geq 5$ and we proved, in the studied examples, that over $S^nX$ \emph{there are left exact sequences
\begin{equation}\label{eq: leftsequencesgeneral} 0 \rTo \W^{\mu^\prime} \rTo \W^{\mu} \rTo \mathcal{L}^\mu(-2 \mu \Delta) \end{equation}
whose restriction to $S^n X_{**}$ is exact}. Here 
$\mu \prec \mu^\prime$ are consecutive partitions in $p_n(k)$, according to the total order $\preceq$.  The existence of such left exact sequences should be true for general $n$ and $k$: we plan to prove it in a forthcoming paper. 
In a future work we also plan to provide evidence that the sequences (\ref{eq: leftsequencesgeneral}) are indeed exact. \end{remark}
\begin{remark}As mentioned in the introduction, it is likely that the previous expectation holds for general tensor powers of tautological bundles or even in the case of general tensor products. Indeed, the filtration $E^\bullet(n,k)$ can be defined in an analogous way for general tensor products of tautological bundles via kernels of operators $\phi_l$ in Krug's work \cite{Krug2014}; the operators $\phi_l$ might be considered the $\perm_k$-equivariant version of operators $D^l_L$. 
As for the filtration $\mc{V}^\bullet$, there should be a natural way of extending it in the case of tensor product of tautological bundles, and maybe even in the case of general tensor products.
\end{remark}

\section{Cohomological consequences.}

\subsection{A spectral sequence for the cohomology.}
Thanks to the main theorems \ref{thm: main2}, \ref{thm: main3}, \ref{thm: main4} of the previous subsections, we can express the cohomology $H^*(X^{[n]}, S^k L^{[n]} \tens \mc{D}_A)$ of the symmetric powers of tautological bundles $S^k L^{[n]} \tens  \mc{D}_A$, twisted by the determinant line bundle $\mc{D}_A$, via the spectral sequence of the cohomology of a graded sheaf. Using definitions and notations from subsection \ref{toward}, we can formulate the results as follows. 
Let, first of all, $l \rMapsto \mu(l)$ be the only order preserving bijection between the totally ordered sets $(\{0, \dots, |p_n(k)|-1 \}, \leq )$ and $(p_n(k), \preceq)$. Then we have
\begin{crl}\label{crl: spectcohom}Let $n \in \mbb{N}^*$ and $k \in \mbb{N}$, such that $n \leq 2$ or $k \leq 4$. Then the cohomology $H^*(X^{[n]}, S^k L^{[n]} \tens~\mc{D}_A)$ is the abutment of the spectral sequence 
$$ E^{p,q}_1 := H^{p+q}(S^n X, \mc{L}^{\mu(p)}(-2 \mu(p) \Delta) \tens \mc{D}_A) \;.$$
\end{crl}
\begin{remark}\label{rmk: spectcohom}Note that, defining the sheaf $L^\mu(-m \mu \Delta)$ over $X^n$ as 
$$ L^\mu(- m \mu \Delta ) := L^\mu \tens \bigcap_{1 \leq i <j \leq l(\mu) } I_{\Delta_{ij}}^{m \mu_j}$$for $\mu$ a partition of $k$ 
of length less or equal than $n$
and $m \in \mbb{N}$, the previous spectral sequence can be rewritten as 
$$ E^{p,q}_1 \simeq H^{p+q}(X^n, L^{\mu(p)}(-2\mu(p) \Delta)  \tens A^{\boxtimes n})^{\Stab_{G}(\mu(p))} \;.$$The cohomology groups in question are the $\Stab_{G}(\mu)$-invariant parts of 
cohomology groups of the sheaves $L^{\mu}(-2\mu \Delta)  \tens A^{\boxtimes n}$, which are essentially (up to tensorization by line bundles) ideal sheaves over the smooth variety $X^n$. In other words the difficulty of understanding the cohomology 
of symmetric powers of tautological bundles has been reduced (at least for $k \leq 4$ or $n \leq 2$) to the difficulty of understanding ideal sheaves of the form $$\FS_{X^n}(-2\mu \Delta) := \bigcap_{1 \leq i <j \leq l(\mu)} I_{\Delta_{ij}}^{2 \mu_j}$$over the smooth variety $X^n$. 
After what we explained in subsection \ref{toward}, we believe that this might be done in general. 
\end{remark}

\subsection{Vanishing.}
As an application of theorems \ref{thm: main2}, \ref{thm: main3} and \ref{thm: main4} and of corollary \ref{crl: spectcohom} we give here, when $X$ is projective,  some \emph{effective vanishing results} about the cohomology of twisted symmetric powers of tautological bundles. 
Since $$H^p(X^{[n]}, S^k L^{[n]} \tens \mc{D}_A) \simeq H^p(S^nX, \mu_*(S^k L^{[n]}) \tens \mc{D}_A) $$by Leray spectral sequence, it is  
 indeed clear that, for $A$ sufficiently ample (or for $L$ and $A$ sufficiently ample), the higher cohomology $H^p(X^{[n]}, S^k L^{[n]} \tens \mc{D}_A)$ has to vanish, since the sheaf $A^{\boxtimes n}$ is ample on $X^n$ and hence $\mc{D}_A$ is ample on $S^nX$; this fact is even more evident in corollary \ref{crl: spectcohom} and remark \ref{rmk: spectcohom}.   Here, for $k \leq 4$ or $n \leq 2$, \emph{we will give an explicit  bound  for the positivity of $L$ and $A$ in order to achieve the vanishing}. 
We will need 
the notion
 of a $m$-very ample line bundle \cite{BeltramettiSommese1991}. 

We start by recalling that we denoted with $\Delta_n$ the big diagonal in $X^n$, that is, the closed subscheme of $X^n$ defined by the ideal 
$ \cap_{1 \leq i <j \leq n} I_{\Delta_{ij}}$ (see notation \ref{notat: diagm}). 

\begin{remark}Indicate now with $E$ the exceptional divisor (or the boundary) of $X^{[n]}$: it is the exceptional divisor for the Hilbert-Chow morphism and the branching divisor for the map $q: B^n \rTo X^{[n]}$. It is well known \cite[Lemma 3.7]{Lehn1999} that 
$$ \FS_{ X^{ [n] } } (-E) \simeq (\det \FS_{X}^{[n]} )^{\tens 2} \:. $$As a consequence there exists a divisor $e$ on the Hilbert scheme $X^{[n]}$ such that $E = 2e$, and such that $\FS_{X^{[n]}}(-e) \simeq \det \FS_X^{[n]}$. It is also well known that $\det L^{[n]} \simeq \mc{D}_L \tens \det \FS_X^{[n]}$, which can be rewritten, with the notations just explained, as \begin{equation} \label{eq: detL[n]}\det L^{[n]} \simeq \mc{D}_L(-e) \;.\end{equation}
\end{remark}
Denote now with $E_B$ the exceptional divisor over the isospectral Hilbert scheme, that is, the exceptional divisor 
for the map $p : B^n \rTo X^n$. We have $\FS_{B^n}(-E_B) \simeq q^*\FS_{X^{[n]}}(-e)$, and, as $G$-equivariant sheaves: \begin{equation}\label{eq: projectionformulaideali} \B{R}p_* \FS_{B^n}(-l E_B) = I^l_{\Delta_n} \tens \epsilon_G^{l }\;,\end{equation}where $\epsilon_G$ is the alternating representation of $G \simeq \perm_n$. 
Taking $G$-invariants we immediately get
\begin{equation}\label{eq: projectionformula[n]} 
\B{R}\mu_*\FS_{X^{[n]}}(-l e) = \pi^G_*(  I^l_{\Delta_n} \tens \epsilon_G^{l} ) \;.
\end{equation}The proof of (\ref{eq: projectionformulaideali}) appeared in \cite[Theorem 1.7]{Scaladiagarxiv2015}.

We now study the positivity of the line bundle $\det L^{[n]}$ on $X^{[n]}$  in terms of the positivity of $L$. The following example (communicated to us by Eyal Markman) proves that 
the fact that $L$ is ample (or even very ample) is not sufficient for $\det L^{[n]}$  to be big and nef.  
\begin{example}\label{ex: markman}Let $X$ be a complex projective K3 surface and let $L$  be a line bundle on $X$. Note that, in our notations, we have, at the cohomology level,  the isomorphism of lattices \cite{Beauville1983}
\begin{equation} \label{eq: 2cohom} H^2(X^{[n]}, \mbb{Z}) \simeq H^2(X, \mbb{Z}) \oplus^{\perp} \mbb{Z} e \;, \end{equation} 
when we equip the left hand side with the Beauville-Bogomolov form $q_n$ and the right hand side with the form $(q, 2-2n)$, where $q$ is the intersection form.  The isomorphism (\ref{eq: detL[n]}) implies precisely that the class $c_1(\det L^{[n]})$ corresponds to the couple $(c_1(L), -e)$ under (\ref{eq: 2cohom}). Therefore  $$q_n(  c_1(\det L^{[n]} ) )=  \int_X c_1(L)^2   + \: 2 - 2n\;.$$The Beauville-Bogomolov form  $q_n( c_1(\det L^{[n]}))$ is now linked to the self-intersection number of $\det L^{[n]}$ by the formula \cite{Beauville1983, Fujiki1987}
$$ \int_{X^{[n]}} c_1(\det L^{[n]} )^{2n} = \frac{2n}{n!2^n} q_n(c_1( \det L^{[n]}))^n \;.$$Consequently, even if $L$ is ample or very ample, but $n$ is a sufficiently big odd natural number, the integral at the left is 
negative; therefore $\det L^{[n]}$ can't be big and nef in this case. 
\end{example}

On the other hand 
the notion of $m$-very ample line bundle \cite{BeltramettiSommese1991, CataneseGoettsche1990} will be fundamental to prove the effective vanishing results. 
\begin{definition} Let $L$ be a line bundle on a smooth complex projective algebraic surface $X$. Then $L$ is called $m$-very ample 
if for all $\xi \in X^{[m+1]}$ the restriction map $\Gamma(L) \rTo \Gamma(L_\xi)$ is surjective. 
\end{definition}The concept of a $m$-very ampleness  generalizes the notion of very ampless. Indeed, with the terminology just introduced, very ampleness is precisely $1$-very ampleness. In the case the map $\Gamma(L) \rTo \Gamma(L_\xi)$ is surjective, let's indicate with $N_\xi$ its kernel. It is a subspace of $\Gamma(L)$ of dimension $h^0(L) - m-1$. 
\begin{remark}If $L$ is $m-1$-very ample, the assignment $\xi \rMapsto N_\xi$ defines a morphism 
$$ \phi_m : X^{[m]} \rTo \Grass(h^0(L) -m, \Gamma(L)) \;.$$Consider now the Pl\"ucker embedding
$ P:  \Grass(h^0(L) -m, \Gamma(L))   \rTo  \mbb{P}(\Lambda^{h^0(L) -m}\Gamma(L))=:\mbb{P}^M$, sending a subspace
 $W$  to $\Lambda^{h^0(L) - m} W$ and where $M = \binom{h^0(L)}{m} -1$. Denote with $\mc{N}$ the universal bundle on $\Grass(h^0(L) -m, \Gamma(L))$ and $\Q$ the universal quotient, that is, the sheaf $\Q$ fitting in the exact sequence
$$ 0 \rTo \mc{N} \rTo \Gamma(L) \tens \FS_{\Grass(h^0(L) -m, \Gamma(L))} \rTo \Q \rTo 0 \;.$$By construction, under the Pl\"ucker embedding $P$, we have $P^* \FS(-1) \simeq \det \mc{N}$, and hence $P^* \FS(1) \simeq \det \Q$. 
\end{remark}
\begin{lemma}If $L$ is $m-1$-very ample, we have that $\phi_m^* \Q \simeq L^{[m]}$. 
\end{lemma}\begin{proof}Indeed, consider the two projections $p_{X^{[m]}}$ and $p_X$ of the universal subscheme $\Xi \subseteq X^{[m]} \times X$ over $X^{[m]}$ and $X$, respectively. Tensorizing
the exact sequence 
$ 0 \rTo \mc{I}_\Xi \rTo \FS_{X^{[m]} \times X} \rTo \FS_{\Xi} \rTo 0 $ by $p_X^* L$ 
and pushing everything forward on $X^{[m]}$ via $p_{X^{[m]}}$ we get the exact sequence
\begin{equation}\label{eq: familysubspaces} 0 {\rTo} {p_{X^{[m]}}}_* (\mc{I}_\Xi \tens p_X^* L) \rTo \Gamma(L) \tens_{\mbb{C}} \FS_{X^{[m]}} \rTo L^{[m]} \rTo 0 \;.\end{equation}
The surjectivity on the right follows from  the fact that $L$ is $m-1$-very ample and by Nakayama lemma. The sequence 
(\ref{eq: familysubspaces}) provides a family of subspaces of $\Gamma(L)$ of the right dimension which corresponds to the embedding $\phi_m$. Therefore $\phi_m^* \mc{N} \simeq {p_{X^{[m]}}}_* (\mc{I}_\Xi \tens p_X^* L)$ and hence $\phi^*_m \Q \simeq L^{[m]}$. 
\end{proof}
\begin{crl}\label{crl: mample}If $L$ is $m-1$-very ample, then $\det L^{[m]}$  is nef. 
If $L$ is $m$-very ample, then $\det L^{[m]}$ is very ample. 
\end{crl}\begin{proof}If $L$ is $m-1$-very ample then the map $\phi_m$ is everywhere defined. 
Under the hypothesis, \begin{equation}\label{eq: relation} \det L^{[m]} \simeq \phi_m^* \det \Q = (P \circ \phi_m)^* \FS(1) \:. \end{equation}Hence  $\det L^{[n]}$ is the pull back of an ample line bundle by the morphism $P \circ \phi_m$, and consequently nef. 
If, additionaly,  $L$ is $m$-very ample, in particular it is $m-1$-very ample \cite{BeltramettiSommese1991}: hence (\ref{eq: relation}) holds. But if $L$ is $m$-very ample, the morphism $\phi_m$ is one to one \cite{BeltramettiSommese1991}, and actually a closed embedding \cite{CataneseGoettsche1990}: hence 
(\ref{eq: relation}) implies that $\det L^{[m]}$ is very ample.  
\end{proof}

\begin{notat}If $A$ is a line bundle on $X$, by abuse of notation we will still indicate the line bundle 
$q^* \mc{D}_A = p^* A^{\boxtimes n}$ on the isospectral Hilbert scheme $B^n$ with $\mc{D}_A$.
\end{notat}
\begin{pps}\label{pps: DLveryample}If $L  = \Tens_{i=1}^l B_i$, with $B_1$ $m$-very ample and $B_i$, $i=2, \dots, l$,  $(m-1)$-very ample line bundles on $X$, then $\mc{D}_L(-le)$ is ample on $X^{[m]}$. If all the $B_i$'s are $m$-very ample then $\mc{D}_L(-le)$ 
is $l$-very ample on $X^{[m]}$. 
\end{pps}
\begin{proof}Under the hypothesis, $\mc{D}_L(-le) = \Tens_{i=1}^l \mc{D}_{B_i}(-e)$; after corollary \ref{crl: mample}, the line bundle
$\mc{D}_{B_1}(-e) \simeq \det B_1^{[m]}$ is very ample on $X^{[m]}$,  while 
each of the $\mc{D}_{B_i}(-e) \simeq \det B_i^{[m]}$, $i = 2, \dots, l$  is nef. Hence $\mc{D}_L(-le)$ is ample on $X^{[m]}$, since it is the tensor product of an ample with a number of nef line bundles. If all the $B_i$'s are $m$-very ample, then all the $\det B_i^{[m]}$ are very ample and hence $\mc{D}_L(-le) $ is $l$-very ample since it is the tensor product of $l$ $1$-very ample 
line bundles \cite{BeltramettiSommese1991}. \end{proof}
\begin{crl}\label{crl: ampleB}If $L = \Tens_{i=1}^{l}B_i$, with $B_1$ $m$-very ample and $B_i$, $i=2, \dots, l$,  $(m-1)$-very ample line bundles on $X$,
 then $\mc{D}_L(-l E_B)$ 
is ample on the isospectral Hilbert scheme $B^m$. \end{crl}
\begin{proof}Since $\mc{D}_L (-l E_B) = q^* \mc{D}_L(-le)$, and since $q$ is finite, we have that $\mc{D}_L(-l E_B)$ is ample if  $\mc{D}_L(-le)$ is. Under the hypothesis, proposition \ref{pps: DLveryample} assures that this is the case. \end{proof}
\begin{pps}\label{pps: vanishingx[m]}
Let $X$ be a smooth projective surface. 
If $L \tens \omega_X^{-1} = \Tens_{i=1}^l B_i$, with $B_1$ $m$-very ample and $B_i$, $i =2, \dots, l$,  $(m-1)$-very ample line bundles on $X$, then $H^i(X^{[m]}, \mc{D}_L(-le) )=0$ for all $i>0$. 
\end{pps}
\begin{proof}Since the canonical line bundle $\omega_{X^{[m]}}$ of the Hilbert scheme of points $X^{[m]}$ is isomorphic to 
the line bundle $\mc{D}_{\omega_X}$,  we have $H^i(X^{[m]}, \mc{D}_L(-le) )=H^i(X^{[m]}, \mc{D}_{L \tens \omega_X^{-1}}(-le) \tens \omega_{X^{[m]}}) = 0$ for all $i>0$ by Kodaira vanishing theorem, since $\mc{D}_{L \tens \omega_X^{-1}}(-le)$ is  ample, in our hypothesis, by proposition \ref{pps: DLveryample}. \end{proof}
\begin{pps}\label{pps: vanishingmu}Let $X$ be a smooth projective surface. Let $\mu$ be  a partition of $k$ of length $l(\mu)=m$ and let $l \in \mbb{N}$.  Let $A$ and $L$ be line bundles on $X$ such that 
\begin{itemize} \item[i)] $L$ is nef \item[ii)] there exists an integer $r$, such that $1 \leq r \leq \mu_m$ and that $ L^{r} \tens A \tens \omega^{-1}_X = \Tens^{l+1}_{i=1} B_i$, where $B_1$ is $m$-very ample and $B_i$, $i = 2, \dots, l+1$,  are $m-1$-very ample line bundles on $X$.   \end{itemize}
Suppose moreover that the isospectral Hilbert scheme $B^m$ has log-canonical singularities. Then
\begin{enumerate} \item $H^i(X^m, L^\mu(- l \Delta) \tens A^{\boxtimes m}) = 0$ for all $i>0$; 
\item $H^i(S^m X, \mathcal{L}^\mu (-l \Delta) \tens \mc{D}_A) =0 $ for all $i>0$. 
\end{enumerate}
\end{pps}\begin{proof}Let's prove the first statement. Let $\mu^\prime$ the partition of $k-m r$ defined by 
$\mu^\prime_i = \mu_i -r$ for all $i=1, \dots, m$. Then 
$$ L^\mu(-l \Delta) \tens A^{\boxtimes m} = (L^{r} \tens A)^{\boxtimes m}(-l \Delta) \tens L^{\mu^\prime} \;,$$
where we simply wrote $(L^{r} \tens A)^{\boxtimes m}(-l \Delta)$ instead of $(L^r \tens A)^{1^m}(-l \Delta)$. 
By projection formula and by  (\ref{eq: projectionformulaideali}) we have, forgetting the $G$-action: 
$$ \B{R}p_*(\mc{D}_{L^{r} \tens A}(-l E_B) \tens p^* L^{\mu^\prime}) = L^\mu(-l \Delta) \tens A^{\boxtimes m} \;.$$Hence
we have: 
\begin{align*} H^i(X^m,  L^\mu(-l \Delta) \tens A^{\boxtimes m}) \simeq & \: H^i(B^m, \mc{D}_{L^{r} \tens A}(-l E_B) \tens p^* L^{\mu^\prime}) \\
\simeq & \: H^i(B^m, \mc{D}_{L^{r} \tens A \tens \omega^{-1}_X} (-l E_B) \tens p^* L^{\mu^\prime} \tens \mc{D}_{\omega_X} ) \\
\simeq & \:  H^i(B^m, \mc{D}_{L^{r} \tens A \tens \omega^{-1}_X} (-(l+1)E_B) \tens p^* L^{\mu^\prime} \tens \omega_{B^n} )
\end{align*}since $\omega_{B^n} = \mc{D}_{\omega_X}(E_B)$. Now the line bundle $\mc{D}_{L^{r} \tens A \tens \omega^{-1}_X} (-(l+1)E_B)$ is ample, under the hypothesis and after corollary \ref{crl: ampleB}; moreover the line bundle $p^* L^{\mu^\prime}$ is nef, since pull-back of a nef. Their tensor product is then ample  and we conclude by Kodaira vanishing 
on a variety with log-canonical singularities \cite[Theorem 4.4]{Fujino2009}. 

As for the second statement, it follows from lemma \ref{lemma: invariantsFDA} and from what we said above, noting that: 
$$ H^i(S^m X, \mc{L}^\mu(-l \Delta) \tens \mc{D}_A) \simeq  H^i ( S^m X, \pi_* (L^\mu(- l \Delta) \tens A^{\boxtimes m})) ^{\Stab_G(\mu)} \simeq H^i(X^m, L^\mu(- l \Delta) \tens A^{\boxtimes m})^{\Stab_G(\mu)} ,$$
because $\pi$ is a finite morphism. 
\end{proof}
\begin{remark}The isospectral Hilbert scheme $B^n$ is obviously smooth for $n=1,2$. We proved in \cite{Scala2015isospectral} that $B^n$ has canonical singularities for $n \leq 5$. However, it is not difficult to see that $B^n$ can't have canonical singularities for high $n$. Already for $n \geq 9$ its singularities  are not even log-canonical. \end{remark}

\begin{theorem}\label{thm: vanishing}Let $X$ be a smooth projective surface. Let $L$, $A$ be line bundles on $X$ such that $L$ is nef and $A \tens \omega_X^{-1}$ is big and nef. 
Then 
\begin{itemize}\item $ H^i(X^{[2]}, S^k L^{[2]} \tens \mc{D}_A) = 0$ \qquad for all $i >0$ \quad if $L \tens A \tens \omega^{-1}_X = \Tens_{j=1}^{k+1} B_j$; 
\item $H^i(X^{[n]}, S^3 L^{[n]} \tens \mc{D}_A) = 0$ \qquad for all $i>0$ \quad if $L \tens A \tens \omega^{-1}_X = \Tens_{j=1}^5 B_j$; 
\item $H^i(X^{[n]}, S^4 L^{[n]} \tens \mc{D}_A) = 0$ \qquad for all $i>0$ \quad if $L \tens A \tens \omega_X^{-1} = \Tens_{j=1}^7 B_j$\;,
\end{itemize}where $B_j$ are very ample line bundles on $X$. 
\end{theorem}\begin{proof}
By theorems \ref{thm: main2}, \ref{thm: main3}, \ref{thm: main4} and by the spectral sequence for the cohomology of a filtered sheaf, that is, corollary \ref{crl: spectcohom}, to prove the vanishing of the higher cohomology of $S^k L^{[n]} \tens \mc{D}_A$ for $k$ and $n$ as in the hypothesis, we just need to prove that 
$$ H^i (S^n X, \gr^{\W \tens \mc{D}_A}_\mu) = 0$$for $i>0$ and for any $\mu \in p_n(k)$. In notation \ref{notat: lDelta}, we can write 
$ \gr^{\W \tens \mc{D}_A}_\mu \simeq \mc{L}^\mu(-2 m_\mu \Delta)$, where $m$ is the length of the partition $\mu$. Then, by K\"unneth formula, we have 
$$ H^*(S^nX, \gr^{\W \tens \mc{D}_A}_\mu) \simeq H^*(S^mX, \mc{L}^\mu_m(-2 m_\mu \Delta) \tens \mc{D}_A) \tens H^*(S^{m-n}X, \mc{D}_A) \;.$$
where we have emphasized that $\mc{L}^\mu_m(-2 m_\mu \Delta)$ is now a sheaf on $S^m X$. The hypothesis that $A \tens \omega_X^{-1}$ is big and nef guarantees that the higher cohomology of $\mc{D}_A$ over $S^{n-m}X$ vanishes by Kawamata-Viehweg. On the other hand, for the vanishing 
$$ H^i(S^m X, \mc{L}^\mu_m(-2 m_\mu \Delta)\tens \mc{D}_A) =0 \qquad \forall i >0$$it is sufficient, if $\mu \neq (h, \dots, h)$, by proposition \ref{pps: vanishingmu}, the condition that $L$ is nef and that \\

$(*)$  $\exists r \in \mbb{N} \:, 1 \leq r \leq \mu_m$ such that  $L^r \tens A \tens \omega_X^{-1} = \Tens_{i=1}^{2 m_\mu +1} B_i$ with $B_1$ $m$-very ample and $B_i$ $(m-1)$-very ample for $i = 2, \dots, 2 m_\mu+1$; 

\vspace{0.3cm}
\noindent
If $\mu = (h, \dots, h)$, we can use that 
$$ H^*(S^m X, \mc{L}^\mu_m(-2 m_\mu \Delta) \tens \mc{D}_A) \simeq H^*(X^{[m]}, \mc{D}_{L^h \tens A \tens \omega_X^{-1}} (-2 h e) \tens \omega_{X^{[m]}} ) \;,$$by projection formula and by (\ref{eq: projectionformula[n]}), the fact that $L$ is nef  and the fact that, after proposition \ref{pps: vanishingx[m]}, condition $(*)$ can be replaced by condition \\

$(**)$ $\exists h^ \prime \in \mbb{N} \:, 1 \leq h^ \prime \leq h$, such that $L^{h^\prime} \tens A \tens \omega_X^{-1} = \Tens_{i=1}^{2 h} B_i$ with $B_1$ $m$-very ample and $B_i$ $(m-1)$-very ample for $i = 2, \dots, 2 h$. 

\vspace{0.3cm}
\noindent
It is now easy to see that, if $n =2$ or if $k \leq 4$, the hypothesis on $L \tens A \tens \omega_X^{-1}$ in the statement implies both conditions $(*)$ and $(**)$ for any $\mu \in p_n(k)$, for $n \leq 2$ or $k \leq 4$. 
\end{proof}
For $n=2$, we can prove a variant of theorem \ref{thm: vanishing} above in which we achieve the vanishing with a positivity bound on $L$ and $A$ independent on  $k$. 
\begin{theorem}\label{thm: vanishingvariant}Let $X$ be a smooth projective surface and $L$, $A$ be line bundles on $X$.  Then $$H^i(X^{[2]}, S^k L^{[2]} \tens \mc{D}_A) = 0 \qquad \mbox{ for all $i >0$}$$
if both $L$ and $A \tens \omega_X^{-1}$ are tensor product of two very ample line bundles. 
\end{theorem}\begin{proof}
After conditions $(*)$ and $(**)$ in the proof of theorem \ref{thm: vanishing}, for the vanishing of the graded sheaves 
$\gr_\mu^{\W \tens \mc{D}_A}$, with $\mu = (k-i, i)$ with $1 < i < k/2$,  it is sufficient that $L^i \tens A \tens \omega_X^{-1}$ is a product of $(2i +2)$ $1$-very ample line bundles; analogously, if $k = 2h$, for the vanishing of the graded sheaf $\mc{L}^{h, h}(-2h\Delta) \tens \mc{D}_A$ it is sufficient that $L^h \tens A \tens \omega_X^{-1}$ is a tensor product of  $(2h+1)$ $1$-very ample line bundles. These conditions are implied by the hypothesis that both $L$ and $A \tens \omega_X^{-1}$ are tensor products of two very ample line bundles. 
\end{proof}
\begin{remark}The same proof of theorem \ref{thm: vanishing}, together with remark \ref{rmk: k2}, yields the vanishing 
$H^i(X^{[n]}, S^2 L^{[n]} \tens \mc{D}_A)=0$ for  $i>0$, provided that $L$ is nef, that $A \tens \omega_X^{-1}$ is big and nef and that $L \tens A \tens \omega_X^{-1}$ is a tensor product of three $1$-very ample line bundles. 
However, using \cite[Theorem 5.1.6]{Scala2009D}, other hypothesis on $L$ and $A$ might be used to achieve the vanishing. 
\end{remark}
The preceding results can be rephrased in an easier way if the Picard number of the surface is one. For example theorem \ref{thm: vanishing} becomes: 
\begin{crl}Let $X$ be a smooth  projective surface with Picard group $\Pic(X) \simeq \mbb{Z} B$, where $B$ is the ample generator. Let $r$ be the minimum positive power of $B$ such that $B^r$ is very ample. Suppose, moreover, that $\omega_X \simeq B^w$, for some integer $w$. Let $l$ and $a$ be integers such that $l >0$ and $a >w$. Then, for $n = 2$ or $k \leq 4$ 
$$ H^i(X^{[n]}, S^k (B^l)^{[n]} \tens \mc{D}_{B^a} ) = 0 \qquad \mbox{for $i >0$} \quad \mbox{if $l+a \geq c$,} $$where the constant $c$ is: $c = r(k+1)+w$ if $n=2$; $c = 5r+w$ if $k=3$, $c=7r+w$ if $k=4$. 
\end{crl}
\begin{remark}
Of course over $X = \mbb{P}_2$ we can state more simply that if $l >0$, $a>-3$, 
then we have the vanishing $ H^i(\mbb{P}_2^{[n]}, S^k \FSP{2}(l)^{[n]} \tens \mc{D}_{\FSP{2}(a)}) = 0$ for $i>0$ if $l+a \geq c$, where $c =k-2$, $c=2$, $c=4$, if $n=2$, $k=3$, $k=4$, respectively. 
\end{remark}

\subsection{The Euler-Poincar\'e characteristic.}
For the sake of completeness, we give here formulas for the Euler-Poincar\'e characteristic of twisted symmetric powers $S^k L^{[n]} \tens \mc{D}_A$ of tautological bundles for $n=2$ and general $k$, or for $k=3,4$, and general $n$.

\begin{notat}For $l \in \mbb{N}$, denote with $\mc{K}^1_{(1)(1)}(- l \Delta)$ the sheaf over $S^n X$ defined as: 
$$ \mc{K}^1_{(1)(1)}(-l \Delta) = \pi_* \Big(\mc{K}^1_{((101), \{1,2 \})} \tens I^l_{\Delta_{23}} \Big)= \pi_*\Big[ ( (\Omega^1_X \tens L^3)_{\Delta_{12}}  \boxtimes L^{1} ) \tens I^l_{\Delta_{23}} \Big] \;.$$where we remember that $L^{1}$, according to notation \ref{notat: Lmu},  is the sheaf $L \boxtimes \FS_{X^{n-3}}$ on $X^{n-2}$. 
\end{notat}
\begin{notat}Let $l \in \mbb{N}$, $0 \leq l \leq n$. Denote with $v^l$ the map from the partial quotient to the total quotient
$v^l : S^l X \times S^{n-l} X \rTo S^n X$. 
\end{notat}The proof of the following lemma can be found in \cite[Prop. 4.17]{Scaladiagarxiv2015}. 
\begin{lemma}\label{lmm: exactL211}We have the exact sequence over $S^n X$: 
\begin{multline} 0 \rTo \mc{L}^{2,1,1}(-2 \Delta) \rTo \mc{L}^{2,1,1}(-\Delta) \rTo \mc{K}^1_{(1)(1)}(-2 \Delta) \rTo \\ \rTo v^3_*( (S^3 \Omega^1_X \tens L^4)_{\{123\}} \boxtimes \FS_{S^{n-3}X} ) \rTo 0 \;,\end{multline}where the middle map is the differential $d^1_{\Delta}$ and the fourth one is the composition: \small
\begin{equation*}\begin{split} \K^1_{(1)(1)} (-2 \Delta)  = \pi_*\Big[ ( (\Omega^1_X \tens L^3)_{\Delta_{12}}  \boxtimes L^{1} ) \tens I^2_{\Delta_{23}} \Big]^{ \perm(\{3, \dots, n\}) } \! \! \! \! \! \! \!\rTo \pi_*\Big[ ( (\Omega^1_X \tens L^3)_{\Delta_{12}}  \boxtimes L^{1} ) \tens I^2_{\Delta_{23}} / I^3_{\Delta_{23}} \Big]^{\perm(\{3, \dots, n\})} \simeq \\ 
\simeq v_*^3 \Big[  (\Omega^1_X \tens S^2 \Omega^1_X \tens L^4)_{\{123\}}  \boxtimes \FS_{S^{n-3}X } \Big] \rTo 
v^3_* \Big[  (S^3 \Omega^1_X \tens L^4)_{\{123\}}  \boxtimes \FS_{S^{n-3}X}  \Big] \;.\end{split}\end{equation*}\normalsize
\end{lemma}

\begin{notat}Let $M$ a line bundle on $X$. Denote with 
$\mc{A}_4(M)$ the sheaf on $S^4 X$ defined by: 
$$ \mc{A}_4(M):= [ \pi_* (M_{12} \boxtimes M_{34}) \tens \Res_{\Stab_{\perm_4}(\{ \{1,2\}, \{3,4\} \})} \epsilon_\Sigma ]^{\Stab_{\perm_4}(\{ \{ 1, 2 \}, \{3,4 \} \} ) } \;.$$where $\Sigma$ is the symmetric group $\perm(\{1, 2 \}, \{3,4 \}) \simeq \perm_2$ and where $\epsilon_\Sigma$ is its associated alternating representation. Here the restricted representation $ \Res_{\Stab_{\perm_4}(\{ \{1,2\}, \{3,4\} \})} \epsilon_\Sigma$ is induced by the natural homomorphism 
$ \Stab_{\perm_4}(\{ \{1,2\}, \{3,4\} \}) \rTo \Sigma$. 
The local sections of $ \mc{A}_4(M)$ over an open set of the form $S^4 U$, where $U$ is an open set of $X$, are $H^0(S^4 U,  \mc{A}_4(M)) = \Lambda^2 H^0(U, M)$. 
\end{notat}
\begin{lemma}\label{lemma: Ktheory}In the K-theory $K(S^n X)$ of the symmetric variety $S^n X$ we have: 
\begin{align*} 
[\mathcal{L}^{1,1,1}(-\Delta)] = & \: [\mathcal{L}^{1,1,1}] - [\mc{K}^0_{(1)(1)}] + [v^3_*( (\Omega^1_X \tens L^3)_{\{1, 2, 3 \}} \boxtimes \FS_{S^{n-3}X} ) ]  \\
[\mathcal{L}^{2,1,1}(-\Delta)] = & \: [\mathcal{L}^{2,1,1}] - [\mc{K}^0_{(2)(1)}] - [\mc{K}^0_{(1)(2)}] + [v^3_*((L^4)_{\{1,2,3\}} \boxtimes \FS_{S^{n-3}X}) ]+ [v^3_*( (\Omega^1_X \tens L^4)_{\{1, 2, 3 \}} \boxtimes \FS_{S^{n-3}X} ) ] \\
[\mathcal{L}^{1,1,1,1}(-\Delta)] = & \: [\mathcal{L}^{1,1,1,1}] - [\mc{K}^0_{(1)(11)}] + [v^4_*(\mc{A}_4(L^2) \boxtimes \FS_{S^{n-4} X})] + 
[v^4_*(  (\Omega^1_X \tens L^3)_{\{ 1,2,3 \}} \boxtimes \mathcal{L} )] \\ & \qquad  - [v^4_* (  (\Omega^1_X \tens L^4)_{\{ 1,2,3,4 \}}  \boxtimes \FS_{S^{n-4} X}) ] -[v^4_*( (K_X \tens L^4)_{\{ 1, 2, 3, 4 \}} \boxtimes \FS_{S^{n-4}X }  ) ] \\ & \qquad \qquad -  [v^4_*( (S^3 \Omega^1_X \tens L^4)_{\{ 1, 2, 3, 4 \}} \boxtimes \FS_{S^{n-4}X }  )]
\end{align*}
\end{lemma}The previous lemma is an immediate application of results appeared in \cite[Cor. 3.24, Cor. 4.16]{Scaladiagarxiv2015}. It is a consequence of the detailed analysis of the hyperderived spectral sequence associated to the $\perm_n$-invariant derived tensor product $\Big(\Tens^L_{| I | = 2} I_I \Big)^{\perm_n}$ of ideals of pairwise diagonals in $X^n$, for $n =3,4$: it is quite easy for $n=3$, but it becomes longer and more technical for $n=4$. 

In the following theorem, if $F$ is a coherent sheaf on a projective variety  $Y$, we will just denote with $\chi(F)$  the Euler-Poincar\'e characteristic of $F$ over $Y$. 

\begin{theorem}Let $n \in \mbb{N}$, $n \geq 1$. Let $X$ be a smooth projective surface. Let $L$ and $A$ be line bundles on $X$. 
Define $\delta_k$ as $1$ if $k$ is even, as $0$ if $k$ is odd. Then 
we have: 
\begin{align*} \chi(X^{[2]}, S^k L^{[2]} \tens \mc{D}_A) = & \: \delta_k \binom{\chi(L^{[k/2]} \tens A)+1}{2}+\sum_{i =0}^{[(k-1)/2]} \chi(L^{k-i} \tens A) \chi(L^{i} \tens A) \\
& \qquad \qquad \qquad \qquad \qquad \qquad - \sum_{j=0}^{k-2} \Big[ \frac{k-j }{2 } \Big] \chi(S^j \Omega^1_X \tens L^k \tens A^2) \\
\chi( X^{[n]}, S^3 L^{[n]} \tens \mc{D}_A)  =  & \; \binom{\chi(A) + n-2}{n-1}\chi(L^3 \tens A) \\ &  \; \; +\binom{\chi(A) + n-3}{n-2}\left[ \chi(L^2 \tens A) \chi(L \tens A) - \chi(L^3 \tens A^2) - \chi(\Omega^1_X \tens L^3 \tens A^2) \right] \\ & \; \; + 
\binom{\chi(A) + n-4}{n-3} \left[ 
\binom{\chi(L \tens A) + 2}{3} - \chi(L^2 \tens A^2) \chi(L \tens A) + \chi(\Omega^1_X \tens L^3 \tens A^3)
\right]
\end{align*}
\begin{align*} \chi(X^{[n]}, S^4 L^{[n]} \tens \mc{D}_A) = & \: \binom{\chi(A) + n-2}{n-1} \chi(L^4 \tens A) \\
& + \binom{\chi(A) + n-3}{n-2} \Bigg [  \chi(L^3 \tens A) \chi(L \tens A) - 2 \chi(L^4 \tens A^2) -\chi(\Omega^1_X \tens L^4 \tens A^2) + 
 \\ & \qquad \qquad   \qquad \binom{\chi(L^2 \tens A)+1}{2}  - \chi(S^2 \Omega^1_X \tens L^4 \tens A^2)
 \Bigg ] \\
& + \binom{\chi(A) + n-4}{n-3} \Bigg[  \chi(L^2 \tens A) \binom{\chi(L \tens A)+1}{2}- \chi(L^3 \tens A^2) \chi(L \tens A)
\\ & \qquad \qquad -\chi(L^2 \tens A^2)\chi(L^2 \tens A) + \chi(L^4 \tens A^3) -\chi(\Omega^1_X \tens L^3 \tens A^2) \chi(L \tens A) \\ & \qquad   
+ 2 \chi(\Omega^1_X \tens L^4 \tens A^3) + \chi(\Omega^1_X \tens \Omega^1_X \tens L^4 \tens A^3)
+  \chi(S^3 \Omega^1_X \tens L^4 \tens A^3)
\Bigg]  \\ 
&  + \binom{\chi(A) + n-5}{n-4} \Bigg[ \binom{\chi(L \tens A)+3}{4} - \chi(L^2 \tens A^2) \binom{\chi(L \tens A)+1}{2}  \\ & \qquad \qquad \qquad + \binom{\chi(L^2 \tens A^2)}{2}    + \chi(\Omega^1_X \tens L^3 \tens A^3)\chi(L \tens A) - \chi(\Omega^1_X \tens L^4 \tens A^4) \\ & \qquad \qquad \qquad - \chi(K_X \tens L^4 \tens A^4) - \chi(S^3 \Omega^1_X \tens L^4 \tens A^4) \Bigg] 
\end{align*}where we intend that a binomial coefficient $\displaystyle \binom{l}{h}$ is zero if $h <0$. 
\end{theorem}\begin{proof}We will just prove the formula for $k=4$ and $n \geq 4$. The proof for the other cases is similar and easier. 
First of all 
we have,
by Leray spectral sequence and by theorem \ref{thm: main4}: 
\begin{align*} \chi(X^{[n]}, S^4 L^{[n]} \tens \mc{D}_A) = & \sum_{\mu \in p_n(4)} \chi(S^n X, \gr^{\mc{W \tens \mc{D}_A}}_\mu ) =  \sum_{\mu \in p_n(4)} \chi(S^n X, \mathcal{L}^\mu (-2 m_\mu \Delta) \tens \mc{D}_A ) \\ = & \: \chi(\mathcal{L}^{1,1,1,1}(-2 \Delta) \tens \mc{D}_A) + \chi(\mathcal{L}^{2,1,1}(-2\Delta) \tens \mc{D}_A) \\ & \quad + \chi(\mathcal{L}^{2,2}(-4 \Delta) \tens \mc{D}_A)   + \chi(\mathcal{L}^{3,1}(-2\Delta) \tens \mc{D}_A )+ \chi(\mathcal{L}^4 \tens \mc{D}_A) \;.\end{align*}
In the computation of the Euler-Poincar\'e characteristic $\chi(\gr^{\W \tens \mc{D}_A}_\mu)$ of the graded pieces we will use repeatedly lemma \ref{lemma: invariantsFDA} and the fact 
that the Euler-Poincar\'e characteristic of the direct image of a sheaf via a finite morphism is equal to the Euler-Poincar\'e characteristic of the original sheaf. 
We have: 
$$ \chi(\mathcal{L}^4 \tens \mc{D}_A) = \chi(L^4 \tens A) \chi(S^{n-1}X, \mc{D}_A) = \chi(L^4 \tens A) \binom{\chi(A)+n-2}{n-1} \;.$$
Moreover, from the exact sequence 
$$ 0 \rTo \mathcal{L}^{3,1}(-2 \Delta) \rTo \mathcal{L}^{3,1}(-\Delta) \rTo v^2_* [ (\Omega^1_X \tens L^4)_{\{1, 2 \}} \boxtimes  \FS_{S^{n-2}X} ] \rTo 0 $$we get\begin{align*} \chi(\mathcal{L}^{3,1}(-2 \Delta) \tens \mc{D}_A) = & \: \chi(\mathcal{L}^{3,1}(-\Delta) \tens \mc{D}_A) - \chi(\Omega^1_X \tens L^4 \tens A^2 )\chi(S^{n-2}X, \mc{D}_A) \\
= & \: \chi(\mathcal{L}^{3,1} \tens \mc{D}_A) - \chi(L^4 \tens A^2) \chi(S^{n-2} X, \mc{D}_A) - \chi(\Omega^1_X \tens L^4 \tens A^2)\chi(S^{n-2}X, \mc{D}_A)  \\
= & \: \Big[ \chi(L^3 \tens A) \chi(L \tens A) - \chi(L^4 \tens A^2) -\chi(\Omega^1_X \tens L^4 \tens A^2) \Big ] \chi(S^{n-2}X, \mc{D}_A) 
 \end{align*}As for $\chi(\mathcal{L}^{2,2}(-4 \Delta) \tens \mc{D}_A)$ we use the exact sequences: 
 \begin{gather*} 0 \rTo \mathcal{L}^{2,2}(-4 \Delta) \rTo \mathcal{L}^{2,2}(-2 \Delta) \rTo v^2_* [ (S^2 \Omega^1_X \tens L^4)_{\{ 1, 2 \}}
 \boxtimes \FS_{S^{n-2}X} ] \rTo 0 \\
 0 \rTo \mathcal{L}^{2,2}(-2 \Delta) \rTo \mathcal{L}^{2,2} \rTo v_*^2[ L^4_{\{ 1, 2\}} \boxtimes \FS_{S^{n-2}X}]  \rTo 0
 \end{gather*}to compute: 
 \begin{align*}
 \chi(\mathcal{L}^{2,2}(-4 \Delta) \tens \mc{D}_A) = & \: \chi(\mathcal{L}^{2,2} \tens \mc{D}_A) - \chi(L^4 \tens A^2) \chi(S^{n-2}X, \mc{D}_A) - 
 \chi(S^2 \Omega^1_X \tens L^4 \tens A^2)  \chi(S^{n-2}X, \mc{D}_A)  \\ 
 = & \: \Big[ \binom{\chi(L^2 \tens A)+1}{2} - \chi(L^4 \tens A^2) - \chi(S^2 \Omega^1_X \tens L^4 \tens A^2)  \Big]  \chi(S^{n-2}X, \mc{D}_A) \;.
 \end{align*}
\end{proof}For $\chi(\mathcal{L}^{2,1,1}(-2 \Delta))$, we use lemma \ref{lmm: exactL211} and lemma \ref{lemma: Ktheory} to get: 
\begin{align*} \chi(\mathcal{L}^{2,1,1}(-2 \Delta) \tens \mc{D}_A) = & \: \chi(\mathcal{L}^{2,1,1}(-\Delta) \tens \mc{D}_A) - \chi(\mc{K}^1_{(1)(1)}(-2 \Delta)\tens \mc{D}_A) + 
\chi(S^3 \Omega^1_X \tens L^4 \tens A^3)\chi(S^{n-3}X, \mc{D}_A) \\
= & \: \chi(\mathcal{L}^{2,1,1} \tens \mc{D}_A) - \chi(\mc{K}^0_{(2)(1)} \tens \mc{D}_A) - \chi(\mc{K}^0_{(1)(2)} \tens \mc{D}_A) + \chi(L^4 \tens A^3) \chi(S^{n-3}X, \mc{D}_A)\\  &  + \chi(\Omega^1_X \tens L^4 \tens A^3) \chi(S^{n-3}X, \mc{D}_A) - 
\chi(\mc{K}^1_{(1)(1)}(- 2 \Delta) \tens \mc{D}_A) \\ & + 
\chi(S^3 \Omega^1_X \tens L^4 \tens A^3) \chi(S^{n-3}X, \mc{D}_A) \;.\end{align*}
Since \begin{gather*} 
\chi(\mathcal{L}^{2,1,1} \tens \mc{D}_A) = \chi(L^2 \tens A) \binom{\chi(L \tens A)+1}{2} \chi(S^{n-3}X, \mc{D}_A) \\
\chi(\mc{K}^0_{(2)(1)} \tens \mc{D}_A) = \chi(L^3 \tens A^2) \chi(L \tens A) \chi(S^{n-3}X, \mc{D}_A) \\
\chi(\mc{K}^0_{(1)(2)} \tens \mc{D}_A) = \chi(L^2 \tens A^2)\chi(L^2 \tens A)  \chi(S^{n-3}X, \mc{D}_A) 
\end{gather*}and since
\begin{align*}\chi(\mc{K}^1_{(1)(1)}(- 2 \Delta) \tens \mc{D}_A)  = & \: \chi(\mc{K}^1_{(1)(1)}(-  \Delta) \tens \mc{D}_A) - 
\chi(\Omega^1_X \tens \Omega^1_X \tens L^4 \tens A^3) \chi(S^{n-3}X, \mc{D}_A) \\  
= & \: \Big[ \chi(\Omega^1_X \tens L^3 \tens A^2) \chi(L \tens A)  - \chi(\Omega^1_X \tens L^4 \tens A^3) 
\\ & \qquad - \chi(\Omega^1_X \tens \Omega^1_X \tens L^4 \tens A^3) \Big] \chi(S^{n-3}X, \mc{D}_A)
\end{align*}

we finally get: 
\begin{multline*}\chi(\mathcal{L}^{2,1,1}(-2 \Delta) \tens \mc{D}_A) =\Big[  \chi(L^2 \tens A) \binom{\chi(L \tens A)+1}{2}- \chi(L^3 \tens A^2) \chi(L \tens A)
-\chi(L^2 \tens A^2)\chi(L^2 \tens A) \\ + \chi(L^4 \tens A^3) + 2 \chi(\Omega^1_X \tens L^4 \tens A^3)  
-\chi(\Omega^1_X \tens L^3 \tens A^2) \chi(L \tens A) \\  
+ \chi(\Omega^1_X \tens \Omega^1_X \tens L^4 \tens A^3) +  \chi(S^3 \Omega^1_X \tens L^4 \tens A^3)
\Big] \chi(S^{n-3}X, \mc{D}_A) \;.
\end{multline*}
As for $\chi(\mathcal{L}^{1,1,1,1}(-2\Delta) \tens \mc{D}_A)$, using lemma \ref{lemma: Ktheory} we compute first:
\begin{align*}
\chi(\mc{K}^0_{(1)(11)} \tens \mc{D}_A) = & \: \chi(L^2 \tens A^2) \binom{\chi(L \tens A)+1}{2} \chi(S^{n-4}X, \mc{D}_A) \\
\chi(v^4_*(\mc{A}_4(L^2) \boxtimes \FS_{S^{n-4} X}) \tens \mc{D}_A) = & \chi(S^4 X, \mc{A}_4(L^2 \tens A^2)) \chi(S^{n-4}X, \mc{D}_A) \\ = & \: \binom{\chi(L^2 \tens A^2)}{2}  \chi(S^{n-4}X, \mc{D}_A) \;.
\end{align*}Hence \begin{align*}
\chi(\mathcal{L}^{1,1,1,1}(-2 \Delta) \tens \mc{D}_A) = &  \: \Big[ \binom{\chi(L \tens A)+3}{4} - \chi(L^2 \tens A^2) \binom{\chi(L \tens A)+1}{2} + \binom{\chi(L^2 \tens A^2)}{2}  \\& \qquad + \chi(\Omega^1_X \tens L^3 \tens A^3)\chi(L \tens A) - \chi(\Omega^1_X \tens L^4 \tens A^4) - \chi(K_X \tens L^4 \tens A^4) \\ & \qquad - \chi(S^3 \Omega^1_X \tens L^4 \tens A^4) \Big] \chi(S^{n-4}X, \mc{D}_A)\;.
\end{align*}Putting all terms together we get the formula in the statement.

\newpage

\appendix 

\section{Higher differentials in the spectral sequence of a bicomplex}
We present here a brief review of how higher differentials in a spectral sequence are defined
and a couple of technical lemmas needed in the main text. 
 Let $(\comp{T}, d)$ a complex (in an abelian category) equipped with a (decreasing) filtration $F^{\bullet}
T$, preserved by the differential $d$, in the sense that $d F^p T^n
\subseteq F^p T^{n+1}$. We indicate the associated graded object with 
$\gr^p_F T^{n} := F^p T^{n} / F^{p+1} T^{n}$. The level zero of  the spectral sequence associated to $(T^\bullet,d)$ and
the filtration $F^{\bullet}$ is
defined as
$ E^{p,q}_0 := \gr^p_F T^{p+q} $.
The groups $Z^{p,q}_r$ and $B^{p,q}_r$ are defined as
$$ Z^{p,q}_r := \{ x \in F^p T^{p+q} \: | \: dx \in F^{p+r} T^{p+q+1}
\} \;,\qquad B^{p,q}_r := Z^{p+1, q-1}_{r-1} + d Z^{p-r+1, q+r-2}_{r-1}
 \;.$$
By definition one  has
$B^{p,q}_r \subseteq
Z^{p,q}_r$, since the two summands $Z^{p+1, q-1}_{r-1} $ and $d Z^{p-r+1, q+r-2}_{r-1}$ are in $ Z^{p,q}_r$. Note moreover
that $d$ defines a differential
$ d:  Z^{p,q}_r \rTo  Z^{p+r,q-r+1}_r $ by definitions of the terms 
$Z^{p,q}_r$. Now 
$d B^{p,q}_r \subseteq  B^{p+r,q-r+1}_r$ by definition of $B^{p,q}_r$,
and hence one can define the higher differential
$$ d^r := E^{p,q}_r := Z^{p,q}_r/  B^{p,q}_r \rTo Z^{p+r,q-r+1}_r/
B^{p+r,q-r+1}_r:=E^{p+r, q-r+1}_r \;.$$

Let us now apply these facts in the case of a bicomplex. Let $(K^{\bullet, \bullet}, \delta, \partial)$ be a double complex
with $\delta$ the vertical differential and $\partial$ the horizontal one. It is intended here that the differentials $\delta$ and
$\partial$ commute. Let moreover  
$T^n = \p_{p+q = n}K^{p,q}$ be the associated total complex, with
differential $d$, defined on the component $K^{p,q}$ as
$\delta + (-1)^p\partial$, equipped
with the natural filtrations
$$ F^p T^n = \bigoplus_{\substack{r + s =n \\
    r  \geq p}} K^{r,s} \; ,  \qquad  W^q T^n = \bigoplus_{
\substack{r + s =n \\
    s  \geq q}} K^{r,s}
$$
The total differential $d$ 
respects the filtrations. We have
$\gr^p_F T^{p+q} = F^p T^{p+q} / F^{p+1} T^{p+q} \simeq
  K^{p,q}$, and, analogously, $
\gr^q_W T^{p+q}  \simeq
K^{p,q}$. In the sequel, spectral sequences will always be associated to the filtration $F^{\bullet}$. 
In this case \begin{equation}\label{eq: zpq} Z^{p,q}_r = \Bigg \{ x \in \bigoplus_{\substack{u+ v =p+q \\
    u  \geq p}} K^{u, v} \: \Big | \: d x \in \bigoplus_{\substack{u+ v =p+q+1 \\
    u  \geq p+r}} K^{u, v} \Bigg \} \;.\end{equation}
\begin{notata}Let $(T^\bullet, d)$ the total complex of the bicomplex $(K^{\bullet, \bullet}, \delta, \partial)$. 
If $x \in T^{n}$  we will indicate with $x_s$ its component in $K^{s, n-s}$ and 
with $x_{\geq r}$ the sum of the components in $\oplus_{l \geq r} K^{l, n-l}$. Analogously for $x_{\leq r}$, $x_{> r}$, etc. 
\end{notata}

\begin{lemmaa}Let $(T^\bullet, d)$ the total complex of the bicomplex $(K^{\bullet, \bullet}, \delta, \partial)$, equipped with the filtration $F^\bullet$ described above. Suppose $r \geq 1$. 
If $[x] \in E^{p,q}_r$, $x \in Z^{p,q}_r$, 
then $d^{p,q}_r [x] = [\partial x_{p+r-1}]$ in  $E^{p+r,q-r+1}_r$. 
\end{lemmaa}\begin{proof}
By (\ref{eq: zpq}) we have that $Z^{p,q}_r$ is the 
set of $x \in F^p T^{p+q}$ such that $(dx)_{\leq p+r-1}=0$. If $x \in Z^{p,q}_r$ then
$d_r^{p,q} [x]$ is the class of $dx$ in
$E^{p+r,q-r+1}_r= Z^{p+r,q-r+1}_r/B^{p+r,q-r+1}_r$, that is the class
of $dx=(dx)_{\geq p+r}=\partial x_{p+r-1}+ d x_{p+r} + d x_{\geq p+r+1}$ in $Z^{p+r,q-r+1}_r/B^{p+r,q-r+1}_r$. 
We have that $B^{p+r,q-r+1}_r= d Z^{p+1,q-1}_{r-1} + Z^{p+r+1, q-r}_{r-1}$. 
Now it is immediate to prove that $dx_{\geq p+r+1}  \in Z^{p+r+1,
  q-r}_{r-1}$, since $d (d x_{\geq p+r+1})=0$; moreover $
d x_{p+r}$ is in $d Z^{p+1,q-1}_{r-1}$,
since  $x_{p+r} \in Z^{p+1,q-1}_{r-1}$. Hence
$d^{p,q}_r x = \partial x_{p+r-1}$  modulo $B^{p+r,q-r+1}_r$. \end{proof}
\begin{lemmaa}\label{lemma: highdiffspec}Let $(T^\bullet, d)$ the total complex of the bicomplex $(K^{\bullet, \bullet}, \delta, \partial)$, equipped with the filtration $F^\bullet$ described above. Suppose that $[x] \in E^{p,q}_r$ with $d^{p,q}_r[x] = 0$. Suppose that $r \geq 1$. 
Suppose moreover that there exists $w_{p+r} \in K^{p+r, q-r}$ such that $\partial x_{p+r-1}=\delta w_{p+r}$. Then
$d_{r+1}^{p,q}[x] = [\partial w_{p+r}]$ in $E^{p+r+1, q-r}_{r+1}$. 
\end{lemmaa}
\begin{proof}
If $d^{p,q}_r[x] = 0$, then $d x= \gamma + d \theta \in B^{p+r,q-r+1}_r= Z^{p+r+1, q-r}_{r-1} + d Z^{p+1,q-1}_{r-1} $, with $\theta \in Z^{p+1,q-1}_{r-1} \subseteq B^{p,q}_r$. Hence, setting $\alpha:= x - \theta$, we have that  $[\alpha] = [x] \in E^{p,q}_r$ and $\alpha \in Z^{p,q}_{r+1}$, since $\gamma \in F^{p+r+1}T^{p+q+1}$. 
Set now $w = w_{p+r}$ and consider  $\beta = x_{\leq p+r-1} + w$. It is clear from the hypothesis that $ \beta \in Z^{p,q}_{r+1}$ and hence $\alpha - \beta \in Z^{p,q}_{r+1}$. From the preceding lemma we have 
$d_{r+1}^{p,q} [\alpha] = [\partial x_{p+r}]$, $d_{r+1}^{p,q} [\beta] = [\partial w_{p+r}]$. Therefore it is sufficient to prove 
that $[\alpha]=[\beta]$ in $E^{p,q}_{r+1}$, that is, $\alpha - \beta \in B^{p,q}_{r+1}$. 
But $d(\alpha - \beta)  \in F^{p+r+1}T^{p+q+1}$ since $\alpha - \beta \in Z^{p,q}_{r+1}$. But, trivially, 
$\alpha -\beta \in F^{p+1}T^{p+q}$, hence $\alpha - \beta \in Z^{p+1, q-1}_{r} $, by definition of $Z^{p+1,q-1}_r$. But $Z^{p+1, q-1}_r \subseteq B^{p,q}_{r+1}$. 
\end{proof}
\begin{crla}\label{crla: highdiffspec}Let $(T^\bullet, d)$ the total complex of the bicomplex $(K^{\bullet, \bullet}, \delta, \partial)$, equipped with the filtration $F^\bullet$ described above. Suppose that $r \geq 1$ and that $[x] \in E^{p,q}_r$ with $d^{p,q}_r[x] = 0$. Suppose moreover that 
$E^{p+r, q-r-1}_r = E_1^{p+r, q-r-1}$. Then, via the differential $\delta$, we can lift the element $\partial x_{p+r-1}$ to an element $w_{p+r} \in K^{p+r, q-r}$, with $\delta w_{p+r}= \partial x_{p+r-1}$. Hence
one has $d^{p,q}_{r+1}[x]= [\partial w_{p+r}]$. 
\end{crla}

\section{Koszul complexes, multitors and invariants.} 
\label{section: appendixKoszul}

\begin{remarka}\label{rmka: symalt}Let $V$ be a finite dimensional vector space over $\mbb{C}$. Let $n \in \mbb{N}$, $n \geq 1$. The symetrization map $\sym$ is the canonical 
projection $\sym : V^{\tens q} \rTo S^q V$ sending $v_1 \tens \cdots \tens v_q \rMapsto (1/q!) \sum_{\sigma \in \perm_q}v_{\sigma(1)} \tens \cdots \tens v_{\sigma(q)}$, where the right hand side is an invariant element of $V^{\tens q}$ (for the action of $\perm_q$ permutating the factors) and hence seen in $S^q V$. 
Analogously, the antisymmetrization $\alt: V^{\tens q} \rTo \Lambda^q V$ is defined as $\alt: v_1 \tens \cdots \tens v_q \rMapsto (1/q!) \sum_{\sigma \in \perm_q} (-1)^{\sigma}v_{\sigma(1)} \tens \cdots \tens v_{\sigma(q)}$, where $(-1)^\sigma$ is the signature of the permutation $\sigma$. Then, following convention \ref{conv: sym} for the definition of symmetric and wedge products for elements of $V$, we have the formulas: 
\begin{gather*} \sym (a_1 . \cdots . a_{k} \tens b_1 . \cdots . b_{q-k}) = \frac{k!(q-k)!}{q!} a_1 . \cdots . a_{k} . b_1 \cdots . b_{q-k} \\
\alt(a_1 \w \cdots \w a_{k} \tens b_1 \w \cdots \w b_{q-k})  =
\frac{k!(q-k)!}{q!} a_1 \w \cdots \w a_{k} \w b_1 \cdots \w b_{q-k}
\end{gather*}
\end{remarka}

 \begin{notata}\label{notata: regularrepr}In what follows $R_k \simeq \mbb{C}^k$ will always denote the natural representation of~$\perm_k$, $\rho_k$ will denote the standard and $\epsilon_k$ the alternating one. The representation $R_k$ splits as $R_k \simeq \rho_k \oplus 1_k$, where $1_k$ is the trivial representation. 
 Let now $e_i$, $i = 1, \dots, k$ be  the standard basis of $\mbb{C}^k$. The vector $\sigma_k := \sum_{i=1}^k e_i \in R_k$ is invariant for the natural action of $\perm_k$ on $R_k$ and generates the trivial representation $1_k$. The representation $\Lambda^{k-1} \rho_k$ is $1$-dimensional, coincides with the
alternating representation of $\perm_k$ and is generated by the
element
$\omega_{k-1} = \sum_{i=1}^k (-1)^{k-i}\widehat{e}_i$ where $\widehat{e}_i =
e_1 \w \cdots \w e_{i-1} \w e_{i+1} \w \cdots \w e_k$. 
  \end{notata}
\begin{lemmaa}\label{lmm: antiinv}Let $V$ be a complex vector space of dimension $2$ and consider the $GL(V)\times \perm_k$-- representation $\Lambda^q(V \tens \rho_k)$. It has 
$\perm_k$-anti-invariants if and only if
$q=k-1$: in this case they are isomorphic, as
$GL(V)$-representation, to $S^{k-1}V$. \end{lemmaa}
\begin{proof}
  The dimension of the space of anti-invariants $
[\Lambda^q(V \tens \rho_k) \tens \epsilon_k]^{\perm_k}$ is 
        given by the scalar product 
	$\langle
\chi_{\Lambda^q(V \tens \rho_k)\tens \epsilon_k}, \chi_1 \rangle $ between the character of the representation $\Lambda^q(V \tens \rho_k) \tens \epsilon_k$ and the character of the trivial one. 
With a proof analogous to \cite[Lemma B.5]{Scala2009D}, it is easily computed to be $$ 
\dim 
[\Lambda^q(V \tens \rho_k) \tens \epsilon_k]^{\perm_k} =
\langle
\chi_{\Lambda^q(V \tens \rho_k)\tens \epsilon_k}, \chi_1 \rangle = \left \{ \begin{array}{lc} k & 
\mbox{if $q = k-1$}\\ 0 & \mbox{if $q 
\neq k-1$.} \end{array} \right.
$$
Recall now the isomorphism of $GL(V) \times \perm_k$-representations~(see \cite{FultonHarrisRT}, ex. 6.11): 
	$$ \Lambda^q(V \tens \rho_k) \simeq \bigoplus_{\lambda} S^{\lambda}V \tens 
	S^{\lambda^{\prime}}\rho_k $$where $\lambda^{\prime}$ denotes the conjugate partition to $\lambda$ and 
where
the sum is on the partitions  of $q$
	having at most $\dim V$ rows and $\dim \rho_k$ columns. The anti-invariants are: \begin{equation*} 
		[\Lambda^q(V \tens \rho_k)\tens \epsilon_k]^{ \perm_k}   \simeq \bigoplus_{\lambda} S^{\lambda}V \tens [
	S^{\lambda^{\prime}}\rho_k \tens \epsilon_k]^{ \perm_k}
        \;.\end{equation*}
We know that the anti-invariants are zero if $q 
\neq k-1$. If $q = k-1$,  the anti-invariants contain the summand indexed by the partition $\lambda = (k-1)$ of $k-1$:
$$S^{k-1}V \tens [\Lambda^{k-1} \rho_k \tens \epsilon_k]^{\perm_k} \simeq S^{k-1}V
\tens [\epsilon_k^2]^{\perm_k} \simeq S^{k-1}V \;.$$
But this summand has dimension $k$, therefore there are no more invariants.
\end{proof}Let now $Y$ be a smooth subvariety of a smooth variety $X$ and let $i: Y \rInto X$ be the closed immersion. Denote with $N_{Y/X}$ the normal bundle
of $Y$ in $X$.
Consider a line bundle $L$ on $X$ and let
$L_Y$ be the restriction of  $L$ to the subvariety $Y$. The symmetric group $\perm_l$ acts on the $l$-multitor $\Tor_{q}(L_Y, \dots, L_Y)$ by permutation of the factors; it turns out that, as a $\perm_l$-sheaf
\cite[Lemma B.3]{Scala2009D} $$\Tor_{q}(\underset{l-{\rm times}}{\underbrace{L_Y, \dots, L_Y}}) \simeq \Lambda^q (N_{Y/X}^* \tens \rho_l) \tens L_Y^{\tens l}\;.$$
\begin{crla}\label{crl: torantiinv}If $\codim Y =2$ the 
$l$-multitor
$\Tor_{q}(L_Y, \dots, L_Y)$  has nontrivial $\perm_l$-anti-invariants if and only if $q = l-1$. In this case the anti-invariants 
are given by $ (S^{l-1} N^*_{Y/X}) \tens L_Y^{\tens l}$. 
\end{crla}

\begin{remarka}\label{rmka: tensorKoszul}
Let now $F$ be a rank $2$-vector bundle on the smooth variety $X$ and let $s$ a section of $F$ transverse to the zero section. Consider the Koszul complex $K^\bullet := K^\bullet (F, s)$.  The symmetric group $\perm_k$ acts on the $k$-fold tensor product $K^\bullet(F,s) \tens \dots \tens K^\bullet(F,s)$ by permutation of the factors. Note that 
$ K^\bullet(F,s) \tens \dots \tens K^\bullet(F,s)$ is naturally isomorphic, as a $\perm_k$-equivariant complex, to the Koszul complex 
$K^\bullet (F \tens R_k, s \tens \sigma_k)$. The latter is isomorphic to the tensor product $K^\bullet(F, s) \tens K^\bullet (F \tens \rho_k, 0)$. 
\end{remarka}
By lemma \ref{lmm: antiinv} we can then deduce: 
\begin{crla}\label{crla: invkoszul}The $\perm_k$-anti-invariants of the $k$-fold tensor product $K^\bullet(F,s) \tens \dots \tens K^\bullet(F,s)$ are given by the complex: 
$$ [K^\bullet(F,s) \tens \dots \tens K^\bullet(F,s) \tens \epsilon_k ]^{\perm_k} \simeq [ K^\bullet (F \tens R_k, s \tens \sigma_k) \tens \epsilon_k]^{\perm_k} \simeq K^\bullet(F, s) \tens S^{k-1}F [k-1] \;.$$
\end{crla}

\begin{crla}\label{cola: ext}Let $V$ be a complex vector space of dimension $2$.
Then the $\perm_k$-representation $\Lambda^q(V \tens R_k)$ has
anti-invariants if and only if $k-1 \leq q \leq k+1$. In these cases the anti-invariants are:
$$ [\Lambda^q(V \tens R_k) \tens \epsilon_k]^{\perm_k} = \left \{
\begin{array}{ll}
S^{k-1}V  & \; \mbox{if $q=k-1$} \\
S^{k-1} V \tens V &\;  \mbox{if $q=k$} \\
S^{k-1} V \tens \Lambda^2 V &\;  \mbox{if $q=k+1$} 
\end{array} \right.
$$
\end{crla}

\begin{remarka}\label{rmka: inclusion}Note that, as in the proof of \ref{lmm: antiinv}, for $k-1 \leq q \leq k+1$,  one has a natural inclusion 
$$ S^{k-1} V \tens \Lambda^{q-k+1} V \rInto \Lambda^q(V \tens R_k) \tens \epsilon_k$$which induces 
an isomorphism at the level of $\perm_k$-invariants (here $\perm_k$ acts trivially on the left hand side). 
More precisely, the inclusion is the composition: 
\begin{multline*} S^{k-1} V \tens \Lambda^{q-k+1} V {\rTo^{\simeq}} [ S^{k-1} V \tens \Lambda^{k-1} \rho_k \tens \epsilon_k ] \tens [ \Lambda^{q-k+1} V \tens 1_k ] {\rTo} \\ \rTo \Lambda^{k-1}(V \tens \rho_k) \tens \epsilon_k \tens \Lambda^{q-k+1}(V \tens 1_k) \rTo 
\Lambda^q (V \tens R_k) \tens \epsilon_k \;.\end{multline*}
The first isomorphism is non canonical and depends on the choice of a basis of $\Lambda^{k-1} \rho_k \tens \epsilon_k$ and $1_{k}$. 
To simplify constants, it turns out that 
it is better to choose $\widehat{\omega}_{k-1}= \omega_{k-1}/(k-1)!$ and $\sigma_k$ as basis of $\Lambda^{k-1} \rho_k \tens \epsilon_k$ and $1_{k}$, respectively. 
The first arrow is then given by: 
$u_1 . \cdots . u_{k-1} \tens v_1 \w \cdots \w v_{q-k+1} \rMapsto ( u_1. \cdots . u_{k-1} \tens \widehat{\omega_{k-1}}) \tens (v_1 \w \cdots \w v_{q-k+1}  \tens \sigma_k)$, the last one is given by the alterating map, considered in remark \ref{rmka: symalt}. 
For the second one we have the lemma below. \end{remarka}
\begin{lemmaa}\label{lemma: inclusion}Let $k$ a positive integer and let $V$, $W$ be finite dimensional vector  spaces, with
  $k \leq \dim W$.
The natural inclusion $S^k V \tens \Lambda ^k W \rInto \Lambda^k(V
\tens W)$ is given by: 
$$ u_1 \cdots u_k \tens v_1 \w \cdots \w v_k \rMapsto \sum_{\tau \in \perm_k} (u_{\tau(1)} \tens v_1 ) \w \cdots \w
(u_{\tau(k)} \tens v_k) $$
\end{lemmaa}
\begin{proof}
Indeed \begin{align*}
u_1 \cdots u_k \tens v_1 \w \cdots \w v_k  = \: & 
\sum_{\sigma \in \perm_k } \sum_{\tau^\prime \in \perm_k}  (-1)^{\sigma} u_{\tau^\prime(1)} \tens
\cdots \tens u_{\tau^\prime(k)} \tens v_{\sigma(1)} \tens \cdots \tens
v_{\sigma(k)} \\
= \: &   \sum_{\sigma \in \perm_k } \sum_{\tau
  \in \perm_k} (-1)^{\sigma}
  u_{  \tau \sigma
  (1)} \tens 
\cdots \tens u_{ \tau \sigma
  (k)} \tens v_{\sigma(1)} \tens \cdots \tens
v_{\sigma(k)} \\
= \: & \sum_{\tau \in \perm_k} (u_{\tau (1)} \tens v_1)
\w 
\cdots \w (u_{\tau (k)} \tens v_k)
\end{align*}

\vspace{-0.3cm}
\end{proof}
\begin{remarka}\label{rmka: last}The last expression can be also rewritten as: 
$$  \sum_{\tau \in \perm_k} (u_{\tau(1)} \tens v_1)
\w 
\cdots \w (u_{\tau(k)} \tens v_k) = 
\sum_{\tau \in \perm_k} (-1)^{\tau} 
(u_{1} \tens v_{\tau(1)})
\w 
\cdots \w (u_{k} \tens v_{\tau(k)})\;.$$\end{remarka}
\begin{lemmaa}
The element $\omega_{k-1}$ identifies to the element 
$$\sum_{\tau \in \perm_k} (-1)^{\tau} \tau_*(e_1 \tens \cdots \tens
e_{k-1}) 
= \sum_{\tau \in \perm_k} (-1)^{\tau} e_{\tau(1)} \tens \cdots \tens
e_{\tau(k-1)} \;.$$
\end{lemmaa}
\begin{proof}The group $\perm_k$ is
generated by $\perm_{k-1}$ and the transpositions $(ik)$, $i \in \{1,
k-1\}$. Therefore the cosets $\perm_k/\perm_{k-1}$ are in bijection
with $\{ (ik), i \in \{1,  \dots, k\} \}$. 
In other words $\perm_k = \coprod_{i=1}^k (ik)\perm_{k-1}$. 
Therefore
\begin{align*}
 \sum_{\tau \in \perm_k} (-1)^{\tau} \tau_*(e_1 \tens \cdots \tens
e_{k-1}) 
= \:& \sum_{i=1}^k \sum_{\sigma \in \perm_{k-1}} (-1)^{(ik) \sigma}
e_{(ik)\sigma(1)} \tens \cdots \tens e_{(ik) \sigma(k-1)} \\
= \:& - \sum_{i=1}^k e_{(ik)(1)} \w \cdots \w e_{(ik)(k-1)} \\
= \:& - \sum_{i=1}^k e_1 \w \cdots e_{i-1} \w e_k \w e_{i+1} \w \cdots
\w e_{k-1} \\
= \:& - \sum_{i=1}^k (-1)^{k-i-1} \widehat{e}_i = \omega_{k-1} \;. 
\end{align*}

\vspace{-0.3cm}
\end{proof}
\begin{lemmaa}
The element
$$ \sum_{[\tau] \in \perm_k/ \perm_{k-1}} (-1)^{\tau} \tau_*[\omega_{k-2} \tens
\sigma_{k-1}] $$is well defined and coincides with $\omega_{k-1}$
\end{lemmaa}
\begin{proof}If we replace the representatives $\tau$ of the classes $[\tau] \in \perm_k / \perm_{k-1}$ with 
representatives $\tau \sigma$, where $\sigma \in \perm_{k-1}$, the sum doesn't change, due to the fact that 
$\omega_{k-2}$ is $\perm_{k-1}$-anti-invariant and $\sigma_k$ is $\perm_{k-1}$-invariant. Hence we can take 
as a set of representatives of $\perm_k/ \perm_{k-1}$ the set $\{ (ik), i \in \{1, \dots, k\} \}$. The sum in the statement can be
written as: \small
\begin{align*} - \sum_{i=1}^k  (ik)_*[ \omega_{k-2} \tens
  \sigma_{k-1}]
  = & - \sum_{i=1}^k \sum_{j=1}^{k-1} (ik)_*[ \omega_{k-2} \tens
  e_j] \\
= & - \sum_{i=1}^k  \sum_{j=1}^{k-1}  (ik)_* \sum_{\sigma \in
  \perm_{k-1}} (-1)^\sigma
e_{\sigma(1) } \tens \cdots \tens e_{\sigma(k-2)} \tens e_j \\
= & - \sum_{i=1}^k  \sum_{j=1}^{k-1} \sum_{\sigma \in
  \perm_{k-1}} (-1)^\sigma
e_{(ik)\sigma(1) } \tens \cdots \tens e_{(ik)\sigma(k-2)} \tens
e_{(ik)j} \\
= & - \sum_{i=1}^k   \sum_{\substack{1 \leq j \leq k-1 \\ j \in \{ \sigma(1),
    \dots, \sigma(k-2) \} }} \sum_{\sigma \in
  \perm_{k-1}}  (-1)^\sigma
e_{(ik)\sigma(1) } \tens \cdots \tens e_{(ik)\sigma(k-2)} \tens
e_{(ik)j} \\ & - \sum_{i=1}^k   \sum_{\substack{1 \leq j \leq k-1 \\ j \not \in \{ \sigma(1),
    \dots, \sigma(k-2) \} } }\sum_{\sigma \in
  \perm_{k-1}}  (-1)^\sigma
e_{(ik)\sigma(1) } \tens \cdots \tens e_{(ik)\sigma(k-2)} \tens
e_{(ik)j} \;.
\end{align*}\normalsize
The sum 
$$\sum_{\substack{1 \leq j \leq k-1 \\ j \in \{ \sigma(1),
    \dots, \sigma(k-2) \} }} \sum_{\sigma \in
  \perm_{k-1}}  (-1)^\sigma
e_{(ik)\sigma(1) } \tens \cdots \tens e_{(ik)\sigma(k-2)} \tens
e_{(ik)j}$$can be rewritten, for some $l \in \{1, \dots, k-2\}$, as
\begin{align*}
    \sum_{l=1}^{k-2} \sum_{\sigma \in
  \perm_{k-1}} (-1)^\sigma
e_{(ik)\sigma(1) } \tens \cdots \tens e_{(ik)\sigma(k-2)} \tens
e_{(ik)\sigma(l)} = &    \sum_{l=1}^{k-2} 
e_{(ik)1 } \w \cdots \w e_{(ik) (k-2)} \w
e_{(ik)l} =0 \;,
\end{align*}and hence  is zero. 
Consider the second sum. If $j \neq \sigma(i)$, for $i \in \{1, \dots,  k-2\}$,
then, necessarily, $j=\sigma(k-1)$. The second sum becomes: \small
\begin{align*}
- \sum_{i=1}^k  \sum_{\sigma \in
  \perm_{k-1}}  (-1)^\sigma
e_{(ik)\sigma(1) } \tens \cdots \tens e_{(ik)\sigma(k-2)} \tens
e_{(ik)\sigma(k-1)}  = &
- \sum_{i=1}^k 
e_{(ik)(1) } \w \cdots \w e_{(ik)(k-2)} \w
e_{(ik)(k-1)} \\ = &
 - \sum_{i=1}^k
e_{1 } \w \cdots \w e_{i-1} \w e_k \w e_{i+1} \w 
e_{k-1}  \\ 
= \: & -\sum_{i=1}^k  (-1)^{k-i-1}
e_{1 } \w \cdots \w e_{i-1} \w e_{i+1} \w 
e_{k-1} \w e_k \\ = & \sum_{i=1}^k  (-1)^{k-i}
\widehat{e}_i = \omega_{k-1}
\end{align*}\normalsize

\vspace{-0.3cm}
\end{proof}

\begin{notata}\label{notata: rki}
Recall that $R_k \simeq \mbb{C}^k$ is the natural representation of $\perm_k$. We will indicate
with $R_{k-1}(i)$, for $i \in \{1, \dots,  k\}$, the vector space: $R_{k-1}(i):=
R_k / \langle e_i \rangle$. It is isomorphic to the natural
representation of $\perm_{k-1}(i) := \perm(\{1, \dots, k \} \setminus \{ i \} ) \simeq \perm_{k-1}$.  
We will indicate with $\rho_{k-1}(i)$ the standard representation of
$\perm_{k-1}(i)$, embedded in $R_{k-1}(i)$, and with $\sigma_{k-1}(i)$
the invariant element $\sigma_{k-1}(i):= \sum_{1 \leq i \leq k,  i \neq
  j} e_i$. It generates the trivial representation $1_{k-1}(i)$ of $\perm_{k-1}(i)$,
seen embedded in $R_{k-1}(i)$. We indicate moreover with $\epsilon_{k-1}(i)$ the alternating representation of $\perm_{k-1}(i)$, that is, $\epsilon_{k-1}(i) = \Lambda^{k-2} \rho_{k-1}(i)$.  
\end{notata}
\begin{lemmaa}\label{lemma: mappainvariants} Let $V$ a vector space of dimension 2. Consider the map
  of $\perm_k$-representations
$$ \bigoplus_{i=1}^k f_i: \bigoplus_{i=1}^k \Lambda^{k-1}(V \tens
R_{k-1}(i)) \tens \epsilon_{k-1}(i) \rTo \Lambda^{k-1}(V \tens R_k)
\tens \epsilon_k \;.$$
The induced map of $\perm_k$-invariants can be identified\footnote{with the choices made in remark \ref{rmka: inclusion}} with the map $(k-1) \sym$ where 
$\sym:  S^{k-2}V \tens V \rTo S^{k-1}V $ is the symmetrization map.
We have: $(k-1) \sym ( u_1 \cdots u_{k-2} \tens v ) = u_1 \cdots u_{k-2}  v$.  
\end{lemmaa}\begin{proof}By the preceding lemmas, the map of invariants can be identified with
a map $S^{k-2}V \tens V \rTo S^{k-1}V$. Let us now determine this map
precisely. By remark \ref{rmka: inclusion} the map of $\perm_k$ invariants $(\oplus_{i=1}^k f_i)^{\perm_k}$
can be identified with the map of $\perm_k$-invariants of the map 
$$ \bigoplus_{i=1}^k g_i : [(S^{k-2}V \tens \Lambda^{k-2} \rho_{k-1}(i))
\tens \epsilon_{k-1}(i)] \tens [V \tens 1_{k-1}(i)] \rTo \Lambda^{k-1}(V \tens R_k)
\tens \epsilon_k \;,$$where $g_i$ is the $\perm_{k-1}$-equivariant composition: 
$$ g_i : [(S^{k-2}V \tens \Lambda^{k-2} \rho_{k-1}(i))
\tens \epsilon_{k-1}(i)] \tens [V \tens 1_{k-1}(i)] \rTo  \Lambda^{k-1}(V
\tens R_{k-1}(i))  \tens \epsilon_{k-1}(i) \rTo^{f_i} \Lambda^{k-1}(V
\tens R_k) 
\tens \epsilon_k $$
of $f_i$ with the natural inclusion 
\begin{equation}\label{eq: j1}[(S^{k-2}V \tens \Lambda^{k-2} \rho_{k-1}(i))
\tens \epsilon_{k-1}(i)] \tens [V \tens 1_{k-1}(i)] \rTo \Lambda^{k-1}(V
\tens R_{k-1}(i)) \tens \epsilon_k \;.\end{equation}Now the inclusion (\ref{eq: j1}) factorizes through:
\begin{diagram}
[(S^{k-2}V \tens \Lambda^{k-2} \rho_{k-1}(i))
\tens \epsilon_{k-1}(i)] \tens [V \tens 1_{k-1}(i)] & \rTo & \Lambda^{k-1}(V
\tens R_{k-1}(i))  \tens \epsilon_{k-1}(i) \\
\dTo^{i_i \tens \id_{V
  \tens 1_{k-1}}} & \ruTo^{\theta_i} & \\
\Lambda^{k-2}(V \tens R_{k-1}(i)) \tens \epsilon_{k-1}(i) \tens [V \tens
1_{k-1}(i)] & & 
\end{diagram}
where $i_i$ is the natural inclusion: 
$$ i_i : S^{k-2}V \tens \Lambda^{k-2} \rho_{k-1}(i) \rInto
 S^{k-2}V \tens \Lambda^{k-2} R_{k-1}(i) \rInto
\Lambda^{k-2}(V \tens R_{k-1}(i)) \;.$$
On the other hand, the map $\theta_i$ is just the
anti-symetrization. Hence the map $g_i$ is the composition $g_i = f_i \circ \theta_i \circ
(i_i \tens \id_{V
  \tens 1_{k-1}}) = \eta_i \circ (i_i \tens \id_{V
  \tens 1_{k-1}})$, where the map $\eta_i$ is defined as $\eta_i := f_i \circ \theta_i$. By Danila's lemma for morphisms, 
in the identification $S^{k-2}V \tens V \simeq [ \oplus_{i=1}^k \Lambda^{k-1}(V \tens
R_{k-1}(i)) \tens \epsilon_{k-1}(i)]^{\perm_k}$ explained in remark \ref{rmka: inclusion}, 
the map of invariants $[\oplus_{i=1}^k g_i]^{\perm_k}$ is given by\begin{align*} [\oplus_{i=1}^k g_i]^{\perm_k}(u_1 \cdots u_{k-2} \tens v) = & 
\sum_{[\nu] \in \perm_k / \perm_{k-1} } (-1)^{\nu} \nu_* g_k (
u_1 \cdots u_{k-2} \tens  \widehat{\omega_{k-2}}(k) \tens v \tens
\sigma_{k-1}(k)) \\ 
= & \sum_{[\nu] \in \perm_k / \perm_{k-1} } (-1)^{\nu} \nu_*  \eta_k [i_k (u_1 \cdots u_{k-2} \tens
\widehat{\omega_{k-2}}(k)) \tens v \tens \sigma_{k-1}(k) ]
\end{align*}Consider now, as set of representatives  of $\perm_{k-1}/\perm_{k-2}$ the cycles $\{ \gamma_i, i=1, \dots, k-1 \}$, where $\gamma_i = (i, i+1, \dots, k-2, k-1)$. 
The element $\widehat{\omega_{k-2}}(k)$
is nothing but the sum
$$\widehat{\omega_{k-2}}(k) = \frac{1 }{(k-2)!} \sum_{i=1}^{k-1} (-1)^{\gamma_i} e_{\gamma_i(1)} \w
\cdots \w e_{\gamma_i(k-2)} \;.$$Hence by lemma \ref{lemma: inclusion} and by remark \ref{rmka: last}
\begin{align*}
i_k (u_1 \cdots u_{k-2} \tens
\widehat{\omega_{k-2}}(k))= \: & \frac{1}{(k-2)!}\sum_{i=1}^{k-1} (-1)^{\gamma_i} i_k (u_1 \cdots u_{k-2} \tens
e_{\gamma_i(1)} \w
\cdots \w e_{\gamma_i(k-2)})\\
= \: &\frac{1}{(k-2)!}\sum_{i=1}^{k-1} (-1)^{\gamma_i} \sum_{\theta \in \perm_{k-2}}
(-1)^{\theta} (u_1 \tens e_{\gamma_i \theta (1) }) \w \cdots \w (u_{k-2}
\tens e_{\gamma_i \theta (k-2)})  \\
= &\frac{1}{(k-2)!} \sum_{\tau \in \perm_{k-1}} (-1)^{\tau} 
(u_1 \tens e_{\tau (1) }) \w \cdots \w (u_{k-2}
\tens e_{\tau (k-2)})  \\
\end{align*}Hence the term  $\eta_k [i_k (u_1 \cdots u_{k-2} \tens
\widehat{\omega_{k-2}}(k)) \tens (v \tens \sigma_{k-1}(k) )]$ can be rewritten as 
\begin{align*}
  & \frac{1}{(k-1)!} \sum_{\tau \in \perm_{k-1}} (-1)^{\tau} 
(u_1 \tens e_{\tau (1) }) \w \cdots \w (u_{k-2}
\tens e_{\tau (k-2)}) \w (v \tens \sigma_{k-1}(k)) \\
= \: & \frac{1}{(k-1)!} \sum_{\tau \in \perm_{k-1}} (-1)^{\tau} 
\sum_{\sigma \in \perm_{k-1}} (-1)^\sigma \sigma_* [ (u_1 \tens e_{\tau(1)}) \tens \cdots \tens (u_{k-2} \tens e_{\tau(k-2)}) 
 \tens (v \tens \sigma_{k-1}(k))]  \\
= \: & \frac{1}{(k-1)!} \sum_{\tau \in \perm_{k-1}} (-1)^{\tau} 
\sum_{\sigma \in \perm_{k-1}} (-1)^\sigma \sigma_* (u_1 \tens \cdots \tens u_{k-2} \tens
v) \tens \sigma_* (e_{\tau (1) } \tens \cdots 
\tens e_{\tau (k-2)} \tens \sigma_{k-1}(k)) \\
= \: & \frac{1}{(k-1)!} \sum_{\sigma \in \perm_{k-1}} (-1)^\sigma \sigma_* (u_1 \tens \cdots \tens u_{k-2} \tens
v) \tens \sigma_* \left ( \sum_{\tau \in \perm_{k-1}} (-1)^{\tau}  e_{\tau (1) } \tens \cdots 
\tens e_{\tau (k-2)} \tens \sigma_{k-1}(k) \right ) \\
 = \: &  \frac{1}{(k-1)!} \sum_{\sigma \in \perm_{k-1}} (-1)^\sigma \sigma_* (u_1 \tens \cdots \tens u_{k-2} \tens
v) \tens \sigma_* \left ( \omega_{k-2}(k) \tens \sigma_{k-1}(k) \right) \\
=\: & \frac{1}{(k-1)} \sum_{\sigma \in \perm_{k-1}} (-1)^\sigma \sigma_* (u_1 \tens \cdots \tens u_{k-2} \tens
v) \tens \sigma_* \left ( \widehat{\omega_{k-2}}(k) \tens \sigma_{k-1}(k) \right) \\
= \: & \frac{1}{(k-1)} \sum_{\sigma \in \perm_{k-1}}  \sigma_* (u_1 \tens \cdots \tens u_{k-2} \tens
v) \tens  \left ( \widehat{\omega_{k-2}}(k) \tens \sigma_{k-1}(k) \right ) \\
= \: &  \frac{1}{(k-1)} u_1 \cdots u_{k-2} v \tens ( \widehat{\omega_{k-2}}(k) \tens \sigma_{k-1}(k) )
\end{align*}
Finally, $[\oplus_{i=1}^k g_i]^{\perm_k}(u_1 \cdots u_{k-2} \tens v)$ can computed as: 
\begin{align*}
 [\oplus_{i=1}^k g_i]^{\perm_k}(u_1 \cdots u_{k-2} \tens v) = \:  & \frac{1}{k-1} \sum_{[\nu] \in \perm_k / \perm_{k-1} }  (-1)^{\nu} \nu_* [  u_1 \cdots u_{k-2} v
  \tens \left ( \widehat{\omega_{k-2}}(k) \tens \sigma_{k-1}(k) \right) ] \\
  \:  & \frac{1}{(k-1)!} \sum_{[\nu] \in \perm_k / \perm_{k-1} }  (-1)^{\nu} \nu_* [  u_1 \cdots u_{k-2} v
  \tens \left ( \omega_{k-2}(k) \tens \sigma_{k-1}(k) \right) ] \\
= \: &
  \frac{1}{(k-1)!}u_1 \cdots u_{k-2} v
  \tens  \sum_{[\nu] \in \perm_k / \perm_{k-1} }  (-1)^{\nu}  \nu_* \left ( \omega_{k-2}(k) \tens \sigma_{k-1}(k) \right) ] \\
  = \: &  \frac{1}{(k-1)!} u_1 \cdots u_{k-2} v \tens \omega_{k-1} = u_1 \cdots u_{k-2} v \tens \widehat{\omega_{k-1}} 
\end{align*}which is $u_1 \cdots u_{k-2} v$ in the identification 
$[\Lambda^{k-1}(V \tens R_k) \tens \epsilon_k]^{\perm_k} \simeq
S^{k-1}V \tens [\Lambda^{k-1} \rho_{k} \tens \epsilon_k]^{\perm_k} 
\simeq S^{k-1}V $ determined by choosing as a basis of $\Lambda^{k-1} \rho_{k}$ the vector $\widehat{\omega_{k-1}}$. 
\end{proof}

\vspace{1cm}
\noindent
{Departamento de  Matem\'atica, Puc-Rio, Rua Marqu\^es S\~ao Vicente 225, 22451-900 G\'avea, Rio de Janeiro, RJ, Brazil}

\noindent 
{\it Email address:} {\tt lucascala@mat.puc-rio.br}

\footnotesize

\begin{thebibliography}{{Gra}06}

\bibitem[Abe10]{Abe2010}
Takeshi Abe.
\newblock Deformation of rank 2 quasi-bundles and some strange dualities for
  rational surfaces.
\newblock {\em Duke Math. J.}, 155(3):577--620, 2010.

\bibitem[Bea83]{Beauville1983}
Arnaud Beauville.
\newblock {V}ari\'et\'es {K}\"ahleriennes dont la premi\`ere classe de {C}hern
  est nulle.
\newblock {\em J. Differential Geom.}, (18):755--782, 1983.

\bibitem[BKR01]{BridgelandKingReid2001}
Tom Bridgeland, Alastair King, and Miles Reid.
\newblock The {M}c{K}ay correspondence as an equivalence of derived categories.
\newblock {\em J. Amer. Math. Soc.}, 14(3):535--554 (electronic), 2001.

\bibitem[BS91]{BeltramettiSommese1991}
M~Beltrametti and Andrew~J Sommese.
\newblock Zero cycles and k-th order embeddings of smooth projective surfaces.
\newblock In {\em Problems in the theory of surfaces and their classification,
  Symposia Math}, volume~32, pages 33--48, 1991.

\bibitem[CG90]{CataneseGoettsche1990}
Fabrizio Catanese and Lothar G{\oe}ttsche.
\newblock d-very-ample line bundles and embeddings of {H}ilbert schemes of
  0-cycles.
\newblock {\em Manuscripta Mathematica}, 68(1):337--341, 1990.

\bibitem[Dan99]{Danilathese}
Gentiana Danila.
\newblock {\em Formule de Verlinde et dualit\'e \'etrange sur le plan
  projectif}.
\newblock PhD thesis, Universit\'e Paris 7, 1999.

\bibitem[Dan00]{Danila2000}
Gentiana Danila.
\newblock Sections du fibr\'e d\'eterminant sur l'espace de modules des
  faisceaux semi-stables de rang 2 sur le plan projectif.
\newblock {\em Ann. Inst. Fourier (Grenoble)}, 50(5):1323--1374, 2000.

\bibitem[Dan01]{Danila2001}
Gentiana Danila.
\newblock Sur la cohomologie d'un fibr\'e tautologique sur le sch\'ema de
  {H}ilbert d'une surface.
\newblock {\em J. Algebraic Geom.}, 10(2):247--280, 2001.

\bibitem[Dan02]{Danila2002}
Gentiana Danila.
\newblock R\'esultats sur la conjecture de dualit\'e \'etrange sur le plan
  projectif.
\newblock {\em Bull. Soc. Math. France}, 130(1):1--33, 2002.

\bibitem[Dan04]{Danila2004}
Gentiana Danila.
\newblock Sur la cohomologie de la puissance sym\'etrique du fibr\'e
  tautologique sur le sch\'ema de {H}ilbert ponctuel d'une surface.
\newblock {\em J. Algebraic Geom.}, 13(1):81--113, 2004.

\bibitem[Dan07]{Danila2007}
Gentiana Danila.
\newblock Sections de la puissance tensorielle du fibr{\'e} tautologique sur le
  sch{\'e}ma de {H}ilbert des points d'une surface.
\newblock {\em Bull. Lond. Math. Soc.}, 39(2):311--316, 2007.

\bibitem[DN89]{drezetNarasimhan1989}
J.-M. Drezet and M.~S. Narasimhan.
\newblock Groupe de {P}icard des vari\'et\'es de modules de fibr\'es
  semi-stables sur les courbes alg\'ebriques.
\newblock {\em Invent. Math.}, 97(1):53--94, 1989.

\bibitem[EGL01]{EllingsrudGoettscheLehn2001}
Geir Ellingsrud, Lothar G{\"o}ttsche, and Manfred Lehn.
\newblock On the cobordism class of the {H}ilbert scheme of a surface.
\newblock {\em J. Algebraic Geom.}, 10(1):81--100, 2001.

\bibitem[FH91]{FultonHarrisRT}
William Fulton and Joe Harris.
\newblock {\em Representation theory}, volume 129 of {\em Graduate Texts in
  Mathematics}.
\newblock Springer-Verlag, New York, 1991.
\newblock A first course, Readings in Mathematics.

\bibitem[Fuj87]{Fujiki1987}
Akira Fujiki.
\newblock On the de {R}ham cohomology group of a compact {K}\"ahler symplectic
  manifold.
\newblock In {\em Algebraic geometry, {S}endai, 1985}, volume~10 of {\em Adv.
  Stud. Pure Math.}, pages 105--165. North-Holland, Amsterdam, 1987.

\bibitem[Fuj09]{Fujino2009}
Osamu Fujino.
\newblock On injectivity, vanishing and torsion-free theorems for algebraic
  varieties.
\newblock {\em Proc. Japan Acad. Ser. A Math. Sci.}, 85(8):95--100, 2009.

\bibitem[GNY09]{GNY2009}
Lothar G{\"o}ttsche, Hiraku Nakajima, and K{\=o}ta Yoshioka.
\newblock {$K$}-theoretic {D}onaldson invariants via instanton counting.
\newblock {\em Pure Appl. Math. Q.}, 5(3, Special Issue: In honor of Friedrich
  Hirzebruch. Part 2):1029--1111, 2009.

\bibitem[G{\"o}t98]{Goettsche1998}
Lothar G{\"o}ttsche.
\newblock A conjectural generating function for numbers of curves on surfaces.
\newblock {\em Comm. Math. Phys.}, 196(3):523--533, 1998.

\bibitem[{Gra}06]{Gray2006}
Robert~M. {Gray}.
\newblock {\em {{T}oeplitz and {C}irculant {M}atrices: {A} {R}eview.}}
\newblock Boston, MA., {F}oundations and {T}rends in {C}ommunications and
  {I}nformation {T}heory {V}ol. 2, no. 3 (2005) edition, 2006.

\bibitem[Gre94]{Green1994}
Mark~L. Green.
\newblock Infinitesimal methods in {H}odge theory.
\newblock In {\em Algebraic cycles and {H}odge theory ({T}orino, 1993)}, volume
  1594 of {\em Lecture Notes in Math.}, pages 1--92. Springer, Berlin, 1994.

\bibitem[Gro67]{EGAIV.4}
Alexander Grothendieck.
\newblock \'{E}l\'ements de g\'eom\'etrie alg\'ebrique. {IV}. \'{E}tude locale
  des sch\'emas et des morphismes de sch\'emas, {Q}uatri\`eme partie.
\newblock {\em Inst. Hautes \'{E}tudes Sci. Publ. Math.}, (32), 1967.

\bibitem[GS14]{GoettscheShende2014}
Lothar G{\"o}ttsche and Vivek Shende.
\newblock Refined curve counting on complex surfaces.
\newblock {\em Geom. Topol.}, 18(4):2245--2307, 2014.

\bibitem[Hai01]{Haiman2001}
Mark Haiman.
\newblock Hilbert schemes, polygraphs and the {M}acdonald positivity
  conjecture.
\newblock {\em J. Amer. Math. Soc.}, 14(4):941--1006 (electronic), 2001.

\bibitem[He98]{He1998}
Min He.
\newblock Espaces de modules de syst\`emes coh\'erents.
\newblock {\em Internat. J. Math.}, 9(5):545--598, 1998.

\bibitem[Kru14]{Krug2014}
Andreas Krug.
\newblock Tensor products of tautological bundles under the
  {B}ridgeland-{K}ing-{R}eid-{H}aiman equivalence.
\newblock {\em Geom. Dedicata}, 172:245--291, 2014.

\bibitem[KST11]{KoolShendeThomas2011}
Martijn Kool, Vivek Shende, and Richard~P. Thomas.
\newblock A short proof of the {G}\"ottsche conjecture.
\newblock {\em Geom. Topol.}, 15(1):397--406, 2011.

\bibitem[Leh99]{Lehn1999}
Manfred Lehn.
\newblock Chern classes of tautological sheaves on {H}ilbert schemes of points
  on surfaces.
\newblock {\em Invent. Math.}, 136(1):157--207, 1999.

\bibitem[LP96]{LePotier1996}
Joseph Le~Potier.
\newblock Syst\`emes coh\'erents et polyn\^omes de {D}onaldson.
\newblock In {\em Moduli of vector bundles (Sanda, 1994; Kyoto, 1994)}, volume
  179 of {\em Lecture Notes in Pure and Appl. Math.}, pages 103--128. Dekker,
  New York, 1996.

\bibitem[Sca05]{ScalaPhD}
Luca Scala.
\newblock {\em Cohomology of the {H}ilbert scheme of points on a surface with
  values in representations of tautological bundles. {P}erturbations of the
  metric in {S}eiberg-{W}itten equations}.
\newblock PhD thesis, Universit\'e Paris 7, 2005.

\bibitem[Sca06]{Scala2006}
Luca Scala.
\newblock Cohomology of the {H}ilbert scheme of points on a surface with values
  in the double tensor power of a tautological bundle.
\newblock {\em C. R. Math. Acad. Sci. Paris}, 343(11-12):741--746, 2006.

\bibitem[Sca07]{ScalaTrento2006}
Luca Scala.
\newblock Mc{K}ay correspondence, {H}ilbert schemes and {S}trange {D}uality.
\newblock In {\em {V}ector {B}undles and {L}ow {C}odimensional {S}ubvarieties:
  {S}tate of the {A}rt and {R}ecent {D}evelopments ({G}. {C}asnati, {F}.
  {C}atanese and {R}. {N}otari eds.)}, volume~21 of {\em Quad. Mat.}, pages
  355--373. Dept. Math., Seconda Univ. Napoli, Caserta, 2007.

\bibitem[Sca09a]{Scala2009D}
Luca Scala.
\newblock {C}ohomology of the {H}ilbert scheme of points on a surface with
  values in representations of tautological bundles.
\newblock {\em Duke Math. J.}, 150(2):211--267, 2009.

\bibitem[Sca09b]{Scala2009GD}
Luca Scala.
\newblock Some remarks on tautological sheaves on {H}ilbert schemes of points
  on a surface.
\newblock {\em Geom. Dedicata}, 139(1):313--329, 2009.

\bibitem[Sca10]{Scala2010}
Luca Scala.
\newblock {\em Research Statement}.
\newblock October 2010.

\bibitem[Sca15a]{Scaladiagarxiv2015}
Luca Scala.
\newblock {N}otes on diagonals of the product and symmetric variety of a
  surface.
\newblock {\em {\rm arxiv: 1510.04889}}, 2015.

\bibitem[Sca15b]{Scala2015isospectral}
Luca Scala.
\newblock Singularities of the {I}sospectral {H}ilbert {S}cheme.
\newblock {\em {\rm arXiv: 1510.03071}}, 2015.

\bibitem[Yua12a]{Yao2012A}
Yao Yuan.
\newblock Determinant line bundles on moduli spaces of pure sheaves on rational
  surfaces and strange duality.
\newblock {\em Asian J. Math.}, 16(3):451--478, 2012.

\bibitem[Yua12b]{Yao2012}
Yao Yuan.
\newblock Estimates of sections of determinant line bundles on moduli spaces of
  pure sheaves on algebraic surfaces.
\newblock {\em Manuscripta Math.}, 137(1-2):57--79, 2012.

\end{thebibliography}
\end{document}